\documentclass[11pt]{amsart}
\usepackage[utf8]{inputenc}
\usepackage[english]{babel}
\usepackage[T1]{fontenc}
\usepackage{indentfirst,microtype}
\usepackage[margin=2.3cm]{geometry}
\usepackage{amsfonts,amsthm,amsmath,amssymb,amsrefs,amscd,cite}
\usepackage{enumerate}
\usepackage[colorlinks=true,linkcolor=blue,citecolor=magenta]{hyperref}
\usepackage{eurosym}
\usepackage{xcolor}
\usepackage[T1]{fontenc} \usepackage{lmodern}
\usepackage{verbatim}
\usepackage{amsmath}
\usepackage{tikz}
\usepackage{tikz-cd} 
\usepackage{graphicx}
\newcommand{\Zo}{\mathcal{Z}}

\newcommand{\Norm}{\mathcal{N}}
\newcommand{\Pro}{\mathrm{P}}
\newcommand{\ud}{\mathrm{d}}
\newcommand{\Diff}{\mathrm{Diff}}

\newcommand{\R}{\mathbb{R}}

\newcommand{\mV}{\mathcal{V}}
\newcommand{\Sing}{\mathrm{Sing}}

\newcommand{\Ker}{\mathrm{Ker}}
\newcommand{\mesh}{\mathrm{mesh}}

\def\va{ \varepsilon}
\newtheorem{thm}{Theorem}[section]

\newtheorem*{thm*}{Theorem.}

\newtheorem*{thma}{Theorem A}

\newtheorem*{thmb}{Theorem B}
\newtheorem*{thmc}{Theorem C}
\newtheorem*{thmd}{Theorem D}
\newtheorem*{thme}{Theorem E}
\newtheorem{proposition}{Proposition}[subsection]

\newtheorem{cor}{Corollary}[subsection]
\newtheorem{lemma}{Lemma}[subsection]

\newtheorem{definition}{Definition}[subsection]

\newtheorem{remark}{Remark}[subsection]

\newcommand{\RL}[1]{Z_{#1}}

\newcommand{\Ra}{\mathcal{R}}
\newcommand{\V}{\mathcal{V}}

\newcommand{\RLp}[1]{Z^{(#1)}}
\newcommand{\kei}[1]{{k(I)}}

\newcommand{\norm}[1]{\Vert #1 \Vert}

\newcommand{\D}{\mathrm{D}}

\newcommand{\vect}[1]{#1}

\newcommand{\be}{\begin{equation}}
\newcommand{\ee}{\end{equation}}
\newcommand{\bes}{\begin{equation*}}
\newcommand{\ees}{\end{equation*}}
\newcommand{\st}{\,  | \, \, }

\title{A priori bounds for GIETs, affine shadows  and rigidity of foliations  in genus two}

\author{Selim Ghazouani}
\author{Corinna Ulcigrai}

\begin{document}

\begin{abstract}
We prove a rigidity result for foliations on surfaces of genus two, which {can be seen} as a generalization to higher genus of Herman's theorem on circle diffeomorphisms and, correspondingly, flows on the torus. We prove in particular that, if a smooth, orientable foliation with non-degenerate {(Morse)} singularities on a closed surface of genus two is minimal, then, under a full measure condition for the rotation number, it is \emph{differentiably} conjugate to a \emph{linear} foliation. 

The corresponding result at the level of Poincar{\'e} sections is that, for a full measure set of (standard) interval exchange transformations (IETs for short) with $d=4$ or $d=5$ continuity intervals and irreducible combinatorics, any generalized interval exchange transformation (GIET for short) which is topologically conjugate to a standard IET from this set and satisfies an obstruction expressed in terms of boundary operator (which is automatically satisfied when the GIET arises as a Poincar{\'e} map of a smooth foliation) is $\mathcal{C}^1$-conjugate to it. This in particular settles a conjecture by Marmi, Moussa and Yoccoz in genus two. Our results also show that this conjecture on the rigidity of GIETs can be reduced to the study of affine IETs, or more precisely of {Birkhoff sums of piecewise constant observables over standard IETs}, in genus $g \geq 3$. 
 
Our approach is via renormalization, namely we exploit a suitable  
 acceleration of the Rauzy-Veech induction ({\color{black}an acceleration which makes  Oseledets generic \emph{effective}}) on the space of GIETs. For infinitely renormalizable, irrational GIETs of \emph{any} number of intervals $d\geq 2$ we prove a dynamical dichotomy on the behaviour of the orbits under renormalization, by proving that either an orbit is \emph{recurrent} to certain bounded sets in the space of GIETs, or it diverges and it is approximated (up to lower order terms) by the orbit of an affine IET (a case that we refer to as \emph{affine shadowing}). This result can in particular  be used, in conjunction with previous work by Marmi-Moussa and Yoccoz on the existence of wandering intervals for affine IETs,  to prove, \emph{a priori bounds} in genus two and  is  therefore at the base of the rigidity result.
\end{abstract}

\maketitle 

\tableofcontents

\section{Introduction and main results}
In this article we extend some aspects of the theory of \emph{circle diffeomorphisms} and of \emph{flows on the torus} to the context of \textit{generalized interval exchange transformations} and of \textit{foliations on higher genus surfaces}. Exploiting a renormalization approach, we are able in particular to prove a \emph{rigidity} result in genus two which can be seen as a generalization of { a celebrated theorem by Herman} in genus one and proves a conjecture  by Marmi, Moussa and Yoccoz in \cite{MMY3} in the case of generalized interval exchange transformations which correspond to minimal surface flows in genus two. We start by giving an introduction to {geometric} rigidity problems in dynamics and some key results on circle diffeomorphisms and (generalized) interval exchange transformations.

\subsection{{Geometric rigidity} in dynamics.}
A natural problem in the theory of smooth dynamical systems is to establish which classes of dynamical systems are geometrically rigid in the following sense. 
We say that a class of dynamical systems (whether it be an endomorphism of a manifold, a foliation or a flow) is  \textit{geometrically rigid} (or just \textit{rigid}) if a topological conjugacy (namely a homeomorphism which intertwines the dynamics on the two systems, see below for the definition) between two elements in this class is necessarily differentiable. 
 A natural problem in the theory of smooth dynamical systems is to establish which classes of dynamical systems are \textit{geometrically rigid}.  

\smallskip
It is well-known that periodic orbits provide obstructions to geometric rigidity for \emph{hyperbolic dynamical systems}: for a diffeomorphism, the product of the derivatives along a period orbit is a $\mathcal{C}^1$-conjugacy invariant. In particular, Anosov diffeomorphisms or flows can easily be deformed to modify this conjugacy invariant, without changing the topological structure (by structural stability). 
{\color{black}The absence of periodic orbits is on the other hand possible (and actually prevalent) in \emph{entropy zero dynamics}, which is therefore a natural setting to investigate geometric rigidity.} 

\smallskip
A fundamental class of (entropy zero)  systems in which geometric rigidity has been shown to be prevalent are (minimal) \emph{circle diffeomorphisms}, a class of dynamical systems that have played a central role in the development of the theory since the work of Poincaré onwards.   
In addition to asking that there are no-periodic orbits (which in the setting of circle diffeomorphisms is equivalent to the assumption that the rotation number is \emph{irrational}),  to prove that a diffeomorphism $T:M\to M$  is rigid one often needs to impose a  \textit{quantitative} version of the absence of periodic orbits, for example asking that there exists $c>0$ and $\alpha>0$ such that
\begin{equation}\label{DC} d(x, T^n(x) ) \geq \frac{c}{n^{\alpha}}, \qquad \mathrm{for\ all}\ n\in\mathbb{N}, \ x\in M
\end{equation} 
(where $d$ is a distance function on $M$). When $M$ is the \emph{circle} $\mathbb{S}^1=\mathbb{R}/\mathbb{Z}$, it is well known that a diffeomorphism $T: \mathbb{S}^1\to \mathbb{S}^1$ satisfies \eqref{DC} for a \emph{full measure} set  of \emph{rotation numbers} (actually with $\alpha>1$). In this setting, indeed, \eqref{DC} is equivalent to assuming that the rotation number $\rho=\rho(T)$ of $T$ is \emph{Diophantine}\footnote{We  recall that one says that $\rho\in\mathbb{R}$ is \emph{Diophantine} with Diophantine exponent $\tau>0$ iff there exists $c>0$ such that  for all $p,q\in\mathbb{Z}$, $q\neq 0$, $\left|\rho- \frac{p}{q}\right|\geq \frac{c}{q^{2+\tau}} $. Since rotations are homogeneous, multiplying by $q$ one gets \eqref{DC} with $\alpha=1+\tau$.} or, as equivalent terminology, satisfies a \emph{Diophantine Condition}. 
(For this reason, an assumption like \eqref{DC} is sometimes called a \emph{Diophantine-type} condition). 

\smallskip
A celebrated result by Michael Herman \cite{Herman}, combined with later work by Jean-Christophe Yoccoz \cite{Yoccoz}, then shows that circle diffeomorphisms which satisfy a Diophantine Condition are \emph{geometrically rigid}.


 


\smallskip
\noindent Let us briefly summarize some of the works and settings in which geometric rigidity has been verified, {that are perhaps surprinsigly few:}
\begin{enumerate}
\item {\color{black}In the above mentioned setting of circle diffeomorphisms}, the first (local) result in this direction of rigidity was obtained by Arnol'd in \cite{Arnold} by applying methods from KAM theory. {\color{black}The} global theory was brought about by the work of Herman \cite{Herman} and completed by Yoccoz \cite{Yoccoz}. It was later revisited in terms of renormalization theory, see \cite{KS:Her, KT:Her}. 

\smallskip
\item If one allows to replace the ambient Riemannian manifold by a minimal invariant closed set, certain smooth \textit{unimodal maps} of the interval $[0,1]$ are known to be geometrically rigid. These were first numerically discovered by physicists Feigenbaum \cite{Feigenbaum} and Coullet-Tresser \cite{CoulletTresser} in the late 1970s. A deep and rigorous theory,  nowadays sometimes referred to as \emph{Sullivan-McMullen-Lyubich theory} was established only later, in the Nineties, through the introduction of complex methods into the picture, see for instance Sullivan \cite{Sullivan}, McMullen \cite{McMullen1,McMullen2}, Lyubich \cite{Lyubich} and Avila-Lyubich \cite{AvilaLyubich}.
\smallskip
\item In \emph{one-dimensional dynamics}, several other classes
of rigid dynamical systems were discovered and much studied, such as circle maps with a \emph{critical point} (see e.g. the works by de Faria and de Melo \cite{deFariadeMelo, deFariadeMelo2} {\color{black}or by Yampolsky \cite{Ya:att,Ya:hyp}}), circle homeomorphisms which are differentiable away from a point, {known as circle maps with \emph{breaks} (see  for example the works} \cite{KhaninKhmelev, KhaninTeplinsky, KhaninKocicMazzeo} by Khanin, Khmelev, Teplinsky, Kocic, Mazzeo, {\color{black}or Cunha and Smania \cite{CS:ren, CS:rig} for more breaks}) or Lorenz maps (see Martens and Winckler, \cite{Winckler, MartensWinckler}).

\smallskip
\item KAM theory establishes the \textit{local} rigidity of Diophantine translations on the torus $\mathbb{T}^d = \mathbb{R}^d / \mathbb{Z}^d$ (see for instance \cite{Karaliolios} and references therein). However, no global rigidity result is known in this context when $d \geq 2$. 

\end{enumerate} {
\noindent In the setting $(4)$, geometric rigidity was proposed as a conjecture by Raphaël Krikorian (see \cite{Kh:sur}), who asked,  when $M=\mathbb{T}^d$, with $d \geq 2$, is a higher dimensional torus, whether given any $\mathcal{C}^\infty$ diffeomorphism $T:M\to M$ which  is  topologically conjugate to a translation of $\mathbb{T}^d$ and whose rotation vector satisfies a Diophantine Condition,\footnote{{\color{black}A translation of $\mathbb{T}^d=\mathbb{R}^d/\mathbb{Z}^d$ with rotation vector $\rho=(\rho_1,\dots,\rho_d)$ is the map which sends $x=(x_1,\dots,x_d) \in \mathbb{T}^d$ to $x+\rho=(x_1+\rho_1,\dots, x_d+\rho_d)\in \mathbb{T}^d$. We say that the vector $\rho$ satisfies a Diophantine condition with exponent $\tau>0$ if there exists $c>0$ such that for every non-zero integer vector $k=(k_1,\dots, k_d)\in\mathbb{Z}^d$, $|\langle k , \rho \rangle|=|k_1 \rho_1 + \dots k_d \rho_d|\geq c/||k||^\tau$.}} then the conjugacy is $\mathcal{C}^1$ and actually $\mathcal{C}^\infty$. A bold generalization of this conjecture was suggested by Konstantin Khanin in his ICM address \cite{Kh:sur}, namely that any minimal $T \in \mathrm{Diff}^{\infty}(M)$ where  $M$ is a smooth, closed Riemannian manifold which satisfies a Diophantine-type condition as in \eqref{DC} is  geometrically rigid. Little evidence  is available towards this conjecture and further obstructions other than periodic orbits (related for example to the presence of invariant distributions) may play a role in this greater generality.}


\medskip
\subsection{Geometric rigidity in genus two}
In this article we provide a new class of \emph{geometrically rigid} dynamical systems, by  proving a \emph{global} rigidity theorem for foliations on surfaces of genus $2$ (Theorem A here below), or, more in general, for a broader class of \textit{interval exchange maps} (see Theorem B and the remarks afterwards). The result for foliations is the following. We explain the meaning of some key words just below (and refer the reader to \S~\ref{foliations} for precise definitions  of notions involved in the statement).

{
\begin{thma}
\label{thmA}
Let $S$ be a closed orientable surface of genus $2$ and $\mathcal{F}$ a  {\color{black}\emph{minimal}} {orientable} foliation on $S$ of class $\mathcal{C}^3$, with non-degenerate (Morse type) singularities. 
Under a full-measure\footnote{{The space of topological conjugacy classes of such foliations can be parametrised by finitely many parameters, and here \emph{full-measure} means 'for almost every parameter' with respect to the Lebesgue measure. This notion of full measure on minimal foliations is also related to the \emph{Katok fundamental class}, see \S~\ref{sec:foliationmeasure} for details. Furthermore, it corresponds to  a full measure set of (combinatorial) rotation numbers (see Definition~\ref{def:rotnumber}) in the sense of Definition~\ref{def:fullmeasure}.}} Diophantine-type condition, $\mathcal{F}$ is geometrically rigid.
\end{thma} }

\noindent This result can be seen as a generalisation of Herman's global rigidity theorem for circle diffeomorphisms, reformulated in the language of minimal foliations on the torus (as we explain at the end of the next section of this introduction). { We recall that  in higher genus, foliations are necessarily \emph{singular}. \emph{Minimality} in this context means that {all \emph{bi-infinite} leaves are dense (see Definition~\ref{def:minimal})}. By \emph{Morse-type} singularities we mean that the leaves of the foliation in a neighborhood of a singularity are level-sets of a Morse function (i.e.~a function with non-degenerate zeros).  This is a \emph{generic} (open and dense) condition. These assumptions imply in particular that singular points of the foliation are \emph{saddles}\footnote{{Morse type singularities are \emph{simple saddles} (with $4$-separatrices or \emph{prongs}) and \emph{centers}: centers are excluded since 
if there is a center, the foliation has closed orbits in the neighbourhood of the center and thefore is not minimal. 
A minimal foliation in  genus is two can have either $2$ simple saddles (with $4$ prongs each), or one (degenerate) saddle with $6$ prongs. The assumption of Morse singularities implies that we are in the first case, but is included in Theorem A only to have a simpler statement: the case of one saddle with $6$-prongs is also covered by Theorem $B$ stated below, see also \S~\ref{foliations}.}} and  that the \emph{holonomy} of the foliation around each saddle points is zero (see \S~\ref{foliations} for details).} {Finally, we define the measure class on minimal foliations with respect of which the Diophantine-type condition (which is given by Definition~\ref{def:DCfol}) has \emph{full measure} in \S~\ref{sec:foliationmeasure}.}

\smallskip
 Note that all the examples of geometrically rigid dynamical systems  that we listed above, with the exception of the \emph{local} rigidity results given by KAM theory, 
are   one-dimensional and combinatorially equivalent to either a translation on a torus (in the case of circle diffeomorphisms, critical circle maps, or circle maps with breaks), or an odometer (in the case of unimodal maps and Lorenz maps). Our result is, to the best of our knowledge, the first (global) rigidity result on surfaces of higher genus, which have a much richer\footnote{{\color{black}This \emph{richness} can be formalized in various ways: flows on surfaces are described by more \emph{frequencies}; the combinatorial information in this setting can be described  \emph{higher dimensional continued fraction algorithms}, which produce cocycles in $SL(d,\mathbb{Z})$ with $d\geq 2$; the combinatorial information can also be encoded in a Bratteli-Vershik diagram with $d\geq 2$ vertices, while odometers and rotations both correspond to Bratteli-Vershik diagrams with $d=2$. 
 Finally, in virtue of this higher dimensional nature, one lacks in general Denjoy-Koksma inequality and a priori bounds, see a later subsection of this introduction.}}  combinatorial structure.

Before we state the other main results of this article (in \S~\ref{sec:GIETrigidityintro} and \S~\ref{sec:renormalizationintro}), we summarise the main results on  diffeomorphisms of the circle and (generalized) interval exchange transformations that motivated our work.

\subsection{Diffeomorphisms of the circle and foliations on the torus.}

Flows and foliations on surfaces have been a topic of interest since the work of Poincaré, who singled out the analysis of flows on the torus as the simplest toy-model to investigate the stability of the solar system.
Poincaré introduced the \textit{rotation number}, which is an invariant which fully accounts for the combinatorial structure of orbits of circle diffeomorphisms (and equivalently flows on the torus).

The rigidity theory of circle diffeomorphisms was started by Denjoy in \cite{Denjoy}. Recall that two homeomorphisms $f,g:S^1\to S^1$ of the circle  $S^1$ are \emph{topologically conjugate} if there exists a homeomorphism $h:S^1\to S^1$ (the \emph{conjugacy} map) such that $f\circ h=g\circ f$.  
Denjoy in \cite{Denjoy}
proved that a sufficiently regular circle diffeomorphism $f$ with irrational rotation number must be \emph{topologically conjugate} to the \emph{rigid rotation} $R_\rho$ with the same rotation number $\rho$, given by $x \mapsto R_\rho(x):= x + \rho$. 
{ The existence of a topological conjugacy} implies in particular that $f$ cannot have \emph{wandering intervals}, namely there does not exist intervals $I\subset S^1$ such that the iterates $f^n(I)$, $n\in\mathbb{Z}$ are all disjoint. 

A landmark result is the \emph{local rigidity} theorem of Arnol'd \cite{Arnold}, who successfully applied KAM theory to show that under a suitable Diophantine-type condition on the rotation number $\rho$, sufficiently small analytic deformations of $x \mapsto x + \rho$, whose rotation number is equal to $\rho$, must be \textit{analytically} conjugate to $x \mapsto x + \rho$. Arnol'd went on to conjecture that such a statement should hold true without any assumption on the closeness to rotations. 

\smallskip This \emph{global rigidity} conjecture was proved to be true { in the (more general) $\mathcal{C}^{\infty}$ setting}\footnote{Herman in his thesis \cite{Herman} considers not just the not just $C^{\infty}$ regularity, but also $C^r$, for $r \geq 3$ as well as the analytic settings.} by {\color{black}Michael} Herman \cite{Herman} in a spectacular treaty, 
whose legacy still lives on. A few years later, Jean-Christophe Yoccoz \cite{Yoccoz} succeeded in showing that 
Herman’s
result indeed extends to all \emph{Diophantine numbers}, { thus providing the optimal arithmetic condition} { in the smooth setting (and later on also the optimal condition in the analytic setting, see \cite{Yo:CIME})}. Combining Herman's \cite{Herman}  and Yoccoz'  \cite{Yoccoz} results, we have the following theorem:
\begin{thm*}[Herman \cite{Herman}, Yoccoz \cite{Yoccoz}]
If $\rho$ is a Diophantine number, then any $T \in \mathrm{Diff}^{\infty}(S^1)$ of rotation number $\rho$ is $\mathcal{C}^{\infty}$-conjugate to $x \mapsto x + \rho$. In particular, smooth circle diffeomorphisms of Diophantine rotation number are (geometrically) rigid.
\end{thm*}
An equivalent geometric reformulation of the above theorem in the language of foliations on surfaces is the following. 
Let $S$ be a torus, namely a \emph{genus one} closed orientable surface. Then, given  any \emph{minimal} (orientable) foliation on $S$ which is topologically conjugate to a linear flow on the torus with Diophantine rotation number\footnote{Let us recall that a \emph{linear flow} on the torus $\mathbb{T}^2=\mathbb{R}^2 / \mathbb{Z}^2$ 
is the flow $(\varphi_t)_{t\in\mathbb{R}}: \mathbb{T}^2 \to \mathbb{T}^2$ 
given by  $\underline{x}\mapsto \underline{x}+t \underline{\alpha}$, where $\underline{x},\underline{\alpha}\in \mathbb{T}^2$. This flow  has rotation number $\rho=\alpha_2 / \alpha_1$.} is smoothly conjugated (in the sense of foliations\footnote{We recall that the regularity of conjugacy of foliations is {\color{black}expressed} in terms of the transverse structure; thus, this is equivalent to the conjugacy of $f$ and $R_\rho$.}) 
to the linear flow foliation. {In particular, the theorem implies} that minimal foliations on  \emph{genus one} surfaces under a full measure  {Diophantine-type} condition are geometrically rigid.   { It is this latter statement that is generalized by Theorem A to \emph{genus two}.}

\subsection{Flows in higher genus and interval exchange transformations.}
The extent to which the theory of diffeomorphisms of surfaces and flows on the torus generalises to flows on higher genus surfaces and their Poincar{\'e} maps is a natural question. 
 The objects which play the role of rigid rotations in this context are \textit{standard interval exchange transformations} (IETs for short), orientation-preserving bijections of $[0,1]$ which are piecewise translations (see Definition~\ref{def:IET}). These transformations  naturally arise as Poincaré maps of \emph{linear flows} on higher genus surfaces (namely \textit{translation flows} on translation surfaces, or, correspondingly,  \textit{measured foliations}\footnote{A translation surface determines indeed a vertical (and a horizontal) measure foliation. The leaves of the vertical linear flow are leaves of the measured foliation and the other foliation determines a transverse invariant measure.}), see \S~\ref{gietandfoliations}.  \emph{Linear flows} on (translation) surfaces in turn play the role of linear flows on (flat) tori. 
The non-linear counterparts are \emph{generalized interval exchange transformations} (or GIETs), which arise as Poincaré maps of minimal flows on higher genus surfaces. Notice that in higher genus the presence of singularities is unavoidable and the corresponding (orientable) foliations have singularities { (corresponding to \emph{fixed points} for the flow).}

\subsubsection{Renormalization and combinatorics.}
To generalize Poincar{\'e} and Denjoy work, one needs first of all a combinatorial invariant which extends the notion of rotation number. Such an invariant can be produced by recording the combinatorial data of a renormalization process. Renormalization operators in this context, similarly to the case of circle diffeomorphisms,  
are obtained associating to a  
 given a GIET $T:I\to I$ on $I=[0,1]$, 
 another GIET 
which is obtained by suitably choosing an  subinterval $I'\subset I$ and considering the \emph{induced map} of $T$. The interval $I'$ is chosen so that the induced map is well defined and is again a GIET  $T'$ of the same number of intervals. Correspondingly, at the level of (minimal) flows (or foliations) on surfaces, this process corresponds to taking a smaller Poincar{\'e} section. The image $\mathcal{R}(T)$ of $T$ under the renormalization operator is
 then by definition the GIET acting on $I=[0,1]$ obtained by \emph{normalising}, i.~.e.~conjugating by the affine transformation which maps $I'$ to $I$, so that the image is again a GIET on $I$. 

 A classical algorithm to \emph{renormalize} standard IET is the \emph{Rauzy-Veech algorithm}, also called \emph{Rauzy-Veech induction} (whose definition we recall in \S~\ref{Rauzysec}), first introduced by Rauzy \cite{Ra:ech} and used starting from the seminal papers by Veech \cite{Ve:gau, Ve:inI} to study fine ergodic properties of standard IETs, see e.g.~\cite{Zo:dev,  AF:wea, Ch:dis}. The ergodic properties of this renormalization dynamics in parameter space is by now well understood (see e.g.~\cite{Zo:gau,  Bu:dec, AGY, AB:exp}, or \cite{Yoc:B} for a brief survey).
 
Rauzy-Veech induction is well defined also on GIETs with no connections (as defined in \S~\ref{combinatorics}) and can be used (as shown e.g.~in \cite{MMY, MarmiYoccoz}, see also the notes \cite{Yoc:Clay}) to define a notion of \emph{rotation number} (see Definition~\ref{def:rotnumber}) and \emph{irrationality} for IETs and GIETs (see Definition~\ref{def:irrational}). One can then show that two \emph{irrational} GIETs with the same rotation number are \emph{semi-conjugated}, a result that we call Poincar{\'e}-Yoccoz theorem (see Theorem~\ref{thm:PY}).

\subsubsection{Absence of a Denjoy theorem  and wandering intervals.} One of the crucial differences between GIETs and circle diffeomorphisms, though, is the absence of a \emph{Denjoy Koksma inequality}\footnote{We recall that the \emph{Denjoy-Koksma inequality} is an ergodic-theoretic statement which gives boundedness of Birkhoff sums of bounded variation observables at special times: given $f:I\to I$ is a function of bounded variation on $I=[0,1]$ and $R_\rho$ a rotation by $\rho$, if $p_n/q_n$ are the \emph{convergents} of $\rho$ (given by $p_n/q_n=[a_0,\dots, a_n]$ where $[a_0,\dots, a_n,\dots]$ is the continued fraction expansion of $\rho$), the Birkhoff sums $\sum_{k=0}^{q_n}f(x)$ at times $q_n$ are uniformely bounded, independently on $n\in\mathbb{N}$ and $x\in I$.} and \emph{a priori bounds}. This has far-reaching consequences, the most spectacular of which being the absence of Denjoy theorem: there are smooth GIETs that are semi-conjugate to a minimal IET for which the semi-conjugacy is \textit{not} a conjugacy, in other words they have wandering intervals. This phenomenon, first discovered by Levitt but for a non-uniquely ergodic example, was later expored  by Camelier Gutierrez \cite{ CamelierGutierrez}, Cobo \cite{Cobo} and Bressaud, Hubert and Mass \cite{BressaudHubertMaass} for special (families of) uniquely ergodic examples with periodic-type combinatorics. It is important to stress that this is is not a low-regularity phenomena, nor it is related to special  arithmetic assumptions, as these examples exist for AIETs with \emph{almost every} rotation number, as shown later in the  work \cite{MMY2} by Marmi, Moussa and Yoccoz, which actually indicates that the existence of wandering interval is in some sense \emph{typical}\footnote{See for example the statement of Proposition~\ref{AIETwi}, which is taken from \cite{MMY}: wandering intervals are shown to exist for a full measure set of rotation numbers as long as the log-slope vector of the AIET has a typical projection on the Oseledets filtration (and, conjecturally, as long as it projects on any positive Oseledets exponent).}.

\subsubsection{Cohomological equations and obstructions to linearisation.}
  A crucial step in the KAM approach developed by Arnol'd for circle diffeomorphisms is to solve a \textit{linearised} version of the conjugacy equation  $ h \circ R_{\rho} = T \circ h $, which amounts to finding  a (smooth) $f$ which satisfy the functional equation $ f \circ R_{\rho}  - f = g$ for a given  (smooth) $g$.  This equation, known as \textit{cohomological equation}, is easily solved { in the smooth setting}  using Fourier analysis in the case where $R_{\rho} := x \mapsto x + \rho$ is a rotation satisfying a full measure arithmetic condition, under the necessary (and in this setting the only) obstruction that $g$ has to have  zero-mean.
  
 For a long time it has been unknown whether the cohomological equation could be solved under suitable assumptions for IETs (or flows on surfaces such as translation flows), until the pioneering work of Forni \cite{Fo:ann} (see also \cite{Forni2}), who brought to light the existence of a  \textit{finite} number of obstructions to solving it. The existence of obstructions to solve the cohomological equation has been since then discovered to be a characteristic phenomenon in \emph{parabolic dynamics}, see for example the works by Flaminio and Forni on the cohomological equation for horocycle flows  \cite{FF}  and \cite{FF2} for  nilflows on nilmanfolds (which are other key examples of \emph{parabolic} flows, in the sense that they present a subexponential form of sensitive dependence of initial conditions, see for example the surveys~\cite{Ul:she, Fo:asy}). Forni's work is a breakthrough that paved the way for the development of a linearisation theory in higher genus. 
 
Another  breakthrough was  achieved by  Marmi-Moussa-Yoccoz in their work \cite{MMY3} (and related works \cite{MMY, MarmiYoccoz}). In \cite{MMY}, in particular, they reproved Forni's result on the cohomological equation using a renormalization approach based on Rauzy-Veech induction, thus describing explicitely a full measure Diophantine-type condition on the IET (a condition that, in analogy with rotations, they called \emph{Roth type}, see \cite{MMY} and also \cite{MarmiYoccoz} for a variation of this condition). Furthermore, the improved regularity in their result could then be exploited in \cite{MMY}, combined with a generalization of  Herman's \emph{Schwarzian derivative trick}, to prove a linearisation result, showing that 
the  high regularity ($\mathcal{C}^r$ for $r \geq 2$)-local conjugacy classes of smooth IETs  form a submanifold of the expected finite codimension (the codimension being related to the number of obstructions to solve the cohomological equation, see \cite{MMY3}). 
{An analogous result for the $\mathcal{C}^1$-local conjugacy classes was suggested as a conjecture in  \cite{MMY3}.  Recently, the first author has proved it in a special case, namely for the (measure zero) set of IETs which have (hyperbolic) periodic-type rotation number (in the sense of Definition~\ref{def:periodictype} below), which hence correspond to periodic points of the renormalization operator.}

\subsubsection{Rigidity conjecture.}
In \cite{MMY3}, Marmi, Moussa and Yoccoz formulate a number of fundamental open questions and conjectures left open in the theory of linearisation of GIETs (see the Open Problems section~1.2 in \cite{MMY3}). One of them, stated as Problem 2 in \cite{MMY3}, is a geometric rigidity question/conjecture\footnote{Immediately before formulating it as a question, Marmi, Moussa and Yoccoz provide an  heuristic \emph{rationale} which explains why it should be true and, just after, formulate a slight extension of what they now call 'one of the previous two \emph{conjectures}'.}. They ask whether it is true that, for  a full measure set of standard IETs $T_0$, any GIET $T$ of class $\mathcal{C}^4$  to $T_0$ and such that the value of a conjugacy invariant that they call \emph{boundary} (see below and, for a  definition, \S~\ref{sec:boundaryGIET})  is zero\footnote{More precisely, Problem $2$ in \cite{MMY3} is first stated for GIETs which are a \emph{simple deformation} of $T_0$ (i.e.~a deformation which does not perturb $T_0$ in a neighbourhood of the discontinuities and endpoints, see \cite{MMY3}). Being a simple deformation implies in particular that the boundary is the same than the boundary of $T_0$ and the latter is indeed zero. Immediately after, they say that the conjecture can be formulated in a slightly more general setting (not restricted to simple deformations) using the boundary conjugacy invariant that they introduce later.}, is actually also $\mathcal{C}^1$-conjugate to $T_0$. 

{We prove in this paper  that this conjecture is  true in genus two (see Theorem B below)}. We also show  that the result in any genus can be reduced to a statement on dynamical partitions of affine IETs, or equivalently, to problem concerning Birkhoff sums of piecewise constant observables over standard IETs (see the comments below, or \S~\ref{sec:rigidity}  and in particular Proposition~\ref{reduction}).

\subsection{Rigidity result for GIETs}\label{sec:GIETrigidityintro}
We have already anticipated one of our rigidity results in the language of foliations (Theorem A stated above).  We will now formulate our main result in the language of GIETs (Theorem B below).

\smallskip
 We denote by $\mathcal{I}_d$ the space of standard irreducible interval exchange transformations with $d$ branches (see \S~\ref{combinatorics} for the definition of irreducible). This space  is a finite union of $d\!-\!1$ simplexes (see \S~\ref{parameters}) and thus carry a natural Lebesgue measure. Full-measure sets and full-measure Diophantine-type conditions are defined using this measure.  
Associated to a GIET $T$, {there} is an important $\mathcal{C}^1$-conjugacy invariant, called the \emph{boundary} of $T$ and here denoted by  $\mathcal{B}(T)$  (for the definition of  $\mathcal{B}$, which is based on Marmi-Moussa-Yoccoz \emph{boundary} operator from \cite{MMY3, MarmiYoccoz}, see Definition~\ref{GIET:boundary} and \S~\ref{sec:boundarycomb}). 
Our main rigidity result in the language of interval exchange maps is  the following.
\begin{thmb}[Rigidity of GIETs with $d=4$ or $d=5$]
\label{thmAb}
There is a full measure\footnote{{Here \label{Lebfootnote} the measure is the Lebesgue measure on the parameter \emph{standard} IETs $\mathcal{I}_d$, i.e.~a result holds for a full measure set of IETs in $\mathcal{I}_d$, if it holds for all \emph{irreducible} combinatorial data and Lebesgue-almost every choice of \emph{lenghts} of the continuity intervals. See \S~\ref{sec:measures} for details.}}  subset $\mathcal{F} \subset \mathcal{I}_4 \cup \mathcal{I}_5$ such that the following holds. If  $T_0 \in \mathcal{F}$ and a  $\mathcal{C}^3$-generalized interval exchange map $T$, whose boundary $\mathcal{B}(T)$ vanishes, is topologically conjugate to $T_0 $, then the conjugacy between $T$ and $T_0$ is actually a diffeomorphism of $[0,1]$ of class $\mathcal{C}^1$. In other words, almost every standard irreducible IET with $4$ or $5$ continuity intervals is  geometrically rigid.
\end{thmb}
Thus, Theorem B proves the rigidity conjecture by Marmi-Moussa-Yoccoz \cite{MMY3} for {irreducible} IETs with $d=4$ or $d=5$ intervals (which correspond to Poincar{\'e} sections of flows in genus two). {The $d=5$ case implies Theorem $A$ (see \S~\ref{foliations}) and,  more in general, for $d=4$, the analogous statement for minimal orientable foliations in genus two with a degenerate saddle.} 
{\color{black} The set $\mathcal{F}$ of \emph{standard} IETs, which  has full measure with respect to the \emph{Lebesgue} measure on $\mathcal{I}_4$ or $\mathcal{I}_5$ (
(see footnote~\ref{Lebfootnote} and \S~\ref{sec:measures} for details) is described by a \emph{Diophantine-type} condition} that we call $(RDC)$. We comment on its nature below (see \S~\ref{DC} and Definition~\ref{def:RDC} for the precise condition).  

We remark that a great part of the intermediate results which are proved to deduce Theorem B (see e.g.~Theorem C and Theorem E stated below), are proved in greater generality, namely for any $d\geq 2$ (and hence for any genus in the language of foliations in Theorem A). The result which is exploited in the proof and reduces the validity of the rigidity conclusion to $d=4,5$ (and respectively genus two foliations) is a result  on existence of wandering intervals for \emph{affine IETs} (and equivalently on the control of Birkhoff sums of piecewise constant observables), which was proved by Marmi, Moussa and Yoccoz in \cite{MMY2} and known only under a technical condition on Lyapunov exponents (which is automatically satisfied for $d=4,5$). 

{\color{black}
\subsubsection{Regularity of the conjugacy}
The reader familiar with the theory of circle diffeomorphisms will have noticed that Theorem B only gives that the conjugating map is of class $\mathcal{C}^1$. We believe that it should be possible to prove that the regularity is indeed $\mathcal{C}^{1+\alpha}$ (i.e.~that the derivative is $\alpha$-H{\"o}lder) for some $0<\alpha<1$ but that indeed the conjugacy is not typically $\mathcal{C}^2$.  We stress that this is not due to a shortcoming in our approach, rather this loss of regularity is an essential feature of the problem, which corresponds to Forni's and Marmi-Moussa-Yoccoz  non-trivial obstructions to solving the cohomological equation: Marmi-Moussa-Yoccoz have indeed shown that asking for more regular conjugacy forces GIETs to live in positive codimension submanifolds of the $\mathcal{C}^1$-conjugacy class; the codimension is an exact reflection of the aforementioned obstruction. 
In this a sense, GIETs are closer to essentially non-linear rigid dynamical systems, such as unimodal maps and circle map with breaks or critical points, for which the conjugacy is typically no more regular than $\mathcal{C}^1$  and actually $\mathcal{C}^{1+\alpha}$ for some  $0<\alpha<1$ in general).
}


\subsubsection{The boundary assumption}
We  remark that the \emph{boundary condition}  (i.e.~the assumption that $\mathcal{B}(T)$ vanishes) is an essential assumption: two GIET that are topologically conjugate but have different boundaries cannot be differentiably conjugated, simply because the boundary is $\mathcal{C}^1$-conjugacy invariant. We note, for the reader who is familiar with the one-dimensional dynamics literature, that the assumption that $\mathcal{B}(T)$ is zero, in the special case where $T$ is a circle maps with breaks, reduces to the classical assumption that the \emph{non-linearity} $\eta_T$ (see \S~\ref{nonlinearity}) has integral zero and that the special pair $(T_1,T_2)$, where $T_1,T_2$ are the two branches of $T$, corresponds to a diffeomorphism without break points\footnote{The boundary $\mathcal{B}(T)=0$ is indeed a vector in $\mathcal{R}^\kappa$ (see Definition~\ref{GIET:boundary} and \S~\ref{sec:boundarycomb}), where $\kappa:=d-2g$ and $g$ is the genus of any (minimal) surface flow which has $T$ as Poincar{\'e} section, see \S~\ref{gietandfoliations}. Asking that the sum of the entries of   $\mathcal{B}(T)$ is zero is equivalent to asking that   $\int_0^1\eta_T(x)\mathrm{d} x=0$ (see Lemma~\ref{boundaryaffinereduction}); when $T$ is a circle diffeomorphisms with break points (i.e.~when the combinatorics of the GIET is of \emph{rotational type}, the value of the entries of $\mathcal{B}(T)$ are related to the values of the break points, and therefore  asking that $\mathcal{B}(T)$ is the \emph{zero-vector} means asking that there are no breaks, see Remark~\ref{rk:Bbreaks}.}.   The case where  $\mathcal{B}(T)$ does not vanish is equally interesting, and some comments are made in a subsequent paragraph. 

Geometrically, when $T$ is the Poincar{\'e} map of a minimal foliation on a surface, the boundary $\mathcal{B}(T)$ encodes the holonomy around the saddles of the foliation (see \S \ref{foliations}). Thus, the assumption that $\mathcal{B}(T)$ is zero is equivalent to the {asking that the corresponding foliation has trivial holomony around singularities (a condition that is automatic when the singularities are level sets of Morse functions). It is using this remark that Theorem A can be deduced from Theorem B (see \S~\ref{foliations})}.

\textcolor{black}{
\subsubsection{The $\mathcal{C}^0$-conjugacy class in parameter space}
The main result of this article does { not yet give a description of the $\mathcal{C}^0$-conjugacy class of IETs in parameter space. It shows on the other hand that the $\mathcal{C}^0$-conjugacy class of almost every IET agrees with the $\mathcal{C}^1$-conjugacy class.} Marmi, Moussa and Yoccoz conjectured that, for almost every IET, the $\mathcal{C}^1$-conjugacy class is a codimension $(d-1) + (g-1)$ submanifold of class $\mathcal{C}^1$ (see \cite{MMY3}, Problem 1). {As already mentioned earlier, a step towards this conjecture} has been taken in \cite{Selim:loc}, which shows that it is {locally} true for {(hyperbolic) periodic type IETs (see Definition~\ref{def:periodictype})}. The main result of the present article, combined with a  complete proof of {the} conjecture,  would therefore automatically yield also a complete description of the $\mathcal{C}^0$-conjugacy class of almost every {genus two} IET in parameter space.
}

\subsection{A dynamical renormalization dichotomy and the strategy of the proof.}\label{sec:renormalizationintro}
 The proof of the rigidity results (Theorem B and its geometric reformulation in Theorem A)  are based upon renormalization methods. We prove in particular some results on the dynamics of renormalization and its consequences, which we now state, which are valid for infinitely renormalizable GIETs with any number $d\geq 2$ of intervals and we believe are of independent importance. 
 
 \smallskip 
Let $\mathcal{X}^r_d$ denote the space of all GIETs of class $\mathcal{C}^r$ on $d\geq 2$ intervals  with an irreducible combinatorics (see \S~\ref{combinatorics} and \S~\ref{parameters} for definitions). 
We consider, as renormalization operator $\mathcal{R}: \mathcal{X}^r_d \to \mathcal{X}^r_d$, an acceleration of  Rauzy-Veech induction (which is in turn given by suitable, linearly growing iterates of Zorich acceleration of Rauzy-Veech induction).  The general statement that we prove on the dynamics of this renormalization, which is valid  for any $d\geq 2$ (and hence, correspondingly, for Poincar{\'e} sections of flows on surfaces of any genus) is, informally, the following \emph{dynamical dichotomy} (we refer to Theorem \ref{shadowing} for a precise statement):

\begin{thmc}[A priori bounds or affine shadowing dichotomy]
Let $T$ be  a  GIET in $\mathcal{X}^r_d$, $d\geq 2$,  whose rotation number satisfies a full-measure\footnote{{The measure on (combinatorial) rotation numbers is induced here by the Lebesgue measure on standard IETs, see Definition~\ref{def:fullmeasure} and \S~\ref{sec:fullmeasure} for details.}} Diophantine-type condition which we call $(RDC)$. Then there exists a bounded set $K$ such that one of the two possibility holds.

\begin{enumerate}

\item The iterated renormalizations $(\mathcal{R}^n(T))_{n \in \mathbb{N}}$ are recurrent to the bounded set $K$. 

\item The iterated renormalizations $(\mathcal{R}^n(T))_{n \in \mathbb{N}}$ go to infinity at an exponentially rate, and the orbit $(\mathcal{R}^n(T))$ is well-approximated by that of an affine interval exchange map.

\end{enumerate}
\end{thmc}
\noindent  The notion of \emph{full measure} is defined \S~\ref{sec:fullmeasure} (see in particular Definition~\ref{def:fullmeasure}). We comment below (in \S~\ref{sec:DCintro}) on the nature of the Diophantine-type condition. Let us say here though that this full measure condition includes in particular as a (measure zero) special case all \emph{periodic-type} combinatorics (also known as \emph{Fibonacci-type} combinatorics in the one-dimensional literature, see \S~\ref{sec:periodictype} for definitions). A proof of Theorem C in this special case is much easier and is included both for didactical purposes and for the reader not interested in the technical  subtlety of Rauzy-Veech induction (see section \S~\ref{sec:periodiccase}); the periodic-type case (see Definition~\ref{def:periodictype})  also yields a stronger conclusion, namely the approximation in $(2)$ is up to a bounded error (see Proposition~\ref{asymptotic}). 

In the  first case, Case $(1)$, which we call the \emph{recurrent} case, one can show that \emph{a priori bounds} on iterates of renormalization hold (see Proposition~\ref{apriori}).  The heart of the work in  Case $(2)$ is to construct a vector $v=(v_1,\dots, v_d)\in \mathbb{R}^d$, that we call the \emph{affine shadow}. This vector is such that $e^{v_i}$, for $1\leq i\leq d$, are the slopes of an affine IET whose orbit under renormalization  gives the \emph{leading divergent behaviour} of the orbit of $T$ (see Theorem~\ref{shadowing} for a precise statement).   Thus, the quantities $v_i$, $1\leq i\leq d$ play the role of \emph{geometrical scaling invariants} associated\footnote{We remark though that the shadow $v\in \mathbb{R}^d$ is not uniquely defined, but its \emph{unstable component} (which leaves in a space of dimension $g$) is: 
two shadows $v_1,v_2$ of the same GIET $T$ differ by an element of the central stable space $E_{cs}(T_0)$ of the Oseledets filtration of the standard IET $T_0$ semi-conjugated to $T$, see \S~\ref{sec:Oseledetsgeneric}.} to $T$.   \textcolor{black}{Theorem C is an instance of study of an infinitely renormalizable dynamics, whose orbit diverges in parameter space. Even in this context, describing the \emph{way} in which divergence occurr proves to be helpful to control the dynamical behaviour of the system.   
An interesting occurrence of this phenomenon in one-dimensional dynamics, similar at least in spirit, has been recently analysed for certain Cherry flows, in the work of Martens and Palmisano \cite{MartensPalmisano}. }

\subsubsection{Wandering intervals and a priori bounds}
Another key step of the proof is to show that, if one can prove that the affine IET that shadows $T$ given by $(2)$ has wandering intervals by showing that the dynamical partitions associated to the AIET are \emph{exponentially distorted} (a geometric notion that we define in \S~\ref{sec:wandering}, see Definition~\ref{def:expdistorted}), then also the GIET $T$ has wandering intervals (see Proposition~\ref{reduction}). Thus, the problem of existence of wandering intervals for GIETs is reduced by our work to a question\footnote{We remark though that it is not sufficient for us to simply show that the affine shadow has wandering intervals, but we need to show that this happens in a \emph{special way}, namely one needs to show the Birkhoff sums estimates proved by Marmi-Moussa and Yoccoz in \cite{MMY2}, as stated in Proposition~\ref{AIETwi}, or, equivalently, that dynamical partitions are \emph{exponentially distorted} in the sense of Definition~\ref{def:expdistorted}.  It is possible that this is indeed the only way in which a wandering interval can appear in an affine interval exchange transformation, but this may be difficult to prove.}  concerning \emph{affine} interval exchange transformations, {or more precisely Birkhoff sums of piecewise constant functions over standard IETs.}

Since Marmi, Moussa and Yoccoz have shown that a large class of AIETs have exponentially distorted towers and hence wandering intervals (see \S~\ref{sec:AIETwi} and in particular Proposition~\ref{AIETwi}), it follows that all GIETs which are shadowed by AIETs in this class (which includes typical AIETs for any $d\geq 2$, see  Proposition~\ref{AIETwi}) have wandering intervals.  
When the number of exchanged intervals is $d=4$ or $d=5$ (i.e.~when the GIET is a Poincar{\'e} section of a minimal flow on a genus two surface), the result by  Marmi, Moussa and Yoccoz \cite{MMY2} include in particular all AIETs with divergent shadow\footnote{More precisely, the condition $v\in E_2(T)\backslash E_3(T)$ in Proposition~\ref{AIETwi} is automatically satisfied, see the proof of Theorem~\ref{sec:proofGIETrigidity}.}. Thus, in this case, assuming that $T$ is minimal (an assumption which rules out the presence of wandering intervals) forces $T$ to be recurrent, i.e.~Case $1$ of the dynamical dichotomy given by Theorem C to hold. Thus, we can deduce in this case a result on \emph{a priori bounds}:

\begin{thmd}[A priori bounds in genus two]
If $T$ is  a \emph{minimal} GIET in $\mathcal{X}^r_d$ with $d=4$ or $d=5$ whose rotation number satisfies the full-measure  condition $(RDC)$, then the acceleration  $\mathcal{R}$ of the Rauzy-Veech renormalization satisfies \emph{a priori bounds}, namely there exists a constant $K>0$ such that the iterates $\mathcal{R}^m(T)$ of $T$ under renormalization satisfy  
$$K^{-1} \leq ||\D\mathcal{R}^m(T)||_{\infty} \leq K, \qquad \textrm{for\ all}\  m \in \mathbb{N},$$ where $\Vert \cdot \Vert_{\infty}$ denotes the sup norm on $I=[0,1]$.  
\end{thmd}
{\noindent We refer to \S~\ref{sec:apriori} (in particular Propostion~\ref{apriori}) for a more precise formulation. }

\noindent Generalizing the aforementioned result by Marmi, Moussa and Yoccoz \cite{MMY2} (in particular Proposition~\ref{AIETwi}) to cover all divergent shadows\footnote{More precisely, one needs to show that Proposition~\ref{AIETwi} holds for any shadow which has a projection on a positive Lyapunov exponent, which is not necessarily the second as in the case when one assumes that $v\in E_2(T)\backslash E_3(T)$.} for a full measure set of AIETs with any $d\geq 2$ will remove the restriction that $d=4,5$ from the statement of Theorem D.   Notice  on the other hand that \emph{no  boundary condition}  appears in Theorem D, nor in Theorem C.  
The assumption that the boundary $\mathcal{B}(T) $ is zero is indeed only required when we proceed to prove a conjugacy regularity result.

\subsubsection{Boundary obstructions and convergence of renormalization}
The next conceptual step of our proof is to show that, when one is in Case $(1)$, namely the recurrent case of the dynamical dichotomy Theorem C (for example because one has ruled out case $(2)$ by showing that it would imply the presence of wandering intervals and hence non-minimality), one can prove results on exponential convergence of renormalization. More precisely, we show the following result, which holds for any $d\geq 2$:

\begin{thme}[Exponential convergence of renormalization]
Let  $T$ be a GIET in $\mathcal{X}^r_d$  whose rotation number satisfies the full-measure  condition $(RDC)$. Assume that $T$ satisfy the conclusion $(1)$ of Theorem C and that the boundary $\mathcal{B}(T)$ is zero. Then the orbit $(\mathcal{R}^m(T))_{m\in\mathbb{N}}$ of $T$ under renormalization converges exponentially fast, in the $\mathcal{C}^1$ distance, to the subspace $\mathcal{I}_d$ of (standard) IETs.
\end{thme}
\noindent The precise formulation of the theorem and the definition of $\mathcal{C}^1$ distance are given in \S~\ref{sec:convergence} (see in particular Theorem~\ref{convergencethm} and \S~\ref{sec:distances}). 
In this case, we can then conclude, using classical arguments, that $T$ is $\mathcal{C}^1$-conjugated to a standard IET with the same rotation number (as shown in \S~\ref{sec:conjugacy}). 

We consider the proof of Theorem E  to  be a streamlined  presentation and generalization to GIETs of the now classical theory of Herman \cite{Herman} for circle diffeomorphisms. Some of these steps are well known in the literature on circle diffeomorphisms with singularities or are folklore, other require some variations of the arguments which are specially required to deal with the increased complexity of GIETs.  

We first show that, under the assumptions of Theorem E, the dynamical partitions associated to the GIET (whose definition is given in \S~\ref{dynamicalpartitions}) converge exponentially fast to the trivial partition into points (i.e.~their \emph{mesh}, or the size of the largest interval, decay exponentially).  {This can be seen also as a generalization of the arguments by Cunha and Smania in \cite{CS:ren} for a measure zero class of circle diffeos with break points  (those which correspond to \emph{bounded-type}, rotational GIETs)  to almost every rotation number and, more in general, to almost every GIET  which satisfies a priori bounds thanks to the recurrence given by the conclusion of Case $(1)$ of Theorem C.}

Exponential decay of the mesh can be used, as in the classical theory of circle diffeomorphisms and its extensions to diffeos with singularities, to show that iterates of renormalization converge to the space of Moebius IETs (GIETs whose branches are Moebius functions, see~\ref{def:MIET}). These first two steps do not require the assumption that  $\mathcal{B}(T)$ is zero. 

The boundary assumption becomes essential to proceed further. Indeed, requiring that $\mathcal{B}(T) $ is zero  restricts us to a positive codimension, renormalization invariant subset of the total space of GIETs which contains standard IETs. We call this the \textit{linear regime}, in contrast to the  \textit{non-linear regime} (see \S~\ref{sec:convergence} for the precise definitions).  
 In the linear  regime we show indeed that one is attracted to the space of \emph{affine} IETs first,  and actually, in a second step, to the space of standard IETs $\mathcal{I}_d$. 

The distinction between \emph{linear} (boundary zero) and \emph{non-linear} (boundary non-zero) regime  is a generalisation of the difference between standard circle diffeormorphisms and \textit{circle maps with breaks}, whose renormalization theory extends in a non-trivial way that of circle diffeomorphisms. We believe that the study of GIETs and renormalization in the \textit{non-linear regime} is also very interesting and, to the best of the authors knowledge, very little is known in this regime. The renormalization dynamics has in this case  a natural attractor which is the set of \textit{Moebius} IETs,   but the dynamics of the renormalization operator in that case is much more intricate to analyse.

\subsubsection{The {\color{black}{Diophantine-type}} condition}\label{sec:DCintro} Finally, we comment on the  Diophantine-type condition  appearing in Theorem C (which is also the full measure  condition implicitely underlying Theorem A and Theorem B).  The full measure condition, that we call \emph{Regular Diophantine Condition}, or $(RDC)$ for short, is formulated in terms of the Zorich (also known as Zorich-Kontsevich) cocyle over the induction. At each step of renormalization one can associate  a matrix $Z(n) \in \mathrm{SL}(d, \mathbb{Z})$. These matrices can be considered as a multi-dimensional generalisations of the coefficients $a_n$ appearing in the continued fraction expansion of a rotation number. The Diophantine-type condition has two aspects (as many Diophantine-type conditions introduced for IETs and GIETs, see e.g.~\cite{MMY, MarmiYoccoz}): 

\begin{enumerate}
\item A \emph{growth} condition, which straightforwardly generalises arithmetic, Diophantine-type conditions in the genus $1$; one asks that the matrices $Z(n)$ do not grow too fast (subexponentially with $n$ in our case).
\smallskip
\item A \emph{Oseledets} aspect, which is specific to the higher genus case: we demand that the product of the matrices $Z(n) \cdots Z(2) Z(1)$ is generic with respect to Oseledets theorem in a \emph{quantitative} way. 
\item A \emph{quantitative recurrence} aspect, where certain series depending on whole history of the  Zorich-Kontsevich cocycle are required to be uniformely bounded along a subsequence of renormalization iterates.
\end{enumerate} 
Our condition { is in part} reminiscent of the  Roth type condition and restricted Roth type conditions introduced by Marmi-Moussa-Yoccoz (see \cite{MMY} and \cite{MarmiYoccoz} respectively) and fairly similar in spirit (Roth type conditions also have a growth condition, usually denoted condition $(a)$, as well as further conditions, like condition (b) and (c) in the standard Roth type condition, which can be inferred from Oseledets genericity), but it is significantly more restrictive. 

First of all, we require  a \emph{quantitative} version of the conclusion of Oseledets theorem, in which the convergence is made \emph{effective} (see \S~\ref{sec:effectiveOseledets}). For technical reasons, we work with the natural extension (by choosing an arbitrary past for the rotation number) and require the existence of an effective Oseledets generic extension. When the (extended) rotation number is generic with respect to this effective version of Oseledets,  one can show that certain series, which depend on the whole matrices of the cocycle (explicitely given by  the \emph{forward} series \emph{backward} series $(F)$ and $(B) $ in the Definition~\ref{def:RDC} of the $(RDC)$ condition), are \emph{finite}. The above mentioned \emph{recurrence} amounts to the request that infinitely often, along a linearly growing subsequnce of times of the Zorich acceleration, these series are \emph{uniformely bounded}.  Conditions of similar (albeit simpler) nature on standard IETs were used by the second author in her work \cite{Ul:abs} on absence of mixing for special flows over IETs and appear as well in recent results in her joint work with K.~Fraczek \cite{FU:gro} on deviations of Birkhoff averages for locally Hamiltonian flows. 

Examples of arithmetic conditions on classical rotation numbers which do not depend only on the asymptotic behaviour of the continued fraction entries (as \emph{Diophantine} or \emph{Roth-type} conditions) but instead  depend on the \emph{whole} record of the continued fraction entries are for example the \emph{Brjuno}-condition (see e.g.~\cite{Yo:CIME}) or the \emph{Perez-Marco} condition \cite{PM:sur}.  
In the theory of circle diffeos, conditions which require recurrence to a set of rotation numbers with this type of control on the whole history seem to appear in global rigidity results, see for example the Condition $(H)$ defined by Yoccoz (see \cite{Yo:CIME}).  

While our condition is full measure (in the sense of Definition~\ref{def:fullmeasure}), it is {likely} not optimal. It would be interesting, but { probably} very difficult, to describe the optimal Diophantine-type condition for a GIET to satisfy the dynamical dichotomy in Theorem C. 

\subsection{Organization  of the paper and reading guide.} 
\noindent In the \emph{background} Section~\ref{background} we give {basic definitions, in particular defining GIETs (as well as IETs, affine IETs and Moebius IETs), Rauzy-Veech induction for GIETs, infinitely renormalizable GIETs, irrationality and rotation numbers. We also  summarize a number of classical tools and results which are used in the rest of the paper. These include both tools from the classical theory of circle diffeos and one dimensional dynamics (such as distorsion, distorsion bounds, non-linearity and Schwarzian derivative) as well as renormalization tools for IETs and GIETs related to Rauzy-Veech induction, such as Zorich acceleration, invariant measures for the dynamics on parameter spaces, dynamical partitions and Rohlin towers produced by Rauzy-Veech induction, special Birkhoff sums and decomposition of special Birkhoff sums. This section does not contain any new result. The reader familiar with one or both these backgrounds can skip this section or read it only quickly as a notational reference.

\smallskip
\noindent Section $2$ contains the precise formulation and the proof of the \emph{dynamical dichotomy} stated informally in this introduction as Theorem C. In \S~\ref{sec:periodiccase} we  first state and prove a (stronger) dynamical dichotomy in the special case of \emph{bounded type} rotation numbers (or Fibonacci combinatorics), defined in \S~\ref{sec:periodictype}. This proof can be skipped by the reader interested in the full measure result. We decided to present it first, even though it lenghten the paper, since it can be accessible to the reader that is not familiar with Rauzy-Veech induction  and already present all the key difficulties and ideas of the general proof. The general case requires the definition of full measure set of GIETs and rotation numbers and the definition of the Regular Diophantine Condition $(RDC)$, which are given in \S~\ref{sec:generalDC}
In \S~\ref{sec:general} we can then give the precise formulation of Theorem C in the general case, which is Theorem~\ref{shadowing}. The rest of the section is devoted to the proof. An outline of the main steps of the proof are given in \S~\ref{sec:outline}.
\smallskip

\noindent The main result of  Section $3$ is Theorem E on \emph{exponential convergence of renormalization} in the recurrent case. The proof takes the whole section and is split in several steps, such as a priori bounds (\S~\ref{sec:apriori}) exponential decay of the dynamical partitions mesh in \S~\ref{sec:sizePcontrol} and convergence first to Moebius IETs in \S~\ref{convergenceMoebius}, then to AIETs (see \S~\ref{convergenceaffine}) and finally to IETs in \S~\ref{IETsconvergence}. 

\smallskip
\noindent In Section $4$, we prove the \emph{rigidity result for GIETs}, namely Theorem B of this introduction. On one hand we prove that, when one has exponential convergence of renormalization and the $(RDC)$ Diophantine-type condition, one can deduce that the conjugacy is $\mathcal{C}^1$. This is done in \S~\ref{sec:conjugacy}. On the other hand, in \S~\ref{sec:wandering}, we {deduce}  the existence of wandering intervals for a GIET from  \emph{exponential distorsion} of the dynamical partitions of the affine shadow, see Proposition~\ref{reduction}, stated in in \S~\ref{sec:reductionstatement}  and proved in  \S~\ref{sec:reductionproof}. 
Combined with the results on wandering intervals proved by Marmi-Moussa and Yoccoz (recalled in \S~\ref{AIETwi}), this allows us to finish the proof of the rigidity result for GIETs as well as Theorem D on a priori bounds in genus two (in \S~\ref{pf:thmD}).

\smallskip
\noindent In Section $5$ we prove Theorem A on foliations on surfaces of genus two. We first define foliations, their regularity and their holomomies. We then deduce Theorem A for foliations in genus two from Theorem B on GIETs with $d=4,5$.

\smallskip
\noindent
In the Appendix, we include for convenience of the reader the proof of the (extension of) some classical results, such as the distorsion bounds for GIETs, the comparision between some of the distances used in Section~\ref{sec:convergence}, as well as some results from \cite{Selim:loc}, in particular on Lipschitz regularity of the renormalization operator, which are used in Section~\ref{sec:convergence}.

\section{Background material}\label{background}
\subsection{Interval exchange transformations}\label{giet} 
The piecewise differentiable maps which arise as Poincar{\'e} maps of a smooth, orientable foliation on a transversal interval are know as \emph{generalized interval exchange transformations}\footnote{The name \emph{generalized} interval exchange maps is used since they generalize \emph{interval exchange transformations} (see Definition~\ref{def:IET} below), which appear as Poincar{\'e} sections of \emph{measured foliations} on transverse intervals, in suitably chosen coordinates.}. 

\subsubsection{Generalized interval exchange transformations}\label{sec:giets}
Let us start by recalling the  definition of generalized interval exchange transformations, or, for short, GIETs.
\begin{definition}[{\bf GIETs}]
\label{def1}
Let $d\geq 2$ be an integer and $r$ a positive real number. A $\mathcal{C}^r$-generalized interval exchange transformation (GIET) of $d$ intervals, or for short a $d$-GIET of class $r$, is a map $T$ from the interval $[0,1]$ to itself such that:

\begin{itemize}

\item[(i)] there are two partitions (up to finitely many points) of $[0,1] = \bigcup_{i=1}^d{I^t_i} =  \bigcup_{i=1}^d{I^b_i} $  of $[0,1]$ into  $d$ open disjoint subintervals, called the \emph{top} and \emph{bottom} partition; the subintervals are denoted respectively $I_i^t$, for $1\leq i \leq d$, and $I_i^b$, for $1\leq i \leq d$;
\smallskip

\item[(ii)]  for each $1\leq i\leq d$, $T$ restricted to $I_i^t$ is an orientation preserving diffeomorphism onto $I_{i}^b$ of class $\mathcal{C}^r$;
\smallskip

\item[(iii)] $T$ extends to the closure of $I_i^t$ to a $\mathcal{C}^r$-diffeomorphism onto the closure of $I_{i}^b= T(I_i^t)$.

\end{itemize}

\end{definition}
We will call the restriction $T_i:=T|I^t_i$ of $T$ onto $I^t_i$, for $1\leq i\leq d$, a \emph{branch} of $T$. We think of $j\in \{1,\dots, d\}$ as the \emph{label} of the intervals $I_j^t$ and $I_j^b=T(I_j^t)$ and denote by $\mathcal{A}:=\{ 1,\dots, d\}$ be the alphabet consisting of labels.  Notice that $T$ is by construction invertible and that the inverse $T^{-1}$ is also a $\mathcal{C}^r$-GIET, for which the top and bottom partition are reversed.

\subsubsection{Standard, affine and Moebius IETs}\label{affine}
 Special cases of generalized interval exchange transformations include \textit{standard} interval exchange transformations (IETs),   \textit{affine} interval exchange transformations (AIETs) and Moebius interval exchange transformations (MIETs):

\begin{definition}[{\bf IETs}]\label{def:IET}
A GIET $T$ is an (\emph{standard}) \emph{interval exchange transformation} or a $IET$ if $|I^t_i|=|I^b_t|$ for every $1\leq i\leq d$ and the branches $T_i$ of the map $T$, for every $1\leq i\leq d$, are assumed to be \emph{translations}, i.e.~of the form $x\to x+\delta_i$ for some $\delta_i\in\mathbb{R}$.
\end{definition}

\begin{definition}[{\bf AIETs}]\label{def:AIET}
A GIET $T$ is an \emph{affine interval exchange transformation} or an $AIET$ the branches $T_i$ of the map $T$, for $1\leq i\leq d$,  affine map, i.e.~of the form $x\mapsto a_i x + b_i$ for some $a_i,b_i \in \mathbb{R}$.
\end{definition}

\begin{definition}[{\bf MIETs}]\label{def:MIET}
A \emph{Moebius IET} $T$ is a generalized interval exchange transformation $T$  such that the branches $T_i$, for $1\leq i\leq d$,  are  restrictions of \emph{Moebius maps}, i.e.~maps of the form
$$ x \mapsto m(x):=\frac{ax +b}{cx +d}, \qquad \text{where}\quad ad - bc>0.$$
\end{definition}

\noindent  

Interval exchange transformations appear naturally as Poincar{\'e} first return maps of  orientable foliations on a surface on transversal segments. The discontinuities arise indeed from points on the interval which hit a singularity of the foliation (or an endpoint of the transversal interval) and therefore do not return to the transversal, while the intervals $I_j^t$ are continuity intervals of the Poincar{\'e} map. The smoothness $r$ of the branches depends on the regularity of the foliation. When the foliation is a measured foliation, one can choose coordinates so that the Poincar{\'e} map is a standard IET, while affine IETs are Poincar{\'e} maps of \emph{dilation surfaces} (see for example the survey \cite{Selim:survey}).

\subsubsection{Combinatorial data}\label{combinatorics}
To encode the \emph{order} of the intervals (from left to right) at the top and bottom partition of a GIET, we adopt the convention (which became standard after its introduction in \cite{MMY}) of using \emph{two permutations}, $\pi_t$ and $\pi_b$ of $\{1,\dots, d\})$: $\pi_t$ (resp.~$\pi_b$) describes the order of the intervals in the top (resp.~bottom) partition, so that, in order from left to right, they are 
$$I^t_{\pi_t(1)}, I^t_{\pi_t(2)},\dots, I^t_{\pi_t(d)} \qquad \text{(resp.~} I^b_{\pi_b(1)}, I^b_{\pi_b(2)},\dots,I^b_{\pi_b(d)}).$$ 
We call the datum $\pi:=(\pi_t,\pi_b)$ of these pairs of permutation the \emph{combinatorial datum} of $T$, or simply the \emph{permutation} of $T$ (by abusing the terminology, even though it is a actually a \emph{pair} of permutations). The composition {$\pi_b^{-1}\circ \pi_t$}, also called \emph{monodromy}, is a permutation (in the classical sense) which encodes how $T$ rearranges the partition intervals. The choice of keeping track of a \emph{pair} of permutations (instead than only the monodromy, which was used classically for IETs (see \cite{Ve:gau} or \cite{Ra:ech}), 
 allows to keep track of labels of intervals and plays a crucial role  in the definition of irrationality of rotation numbers of GIETs (see Definition~\ref{def:irrational}).


\smallskip
\noindent We will assume that the combinatorial datum is \emph{irreducible}, i.e.~for every $1\leq k<d$ we have
$$
{ \pi_t \{ 1,\dots, k\} \neq  \pi_b \{ 1,\dots, k\}}
$$
(this guarantees in particular that the GIETs cannot be reduced to a GIET of a smaller number of exchanged intervals). We will denote by $\mathfrak{S}_d^0$ the set of irreducible combinatorial data $\pi=(\pi_t,\pi_b)$ with $d$ symbols.

\subsubsection{Singularities}\label{sing}
We denote by $u_i^t$, for $0\leq i\leq d$ the endpoints of the top partition intervals and, respectively by  $u_i^b$, $0\leq i\leq d$, the endpoints of the bottom partition, in their natural order, so that
\begin{align*}
& 0:= u^t_0 < u^{t}_1 <   \cdots <   u^t_{{d-1}}<  u^t_{{d}}:= 1;\\
& 0:= u^t_0 < u^{t_1 }<   \cdots <   u^t_{{d-1}}<  u^t_{{d}}:= 1.
\end{align*}
Then, with the chosen conventions,  we have
{
$$
I^t_{\pi_t(j)}= (u^t_{j-1}, u^t_{j}), \qquad I^b_{\pi_b(j)}= (u^t_{j-1}, u^t_{j}), \qquad \text{for}\ 1\leq j\leq d.
$$
}
We will denote by $|J|$ the length (with respect to the Lebesgue measure) of an interval $J\subset I$, so $|I^t_{\pi_t(j)}| = u_j^t- u_{j-1}^t$ and $|I^b_{\pi_b(j)}| = u_j^b - u_{j-1}^b$. 
The points $u^t_1, \dots, u^t_{d-1}$ separating the top intervals are called the \emph{singularities} 
of $T$. The points $u^b_1, \dots, u^b_{d-1}$ are the singularities of $T^{-1}$. 
 

\subsubsection{Connections}\label{connections} 
A \emph{connection} is  a triple $(u^t_j, u^b_i, m)$ where $m$ is a positive integer such that $T^m (u^b_j)=v^t_i$.  Thus, a connection encodes a finite orbit whose starting point and end point belong to the set of endpoints  points $\{ u_0^t, u_1^t, \cdots, u_d^t \}$. We say that $T$ has \emph{no connections} if no such triple exists. This condition is also called \emph{infinite distinct orbit condition} or \emph{Keane condition} for standard IETs. Keane indeed proved that a standard IET with irreducible $\pi$ and no connections is minimal. 
When $T$ is the Poincar{\'e} map of a transveral a flow along the leaves of a foliation on $S$, connections correspond to \emph{saddle connections} on $S$, i.e.~trajectories of the flow which connect two singularities. Thus, if the flow along the leaves of the foliation has no saddle connections, any GIET obtained as Poincar{\'e} map has no connections.

\subsubsection{GIETs, surfaces and foliations}\label{gietandfoliations}
A generalized interval exchange map $T$ can be \emph{suspended} \cite{Ma:int, Yoc:Clay} (see also Appendix~\ref{sec:boundarycomb}) to an orientable (singular) \textit{foliation} $\mathcal{F} = \mathcal{F}(T)$ on a closed oriented surface $S = S(T)$ such that the singular points of $\mathcal{F}$ are {(possibly degenerate) saddles (with an even number of prongs)}; we can make this construction so that all the discontinuity points of $T$ belong to singular leaves of the foliation  (see  Appendix~\ref{sec:boundarycomb}). Furthermore, if the permutation is \emph{irreducible}, both endpoints also belong to a singular leaf. If $T$ is minimal, the associated foliation is {\it minimal} in the sense that all regular leaves are dense (see Definition~\ref{def:minimal}  in \S~\ref{foliations}). 
We denote by $\Sing = \Sing(T) \subset S$ or $\Sing (S)\subset S$ the the set of saddle points of $\mathcal{F}$. If $g$ is the genus of $S$ and $\kappa = |\Sing(S)|$ we have the following equality 
$$ d = 2g + \kappa - 1.$$ 
Notice that both the genus $g$ and the number $\kappa$ of singularities can be recovered purely combinatorially from the knowledege of the combinatorial datum $\pi$, see \cite{MarmiYoccoz, Yoc:Clay} or \cite{Vi:IET} or Appendix~\ref{sec:boundarycomb} for details. Conversely, $T$ can be recovered from $\mathcal{F}$ by considering a first-return map on a suitably chosen transverse arc $J$  in $S$ joining two singularities in $\mathrm{Sing}(S)$, which we can identify with the interval $[0,1]$ {(see also Lemma~\ref{prop:folgiet})}. We develop this foliation point of view further in Section \ref{foliations}. Notice that choosing  $J$ with endpoints at singularities, or on singular leaves, guarantees that the number of exchanged intervals of $T $ is as small as possible and equal to $d$.  
 
\smallskip
Let $\mathcal{F} $ be a foliation on $S$ which suspends $T$. We can associate to each discontinuity of $T$ a singularity as follows. 
Let us say that $\mathcal{F} $ is a \emph{standard suspension} if both endpoints\footnote{{Usually, in the IETs and translation surfaces literature, one often uses as suspensions the \emph{zippered rectangles} introduced by Veech \cite{Ve:gau} (see e.g.~\cite{Yoc:Clay} or \cite{Vi:IET}) or the polygonal suspensions by Masur \cite{Ma:int}. In zippered rectangles, one endpoint is always at a singularity, while the other is usually not and belong to a separatrix, i.e. a singular leave, that is either ingoing or outgoing. From the foliation point of view, all suspensions whose end points are on singular leaves are equivalent and one can simply slide the singularity and assume that it is at an endpoint. This is convenient since it makes it easier to identify singularities of the foliation with endpoints and discontinuities of the GIETs. Standard (translation surfaces) suspensions are for example explicitely defined in \cite{MarmiYoccoz}; { the construction is included in Appendix~\ref{sec:boundarycomb}}.}}
  suspension  of $I$ are singular points of $\mathcal{F}$. In this case, all the singularities of $T$ are obtained by pulling-back a singular leaf\footnote{If $\mathcal{F}$ is not standard, there can be singularities which are created by an endpoint of the section, namely which belong to the leave passing through an endpoint.}. Then  we can define a map $s$ from the set $\{ u_0^t , \dots , u_d^t \}$ of singularities  and endpoints of $T$ to the singularity set $\Sing(S)$ simply associating to the endpoints $u_0^t $ and $u_d^t$ the corresponding singularity (i.e.~the endpoint of the section in $S$) and to all other $u_i^t$, $1\leq i<d$, the singularity, that we will denote $s(u_i^t)$, that is reached when following  the  oriented leaf emanating from $u_i^t$.

\subsection{Parameter (sub)spaces} \label{parameters}
\noindent We define, for a given differentiability class $r \in \R_+$ and number of intervals $d\geq 2$ the space $\mathcal{X}^r $ 
 of generalized interval exchange transformations of class $\mathcal{C}^r$ with $d$ intervals, namely
$$ \mathcal{X}^r := \bigcup_{\pi \in \mathfrak{S}_d^0}{\mathcal{X}_{\pi}^r }, \qquad \text{where} \quad \mathcal{X}_{\pi}^r: = \{  T, \quad T\ \text{d-GIET of class}\ \mathcal{C}^r \ \text{with associated permation} \ \pi  \}.$$ 
\noindent When there is no ambiguity on the differentiability class $r$, we denote  $\mathcal{X}^r_{\pi}$ and $\mathcal{X}^r$ simply $ \mathcal{X}_{\pi}$ and $ \mathcal{X}$ respectively.

\smallskip
The space of \emph{(standard)} interval exchange transformations (respectively, \emph{affine} interval exchange transformations or \emph{Moebius} interval exchange transformations) 
 with combinatorics $\pi$ will be denoted by $\mathcal{I}_\pi$ (respectively, $\mathcal{A}_\pi$ or  $\mathcal{M}_\pi$) and are subspaces of $\mathcal{X}^r $  for every $r\geq 0$. Similarly, for any $d\geq 2$, let us set
$$
 \mathcal{I}_d := \bigcup_{\pi \in \mathfrak{S}_d}{\mathcal{I}_{\pi} }, \qquad   \mathcal{A}_d := \bigcup_{\pi \in \mathfrak{S}_d}{\mathcal{X}_{\pi}^r }, \qquad  \mathcal{M}_d :=\bigcup_{\pi \in \mathfrak{S}_d}{ \mathcal{M}_{\pi} }. 
$$
Clearly, for every $r>0$, we have the inclusions
$$
 \mathcal{I}_d \subset  \mathcal{A}_d \subset  \mathcal{M}_d   \subset \mathcal{X}^r_d.
$$

\subsubsection{Parameter spaces of IETs} \label{IETs}
If $T\in \mathcal{I}_d$ is a (standard) IET, $T$ is completely determined from the combinatorial datum $\pi$ and the lengths of the top intervals $I^t_j$, $1\leq j\leq d$.  Denote by $\lambda_1, \cdots, \lambda_d$ the lengths of its continuity intervals, so that $\lambda_j:= |I^t_j|= v^t_j-v^t_{j-1}$. Because the top intervals form a partition of $[0,1]$, the lengths must satisfy the following equation:
\begin{equation}\label{simplex} \lambda_1 + \cdots + \lambda_d = 1.\end{equation}
We will denote by $\Delta_{d-1}$ (for $d-1$ dimensional simplex) the set of vectors in $\mathbb{R}^d_+$ which satisfy \eqref{simplex}. We denote by $\lambda(T)$  and call \emph{lengths vector} the vector whose components are lengths of top intervals, namely 
$$\lambda(T): = (\lambda_1,\dots, \lambda_d)= (|I^t_1|, \dots, |I^t_d|)\in \Delta_{d-1}.
$$
Thus, the subspace  $\mathcal{I}_d$ of $d-$IETs is parametrized by $\Delta_{d-1}\times \mathfrak{S}_d$  and $\mathcal{I}_{d}$ is a (finite union of) submanifold(s) of $\mathbb{R}^{d}$ of dimension $d-1$. 


\subsubsection{Parameters of AIETs}\label{AIET}
Let $T$ be an AIET with combinatorial datum $\pi$. Let $\lambda\in \Delta_{d-1}$ be as before the vector of lengths of top intervals. 
 If we denote by $\rho_1, \cdots, \rho_d$ the derivatives of $T$ on intervals of respective lengths $\lambda_1, \cdots, \lambda_d$, we have that for each $1\leq j\leq d$, the length $|I_j^b|$ of the bottom interval $I_j^b=T(I_j^t)$  is $ \rho_j\lambda_j$. Therefore, since also the bottom intervals form a partition, the lengths must also satisfy
$$  \rho_1\lambda_1 + \cdots + \rho_d\lambda_d = 1.$$ 
This equation, together with \eqref{simplex} and the further restrictions that $\forall i, \lambda_i > 0$ identify $\mathcal{A}_{\pi}$ to a submanifold of $\mathbb{R}^{2d}$ of dimension $2d-2$. For any affine interval exchange transformation $T$, we denote by $\lambda(T)$ and $\rho(T)$ respectively
\begin{align*}
\lambda(T) &:= (\lambda_1,\dots, \lambda_d)= (|I^t_1|, \dots, |I^t_d|), \\  \rho(T) & := (\rho_1,\dots, \rho_d)= (D T_1(x_1), \dots, DT_d(x_d)),\qquad x_j\in I_j\, \quad \text{for\ all}\ 1\leq j\leq d,
\end{align*}
where $DT_j(x_j):=T_j'(x)$ is the value of the derivative of the branch $T_j$ of $T$ at any point $x_j$ in the interval $I_j$ and is independent on the choice of $x_j$ since in an affine IET $T'$ is locally constant on $I_j$. 
We call $\lambda(T)$ the \emph{length vector} and $\rho(T)$ the \emph{slope vector} of $T$.


\subsubsection{Shape-profile coordinates for GIETs} \label{coordinates}
 We introduce now a set of coordinates on  $ \mathcal{X}_{\pi}^r$  which allow us to endow  $ \mathcal{X}_{\pi}$  and consequently $ \mathcal{X}^r$ with the structure of a Banach manifold. These coordinates, that we will call \emph{shape-profile} coordinates, where first introduced and used by the first author in \cite{Selim:loc}
 and will play a central role also in the present paper.

 Let $T$ be a $\mathcal{C}^{r}$-GIET, with associated permutation $\pi$ and let $(I_i^t)_{1 \leq i\leq d}$ and $(I_i^b)_{1 \leq i\leq d}$ be the \emph{top} and \emph{bottom} partitions of $[0,1]$ associated to it. We make the two following observations.

\begin{enumerate}
\item  There is  a unique affine interval exchange transformation $A_T$ mapping $I_i^t$ to $I_i^b$. 

\item Furthermore, for all $1\leq i\leq d$, there is a unique element $\varphi_T^i$ of $\mathrm{Diff}^r([0,1])$  such that the restriction of $T$ to $I_i^t$ is equal to 
$$\varphi_T^i:= \mathcal{N}(T_i ):= a_i \circ  T_i \circ b_i $$ where $b_i$ is the unique orientation preserving affine map mapping $I_i^t$ onto $[0,1]$ and  $a_i$ is the unique orientation preserving affine map mapping $[0,1]$ onto $I_{\pi(i)}^b$. 
\end{enumerate}
Notice that, using these coordinates, if $\rho=\rho(A_T) = (\rho_1, \dots , \rho_d)$  is the slope vector of the shape  $A_T$, one has\footnote{This follows from the explicit expression of the GIET in terms of the profiles which is given by 
$$
T(x) = u^b_{{\pi_b(i)-1}}+ |I^b_i| \, \varphi^i_T (b_i(x)), \ \text{where}\ b_i(x)={(x -u^t_{\pi_t(i)-1} )}/{|I^t_i|}, \qquad \text{for\ all}\ x\in I^t_i=(u^t_{\pi_t(i)-1}, u^t_{\pi_t(i)}), 
\quad 1\leq i\leq d.
$$
noticing that $DA_T(x)=|I_i^b|/|I^t_i|$ for every $x\in I^t_i$ and therefore $\rho_i=|I_i^b|/|I^t_i|$.}
\begin{equation}\label{Dincoordinates}
DT(x)=DT_i(x) = \rho_i \, D \varphi^i_T (b_i(x)),   \qquad \text{for\ all}\ x\in I^t_i,\ 1\leq i\leq d, 
\end{equation}
(where $b_i$ is as above the affine map which maps $I^t_i$ to $[0,1]$).

\smallskip
\noindent The operation of associated to a GIET a shape and a profile can be inverted and therefore the map:
$$ T \longmapsto (A_T, \varphi_T),\qquad \varphi_T:=(\varphi_T^1, \cdots, \varphi_T^d) $$ gives an identification between $\mathcal{X}_{\pi}^r$ and $\mathcal{A}_{\pi} \times (\mathrm{Diff}^r([0,1]))^d$ where $\mathcal{A}_{\pi}$ the space of AIETs with permutation $\pi$.  The AIET which appears as projection  onto the first  coordinate, namely  $A_T$, will be called the \emph{shape} of $T$; the vector $\varphi_T:=(\varphi_T^1, \cdots, \varphi_T^d)$ will be called \emph{profile} of $T$.

We denote by $\mathcal{P}_d^r$ (or simply $\mathcal{P}_d$ or $\mathcal{P}$ when the regularity $r$ or the number of intervals $d$ are not relevant or clear from the context) the space $\big(\mathrm{Diff}^r([0,1])\big)^d$ and so we have a canonical identification 
$$ \mathcal{X}_{\pi}^r  = \mathcal{A}_{\pi} \times \mathcal{P}^r_d, \qquad  \mathcal{X}^r_d  = \mathcal{A}_d \times \mathcal{P}^r_d= \bigcup_{\pi\in \mathfrak{S}^0} \mathcal{A}_{\pi} \times \mathcal{P}_d^r .$$ Using this parametrisation, we can endow $\mathcal{X}_{d}^r$ with the structure of a Banach manifold directly inherited from that of  $\mathrm{Diff}^r([0,1])$. Again, where there is no possible ambiguity, we will drop the indexes $\pi$ and $r$ and simply write  $ \mathcal{X} = \mathcal{A} \times \mathcal{P}$.  

\smallskip
\noindent In the sequel we use the following notation for $f : [0,1] \longrightarrow \mathbb{R}$ of class $\mathcal{C}^r$, 
$$ ||f||_{\mathcal{C}^r} = \max_{0\leq i \leq d}{||f^{(i)}||_\infty} $$ where $f^{(i)}$ is the $i$-th derivative of $f$ and $|| \cdot ||_{\infty}$ denotes the sup norm. We extend this norm to $(\mathcal{C}^r([0,1], \mathbb{R}))^d$ simply by taking the sum of the norms on each coordinate. 


\subsection{Renormalization of GIETs}\label{sec:renormalization} 
We introduce now the renormalization operator $\mathcal{V}$ on the space  $\mathcal{X}^r $ of GIET  defined (on the subspace of GIETs with no connections) by Rauzy-Veech induction. Rauzy-induction as a tool to study standard IETs and their ergodic properties appears for the first time in the seminal works by Rauzy \cite{Ra:ech} and Veech \cite{Ve:gau, Ve:inI} and has been a standard tool in the theory since then (see~e.g.~\cite{Zo:dev, AF:wea,  Ul:mix, Ul:abs, KKU,  Ch:dis}, $\dots$). Rauzy-Veech induction as a tool to study GIETs and the notion of rotation number of a GIET appear e.g.~in the works \cite{MMY3, MarmiYoccoz}, see also \cite{Yoc:Clay}. 
This section follows partly \cite{MMY}.

\subsubsection{Elementary step of Rauzy-Veech induction}\label{RV}
Let $T$ be a GIET on $d$ intervals (as in Definition \ref{def1}). Consider the partition endpoints $u^t_j$ and $u^b_j$, see \S~\ref{sing}. Let ${\lambda}_1 := \max\{u_{d-1}^t, u_{d-1}^b\}$.  Thus $[0,\lambda_1]$ is the interval $[0,1]\backslash I^t_{\pi_t(d)}$ if the last interval before the exchange $I^t_{\pi_t(d)}$ is shorter than the last interval after the exchange $I^b_{\pi^t(d)}$, while $[0,\lambda_1]=[0,1]\backslash I^b_{\pi_b(d)}$ otherwise, i.e.~if $|I^b_{\pi_b(d)}|< |I^t_{\pi_t(d)}|$. 

We define $T_1$ to be the first-return of $T$ on the interval $[0,\lambda_1]$. One can  verify that $T_1$ is well defined and is also a GIET on $d$ intervals provided $u_{d-1}^t \neq u_{d-1}^b$. Define $\mathcal{V}(T)$ to be $T_1$ \textit{normalised} to be a map  whose range is $[0,1]$: formally $\mathcal{V}(T):=\mathcal{N}(T_1)$ where $\mathcal{N}(T_1)$ is obtained conjugating $T_1$ by the unique affine map mapping $[0,l]$  to $[0,1]$, namely
\begin{equation}\label{renormalizationeq}
\mathcal{V}(T) := \mathcal{N}(T_1), \qquad \text{where}\quad  \mathcal{N}(T_1)(x):=\frac{1}{\lambda_1}\, T_1 \big( \lambda_1 x\big), \qquad \textrm{for\ all}\ x\in [0,1].
\end{equation}
 The operation consisting in passing from $T$ to $\mV(T)$ is called the \textit{elementary step of the Rauzy-Veech induction}.

 When one performs the elementary step of the induction, the associated permutation $\pi$ changes. Note that the new permutation only depends on the initial one and on whether $u_{d-1}^t > u_{d-1}^b$ or not. If $u_{d-1}^t < u_{d-1}^b$ we say that the \textit{bottom interval wins} and we say that the \textit{top interval wins} otherwise. Furthemore, we can record the label of the interval which wins: if the interval which wins is $I^t_j$ (i.e.~the top interval wins and $I^t_{\pi_t(d)}= I^t_j$ we say that $j$ is the \emph{winner}). Similarly we say that $j$ is the \emph{winner} also if $I^b_j$ wins, i.e.~the bottom interval wins and $I^b_{\pi_b(d)}= I^b_j$.

One can show that if $T$ has no connections, also $T_1$ has no connections and therefore it is possible to apply again an elementary step of the Rauzy-Veech induction to $T_1$. Thus, if a GIET $T$ has \textit{no connections}, Rauzy-Veech induction can be iterated infinitely many times.


\subsubsection{\it Paths on Rauzy diagrams and rotation numbers} \label{rotnumber}
We define now the notion of \emph{rotation number} associated to a GIET, see e.g.~\cite{Yoc:CF, Yoc:Clay}. The rotation number will be an infinite path on a combinatorial graph describing the moves of the renormalization algorithm (see Defiintion~\ref{def:rotnumber}).

We can form an oriented graph whose set of vertices is $\mathfrak{S}_d$ and there is an oriented edge from one permutation $\pi_1$ to $\pi_2$  if and only if there is a GIET with permutation $\pi_1$ whose image by the elementary step of the Rauzy-Veech induction is a GIET with permutation $\pi_2$. This oriented graph is called the \emph{Rauzy diagram}. It has a certain number of connected components (which were classified by Zorich in \cite{Zo:Rau}) and are classically  called \emph{Rauzy classes}. 

 If a GIET has \textit{no connections},  so then it is possible to iterate Rauzy-Veech induction infinitely many times and get an infinite sequence $ T, \mV(T), \mV^2(T), \cdots, \mV^{n}(T), \cdots$. For every $n\in\mathbb{N}$ let $\pi_n$ be the combinatorial datum of $\mV^{n}(T)$ and let $\gamma_n$ be   the arrow from $\pi_{n}$ to $\pi_{n+1}$ which corresponds to the elementary step to pass from $\mV^n(T)$ to $\mV^{n+1}(T)$. To the infinite orbit $\{ \mV^2(T),  n\in\mathbb{N}\}$ we can associate a path $\gamma(T):= \gamma_0\gamma_1\cdots \gamma_n\cdots $ in the associated Rauzy diagram  passing through the vertices
$ \pi_0, \pi_1, \cdots, \pi_n, \cdots $ obtained concatenating the arrows $\gamma_n$ describing the moves of the algorithm.

\begin{definition}[{\color{black}Combinatorial} rotation number]\label{def:rotnumber}
Given a GIET n $T$ with no connections, {\color{black}its \emph{combinatorial rotation number} (or simply \emph{rotation number})} is the datum of the Rauzy path $\gamma(T)=\gamma_1\gamma_2\cdots \gamma_n\cdots $  associated to the orbit $\{ \mV^n(T),  n\in\mathbb{N}\}$.
\end{definition}
\noindent {\color{black}The terminology \textit{rotation number}\footnote{{\color{black}We added the adjective \emph{combinatorial} since it gives a conjugacy invariant which describes the combinatorial structure of orbits (as in other examples from one-dimensional dynamics, like for example the \emph{kneading sequences} for unimodal maps) as well as to distinguish it from other possible generalizations of the notion of rotation number, such as \emph{rotation vectors} for higher dimensional tori or the \emph{Katok fundamental class} which generalizes the \emph{asymptotic cycle} role also played by the rotation number.}} (which was used in the works by Marmi-Moussa and Yoccoz and advertised in the lecture notes by Yoccoz, see~\cite{Yoc:CF, Yoc:Clay}) has been chosen because} for $d=2$, i.~e.~for GIETs with $d=2$ intervals, which correspond to circle homeomorphisms,  this piece of data is equivalent to the datum of the usual rotation number. 
Furthermore, if $T_1$ and $T_2$ are (semi)conjugate, then $\gamma(T_1)=\gamma(T_2)$, i.e.~the rotation number is an invariant of the (semi)conjugacy class. The converse (see Theorem~\ref{thm:PY} below) is also true for \emph{irrational} rotation numbers (to be defined below, see Definition~\ref{def:irrational}), a perhaps more important reason which further supports that it is a good analogue of the classical rotation number. 

\subsubsection{Irrational combinatorial rotation numbers and semi-conjugacy with a standard IET}\label{sec:irrational}
If $T_0$ is a standard IET with no connection, its combinatorial rotation number $\gamma(T_0)$ has an additional property: all indexes $j$ in the alphabet $\mathcal{A}=\{1,\dots, d\} $ of \emph{labels} of intervals are \emph{winners} infinitely many times. This property characterizes paths on the Rauzy diagram which come from standard IETs (or GIETs which are semi-conjugated to standard IETs, see above). 
In particular, if a path $\gamma$ on a Rauzy-diagram has this property that every $j$ appears infinitely many times as a winner of an arrow of $\gamma$ (a path on the Rauzy-diagram with this property is called $\infty$-complete, see \cite{MMY, Yoc:Clay, Yoc:cours}), then there exists\footnote{The IET is not necessarily unique, but any two standard IETs $T_0$ and $T_1$ with $\gamma(T_0)=\gamma(T_1)$ are topologically conjugated, see \cite{Yoc:cours}.} a standard IET $T_0$ such that $\gamma(T_0)=\gamma$.  

Following \cite{MMY, MMY3}, we give the following definition. Let $T$ be a GIET with no connection and $\gamma(T)$ its rotation number.

\begin{definition}[irrational or ($\infty$-complete) rotation numbers]\label{def:irrational}
The (combinatorial) rotation number $\gamma(T)$ is said to be \emph{irrational} (or $\infty$-\emph{complete}) iff every $j\in \mathcal{A}=\{1,\dots, d\}$ is the winner of infinitely many arrows of $\gamma(T)$.
\end{definition}

\noindent The following result, which is proved in the notes \cite{Yoc:Clay} by Yoccoz, extends Poincar{\'e} theorem for circle diffeomorphisms and  further supports the terminology 'rotation number'. We will refer to it as Poincar{\'e} theorem for GIETs.

\begin{thm}[Poincar{\'e} theorem for GIETs, see \cite{Yoc:CF, Yoc:Clay}]\label{thm:PY}
Let $T$ be a GIET with $\infty$-complete rotation number and let $T_0$ be an IET with same rotation number. Then $T$ is semi-conjugate to $T_0$.
\end{thm}
\noindent In the rest of the paper, since we are interested in GIETs which are semi-conjugated to a standard IET, we will always work GIETs with \emph{irrational} rotation number in the sense of Definition~\ref{def:irrational}.

\subsubsection{Periodic-type (or Fibonacci-type) combinatorics} \label{sec:periodictype}
We can now define also GIET of \emph{periodic type} (which are analogous to maps with \emph{Fibonacci type combinatorics} in the one-dimensional dynamics literature): 
\begin{definition}[Periodic type]\label{def:periodictype}
A GIET $T$ is called of \emph{periodic type} if it has no connections and its combinatorial rotation number $\gamma(T)$ is \emph{irrational} and \emph{periodic}, i.e.~there exists a $p>0$ such that $\gamma_{n+p}=\gamma_n$ for every $n\in\mathbb{N}$. The minimal $p$ with such property will be called the \emph{period} of $\gamma(T)$.
\end{definition}
\vspace{2mm}

\subsubsection{Definition of the renormalisation operator}\label{sec:renormalizationop}
\noindent The elementary step of the Rauzy-Veech induction can be used to define an operator acting on an open subset of $\mathcal{X}^r$ defined the following way. Set
$$\mathcal{Y}^r = \{ T \in \mathcal{X}^r \ | \ u_{d-1}^t(T) = u_{d-1}^b(T)   \}.$$ Note that $\mathcal{Y}^r $ is a codimension $1$ smooth submanifold of $\mathcal{X}^r$. In other words,  $\mathcal{Y}^r $ is the subset of those GIETs for which the rightmost top and bottom intervals have same length. It is exactly the set for which the elementary step of the Rauzy-Veech induction is not defined.  Thus the elementary step of the Rauzy-Veech induction defines an operator $\mathcal{V}:\mathcal{X}^r \setminus \mathcal{Y}^r \mapsto \mathcal{X}^r$ given by $T   \longmapsto  \mathcal{V}(T)$.

\smallskip 
A GIET $T$ is said to be \textit{infinitely renormalizable} iff 
$\mathcal{V}^n(T)$ if well-defined for all $n \in \mathbb{N}$. Note that in particular a GIET with irrational (i.e.~$\infty$-complete, see Definition~\ref{def:irrational}) rotation number is infinitely renormalizable. However, not all infinitely renormalizable GIETs have irrational ($\infty$-complete) rotation number (actually in parameter space, the "generic case" is expected not to have irrational combinatorial rotation number).


\subsubsection{Accelerations}\label{sec:accelerations} 
We will consider as renormalization operators $\mathcal{R}$ operators that are obtained \emph{accelerating} $\mathcal{V}$, i.e.~such that $\mathcal{R}^k(T):=\mathcal{V}^{n_k}(T)$ where  is a suitably chosen sequence of iterates of $\mathcal{V}$ (which depends on $T$). For example, when $T$ is of \emph{periodic type} with period $p$ (see Definition~\ref{def:periodictype}), the natural renormalization operator to use is simply $\mathcal{R}:= \mathcal{V}^p$, so that $\mathcal{R}^k(T):=\mathcal{V}^{kp}(T)$ for every $k\in\mathbb{N}$. A classical acceleration of $\mathcal{V}$ is the \emph{Zorich acceleration}, which we will denote by $\mathcal{Z}$ and corresponds to grouping together all successive elementary steps of Rauzy-Veech induction which are equal of type top, or bottom, respectively: given $T$ with irrational $\gamma(T)$, one can show that top and bottom both win infinitely often; therefore, one can define the sequence $(n_k)_{k\in\mathbb{N}}$ such that $\mathcal{Z}^k(T):=\mathcal{V}^{n_k}(T)$ by setting $n_0:=0$ and, recursively, if $n_k$ is such the top (resp.~bottom) interval of $\mathcal{V}^{n_k}(T)$ wins,  setting $n_{k+1}$ to be the first $n>n_k$ such that bottom (resp.~top) interval  of $\mathcal{V}^{n}(T)$  wins. 

Accelerations can  also be obtained considering \emph{inducing} (ie.~first return maps): if $E\subset \mathcal{X}^r$ is a subset, we can obtain an acceleration of $\mathcal{V}$, denoted by $\mathcal{V}_E$, defined on the set of $T\in \mathcal{X}^r$ which visit $E$ infinitely often: if $(n_k)_{k\in\mathbb{N}}$ is the sequence of successive visits of the orbit $(\mathcal{V}^n(T))_{n\in\mathbb{N}}$ to $T$ (i.e.~we set $n_1$ to be the first $n\geq 0$ such that $\mathcal{V}^n(T) \in E$ and, given $n_k$, we set $n_{k+1} $ to be the smallest $n>n_k$ such that $\mathcal{V}^n(T) \in E$), we can defined $\mathcal{V}_E^k(T):= \mathcal{V}^{n_k}(T)$. 


\subsubsection{Dynamical partitions} \label{dynamicalpartitions}
We introduce the notion of \textit{dynamical partitions}. 
Let $T$ be an infinitely renormalizable $T$ and let $\mathcal{R}$ be a renormalization operator obtained by accelerating Rauzy-Veech induction, as described above. Then the orbit $(\mathcal{R}^n(T))_{n\in\mathbb{N}}$ is well defined and, by definition, for every $n\in\mathbb{N}$ the GIET $\mathcal{R}^n(T)$ is obtained  
%
by rescaling 
the first return map of $T$ on an interval of the form $[0,\lambda_n]$. We will denote by $I^{(n)}:=[0,\lambda_n]$ and by $T_n$ the Poincar{\'e} map of $T$ to $I^{(n)}$, so that $\mathcal{R}^n(T) (x)=  T_n(\lambda_n x)/\lambda_n$ or, explicitely, if $I^{t}_j(n)$ denote the continuity intervals for $\mathcal{R}^n(T)$,
\begin{equation}\label{eq:renormalizedmap}
 \mathcal{R}^n(T)(x)= \frac{T_n(\lambda_n x)}{\lambda_n}= \frac{T^{q^{(n)}_j} (\lambda_n x)}{\lambda_n}, \qquad \mathrm{for\ all}\ x\in I^{t}_j(n)= \frac{1}{\lambda_n} I^{(n)}_j.
\end{equation}
Notice that $\{ I^{(n)}, \ n\in\mathbb{N}\} $ are nested intervals with $0$ as a common left endpoint.
By construction $T_n$ is a $d-$GIET. We denote by $I^{(n)}_j$, for $j=1,\dots, d$ its continuity intervals, so that  the interval $I^{(n)}=[0,\lambda_n]$ is partitioned into $I^{(n)} = \cup_{j=1}^d{I^{(n)}_j} $ and for each $1\leq j\leq d$, $T_n$ restricted to $ I^{(n)}_j$ is equal to $T^{q^{(n)}_j}$  where $q_j^{(n)}$ is the first return time of $I^{(n)}_j$ to $I^{(n)}$ under $T$, i.e.~the minimum $q\geq 1$ such that $T^{q}(x)\in I^{(n)}$ for some (hence all) $x\in I^{(n)}_j$. 

\smallskip
\noindent Let us define 
$$ \mathcal{P}_n := \bigcup_{j=1}^d{\mathcal{P}^j_n}, \qquad \text{where}\quad \mathcal{P}^j_n := \{ I^{(n)}_j, T(I^{(n)}_j),  T^2(I^{(n)}_j), \cdots, T^{q^{(n-1)}_j}(I^{(n)}_j)  \}. $$
 One can verify that $\mathcal{P}_n$ is a partition of $[0,1]$ into subintervals  and that, for each $1\leq j\leq d$, the collection  $\mathcal{P}^j_n$ is a \emph{Rohlin tower} by intervals, i.e.~a collection of disjoint intervals which are mapped one into the next by the action of $T$. 

We will say that  $\{ \mathcal{P}_n, \ n\in\mathbb{N}\}$ is the sequence of \emph{dynamical partitions} \emph{associated to} the orbit  $(\mathcal{R}^n)_{n\in\mathbb{N}}$ of $T$ under the renormalization operator $\mathcal{R}$ and, when the renormalization operator is clear, that   $\mathcal{P}_n$  is the {dynamical partition} \emph{of level} $n$.
 We say that the number $q_j^{(n)}$ of intervals in a tower is the \emph{height} of the (Rohlin) tower $\mathcal{P}^j_n$.  
Thus, $\mathcal{P}_n$ also gives a representation of $[0,1]$ as a \emph{skyscraper}, i.e.~a collection of Rohlin towers, for $T$.
Notice that if $n>m$, then the partition  $\mathcal{P}_n$  is a refinement of $\mathcal{P}_m$.

\subsection{One dimensional dynamics toolkit}
We recall here classical and crucial tools in the theory of circle diffeomorphisms and more in general in one-dimensional dynamics, such as non-linearity and Schwarzian derivative.

\subsubsection{Non-linearity} \label{sec:nl}
\label{nonlinearity}
For any $\mathcal{C}^2$ map $f: I \longrightarrow J$ where $I$ and $J$ are open intervals such that $\mathrm{D}f$ does not vanish, one can define the \emph{non-linearity} function $\eta_f$ to be the function $\eta_f:I\to \mathbb{R}$ given by
            $$\eta_f(x):= (\mathrm{D} \log \mathrm{D}f)(x) = \frac{\mathrm{D}^2f(x)}{\mathrm{D}f(x)}.$$
The function $\eta_f$ is called non-linearity since it measures how far $f$ is from being affine and has the property that $\eta_f \equiv 0 $ if and only if $f$ is an affine map. 
Some easy but important consequence for the non-linearity are the following. 

\begin{lemma}[{\bf properties of non-linearity}]\label{lemma:nonlinearity}
Let $f: I \longrightarrow J$, $g : J \longrightarrow K$ be diffeomorphisms of class $\mathcal{C}^2$. Then the following properties hold:
\smallskip
\begin{enumerate}
\item[(i)] \emph{Chain rule for non-linearity}: $\, \eta_{g \circ f} = \eta_f + \mathrm{D}g \,(\eta_{g} \circ f)$;
\medskip
\item[(ii)] \emph{Distribution property}: $\, \int_{I}{\eta_{g \circ f}} = \int_{I}{\eta_f}  +  \int_{J}{\eta_g}$; 
\medskip

\item[(iii)]  \emph{Triangular inequality}:
$\, \int_{I}{|\eta_{g \circ f}|} \leq \int_{I}{|\eta_f|}  +  \int_{J}{|\eta_g|}$.
\smallskip
\end{enumerate}


\end{lemma}
\noindent The second and third points are  consequence of the first, which itself is an application of the chain-rule for differentiable functions. {\color{black}We refer the reader for example to the Appendix of \cite{Martens} for more details.}

\begin{definition}[{\bf mean and total non-linearity of a GIET}]\label{def:Ns} Given a $\mathcal{C}^2$ interval exchange map  $T: [0,1]\to [0,1]$, defined on the continuity intervals $I_j\subset [0,1]$  we  define the non-linearity $\eta_T$ to be the (bounded) piecewise continuous map from $[0,1]$ to $\mathbb{R}$ given by 
$$\eta_T(x):=\eta_{T_j}(x), \qquad \mathrm{if} \ x\in I_j, \quad 1\leq j\leq d,$$ 
where $T_j: I_j \to [0,1]$ are the  branches of $T$ obtained restricting $T$ to its continuity intervals. 
We subsequently define 
$$ \overline{N}(T) := \int_0^1{\eta_T(x)dx}, \qquad |N|(T) := \int_0^1{|\eta_T(x)|dx}.$$ 
We call $\overline{N}(T)$ the \emph{mean non-linearity} of $T$ and $|N|(T)$ the \emph{total non-linearity} of $T$.
\end{definition}

\smallskip
Mean non-linearity and  total non-linearity play an important role in the theory of renormalization, in particular since, seen as functions $\overline{N}(\cdot)$ and $|N|(\cdot)$ on the space $\chi^r$  of GIET (with $r\geq 2$) 
they satisfies   the properties listed in the following proposition.

\begin{proposition}[{\bf properties of mean and total non-linearity}]\label{prop:Nproperties}
For any $r\geq 2$, $\pi \in \mathfrak{S}_r$, the mean non-linearity $\overline{N}(\cdot)$ and the  total non-linearity $ |N|(\cdot )$ have the following properties:
\begin{itemize}
\item[(i)] $|\overline{N}(T)|\leq  |N|(T )$   for every $T\in \mathcal{X}^r$;
\item[(ii)] $\overline{N}(\cdot)$  is invariant under renormalization, i.e.~$\overline{N}(\V(T))=\overline{N}(T)$ for every $T\in \mathcal{X}^r$;
\item[(iii)]  $ |N|(\cdot )$ is decreasing under renormalization, i.e.~$|N|(\V(T))\leq |N|(T)$ for every $T\in \mathcal{X}^r$;
\item[(iv)]  $\overline{N}(T)$ is invariant under rescaling by (restrictions) of affine maps, so that in particular if $a,b$ are  (restrictions of) linear maps, 
$\overline{N}(a\circ T\circ b) = \overline{N}(T)$.
\end{itemize}
\end{proposition}
\begin{proof}
Note that the branches of $\mathcal{V}(T)$  are compositions of restrictions of branches of $T$. These properties are thus easy consequences of Properties $(ii)$ and $(iii)$ of Lemma~\ref{lemma:nonlinearity} 
applied to branches of $\mathcal{V}(T)$.
\end{proof}
\noindent {\color{black}In light of $(iii)$}, the total non-linearity $|N|(T)$ captures how close $T$ is to the set of affine interval maps (see in particular Remark~\ref{rk:nonlinearityasd}). This will play a central role in proofs of convergence of renormalization.

\subsubsection{Distortion bounds}
We now recall standard distortion bounds in one-dimensional dynamics and derive from it an important consequence for the renormalization operator $\mathcal{R}$.

\begin{lemma}[Distorsion bound]
\label{bound1}
Let $T$ be a GIET of class $\mathcal{C}^2$. Let $J \subset [0,1]$ be an interval such that $J, T(J), T^2(J), \cdots, T^{n}(J)$ are pairwise disjoint and do not contain any singularities of $T$. Then we have 
$$\left| \frac{\mathrm{D}(T^n)(x)}{\mathrm{D}(T^n)(y)}\right| \leq \exp{|N|(T)}:=\exp\left(\int_0^1{|\eta_T(x)|dx}\right), \qquad \text{for\ all}\ x,y \in J.$$
\end{lemma}
\noindent The proof is an adaptation to GIETs of the classical proof. We included it in Appendix~\ref{estimates} for complenetess. 

\subsubsection{The Schwarzian derivative}\label{Schwarzian}
We conclude this Section by introducing the Schwarzian derivative which is a most classical tool in one-dimensional dynamics. If $f : I \longrightarrow J$ is a $\mathcal{C}^3$ diffemorphism between two connected intervals $I$ and $J$, define its Schwarzian derivative to be 
$$ \mathrm{S}(f) := \frac{\mathrm{D}^3f}{\mathrm{D}f} - \frac{3}{2} \left(\frac{\mathrm{D}^2f}{\mathrm{D}f}\right)^2.$$ 
Non-linearity and Schwarzian derivative are related by the following equivalent expression for $\mathrm{S}(f)$:
\begin{equation}\label{Sviaeta} \mathrm{S}(f) = \D \eta_f - \frac{1}{2} \eta_f^2.
\end{equation}
The Schwarzian derivative enjoys the two following important properties :
\vspace{1mm}
\begin{itemize}
\item[(S1)] $\mathrm{S}\left(f\right)$ identically vanishes if and only if $f$ is the restriction of a Moebius map to its domain.\vspace{1mm}
\item[(S2)] if $f:I_1 \to I_2$ and $g: I_2\to I_3$, the composition $g\circ f: I_1\to I_3$ satisfies the following \emph{chain rule for the Schwarzian derivative}:
 $$ \mathrm{S}\left(g \circ f\right) = \mathrm{S}(g) \circ f \,(\mathrm{D}f)^2 + \mathrm{S}(f).$$
\end{itemize}
{For $f : I \rightarrow J$, where $I $ and $J$ are intervals, 
 we denote by $\Norm(f)$ the \textit{normalisation} of $f$ given by  $ \Norm(f) := b \circ f \circ a $, where $a$ and $b$ are respectively the only orientation-preserving affine map mapping $[0,1]$ onto $I$ and $J$ onto $[0,1]$. Then, from $(S1)$ and $(S2)$ we can deduce\footnote{Property $(S3)$ follows since by $(S1)$ we have that $\mathrm{S}(a)=\mathrm{S}(b)=0$ and thus, by $(S2)$, $  \mathrm{S}\left(b\circ f\circ a\right) = \mathrm{S}\left( f\circ a \right)= \left( \mathrm{S}( f) \circ a\right) (a')^2$, which gives the desired conclusion since $a'(x)=|I|$ for all $x\in [0,1]$.}  that
\smallskip
\begin{itemize}
\item[(S3)]  $\mathrm{S}\left(\Norm(f)\right) = \mathrm{S}\left(b\circ f\circ a\right) = |I|^2 \, \mathrm{S}(f) \circ a$.
\end{itemize}}

\subsection{The Zorich cocycle}\label{Rauzysec}
In this section we restrict the Rauzy-Veech renormalization $\mathcal{V}$ (which was  defined in \S~\ref{sec:renormalization} 
for GIETs in $\mathcal{X}^r\backslash \mathcal{Y}^r$) 
to the subset $\mathcal{I}_d\subset \mathcal{X}^r$ of \emph{standard} IETs. 
This is the classical setup in which the induction was introduced by Rauzy and Veech \cite{Ra:ech, Ve:gau} and studied 
from the ergodic theory point of view. We introduce the Rauzy-Veech and Zorich cocycle and recall the integrability property and Oseledets theorem for the latter.

\subsubsection{Invariant measures}
\label{sec:measures}
Let us recall that $\mathcal{I}_d$ is isomorphic to $\Delta_{d-1}\times \mathfrak{S}^0_d$ (refer to \S~\ref{IETs}). A natural measure on $\mathcal{I}_d$, which we will call \emph{Lebesgue measure}, is the 
product measure obtained taking the product of the Lebesgue measure on $\mathbb{R}^d$ restricted to the simplex $\Delta_{d-1}$ and the counting measure
\footnote{{\color{black}The counting measure $\delta$ on $\mathfrak{S}^0_d$ is simply the measure defined by setting $\delta(S)$ to be the cardinality of $S$ for any subset $S\subset \mathfrak{S}^0_d$. It is here simply used to put a copy of Lebesgue measure on each \emph{copy} of the simplex $\Delta_d\times \{\pi\}$ indexed by $\pi\in \mathfrak{S}^0_d$.}}
 the measure defined, for any $\pi_1,\pi_2\in \mathfrak{S}^0_d$, by $\delta(\pi_1, \pi_2)=1$ iff $\pi_1=\pi_2$ and $0$ otherwise. Thus, for any $0< \epsilon< 1$, asking that $d_{\mathcal{C}^1}(T_1,T_2)<\epsilon$, where, for $i=1,2$, $T_i $ is a GIET with combinatorial datum $\pi_i$ and shape-profile coordinates $(A_{T_i}, \varphi_{T_i})$, is equivalent to asking that $\pi_1=\pi_2$, $d_{\mathcal{A}}(T_1,T_2)<\epsilon$ and $d_{\mathcal{C}^1}^\mathcal{P}(\varphi_{T_1}, \varphi_{T_2})<\epsilon$.} 
on combinatorial data $\mathfrak{S}^0_d$. We refer to its measure class\footnote{Recall that a \emph{measure class} is an equivalence class of measures which have the same sets of measure zero.} as \emph{Lebesgue measure class}.

The domain of definition of an elementary step of Rauzy-Veech induction $\V$ (acting on standard IETs) is  hence
$$\mathcal{I}_d:= \mathcal{I}_d\backslash \mathcal{Y}^r= \{ T=(\pi,\lambda)\in \mathcal{I}_d, \, \, \text{such\ that\ }  u^t_{d-1} \neq  u^b_{d-1}  \} = \{ T=(\pi,\lambda)\in \mathcal{I}_d, \,\, \text{such\ that\ }   \lambda_{\pi_t(d)} \neq   \lambda_{\pi_b(d)} \} $$
and therefore it is a full measure subset of $\mathcal{I}_d$ with respect to the Lebesgue measure class (defined above).

Le us fix an irreducible $\pi\in \mathfrak{S}_d^0$ (see \S~\ref{combinatorics}) and consider the action of $\V$ restricted to the space $\mathcal{I}_\pi:=\Delta_{d-1} \times  \Ra(\pi)$, where $\Ra(\pi) $ is the Rauzy class\footnote{Let us recall that the Rauzy class of $\pi$ is the subset  of all permutations $\pi'$ of $d$ symbols which appear as permutations  of an IET $T'=(\underline{\lambda}', \pi')$ in the orbit under $\Ra$ of some IET $(\underline{\lambda}', \pi)$  with initial permutation $\pi$.} of  $\pi$ (see \S~\ref{RV}).  
 Veech proved in \cite{Ve:gau} is that the restriction of $\mathcal{V}: \mathcal{I}_\pi\to \mathcal{I}_\pi$ admits an invariant measure which 
 is absolutely continuous with respect to the Lebesgue measure on $\mathcal{I}_d$ (see above),  but which is infinite. Dropping the dependence on $\pi$ (or more precisely on the Rauzy class $\Ra(\pi)$) we will denote this natural measure by $\mu_{\V}$ when $\pi$ is fixed. This seminal result started the study of $\mathcal{V}$ from the ergodic theoretical point of view.  Veech also showed already in \cite{Ve:gau} that $\mathcal{V}: \mathcal{I}_\pi\to \mathcal{I}_\pi$ is conservative and ergodic with respect to $\mu_{\V}$.
 
The acceleration $\mathcal{Z}$ defined by Zorich 
 was introduced to have a \emph{finite} invariant measure: in \cite{Zo:gau} Zorich showed indeed that $\mathcal{Z}$ is defined on a full measure set of  $\mathcal{I}_\pi $ and 
admits a  \emph{finite} invariant measure that we will denote $\mu_{\Zo}$. It follows from the definition and \cite{Ve:gau} that also $\Zo$ is ergodic with respect to $\mu_{\Zo}$. 

\begin{remark}\label{rk:ac}
Since both $\mu_\mathcal{Z}$ and $\mu_\mathcal{V}$ are \emph{absolutely continuous} with respect to the Lebesgue measure  and $\mathcal{I}_d=\cup_{\pi\in \mathfrak{S}_d^0 }\mathcal{I}_\pi $, to show that a property holds for a full measure set of IETs for the Lebesgue measure on $\mathcal{I}_d$ it is sufficient to prove that, for any fixed irreducible combinatorial datum $\pi\in \mathfrak{S}_d^0$, the it holds for $\mu_\mathcal{Z}$ (or $\mu_\mathcal{Z}$) almost every $T$ in $ \mathcal{I}_\pi$.
\end{remark}

\subsubsection{Natural extension.}\label{sec:natextension}
Notice that Rauzy-Veech induction $\mathcal{V}$  is not injective ($\mathcal{V}$ is actually two-to-one) and therefore neither  $\mathcal{V}$ nor its Zorich acceleration $\mathcal{Z}$  are  \emph{invertible} on the space $\mathcal{I}_d$. One can consider, though, its \emph{natural extension}  
$\hat{\mathcal{Z}}$ 
defined on 
the measure space $(\hat{\mathcal{I}}_\pi, \mu_{\hat{\mathcal{Z}}})$: this is a map such that
 $\hat{\mathcal{Z}}$ 
is defined and \emph{invertible} on a full measure 
subset of $ \hat{\mathcal{I}}_\pi$ with respect to the measure  $\mu_{\hat{\mathcal{Z}}}$,
and  
is an \emph{extension} of 
$\mathcal{Z}$ in the sense of ergodic theory, i.e.~there exists a projection  ${p}: \hat{\mathcal{I}}_\pi \to {\mathcal{I}}_\pi$ such that 
$ \hat{p} \hat{\mathcal{Z}} = \mathcal{Z} \circ {p} $ (i.e.~$p$ intertwines the dynamics of $\mathcal{Z}$ and $\hat{\mathcal{Z}}$)  
and the measure ${m}_{\hat{\mathcal{Z}}}$ is the  pull-back of ${m}_{\hat{\mathcal{Z}}}$ via $p$, i.e.~for every measurable set $E\subset \mathcal{I}_\pi$,  ${m}_{\hat{\mathcal{Z}}}(p^{-1}(E))= {m}_{{\mathcal{Z}}}(E)$. 
Notice that the  invariant measure $m_{\hat{\mathcal{Z}}}$  preserved by the map $\hat{\Zo}$  is also finite. 

One can describe explicitely a geometric realization of these natural extensions and the space $\hat{\mathcal{I}}_\pi$ can be identified with the 
 the space of \emph{zippered rectangles} introduced by Veech in \cite{Ve:gau} (consisting of triples $\hat{T}=(\pi, \lambda, \tau)$ where $T=(\pi, \lambda)$ is a standard IET and $\tau$ is a \emph{suspension datum} which contains the information required to define a translation surface which has $T$ as Poincar{\'e} map, as in \S~\ref{gietandfoliations}). We will not make explicit use of this interpretation, so we will simply denote by $\hat{T}$ a point of  $\hat{\mathcal{I}_d}$ such that $p(\hat{T})=T$ (here if $\hat{T}=(\pi, \lambda, \tau)$, $p(\hat{T})=T$ is the IET $T=(\pi,\lambda)$ obtained forgetting the suspension datum $\tau$). 

\subsubsection{Basics on cocycles.}\label{sec:cocycles}
We recall now basic definitions concerning cocycles and their accelerations. We refer the reader for example to \cite{ViaLP} for a comprehensive introduction to cocycles and Lyapunov exponents. 
Let $(X,\mu, F)$ be a discrete dynamical system, where $(X, \mu)$ is a probability space and $ F$ is a $\mu$-measure preserving map on $X$. A a measurable map $A : X  \rightarrow SL(d,\mathbb{C})$ ($d\times d$ invertible matrices) determines a cocycle $A$ on  $(X,\mu, F)$. If we  denote by  $ A_n (x) = A(F^n x)$ and by 
 $A_F^{n}(x) = A_{n-1} (x)  \cdots A_1(x)  A_0 (x) $,  
the following \emph{cocycle identity}
\be \label{cocycleid}
A_F^{m+n}(x) = A_F^{m}(F^n x)A_F^{n}(x)
\ee
holds for all $m,n \in \mathbb{N}$ and  for all $x \in X$. 
If $F$ is invertible, let us set $A_{-n}(x) = A(F^{-n}x)$. The map $A^{-1}(x) = A(x)^{-1}$ gives a cocycle over $F^{-1}$ which we call \emph{inverse cocycle}.
Let us set $A_{-n}(x) =A( F^{-n}x)$.
for $n<0$ we can set $A^{(-n)}(x)= A^{-1}(F^{-1}x) \cdots A^{-1}(F^{-n}x)$,  
 so that (\ref{cocycleid}) holds for all $n,m \in \mathbb{Z}$. Remark that $A^{(-n)}(x)= (A^{(n)}(T^{-n}x))^{-1}$. The \emph{inverse transpose cocycle} $(A^{-1})^T$ is defined by $(A^{-1})^T(x) = (A^{-1})^T(x )$ where $M^T$ denotes the transpose of $M$.


\subsubsection{Induced cocycles and accelerations.}\label{sec:inducing}
If $Y \subset X$ is a measurable subset, the induced map (or first return map, or Poincar{\'e} map) of $F$ on $Y$, which is defined $\mu$-almost everywhere  by Poincar{\'e} recurrence, is the map given by $F^{r_Y(y)}(y)$ where  $r_Y(y):= \min \{ r \st F^r y \in Y\}$. 
The \emph{induced cocycle} $A_Y$ on $Y$ is a cocycle over $(Y, \mu_Y , F_Y)$ where $F_Y$ is the induced map of $F$ on $Y$ and $\mu_Y = \mu/\mu(Y)$  and $A_Y(y)$ is defined for all $y\in Y$ which return to $Y$ and is given by   
\bes
A_Y(y) = A
(F^{r_Y(y) -1 } y
)
\cdots A\left(F y \right) A\left(y\right) ,
\ees
where $r_Y(y)$ is again the first return time. 

The induced cocycle is an \emph{acceleration} of the original  cocycle, i.e. if $\{n_k\}_{k\in \mathbb{N}} $ is the infinite sequence of return times of some $y\in Y$ to $Y$ (i.e. $T^{n} y \in Y$ iff $n=n_k$ for some $k\in \mathbb{N}$ and $n_{k+1}> n_k$) then 
\be\label{acceleration}
(A_Y)_k(y) =A_{n_{k+1}-1 }(y)   \cdots   A_{n_{k}+1}(y) A_{n_k}(y) .   
\ee

We say that $x \in X$ is \emph{recurrent} to $Y$ under $T$ if there exists an infinite increasing sequence $\{n_k\}_{k\in \mathbb{N}} $ such that $T^{n_k} x \in Y$. Let us extend the definition of the induced cocycle $A_Y$ to all $x \in X$ recurrent to $Y$. 
If the sequence  $\{n_k\}_{k\in \mathbb{N}} $ is increasing and  contains all $n\in \mathbb{N}^+$ such that  $T^{n} x \in Y$, let us say that $x $ \emph{recurs to} $Y$ \emph{along}  $\{n_k\}_{k\in \mathbb{N}} $. In this case,  let us set
\bes
A_Y(x) :=  A( y ) A^{n_0}_F (x) , \quad \mathrm{where}\,\,  y:= F^{n_0} x \in Y;  \qquad (A_Y)_{n}(x) :=  (A_Y)_n (y ), \qquad  \mathrm{for}\,\, 
  n \in \mathbb{N}^+. 
\ees
If $F$ is ergodic, $\mu$-a.e.~$x \in X$ is recurrent to $Y$ and hence $A_Y$ is defined on a full measure set of $ X$.

\subsubsection{Integrability}\label{sec:integrability}
Here and in the rest of the paper, we will use the norm  $\norm{A}  = \sum _{ij} |A_{ij}|$  on matrices (more generally), the same results on cocycles hold for any norm on $SL(d,\mathbb{Z})$). Remark that with this choice $\norm{ A } = \norm{ A ^T }$.  

\smallskip
A cocycle over $(X , F, \mu)$ is called \emph{integrable} if $\int_X \ln \|A(x)  \| \ud \mu(x) < \infty$. Integrability is the assumption which allows to apply \emph{Oseledets Theorem}, also known as \emph{multiplicative ergodic theorem} (the conclusion of Oseledets theorem in the setting of the Zorich cocycle is recalled in \S~\ref{sec:Oseledets}; for a more general reference, see e.g.~Section 4 in \cite{ViaLP}). If $A$ is an integrable cocycle over $(X , F, \mu)$ assuming values in $SL(d,\mathbb{Z})$, then one can show that 
the dual cocycle $(A^{-1})^{T}$ and, if $F$ is invertible,  the inverse cocycle $A^{-1}$ over  $(X , F^{-1}, \mu)$  are integrable. Furthermore, 
any induced cocycle $A_Y$ of $A$ on a measurable subset $Y \subset X$ is integrable (see for example \cite{ViaLP}, \S~4.4.1).  

\subsubsection{The Zorich cocycle.}\label{sec:Zorichcocycle}
Consider the Zorich map $\Zo$ on $\mathcal{I}_d$. Given $T=(\lambda, \pi)\in \mathcal{I}_d$ with no connections and irreducible $\pi$,  denote by $\{I^{(k)}, \, {k\in\mathbb{N}}\}$ the sequence of inducing intervals corresponding to  Zorich acceleration $\mathcal{Z}$ of the Rauzy-Veech algorithm $\mathcal{V}$.  Write $\mathcal{Z}^k(T):=(\pi^{(k)}, \lambda^{(k)})$ and 
 let $T_k$ be the (non normalised) induced IET, given by first returns of $T$ to $I^{(k)}$, so that 
$$\mathcal{Z}^k(T)=(\pi^{(k)}, \underline{\lambda}^{(k)}/|I^{(k)}|),\quad \text{where}\ |I^{(k)}|= \sum_{j=1}^d \lambda^{(k)}_j= |\lambda^{(k)}|.$$  
Recall from \S~\ref{dynamicalpartitions} that $T$ can be represented as skyscraper over $I^{(k)}$ and let 
$$
q^{(k)}=q^{(k)}(T):= \left( q_1^{(k)} , \dots, q_d^{(k)} \right)^T,  
$$
be the column vector whose entries $q^{(k)}_j$ are the heights of the Rohlin towers, the vector of \emph{heights} or, equivalently, the vector of first return times, since $q^{(k)}_j$ is also the first return time of $I^{(k)}_j$ to $I$.

For each $T= T^{(0)}$ for which $\Zo (T) = (\pi^{(1)}, \underline{\lambda}^{(1)}/|I^{(1)} |)$ is defined, let us associate to $T$ the matrix $Z=Z(T)$ in $SL(d, \mathbb{Z})$  such that $q^{(1)}= Z \,q^{(0)}$.  The map  $Z$: $X \rightarrow SL(d,\mathbb{Z})$ is a cocycle over $(X, \mu_{\Zo},  \Zo ) $, which we call the \emph{Zorich cocycle}\footnote{Notice that there is also another cocycle, also sometimes called Zorich cocycle, which transforms \emph{lengths} and is actually the transpose inverse of the cocycle here defined.} (also sometimes called \emph{Kontsevich-Zorich} cocycle).  Explicitely, if $T$ has combinatorial datum $\pi=(\pi_t,\pi_b)$ and lengths vector $\lambda$,  
the cocycle $Z=Z(T)$ is given by 
$$
Z= Z(T)=  \begin{cases}
I_d + E_{\pi_b(d) \pi_t(d)} &\qquad \text{if}\, \lambda_{\pi_t(d)}>\lambda_{\pi_b(d)} \quad \text{(ie.~top\ is\ winner)},\\
I_d + E_{ \pi_t(d) \pi_b(d)} &\qquad \text{if}\, \lambda_{\pi_t(d)}<\lambda_{\pi_b(d)} \quad \text{(i.e.~bottom\ is\ winner)},
\end{cases}
$$
where $I_d$ denotes the $d\times d $ identity matrix, while $E_{ij}$ denotes the matrix which has all entries equal to zero, but the entry $ij$ which is equal to $1$. 

Zorich proved in \cite{Zo:gau} that $Z$  is \emph{integrable}. Defining: 
$$\RL{n} = \RL{n} (T) :=  Z(\Zo^n (T)), \qquad \RLp{n} :=\RL{n-1} \cdot \dots \cdot \RL{1} \RL{0}$$  and iterating the above relation, we then get the following matrix product relation:   
\begin{equation} \label{heightsrelation}
 \underline{q}^{(n)} =  \RLp{n}   q^{(0)}, \qquad \mathrm{where} 
 \quad q^{(0)} (T) := \left(1, \dots, 1 \right)^T,
\end{equation}
which gives the heights (i.e.~return times) of the representation of $T$ as a skyscraper over the firt return map $T_n$ to $I^{(n)}$. 
For more general products with $m<n$ we use the notation 
\begin{equation}\label{Zorichproducts}
Q{(m,n)} \doteqdot \RL{n-1}   \RL{n-2}  \, \dots \, 
 \RL{m+1} \RL{m}.\end{equation}
The following \emph{cocycle relation} then holds for any triple of integers $n,m,p$:
\begin{equation}\label{cocyclerel}
Q(n,p)=Q(m,p)\, Q(n,m), \qquad \text{for\ all}\ n<m<p.
\end{equation}	
\noindent Notice that by choice of the norm $|q|= \sum_{j}|q_j|$ on vectors and $\norm{A} = \sum_{i,j}|A_{ij}| $ on matrices and since return times are positive numbers, 
\be\label{normproduct}
\max_j q^{(n)}_j \leq | q^{(n)} |  \leq \Vert Q{(m,n)} \Vert \, \Vert q^{(m)} \Vert, \qquad \text{for\ any}\ m<n.  
\ee

\subsubsection{Dynamical interpretation  of the entries.}\label{sec:entries}
The \emph{entries} of the Zorich cocycle matrices have the following crucial dynamical interpretation: if $T^{(n)}$ is the sequence of IETs obtained inducing on the sequence $I^{(n)}, n\in\mathbb{N}$ of intervals given by the Zorich acceleration (so that $\mathcal{Z}^n (T)$ is obtained normalising $T^{(n)}$ to an interval exchange acting on the unit interval), then  the entry $(Z_n)_{ij}$ of the $n^{th}$ Zorich matrix $Z_n$ gives the number of visits of  the orbit of $x\in I^{(n+1)}_j$ under $T^{(n)}$ to $I^{(n)}_i$ up to its first return to $I^{(n+1)}$. 

\smallskip
Correspondingly, the Zorich cocycle has also an interpretation in terms of \emph{incidence matrices} of Rohlin towers (defined in \S~\ref{dynamicalpartitions}). The Rohlin towers at step $n+1$ can be obtained by a \emph{cutting and stacking}\footnote{We do not give here a precise definition of \emph{cutting and stacking}, which is a standard construction in the study of ergodic theory and in particular of \emph{rank one} and, more in general, \emph{finite rank} dynamical systems.}, which construction from the Rohlin towers at step $n$: more precisely, for any $n\in\mathbb{N}$ and $1\leq i,j\leq d$, the Rohlin tower over $I^{(n)}_j$ is obtained stacking \emph{subtowers} of the Rohlin towers over $I^{(n)}$ (namely sets of the form $\{ T^k J, 0\leq k<q^{(n)}_j\}$ for some subinterval $J\subset I^{(n)}_j$). 
 Then  $ (Z_n)_{ij}$ is the number of subtowers of the Rohlin tower over $I^{(n)}_i$ inside the Rohlin tower over $I^{(n+1)}_j$. It follows that the Rohlin tower over  $I^{(n+1)}_j$ is made by stacking exactly  {\color{black}$\sum_{i=1}^d (Z_n)_{ij}$} subtowers of Rohlin towers of step $n$. Notice that {\color{black}$\sum_{i=1}^d (Z_n)_{ij}$} is the norm of the $j^{th}$ \emph{column} of the matrix $Z_n$.



\subsubsection{Length cocycle}\label{sec:lengthcocycle}
One can check that the length (column) vectors $(\lambda^{(k)})_{k\in\mathbb{N}}$ that give the lengths $\lambda^{(k)}_j=|I^{(k)}_j|$ of the exchanged intervals of the induced map $T_k$ on the sequence of inducing intervals $\{ I^{(k)}, k\in\mathbb{N}\}$ given by the Zorich acceleration also transform via a cocycle and that the cocycle is exactly the transpose inverse $(Z^\dag)^{-1}$ of the Zorich cocyle $Z$. Thus,  we have that
$$
\lambda^{(m)}= Z(m,n)^\dag \lambda^{(n)}= Z(m)^\dag Z(m-1)^\dag \cdots Z(n-1)^\dag \lambda^{(n)}, \qquad \text{for\ every}\  0\leq m<n.
$$
Let us also recall that, for every  cocycle product $B:=Z(0,n)=Z^{(n)}(T)$ , if we define the sub-simplex
\begin{equation}\label{subsimplex}
\Delta_B:= \left\{ \lambda= \frac{B^\dag\lambda}{|B^\dag \lambda|}, \qquad \lambda \in \Delta_{d-1}\right\}\subset \Delta_d\subset \mathbb{R}^+_d,
\end{equation}
(where $B^\dag$ is as above the transpose of $B$), 
then (because of the relation between lengths and Zorich cocycle),  for any  $\lambda'\in \Delta_d$, if $T'$ is the IET with length data $\lambda'$ and same combinatorics $\pi$ than $T$, we have that $Z^{(n)}(T')=Z^{(n)}(T)= B$. 

\subsubsection{Cocycle action on log-slopes of AIETs}\label{slopesaction}
Let us now consider affine interval exchange $T\in \mathcal{A}_d$ that it is infinitely renormalizable, and assume that its  rotation number $\gamma(T)$ is \emph{irrational}. Then, Zorich acceleration is well-defined for $T$ (see~\ref{sec:accelerations}).  
 An important fact is that the action of $\mathcal{Z}$ on the slopes column vector $\rho (T) = (\rho_1, \cdots, \rho_d)^\dag$ satisfies the following: if $\vect{\omega}(T) := \log \rho(T) = \big(\log \rho_1, \cdots, \log \rho_d \big)^\dag$ denotes what we call the \emph{log-slope vector} of $T$ (whose entries are the logarithms of the slopes of $T$ on continuity intervals), we have 
$$ \vect{\omega}(\mathcal{R}T) = Z(T)  \, \vect{\omega}(T), \qquad  \vect{\omega}(\mathcal{R}^n T) = Q(n)  \vect{\omega}(T), \qquad \text{for\ all\ } T \in \mathcal{A}_d, \  n\in \mathbb{N}.$$
Thus, the way log-slope vector $\vect{\omega}(T)$ transform under $\mathcal{V}$ does not depend on the value of $\lambda(T)$ and 
is linear and given by the Zorich cocycle. 

\subsubsection{Lyapunov exponents and Oseledets splittings}\label{sec:Oseledets}
Zorich showed in \cite{Zo:gau} that for every irreducible $\pi\in\mathfrak{S}^0_d$  the Zorich cocycle ${Z}: \mathcal{I}_\pi\to SL(d,\mathbb{Z})$, as well as its traspose dual $(Z^\dag)^{-1}$, are \emph{integrable} (see \S~\ref{sec:integrability}) with respect to the Zorich measure $\mu_{\mathcal{Z}}$ (introduced in \S~\ref{sec:measures}). Let us consider the natural extension $\hat{\mathcal{Z}}:\hat{\mathcal{I}_d}\to \hat{\mathcal{I}_d}$ (see \S~\ref{sec:natextension}) with its invariant measure ${\mu}_{\hat{\mathcal{Z}}}$. The cocycle $Z$ over  $\mathcal{Z}$ can be \emph{extended} to a cocycle, which we still denote by $Z$, over $\hat{\mathcal{Z}}$, by defining $Z:\hat{\mathcal{I}_d}\to SL(d,\mathcal{Z})$ to be given by  $Z(\hat{T}):= Z(p(\hat{T}))$ for every $\hat{T}\in \hat{\mathcal{I}_d}$. Then, $Z$ is now a cocycle over $\mathcal{Z}$ which is still integrable (by definition of $\hat{Z}$ and by construction of $\mu_{\hat{\mu}}$, see in particular property $(ii)$ of the natural extension, see \S~\ref{sec:natextension}). 
Notice also that, by definition of the extension, if $\hat{T}\in p^{-1}(T)$  (since then $Z(\hat{T})= Z(T)$), then $Z_n(T)=Z(\mathcal{Z}^nT)=Z(\hat{\mathcal{Z}}^n\hat{T}) =Z_n(\hat{T})$ for every $n\in\mathbb{N}$.

As a consequence of Oseledets theorem for cocycles over \emph{invertible} transformations and of the celebrated  works\footnote{The \emph{symmetry} of the Lyapunov exponents  (i.e.~the property that for every exponent $\lambda_i$ also $-\lambda_i$ is an exponent), 
is a consequence of the \emph{symplectic} nature of the Zorich cocycle (proved in \cite{Zo:dev}, \cite{Zo:gau}), the \emph{hyperbolicity}, namely the inequality  $\lambda_g>0$, was proved by Forni \cite{Fo:ann} and \emph{simplicity}, namely $\lambda_i< \lambda_{i+1}$  for every $1\leq i<g$ was proved by Avila and Viana \cite{AV:sim}.} \cite{Ve:Teich, Zo:dev, Fo:ann, AV:sim}  it hence follows that there exists $g$ (positive) \emph{Lyapunov exponents} $\lambda_1>\ldots>\lambda_g>0$  (where $g$ is the genus of the suspension $\hat{T}$) such that for $\mu_{\mathcal{Z}}$-almost every $\hat{T}\in \hat{\mathcal{I}}_\pi$, there exists a \emph{splitting}, namely,  denoting by $\kappa$ the number of singularities of the suspension $\hat{T}$ (see \S~\ref{sec:natextension} and \S~\ref{gietandfoliations}), a decomposition  
\begin{equation}\label{splittings} \R^{d}=\bigoplus_{-g\leq i\leq g}E_i(\hat{T}), \qquad \text{where} \ \dim E_i(\hat{T})=\begin{cases}1 & \text{if}\ i\neq 0,\\
\kappa-1 & \text{if}\ i= 0 \end{cases}
\end{equation}
(called \emph{Oseledets splitting}) such that:
\begin{equation}\label{eq:Oseledets}
\lim_{n\to\pm \infty}\frac{1}{n}\log\|Z^{(n)}(\hat{T}) v\|=\begin{cases}\lambda_i&\text{ if }v\in E_i(\hat{T})\text{ and }i>0\\
-\lambda_i&\text{ if }v\in E_i(\hat{T}) \text{ and }i<0,\\
0&\text{ if }v\in E_c(\hat{T}).
\end{cases}
\end{equation}
We define the \emph{stable}, \emph{unstable} and \emph{central}  space for any $n\in\mathbb{N}$ to be respectively 
$$E^{(n)}_{s}(\hat{T})= E_{s}(\mathcal{Z}^n\hat{T}):=  \bigoplus_{-g\leq i\leq 0}E_i(\mathcal{Z}^n\hat{T}) \qquad  E^{(n)}_{s}(\hat{T}) =E_{u}(\mathcal{Z}^n\hat{T}):= \bigoplus_{0\leq i\leq g}E_i(\mathcal{Z}^n\hat{T}), \qquad E_c^{(n)}:= E_c(\mathcal{Z}^n\hat{T}).$$ 
Then invariance of the splitting means that for any $m<n$ we have that 
$$Z(m,n) E^{(m)}_\nu (\mathcal{Z}) = E^{(n)}_\nu (\mathcal{Z}),\qquad \text{where} \ \nu \ \text{in\ any\ index\ in}\ \{u,s,c\}.$$ 
{Furthermore, Oseledets theorem also guarantees a control of the \emph{angle} $\angle$ (see e.g.~Section 4 in \cite{ViaLP}) between stable, unstable and central spaces (where the angle $\angle(V,W)$ between two linear subspaces $V, W\subset {\mathbb{R}^d}$ is defined as the minimum angle $\angle (v,w)$ among all non-zero vectors $v\in V, w\in W$), namely for every $\epsilon>0$ there exists $c=c(\epsilon, \hat{T})>0$ such that
\begin{align}\label{Oseledetsangles}
\lim_{n\to \pm \infty}\frac{ \sin |\angle (E^{(n)}_{\nu_1}(\hat{T}), E^{(n)}_{\nu_2}(\hat{T}))|}{n}= 0 
\qquad \text{for\ all\ \emph{distinct}\ pair\ of\ indexes}\  \nu_1,\nu_2\in \{u,s,c\}.
 \end{align}}
We say that a IET is \emph{Oseledets generic} if it satisfies all the conclusions of Oseledets theorem listed above. 
An application of the Oseledets theorem (or simply ergodicity combined with integrability) also shows that if $T$ is Oseledets generic, then 
\begin{equation}\label{subexpgrowth}
\lim_{n\to\pm \infty}\frac{1}{n}\log\|Z_n(\hat{T})\|= 0.
\end{equation}
[Notice that here we consider the elementary matrix $Z_n(\hat{T})=Z(\hat{\mathcal{Z}}^n\hat{T})$ and not $Z^{(n)}(\hat{T})$ which is a \emph{product} of cocycle matrices.]


{
\subsection{Birkhoff sums and special Birkhoff sums}\label{sec:sums}
The Zorich cocycle plays a crucial role also in the study of Birkhoff sums of functions over an IET,  through the study of \emph{special Birkhoff sums}, which were introduced in the work by Marmi-Moussa-Yoccoz \cite{MMY} as fundamental blocks to decompose and study Birkhoff sums over an IET. We recall here  both the definition of Birkhoff sums and special Birkhoff sums, as well as the connection between special Birkhoff sums and the (extended) Zorich cocycle, while in the next subsection \S~\ref{sec:SBS} we recall how special Birkhoff sums can be used to decompose Birkhoff sums.

\subsubsection{Piecewise continuous functions and Birkhoff sums}\label{sec:obs}
Given a GIET $T$,  let us denote (using a notation inspired by \cite{MMY})
$\mathcal{C}(T):= \mathcal{C}\left(\sqcup_{j=1}^d I^t_j\right)$ (or for short $\mathcal{C}\left(\sqcup I^t_j\right)$,  
 the space of \emph{piecewise continuous} functions $f:[0,1]\to \mathbb{R}$ which are \emph{continuous} on   each continuity interval $I^t_j$, for $1\leq j\leq d$ of $T$ and extend to continuous functions on the closure of each $I^t_j$ (so that right and left limit as $x$ tend to the endpoints of each $I^t_\alpha$ exist).  Notice that as  subspace, it contains the space $\Gamma= \Gamma^{(0)}:=\Gamma\left(\sqcup_{j=1}^d I^t_j\right)$ of functions which are \emph{piecewise constant} on each interval $I^t_j$, $1\leq j\leq d$. 

Given $\mathcal{C}\left(\sqcup_{j=1}^d I^t_j\right)$, the \emph{Birkhoff sum} $S_n f(x)$ of $f$ over $T$ denotes the sum of $f$ along the orbit up of $x$ under $T$ to time $n$, i.e.
$
S_n f(x):= \sum_{j=1}^{n-1}f (T^j x).
$
We stress that in this paper we consider in general Birkhoff sums over a \emph{generalized} IET, while in \cite{MMY} and the consequent works one usually restricts the attention to Birkhoff sums over \emph{standard} IETs.

\subsubsection{Special Birkhoff sums}\label{sec:SBS}
Assume that $T$ is infinitely renormalizable and let $\{I^{(n)}, n\in\mathbb{N}\} $, where $I^{(n)}=[0,a_n]$, be a sequence of intervals obtained by the Zorich (or more in general by a further) acceleration of  Rauzy-Veech induction
 and let $\{ T_n, n\in\mathbb{N}\}$ be the corresponding induced GIETs, $T_n$ being the first return of $T$ on $[0,a_n]$. Let $q^{(n)}={(q^{(n)}_1,\dots, q^{(n)}_d)}^\dag$ be the corresponding vector of first return times, or equivalently, of \emph{heights} of the Rohlin towers in the representation of $T$ as a skyscraper over $T_n$ (see \S~\ref{dynamicalpartitions}).

The (sequence of) \emph{special Birkhoff sums} $f^{(n)}$, $n\in\mathbb{N}$, is  the sequence of functions 
$f^{(n)}: \mathcal{C}(T^{(n)}) = \mathcal{C} (\sqcup I^{(n)}_j)  \to \mathbb{R}$ 
obtained \emph{inducing} $f$ over $T_n$, namely given by
$$
f^{(n)}(x):= S_{q^{(n)}_j}(x) =\sum_{\ell=0}^{q^{(n)}_j-1} f \left(T_n^\ell( x)\right), \qquad \text{if}\ x\in I^{(n)}_j,\quad \text{for\ any}\ 1\leq j\leq d, \ n\in\mathbb{N}.
$$
One can think of $f^{(n)}$ as an \emph{induced function}, obtained inducing the function $f$ on the interval $I^{(n)}$ via the dynamics of the first return map.   
One can check that by construction, for each $n\in\mathbb{N}$, $f^{(n)}$ is continous on each $I^{(n)}_j$, $1\leq j\leq n$ and belongs to $\mathcal{C}(T^{(n)})$. If $x\in I^{(n)}_j$, one can think of $f^{(n)}(x)$ as the Birkhoff sum of $f$ \emph{along the Rohlin tower} of height $q^{(n)}_j$ over $I^{(n)}_j$.

\subsubsection{The extended Zorich cocycle}\label{sec:BS}
In the special case when $f\in \Gamma^{(0)}:= \Gamma(T)$, one then has that $f^{(n)}\in \Gamma^{(n)}:= \Gamma (T^{(n)})$, i.e.~$f^{(n)}$ is piecewise constant on each $I^{(n)}_j$. If we identify each piecewise constant function $f^{(n)}$ in $\Gamma^{(n)}$ with the column \emph{vector}, which by abusing the notation we will still denote by $f^{(n)}$, whose $j^{th}$ entry is the value on (any) point of $I^{(n)}_j$, i.e. 
$$
f^{(n)} = (f^{(n)}(x_1), \dots, f^{(n)}(x_d))^\dag, \qquad \text{where}\ x_j\in I^{(n)}_j, 1\leq j\leq d,
$$
then one can see that $f^{(n)}$ transforms according to the Zorich cocycle defined in \ref{sec:Zorichcocycle} (or the corresponding acceleration if $\{ I^{(n)}, n\in\mathbb{N}\}$ where obtained by an acceleration), i.e.
$$
f^{(n)}= Z^{(n)}f^{(0)}, \qquad f^{(n)}:= Z(n,m) f^{(m)}, \qquad \text{for \ every \ n>m}, n,m\in\mathbb{N}.
$$
where $Z^{(n)}$ and $Q(m,n)$ are the Zorich cocycle matrices defined in \S~\ref{sec:Zorichcocycle} (compare in particular the above relation with 
the height relation \eqref{heightsrelation}).

More in general, given a function $f^{(n)}\in \mathcal{C}(T^{(n)})$ for some $n\in\mathbb{N}$,  we can recover $f^{(n+1)}$ from $f^{(n)}$ and $T_n$ by writing 
\begin{equation}\label{SBSdecomp}
f^{(n+1)} (x) = \sum_{k=0}^{(Z_n)_{ij}-1} f^{(n)} \left(({T}_n)^k(x)\right), \qquad  \text{for\ any}\ x\in I^{(n+1)}_j,
\end{equation}
where $(Z_n)_{ij}$  is the $(i,j)$ entry of the $n^{th}$ Zorich matrix $Z_n$, which, as recalled in \S~\ref{dynamicalpartitions},  gives the number of visits of  the orbit of $x\in I^{(n+1)}_j$ under $T_n$ to $I^{(n)}_i$ up to its first return to $I^{(n+1)}$).  This relation, which can be proved simply recalling the definition of special Birkhoff sums and Birkhoff sums, can be understood  in terms of \emph{cutting and stacking} of Rohlin towers: the relation indeed mimics at the level of special Birkhoff sums the fact (recall in \S~\ref{sec:entries}) that the Rohlin tower over  $I^{(n+1)}_j$ is obtained by stacking $(Z_n)_{ij}$ subtowers of the Rohlin tower over  $I^{(n)}$ and hence, correspondingly, the special Birkhoff sum $f^{(n+1)} (x) $ is obtained as sum of $(Z_n)_{ij}$ values of the special Birkhoff sum $f^{(n)}$ at points of $I^{(n)}_i$. 

\subsubsection{Decomposition of Birkhoff sums}\label{decompBS}
Special Birkhoff sums can be used as follows as fudamental \emph{building blocks} to study Birkhoff sums, see e.g.~\cite{Zo:dev, MMY, Ul:mix, Ul:abs, MarmiYoccoz, MUY}.  
Given an infinitely renormalizable $T$ and  $f\in \mathcal{C}(T)$, let $\{ I^{(n)}, k\in\mathbb{N}\}$ be the sequence of inducing intervals given by Zorich induction. Consider first the special case in which $x_0\in I^{(n_0+1)}$ 
 for some $n_0\in\mathbb{N}$. Then, if follows from \eqref{SBSdecomp} that, if $1\leq j\leq d$ is such that $x_0\in I^{(n_0+1)}_j$, the Birkhoff sum $S_r f(x)$ for any $0\leq r\leq q^{(n_0+1)}_j$ can be decomposed as
$$
S_r f(x_0)=\sum_{\ell=0}^{b_{n_0}-1}  f^{(n_0)} \left(x^{(n_0)}_\ell \right) + S_{r_1}f(x_1), \qquad \text{with\ } x^{(n_0)}_\ell:= \left({T}_{n_0}\right)^{\ell}(x_0), \quad x_1:= \left({T}_{n_0}\right)^{b_{n_0}}(x_0)
$$
where $b_{n_0}=b_{n_0}(x)$ is such that $0\leq b_{n_0}\leq \sum_{i=1}^d (Z_{n_0})_{ij}$ and  $S_{r_1}f(x_1)$ is a \emph{reminder} which is not a special Birkhoff sum of lever $n_0$, i.e.~if $j_0$ is such that $x_1 \in I^{(n_0)}_{j_0}$ then $0\leq r_1<  q^{(n_0)}_{j_0}$. 
 Repeating this decomposition for $S_{r_1} f(x_1)$, we then obtain by recursion the following \emph{geometric decomposition} of the Birkhoff sum $S_r(x_0)$ into special Birkhoff sum:
\begin{equation}\label{geometric}
S_r f(x_0)=\sum_{n=0}^{n_0} \sum_{\ell=0}^{b_{n}-1}  f^{(n)} \left(x^{(n)}_\ell \right) , \qquad \text{where} \ 0\leq b_n\leq \Vert Z_n\Vert , \quad x^{(n)}_\ell \in I^{(n)},\ \text{for}\ 0\leq \ell \leq b_n -1.
\end{equation}
From here, we also get the estimate
$$
|S_r f(x_0)|\leq \sum_{n=0}^{n_0}  \Vert Z_n\Vert \, \Vert f^{(n)}\Vert, \qquad \text{if}\ x\in I^{(n_0)}_j, \ 0\leq r\leq q^{(n_0+1)}_j.
$$
For the general case of a  Birkhoff sum $S_r f(x)$ for any $x\in [0,1]$ and $r\in\mathbb{N}$, we can define   
$n_0=n_0(x,r)$ to be the maximum $n_0\geq 1$ such that $I^{(n)}$ contains at least \emph{two} points of the orbit $\{T^i x, 0\leq i<r\}$. (This guarantees that $r$ is larger than the smallest height of a tower over $I^{(n_0)}$, but at the same time that it is smaller than a over $I^{(n_0+1)}$. Then, if $x_0 =T^{i_0}(x)$ is one of the points in $I^{(n_0)}$ we can split the Birkhoff sum $S_r f(x)$ into two sums of the previous form, one for $T$ and the other for $T^{-1}$ and therefore get
\begin{equation}\label{geometricestimate}
|S_r f(x)|\leq 2 \sum_{n=0}^{n_0}  \Vert Z_n\Vert \, \Vert f^{(n)}\Vert, \qquad \text{for any}\ x\in [0,1]\, \backslash \, I^{(n_0+1)}.
\end{equation}  
}
\subsection{Boundary operators}
\label{boundary}
We introduce here some operators on the space of GIETs, first defined in the work of Marmi, Moussa and Yoccoz \cite{MMY} and known as \emph{boundary operators}, for their correspondence with a boundary operator in cohomology (see \S~\ref{sec:cohomologyboundary}). 

\subsubsection{Boundary operator for observables}\label{sec:obs_boundary}
Let $T \in \mathcal{X}_d^r$ be a GIET and let $f$ be a function in the space $\mathcal{C} (T)$ introduced in the previous section \S~\ref{sec:obs}. By definition of the functional space, on each of the continuity intervals $I^t_{\pi^t(i)} = (u_{i-1}^t(T), u_{i}^t(T))$ of $T$, for $1\leq i\leq d$, $f$  has a \emph{right limit} at $u_{i-1}^t(T)$ and a \emph{left limit} at $u_{i}^t(T)$. We denote by $f^r(u_i)$ and $f^l(u_i)$ respectively the right and left limits at the discontinuity point $u_i$. Explicitely,\footnote{Indeed, since $u_i$ is the left endpoint of $I^t_{\pi_t(i+1)}$, we have to use $f_{\pi_t(i+1)}(x)$ to take the right limit, while to take the left limit we have to see 
$u_i$ as the right endpoint of $I^t_{\pi^t(i)}$ and consider $f_{\pi_t(i)}$.} if we denote by $f_i$  the $i^{th}$ branch of $f$ obtained restricting $f$ to $I^t_i$, we have:
$$
f^r(u_i):= \lim_{x\to u_i^+} f_{\pi_t(i+1)}(x), \qquad f^l(u_i):= \lim_{x\to u_i^-} f_{\pi_t(i)}(x).
$$
We also set by convention $f^l(u_0):=0$ and $f^r(u_d):=0$. 
Let now $S=S(T)$ be the suspended surface corresponding to $T$ (see \S~\ref{gietandfoliations}).  Let $d=2g+\kappa$ where $g$ is the genus of $S$ and $\kappa$ the cardinality of the set $\Sing (T)$ of singularities, which we will label by $\{ 1,\dots,\kappa\}$. Recall that each of the $u_i$ corresponds to (the label of) a singular point $s(u_i) \in \{1,\dots, \kappa\}$ (see~\ref{gietandfoliations}). 

\smallskip
\noindent For each  $f\in \mathcal{C} (\sqcup I^{(n)}_j)$ and for each $1\leq s\leq \kappa$, set 
$$ B_s(f) := \sum_{0\leq i\leq  d \ \text{s.t}\ s(u_i) = s}\left( {f^r(u_i) }-
{f^l(u_i)}\right).$$ This defines by a map 
$$ \begin{array}{ccccc}
B  & := & \mathcal{C}_0\big(\sqcup_i{I_i^t(T)} \big) &  \longrightarrow & \R^{\kappa} \\
   &      & f  & \longmapsto & (B_s(f))_{1\leq s\leq \kappa}.
 \end{array}$$
A combinatorial definition of the correspondence between endpoints and singularities as well as  of the boundary operator following \cite{MarmiYoccoz} is given in the Appendix, see \S~\ref{sec:boundarycomb}.

 \subsubsection{Cohomological interpretation of the boundary operator}\label{sec:cohomologyboundary}
When one restricts the boundary operator to piecewise constant functions, one recovers a standard boundary map in homology. The intervals $(I_i^t(T))_{i \leq d}$ {\color{black}can be put into one-to-one correspondence to curves on  $S_g$ whose endpoints belong to the singularity set $\Sing(T)$ (each of them indeed embed onto a segments $S_g$,  whose endpoints can then be slided along the leaves of the foliation until they become  singularities  in $\Sing$, see Appendix~\ref{sec:boundarycomb})}. The (relative) holonomy classes of these curves   actually form a \emph{base} of the \emph{relative homology group} $\mathrm{H}_1(S_g, \Sing, \mathbb{Z})$ (see e.g.~Viana~\cite{Vi:IET} or \cite{Yoc:Clay}). A function that is constant on each of the $I_i^t(T)$s thus defines a class in $\mathrm{H}_1(S_g, \Sing, \mathbb{R}) = \mathrm{H}_1(S_g, \Sing, \mathbb{Z}) \otimes \R$. The boundary operator defined above is nothing but the standard boundary operator for relative homology 

$$ \partial : \mathrm{H}_{n+1}(S_g, \Sing, \mathbb{R}) \longrightarrow \mathrm{H}_{n}( S, \mathbb{R}) $$ restricted to the case $n = 0$.

\subsubsection{Renormalization invariance of the boundary operator for GIETs}
The following Proposition was proved in in \cite{MMY3}, see also \cite{Yoc:Clay}.
\begin{proposition}[properties of the boundary, see Proposition 3.2 in  \cite{MMY3}]
\label{propB}
{\color{black}The boundary $B: \mathcal{C}_0\big(\sqcup_i{I_i^t(T)} \big)  \to  \R^{\kappa}$ has the following properties:} 
\begin{itemize}
\item[(i)] For any $\psi \in \mathcal{C}(T)$, $B(\psi)=B(\psi\circ T)$.
\item[(ii)] If $T$ is once renormalizable, for any $\psi_0\in \mathcal{C}(T)$, $B(\psi_0)= B(\psi_1)$ where $\psi_1\in \mathcal{C}(\mathcal{V}(T))$ is induced from $\psi_0$ obtained considering special Birkhoff sums.
\end{itemize}
\end{proposition}

{\color{black}
\subsubsection{Boundary of a GIET}\label{sec:boundaryGIET} 
We define now the \emph{boundary} of a GIET {(see also \cite{MMY2}).
\begin{definition}[boundary of a GIET]\label{GIET:boundary} Given $T  \in \mathcal{X}$, we define the \emph{boundary} of $T$ to be $\mathcal{B}(T) := B(\log \D T) $, where $B$ is the boundary operator on $\mathcal{C}_0\big(\sqcup_i{I_i^t(T)} \big)$ defined by Marmi-Moussa-Yoccoz (see \S~\ref{sec:obs_boundary}).
\end{definition}
\noindent The following Remark gives an equivalent expression for the boundary in terms of the shape-profile coordinats (from \S~\ref{coordinates}) which will be useful in the sequel. Recall that give $T \in \mathcal{X}$ we denote by $(A_T, \varphi^1_T, \cdots, \varphi^d_T)$ its coordinates for the product structure $\mathcal{X} = \mathcal{A} \times \mathcal{P}$. We denote $\omega_i:=\omega_i(T)$ the logarithm of the slope of $A_T$ on the $i$-th interval. 
\begin{remark}\label{rk:Bexpression}  Given $T =(A_T, \varphi^1_T, \cdots, \varphi^d_T) \in \mathcal{X}$, 
let $\omega(x)$ and $\log \D \varphi_T(x)$ denote the piecewise continuous function in $\mathcal{C}_0\big(\sqcup_i{I_i^t(T)} \big)$  which are respectively equal to $\omega_i$ and to $\log \D \varphi^i_T$ on $I^t_i$. Then we claim that
$$\mathcal{B}(T)= B({\omega}) + B(\log \D \varphi).$$
This can be seen since the derivative $\D T $ of a GIET $T$ is related to the functions $\D\varphi_T$ and $\omega$   by
$$
\D T = e^\omega D\varphi_T, \qquad \text{or, \ equivalenty,} \quad \D T (x)= e^{\omega_j} \,\log \D \varphi^j_T(x), \ \  \text{for\ all\ } x\in I^t_j, \ 1\leq j\leq d.
$$
Thus, from the  definition of $B$ in terms of right and left limits of a function in $\mathcal{C}_0\big(\sqcup_i{I_i^t(T)} \big)$, we have that 
$ B(\log \D T) = B({\omega}) + B(\log \D \varphi)$. 
\end{remark}

\noindent Proposition~\ref{propB} (proved in \cite{MMY3} as Proposition 3.2) has the following implication for the boundary of a GIET:
\begin{lemma}[properties of GIETs boundary]\label{lemma:Bproperties}
\label{prop2}
Consider $\mathcal{T} \in \mathcal{X}^1$. 
\begin{itemize}
\item[(i)] For any $\psi \in \mathrm{Diff}^1([0,1])$, $\mathcal{B}(\psi^{-1} \circ T \circ \psi) = \mathcal{B}(T)$.
\item[(ii)] If $T$ is once renormalizable, $\mathcal{B}(\mathcal{V}(T)) = \mathcal{B}(T)$.
\end{itemize}
\end{lemma}
\noindent Thus, the Lemma shows that the boundary of a GEIET is both a \emph{conjugacy invariant} (by property $(i)$) and a \emph{renormalization invariant} (by property $(ii)$).
\begin{proof}
Property $(i)$ follows immediately from the definition of boundary, since the values of the (left and right) derivatives at the endpoints are invariant by conjugation. 
Property $(ii)$ is simply a reformulation of property  $(ii)$ of the observable boundary $B$ (see 
Proposition~\ref{propB}, proved in \cite{MMY3} as Proposition 3.2) for the observable $\psi=\log DT$, after remarking that $\varphi_1=\log D (\mathcal{V}(T))$ (which follows  from the definition of the induced map and the chain rule, taking the log and comparing the result with the definition of special Birkhoff sums, see \S~\ref{sec:SBS}). 
\end{proof}}

\section{Affine shadowing}
\label{sec:affineshadowing}
In this section we state and prove the dynamical  dichotomy for orbits of irrational GIETs under renormalization (stated below as Theorem~\ref{shadowing}), which is the main result at the heart of our rigidity results (both \textcolor{black}{Theorem~\ref{maintheorem} and Theorem~\ref{thm:rigidityfoliation})} and can be used to prove  \emph{a priori bounds} for minimal GIETs. The precise formulation requires the definition of a full measure Diophantine condition on the rotation number (which we will call \emph{regular} Diophantine condition, or RDC for short, see \S~\ref{sec:regularDC}) under which it will hold. We will later show (in Section~\ref{sec:fullmeasure} that this Diophantine condition is satisfied for a full measure set of rotation numbers (see Definition~\ref{def:fullmeasure} in \S~\ref{sec:fullmeasure} for the notion of full measure). 
\smallskip

For expository purposes and to increase readability (in particular for the readers who are not familiar with the technicalities of Rauzy-Veech induction) we first treat separately (in \S~\ref{sec:periodiccase}) the case where the rotation number is assumed to be periodic (also sometimes known, in the one-dimensional dynamics literature, as \emph{Fibonacci combinatorics} case). In this case the Diophantine-type conditions simplify drastically and the proof is less technical (and yields a stronger result, see Proposition~\ref{asymptotic}), but all the ideas needed for the general case are already there. 

\smallskip
The general case is then treated in \S~\ref{sec:generalDC}. The statement of the affine shadowing dichotomy is given in Theorem~\ref{shadowing}. The required Diophantine condition (introduced in \S~\ref{sec:arithmetic}) is defined using an \emph{acceleration} of the Zorich induction $\mathcal{Z}$, which we denote by $\widetilde{Z}$ and call \emph{Oseledets regular} since it is obtained requesting that the estimates given by Oseledets theorem are effective (see \S~\ref{sec:effectiveOseledets} for details).

\subsection{Scaling invariants and mean (log-)slope vectors}\label{sec:averageslope}
{ 
Let us first define the \emph{average} slope vector (and the \emph{log-average} vector) associated to a generalized interval exchange map. These quantities will play a central role in our study of renormalization, in particular to define affine shadowing. Recall that in \S~\ref{coordinates} we introduced the \emph{shape}-\emph{profile} coordinates of a GIET, so that we can write $T=(A_T, \mathcal{P}(T))$ where $A_T$ is an AIET called the \emph{shape} of $T$ and $\mathcal{P}(T)=(\varphi_1,\dots, \varphi_d)$ is its \emph{profile} (refer to \S~\ref{coordinates}).
\begin{definition}[average vectors $\vect{\rho}(T)$ and $\vect{\omega}(T)$ for a GIET]\label{def:rhoomega}
Let $T$ be a GIET in $\mathcal{X}^r_d$, with $r\geq 1$. The \emph{shape-slope vector} $\vect{\rho}(T)$ associated to $T$ is by definition
$$
\vect{\rho}(T):= \vect{\rho}(A_T), \qquad \text{if}\ T= (A_T,\varphi_1, \dots, \varphi_d),
$$
i.e.~it is the slope vector of the \emph{affine} interval exchange $A_T$ which gives the shape of $T$. 

\smallskip
\noindent We also define the \emph{shape log-slope vector}, or, for short, the \emph{log-slope vector} $\vect{\omega}(T)$ to be
$$
\vect{\omega}(T)=\log \vect{\rho}(T):= (\log \rho_1, \dots, \log \rho_d), \qquad \text{if}\ \vect{\rho}(T)=(\rho_1,\dots, \rho_d). 
$$  
\end{definition}
\noindent Note that, since for every $1\leq j\leq d$ the \emph{slope} $\rho(A_T)_j$ of the $j^{th}$ branch $(A_T)_j$ is also given by the average value of $T'$ on $I_j$ (see \S~\ref{coordinates}), we can also define explicitly
\begin{equation}\label{rhoasaverage}
\vect{\rho}(T)= \left( \frac{1}{|I_1|}\int_{I_1} D T_1(x) \ud x, \dots, \frac{1}{|I_d|}\int_{I_d} D T_d(x) \ud x \right).
\end{equation}
For the reader familiar with one-dimensional dynamics literature, these quantities can be seen as the key \emph{scaling ratios} that we want to exploit to encode the dynamics of renormalization.

\begin{remark}\label{rk:logvectorRn}
If $\mathcal{R}$ is an acceleration of $\mathcal{V}$ corresponding to inducing on intervals $(I_n)_{n\in\mathbb{N}}$, using the notation in \S~\ref{dynamicalpartitions} and recalling the explicit form of the renormalization $\mathcal{R}^n(T)$ given by \eqref{eq:renormalizedmap}, which relates the branches of $\mathcal{R}^n(T)$ with iterates of the induced map (and since by the chain rule one can see that conjugation with an affine map does not change the derivative) we have that
\begin{align}\label{eq:logvectorRn}
\vect{\rho}(\mathcal{R}^nT) & =  \left( \frac{1}{|I_1^{t}{(n)}|}\int_{I_1^{t}{(n)}} D \mathcal{R}^n(T) (x) \ud x, \right.  \cdots &\left. \frac{1}{|I^{t}_d(n)|}\int_{I^{t}_d(n)} D \mathcal{R}^n(T) (x) \ud x  \right) \\
 &=  \left( \frac{1}{|I^{(n)}_1|}\int_{I^{(n)}_1} D \left(T^{q^{(n)}_1}_1\right) (x) \ud x,  \right.  \cdots & \left. \frac{1}{|I^{(n)}_d|}\int_{I^{(n)}_d} D\left( T^{q^{(n)}_j}\right) (x) \ud x\right)& 
\label{eq:logvectorTn}
\end{align}
\end{remark}
}

\subsection{Affine shadowing in the periodic type (or Fibonacci type) case}\label{sec:periodiccase}
In this section, we assume that $T$ is a infinitely renormalizable generalized interval exchange map of \emph{periodic type} (see Definition~\ref{def:periodictype}) and in addition that it has  hyperbolic Rauzy-Veech period matrix $A$, or, for short, that $T$ is of \emph{hyperbolic periodic type}.  Thus, there exists a $n_0>0$ such that the rotation number is periodic with period $k_0$, namely if $n=i k_0+r$ for some $i\in\mathbb{N}$ and $0\leq r<n_0$ then 
$\gamma(\V^{i k_0+r}(T))=\gamma(\V^r(T))$. (We remark that even though the rotation number is periodic, the orbit of $T$ is \emph{not} in general periodic).

In this case, we will use, as renormalization operator, the acceleration of Rauzy-Veech induction which corresponds to the period $k_0$ of the rotation number $T$, namely the operator which, by abusing the notation we will still denote by $\mathcal{R}$, given by
$$\mathcal{R}^{n}(T):=\V^{n k_0 }(T)=(\V^{ k_0 })^{n}(T) , \qquad \mathrm{for\ all}\ n\in \mathbb{N}.$$ 
Notice that this definition of $\mathcal{R}$ is used only in this section.\footnote{Note also that the operator $\mathcal{R}$ here defined also satisfies the requirements of the acceleration which defines $\mathcal{R}$ in the general case, since in the periodic type case all requests on the acceleration $\mathcal{R}$ are trivially satisfied for every iterate of the orbit of $T$ under $\V$ or any of its powers.}

\smallskip
{ The following proposition gives the dynamical dichotomy for (hyperbolic) periodic-type IETs, so in a setting which is a special case of  Theorem~\ref{shadowing}, but in this special case it actually yields a stronger conclusion (as remarked after the statement).
(as it can be seen comparing the conclusions in Case 2 of Proposition~\ref{asymptotic} here below and Theorem~\ref{shadowing} stated later on).} 
\smallskip 

Let us write $\mathbb{R}^d =E_s\oplus E_c\oplus E_u$ for the splitting of $\mathbb{R}^d$ into respectively the stable space $E_s$,  the central space $E_c$ and the unstable space $E_u$ for the action of $A$ on $\mathbb{R}^d$ (corresponding to eigevectors with norm respectively smaller, equal and greater than $1$).

\begin{proposition}[\bf Affine shadowing dichotomy for periodic type]
\label{asymptotic}
Let $T$ be of hyperbolic periodic type, with rotation number of period $k_0$. Denote by 
$\vect{\omega}_n: = \vect{\omega}(\mathcal{R}^{n}(T) ) \in \mathbb{R}^d $ 
the shape log-slope vector (as defined in \S~\ref{sec:averageslope})  of $\mathcal{R}^{n}(T)$, where here $\mathcal{R}:=\mathcal{V}^{k_0}$ is the acceleration corresponding to the period. Then, either 
\begin{enumerate}
\item[(1)] $(\vect{\omega}_n  )_{n\in\mathbb{N}}$  is bounded, or 
\item[(2)] there exists $\vect{v} \in E_u$ such that 
$$ \vect{\omega}_n = A^n  \vect{v} + O(1) $$ \textit{i.e.} the difference $\vect{\omega}_n - A^n  \vect{v}$ is bounded.
\end{enumerate}
\end{proposition}
\noindent Note that 
cases $(1)$  
 and $(2)$ can be merged, since case $(1)$ can be reformulated in the form of case $(2)$ with the additional request that $\vect{v}=0$. 
The  proposition then states that the sequence $(\vect{\omega}_n)_{n\in\mathbb{N}}$ (of log-slopes vectors for the shape $
\mathcal{R}^n(T)$)  can be \emph{approximated} (up to a \emph{bounded} error) by the linear evolution of a vector $\vect{v}$ under the action of the period matrix $A$.
\noindent The vector $\vect{v}$ will be called the (\emph{affine}) \emph{shadow}. 

We remark that in this  special case of periodic type rotation number, the conclusion is stronger than the result we will prove for a full measure set of rotation numbers, i.e.~Theorem~\ref{shadowing}: here in Proposition~\ref{asymptotic}, the difference $\vect{\omega}_n - A^n \,\vect{v}$ in Case $2$ is \emph{bounded}, while in the general case we will be able to approximate the evolution of $(\vect{\omega}_n)_{n\in\mathbb{N}}$ by a linear evolution of a shadow vector only up to  a lower order (but in general not bounded) error, see the conclusion of Case $2$ in Theorem~\ref{shadowing}.

In the rest of this section we prove Proposition~\ref{asymptotic}. The reader who is interested in the general case and familiar with Rauzy-Veech induction can omit this section and move directly to \S~\ref{sec:arithmetic} (for the Diophantine condition on IETs) and \S~\ref{sec:generalDC}. The following outline of the proof though may be useful also to follow the strategy of the proof in the general case.

\smallskip

\noindent {\it Sketch of the Proof of Proposition~\ref{asymptotic}.} 
First we show, in Lemma~\ref{lemma1} (which is the most important one), that, thanks to  a basic (and classical) application distortion bounds (in the form given by Lemma~\ref{bound1}),  up to some uniformly bounded mistake, the shape log-slope  vector $\vect{\omega}_n$ transform linearly via the cocyle (which is this case is just given by applying repeatingly the period matrix).  We refer  to this result (i.e.~Lemma~\ref{lemma1}) as \emph{linear approximation}. The projection $\Pro _s(\vect{\omega}_n)$ of $\vect{\omega}_n$ onto the stable space $E_s$ is controlled through Lemma~\ref{lemma2}, which is valid for any $T$ with periodic rotation number and shows that the part in the stable space always remains bounded.

We then consider  iterates of renormalizations of $T$ and consider separately two cases: (1) if the slopes  are bounded, we are in Case 1; otherwise, (2) if the slopes are  not bounded, in virtue of the control of the stable part (given by Lemma~\ref{lemma2}), the component  in the unstable space is also unbounded. To prove that in this case we are in Case 2, namely we can build an \emph{affine shadow}, we  wait for a time when this compoment is large compared to the mistake that one makes when comparing the actual growth of the slopes with how it transforms linearly. If  one starts renormalizing from that moment, the slope change almost linearly up to a mistake that is more and more negligible as slopes in the unstable space grow exponentially fast. Thus,  adjusting using smaller and smaller corrections (see \eqref{def:shadow} and \eqref{eq:smallercorrections}) allows to find a vector shadowing the slopes. This is done rigorously through definition \eqref{def:shadow} of the \emph{shadow} and the proof of Proposition~\ref{asymptotic} presented below.

\smallskip

\begin{lemma}[{\bf Linear approximation for periodic type GIETs}]
\label{lemma1}
For any  periodic type $T$  with period matrix $A$ there exists a constant $K_T$ such that $\vect{\omega}_n=\vect{\omega}(\mathcal{R}^n T)$ satisfies
$$ ||\vect{\omega}_{n+1} - A\, \vect{\omega}_n  || \leq K_T, \qquad \textrm{for\ all}\ n \in \mathbb{N}  $$
\end{lemma}
\begin{proof}
{Since $T$ is of periodic-type with period $k_0$ and $\mathcal{R}=\mathcal{V}^{k_0}$, for any $n\geq 0$ we have that for each $x\in [0,1]$,  
$$
\mathcal{R}^{n+1} T (x) =\frac{1}{\lambda} (\mathcal{R}^{n}T )^{k(x)-1}(\lambda x),
$$
where $\lambda$ is a rescaling ratio (more precisely, one has $\lambda=|I^{((n+1)k_0)}|/|I^{(n k_0)}|$) and where 
 $0\leq k(x)\leq K$ is uniformely bounded independently on $n$ (here $k(x)$ is  indeed  constant\footnote{More precisely, recalling the Rohlin towers point of view from \S~\ref{dynamicalpartitions} and the dynamical interpretation of the cocycle entries from \S~\ref{sec:entries}, for every $x$ belonging to the $j^{th}$ continuity interval of $\mathcal{R}^{n+1} T$,  we have that $$k(x)= Z^{(k_0)}_{j} :=\sum_{i=1}^d Z^{(k_0)}_{ij} =\sum_{i=1}^d Q(n k_0, (n+1) k_{0})_{ij}.$$} on each continuity interval for $\mathcal{R}^{n+1}T$). 
Thus, by chain rule, 
we have that $\D(\mathcal{R}^{n}T) (\lambda x)/\lambda) = \D\mathcal{R}^{n}T (\lambda x)$ and, taking logarithms, 
\begin{equation}\label{decompDlogperiodic} \log D\mathcal{R}^{n+1}T (x) =  \sum_{i=0}^{k(x)-1}{\log D\mathcal{R}^{n}T\left((\mathcal{R}^{n}T)^i(y)\right)}, \end{equation} 
where $y:=\lambda x$ is the point that corresponds to $x$ under the linear rescaling.
\smallskip
\noindent If we now consider $\vect{\rho}_n :=  \vect{\rho}(\mathcal{R}^{n}T)$, 
by definition of the  shape slope vector  $\vect{\rho}$ (see\S~\ref{sec:averageslope} and in particular \eqref{eq:logvectorRn} in Remark~\ref{rk:logvectorRn}), 
 for any $n\in\mathbb{R}$ and any $1\leq j\leq d$, there exists by mean value theorem a point $y_{n,j}\in [0,1]$ such that 
\begin{equation}\label{meanvalueptR}
(\rho_n)_j=  D\mathcal{R}^{n}T\left(y_{n,j}\right), \qquad (\omega_n)_j= \log(\rho_n)_j = \log D\mathcal{R}^{n}(T)\left(y_{n,j}\right).
\end{equation}
Thus, if for each $0\leq \ell \leq k(x)-1$, we let $ j_\ell \in \{1,\dots, d\}$ be the index such that $(\mathcal{R}^{n}T)^\ell(y)$ belongs the $(j_\ell)^{th}$ continuity interval $I^t_{j_\ell}(n)$ of $\mathcal{R}^n T$,  
 applying  the distorsion bound given by Lemma \ref{bound1} to $\mathcal{R}^{n}T$ and taking logarithms, we have that 
$$ \left| \log D\mathcal{R}^{n}(T)\left({(\mathcal{R}^{n}T)}^\ell(y)\right) - (\omega_{n})_{j_\ell}\right| \leq D_T $$ where $D_T = \int_0^1{|\eta_T| dt}$ and $(\omega_{n})_{j_\ell}$ is the $j_\ell^{th}$ entry of $\vect{\omega}_n$.   

\smallskip
Applying now the formula \eqref{decompDlogperiodic} to  the point $x=y_{n+1,j}$, where $x=y_{n+1,j}$ is given by \eqref{meanvalueptR}, and denoting by $k_j=y_{n+1,j}$, since we showed at the beginning that $k_j\leq K$, we get
$$  (\omega_{n+1})_j = \sum_{\ell=0}^{k_j-1}{(\omega_{n})_{j_\ell}} + E(n,j) ,\qquad \text{where} \quad |E(n,j)|\leq k_j D_T \leq K D_T.$$  Note  now that $ \sum_{\ell=0}^{k_j-1}({\omega_{n})_{j_\ell}}$ is exactly the $j$-th entry of $A$  (because  by definition $\mathcal({R}^{n}T)^\ell(y)$ belongs the $j_\ell$-th interval of  continuity of $\mathcal{R}^{n}T$). The Lemma is thus proven with $K_T := K D_T$. }
\end{proof}

\begin{lemma}[{\bf Control of the stable part}]
\label{lemma2}
For any hyperbolic periodic type $T$ with periodic matrix $A$ there exists a constant $M_T > 0$ depending only upon the constant $K_T$ in Lemma~\ref{lemma1} above such that for all $n \in \mathbb{N}$
$$ ||\Pro _s(\vect{\omega}_n)|| \leq M_T,$$
where $\Pro _s:\mathbb{R}^d=E_u \oplus E_s  \longrightarrow E_s$ denotes the projection onto the stable space $E_s$ for the action of $A$ on $\mathcal{R}^d$. 
\end{lemma}

\begin{proof}
By Lemma \ref{lemma1}, we have $||\vect{\omega}_{n+1} - A\, \vect{\omega}_n || \leq K_T$. Projecting onto $E_s$, since $P_s$ commutes with the action of $A$, one gets 
$$ || \Pro _s(\vect{\omega}_{n+1}) || \leq  || A \, \Pro _s (\vect{\omega}_n) || + K_T.$$ Since $A$ contracts  $E_s$ by a factor $\gamma< 1$, we get that the sequence $(a_n)_{n\in\mathbb{N}}$  given by $a_n: = ||\Pro _s (\vect{\omega}_n)||$ satisfies 
$$ a_{n+1} \leq \gamma a_n + K_T.$$ One can then check that any sequence with such this property is bounded. This concludes the proof. 
\end{proof}

\begin{proof}[Proof of Proposition \ref{asymptotic}]
Consider now the projection  $\Pro _s:\mathbb{R}^d=E_u \oplus E_s  \longrightarrow E_s$   onto the stable space $E_s$ for the action of $A$ on $\mathcal{R}^d$.
We distinguish on two cases:.  \vspace{3mm}

\noindent \textit{Case 1}: The sequence $\{ \Pro _u(\vect{\omega}_n)\}_{n\in\mathbb{N}}$ is bounded. 

\vspace{2mm} \noindent In this case, by Lemma \ref{lemma2}, $\Pro _s(\vect{\omega}_n)$  is bounded as well thus $\vect{\omega}_n$ is bounded. This shows that we are in Case 1 of Proposition \ref{asymptotic}.

\vspace{3mm}

\noindent \textit{Case 2}: The sequence $\{\Pro _u(\vect{\omega}_n)\}_{n\in\mathbb{N}} $ is unbounded.

\vspace{2mm} \noindent In that case let us show that we can define the shadow $\vect{v}\in \mathbb{R}^d$ to be:  
\begin{equation}\label{def:shadow} \vect{v} :=  \sum_{i=1}^{\infty}{A^{-i}(\Pro _u(\vect{\omega}_{i} - A\vect{\omega}_{i-1}))} + \Pro _u(\vect{\omega}_0). \end{equation}
We will show at the same time that the series converge and hence $\vect{v}$ is well defined. Let us formally compute $A^n\vect{v}$ which, splitting the series  and changing index in the finite sum (using $k= n-i$) to exploit the telescopic nature of the finite sum, gives:
\begin{align*} 
A^n\vect{v} &=   \sum_{k=0}^{n-1}{\Pro _u(A^k\vect{\omega}_{n-k} - A^{k+1}\vect{\omega}_{n-{k-1}})}+ \sum_{i=n}^{\infty}{A^{n-i}(\Pro _u(\vect{\omega}_{i} - A\vect{\omega}_{i-1}))} + A^n (\Pro _u(\vect{\omega}_0)),\\ & =  {\Pro _u(\vect{\omega}_{n}) - \Pro _u(A^{n}\vect{\omega}_{0})} +  \sum_{i=n}^{\infty}{A^{n-i}(\Pro _u(\vect{\omega}_{i} - A\vect{\omega}_{i-1}))}+ A^n (\Pro _u(\vect{\omega}_0)).
\end{align*}
Since $E^u$ is invariant by the action of $A$, $\Pro _u(A^{n}\vect{\omega}_{0})=A^n (\Pro _u(\vect{\omega}_0))$. Moreover, since by definition $\vect{v}\in E^u$, we get that
\begin{equation}\label{Aseries}\Pro _u(A^n \vect{v} - \omega_n) = A^n \vect{v} - \Pro _u(\vect{\omega}_n)  = \sum_{i=n}^{\infty}{A^{n-i}(\Pro _u(\vect{\omega}_{i} - A\vect{\omega}_{i-1}))}.\end{equation} Since the map $A$ is uniformly expanding on $E_u$ and, by Lemma \ref{lemma1}, $||\vect{\omega}_{n+1} - A\, \vect{\omega}_n || \leq K_T$,  there exists $c <1$ such that
\begin{equation}\label{eq:smallercorrections} ||A^{n-i}(\Pro _u(\vect{\omega}_{i} - A\vect{\omega}_{i-1})) || \leq K_T c^{i-n} ,\qquad \text{for \ all} \ i \geq n .\end{equation} Using this estimate in \eqref{Aseries} gives that
$$ || \Pro _u(A^n \vect{v} - \vect{\omega}_n) || \leq K_T \sum_{j=0}^{\infty}{c^j}.$$ Since $v\in E^u$ and $\Pro _s(\vect{\omega}_n)$ is bounded by Lemma \ref{lemma2}, we have that, for some $C_A>0$ (depending on the angle between $E^s$ and $E^u$) 
\begin{align*}
|| A^n \vect{v} - \vect{\omega}_n || & \leq C_A  \left( || \Pro _u(A^n \vect{v} - \vect{\omega}_n) || +  || \Pro _s(A^n \vect{v} - \vect{\omega}_n) ||\right) \\ & = C_A  \left( || \Pro _u(A^n \vect{v} - \vect{\omega}_n) || +  || 0-  \Pro _s(\vect{\omega}_n) ||\right) \leq C_A \left( K_T \sum_{j=0}^{\infty}{c^j} + M_A\right)<+\infty.
\end{align*}
This shows at the same time (taking $n=0$) that the series defining $\vect{v}$ converges and hence $\vect{v}$ is well defined and that the difference $A^n \vect{v} - \vect{\omega}_n $ is uniformely bounded, so property $(2)$ of the Proposition~\ref{asymptotic} holds. This concludes the proof.
\end{proof}

\subsection{The Diophantine-type condition for the general case}\label{sec:arithmetic}
We  now turn to the general case. This section is devoted to the definition of   
 the Diophantine-type condition  under which we will prove the general case of the affine shadowing dichotomy in \S~\ref{sec:regularDC}. 
The main difference with the periodic-type case, is that some return times that were bounded in the special case, are now only bounded on average; the Diophantine condition  for the general case (see Definition~\ref{def:RDC}) is devised to provide a not too sparse sequence where they are nevertheless uniformly bounded, exploiting the hyperbolicity of the (Zorich acceleration of the) Rauzy-Veech cocycle. The renormalization operator which will be used is an \emph{acceleration} of the Zorich renormalization $\mathcal{Z}$  corresponding to accelerating along this sequence. 

The section is organized as follows. We define first a notion of \emph{full measure} on irrational rotation numbers, see \S~\ref{sec:fullmeasure}.
We then introduce in  \S~\ref{sec:Oseledetsgeneric} a notion of Oseledets genericity (corresponding to having an Oseledets generic extension, see Definition~\ref{def:Oseledetsextension} and the comments thereafter). In  \S~\ref{sec:goodreturns} we define  sequences of \emph{good return times}; this is a technical condition (which~corresponds to occurences of two consecutive bounded \emph{positive} matrix in the cocycle, see Definition~\ref{def:goodreturns}) which we want to assume on the accelerating sequence since it will help to control the size of the dynamical partitions. Finally, in \S~\ref{sec:regularDC}, we define the Regular Diophantine Condition (see Definition~\ref{def:RDC}).
 

\subsubsection{Notion of full measure}\label{sec:fullmeasure}
Let $T$ be an infinitely renormalizable GIET with \emph{irrational} (i.e.~\emph{infinitely complete}, see Definition~\ref{def:irrational}) rotation number $\gamma(T)$. By Poincar{\'e}-Yoccoz Theorem~\ref{thm:PY}, there exists\footnote{The IET $T_0$ is not necessarily unique, but it is unique for a full measure set of rotation numbers.} a standard IET $T_0$  with the same rotation number $\gamma(T)=\gamma(T_0)$ such that $T$ is semi-conjugated to $T_0$.

Recall that the \emph{Lebsgue measure class} on $\mathcal{I}_d=\Delta_{d-1}\times \mathfrak{S}_d^0$ (refer to \S~\ref{IETs} for the notation) is the measure class of the restriction of the Lebesgue measure on 
$\Delta_{d-1}\subset \mathbb{R}^d_+$ and the counting measure on combinatorial data in $\mathfrak{S}_d^0$.

\begin{definition}[full measure for rotation numbers]\label{def:fullmeasure}
We say that a set $\mathcal{F}$ of rotation numbers has \emph{full measure} if it contains the set of rotation numbers of a full measure set of classical IETs, i.e.
$
\mathcal{F}\supset \{ \gamma(T),\  T\in \mathcal{I}'_d \}, 
$
where $\mathcal{I}'_d\subset \mathcal{I}_d$ is full measure subset of the set $\mathcal{I}=\Delta_{d-1}\times \mathfrak{S}_d^0$ with respect to Lebesgue measure class on $\mathcal{I}'_d$.
\vspace{1mm}

\noindent We say that a property holds for \emph{almost every} \emph{irrational} GIET 
 if and only if the set of rotation numbers $\{ \gamma(T), T\in \mathcal{G}\}$ for which it holds has full measure in the sense above.
\end{definition}
\noindent Equivalently, we could have asked that, for each given irreducible combinatorial datum $\pi \in \mathfrak{S}_d^0$, $\mathcal{F}$ contains  the rotation number of almost every IET in $\mathcal{I}_\pi$ with respect to the Zorich $\mu_\mathcal{Z}$ or Masur-Veech measure  $\mu_\mathcal{V}$ (see Remark~\ref{rk:ac} in \S~\ref{sec:measures}). 
\smallskip


\subsubsection{Oseledets generic extensions}\label{sec:Oseledetsgeneric}
Let $T\in \mathcal{I}_d$ be a (standard) IET for which Zorich acceleration $\mathcal{Z}$ is defined. 
The following definition summarizes the Oseledets genericity-type properties that we will require in the Regular Diophantine Condition.

\begin{definition}[Oseledets generic extensions]\label{def:Oseledetsextension}
We say that an IET $T$ has an \emph{Oseledets generic extension} if there exists a sequence of \emph{invariant splittings}, i.e.~decompositions
\begin{equation}
\label{splittingsextension}
\mathbb{R}^d = \Gamma^{(n)}_s \oplus \Gamma^{(n)}_c\oplus \Gamma_u^{(n)}, \qquad n\in\mathbb{N} 
\end{equation}
into spaces $\Gamma^{(n)}_a$ with $a \in\{s,c,u\}$ which are invariant under the dynamics, i.e.~such that
$$
Q(m,n)\, \Gamma^{(m)}_a =\Gamma^{(n)}_a , \qquad \forall a \in\{s,c,u\},\quad \forall m<n,
$$
of dimension 
\begin{equation}
\tag{H}  \dim \Gamma^{(0)}_s = \dim \Gamma^{(n)}_s = g, \qquad \dim \Gamma^{(0)}_u = \dim \Gamma^{(n)}_u = g, \qquad \forall \ n\in\mathbb{N},
\end{equation}
such that, for some $\theta>0 $
and $C=C(T)>0$,  
\begin{align*}
\tag{O-s}& \|Z^{(n)}\, v\|=\|Q(0,n)\, v\| \leq C\, {e^{-\theta n}} &  \text{for\ all}\ n\in\mathbb{N}, \ \text{for\ all}\ v\in \Gamma^{(0)}_s,  \label{O-s} \\
\tag{O-u} & \|(Z^{(n)})^{-1}\, v\|=\|Q(0,n)^{-1} v\| \leq C \, {e^{-\theta n}} & \text{for\ all}\ n\in\mathbb{N}, \ \text{for\ all}\ v\in \Gamma^{(n)}_u, \label{O-u} \\ 
\tag{O-c} & \lim_{n\to+\infty}\frac{1}{n}\log\|Z^{(n)}v\| = 0, &  \text{for\ all}\ v\in \Gamma^{(0)}_c,\label{O-c}
\end{align*}
and furthermore, for any $\epsilon>0$ there exists $c=c(\epsilon,T)$ such that
\begin{align}\label{O-a}
\tag{O-a} | \sin \angle (\Gamma^{(n)}_{a_1}(\hat{T}), \Gamma^{(n)}_{a_2}(\hat{T}))| \geq c(\epsilon) e^{-\epsilon n},& \qquad \text{for\ all}\ a_1\neq a_2, \quad a_1,a_2\in \{s,c,u\},
 \end{align}
where the angle $\angle(V,W)$ between two linear subspaces $V, W\subset {\mathbb{R}^d}$ was  defined in \S~\ref{sec:Oseledets}.
\end{definition}
\noindent[The choice of labels \ref{O-u}, \ref{O-s}, \ref{O-c} is for \emph{Osedelets Stable},  \emph{Osedelets Unstable} and \emph{Osedelets Central} conditions, while \ref{O-a} stands for \emph{Oseledets angles} condition.]

\smallskip
The reader has certainly noticed the similarity with the conclusion of Oseledets theorem (as recalled in \S~\ref{sec:Oseledets}): we will indeed show that $T$ has a \emph{Oseledets generic extension} if the conclusion of Oseledets theorem holds for an \emph{extension} $\hat{T}$ of $T$. The spaces $\Gamma^{(n)}_s$, $\Gamma^{(n)}_c$, $\Gamma^{(n)}_u$, $n\in\mathbb{N}$,  will then be respectively the \emph{stable}, \emph{unstable} and \emph{central}  space for (the $n^{th}$ iterate $\hat{\mathcal{Z}}^n(\hat{T})$ of) the (extended) Zorich cocycle.
By exploiting Oseledets theorem for the natural extension of $\mathcal{Z}$ we will prove in \S~\ref{sec:fullmeasure} that $\mu_\mathcal{Z}$-almost every $T$ has an Oseledets generic extension.

Remark that while the sequence $\{ \Gamma_s^{(n)}, \ n\in \mathbb{N}\}$ is uniquely defined by $T$, the sequence  $\{ \Gamma_u^{(n)},  \ n\in \mathbb{N}\}$ is not; on the other hand, we only require the \emph{existence} of a such sequence.  
The choice of an (invariant) sequence of unstable spaces satisfying \ref{O-u} and \ref{O-a}is equivalent to the choice of an extension $\hat{T}$ of $T$.


\subsubsection{Good return times}\label{sec:goodreturns}
We introduce now a  property of a sequence of iterates of the induction that will play an important technical role in the proof of exponential convergence of renormalization.

\smallskip 
We say that a matrix $A\in SL(d, \mathbb{Z})$ is \emph{positive matrix} if its entries $A_{ij}$  are strictly positive for each $1\leq i,j\leq d$. We say that $A$ is a \emph{Zorich cocycle matrix} if it is a product of matrices of the Zorich cocycle, i.e.~there exists a $p>0$ and  $T\in\mathcal{I}_\pi$ such that $A=Q(0,p)(T)$.  We say in this case that $p$ is the \emph{Zorich length} of $A$.

\smallskip

 Good return times are those that correspond to a \emph{double} occurrence of a fixed positive matrix $A$:
\begin{definition}\label{def:goodreturns} Given a positive integer $p>0$, the sequence $(n_k)_{k\in\mathbb{N}}$ is z sequence of $p$-\emph{good returns} of the Zorich cocycle if there exists a positive Zorich cocycle matrix $A$ of length $p$ such that
$$
Q(n_k , n_{k}+2p) = A A, \qquad \text{for\ all}\ k\in\mathbb{N}
$$ 
and $n_{k+1}-n_k\geq 2p$ so that $Q(n_{k},n_{k+1}) = Q_k A A $ for some non-negative matrix $Q_k\in SL(d,\mathbb{Z})$. We say that $(n_k)_{k\in\mathbb{N}}$ is  a sequence of \emph{good returns} if they are $p$-\emph{good returns} for some positive integer $p$. We also say that $(n_k)_{k\in\mathbb{N}}$ is sequence of $A$-good return times if we want to specify the matrix $A$.  

\end{definition}
From ergodicity of $\mathcal{Z}$, one can easily show that almost every IET admits a sequence of good returns. Recurrence of (fixed) positive matrices in the cocycle  are useful in the study of standard IETs to guarantee some \emph{balance} in the size of the floors (and heights) of the Rohlin towers in the dynamical partitions, a key property  
exploited in almost all works on IETs starting from the seminal work of Veech \cite{Ve:gau}. 
We will show in Section~\ref{sec:convergence} that it can also be used to get some estimates on the size of the dynamical partitions at recurrence times for the renormalization, when combined with a priori bounds (see in particular \S~\ref{sec:sizePcontrol}).

The notion of good return may look quite special, but actually a weaker notion (see Remark~\ref{rk:positiveweaker} just below) is sufficient (and essentially corresponds to returns to bounded sets\footnote{In the case of standard IETs, a (future) occurrence of a positive matrix with bounded norm corresponds indeed to returns to a compact set in simplex $\Delta_{d}$  of lengths vectors, while a time which follows an occurence of a positive matrix with bounded norm correspond to a compact set in the suspension datum space of parameters (see for example \cite{Ul:mix}). Therefore, at time $n_k+p$ the induction visits a compact set for the natural extension domain $\hat{\mathcal{I}_d}$. In \S~\ref{sec:sizePcontrol}, we show that the double occurrence of a positive matrix, \emph{together} with a priori distorsion bounds, can also be used to prove some geometric control on dynamical partitions.}).  
On the other hand, proving the stronger form in Definition~\ref{def:goodreturns} costs no additional effort from the technical point of view (see \S~\ref{sec:fullmeasure}) and simplifies the notation.
\begin{remark}\label{rk:positiveweaker}
The definition of good return can be weakened, by considering $(n_k)_{k\in\mathbb{N}}$ such that for each $k$ we can write
$Q(n_k, n_{k+1})=  Q_k A_k B_k$, where $Q_k$ is a non negative matrix, while $A_k$ and $B_k$ are two \emph{positive} cocycle matrices such that $||A_k||$ and $|| B_k||$ are uniformely bounded in $k$. 
\end{remark}

\subsubsection{Notation for Zorich accelerations}\label{sec:acceleration}
Let us introduce the following terminology and notation for accelerations of the Zorich cocycle. If $(n_k)_{k\in\mathbb{N}}$ is a sequence with $n_0:=0$ (which will be in our case a sequence of good returns on which Oseledets theorem can be made effective, see \S~\ref{sec:effectiveseqs}), let us denote by  $\widetilde{\mathcal{Z}}$ and $\widetilde{Z}$ respectively the acceleration of the Zorich map  $\mathcal{Z}$ and the associated acceleration of the Zorich cocycle (see \S~\ref{sec:accelerations} \S~\ref{sec:inducing})  given by:
$$
\widetilde{\mathcal{Z}}^k(T):=\mathcal{Z}^{n_k}(T), \ \forall \ k\in \mathbb{N}, \qquad \widetilde{Z} = \widetilde{Z}(T)= Q(0, n_{1}). 
$$
Then $\widetilde{{Z}}$ is a cocycle over $\widetilde{\mathcal{Z}}$. We will say that $\widetilde{\mathcal{Z}}$ and  $\widetilde{{Z}}$ are \emph{accelerations along the sequence} $(n_k)_{k\in\mathbb{N}}$. We will also denote by $\widetilde{Z}_k$ and $\widetilde{Q}(k,k')$ for $k'>k$ 
\begin{equation}\label{Zorichaccelerations}
\widetilde{Q}(k,k'):= Q(n_k,n_{k'}), \qquad  \widetilde{Z}_k=\widetilde{Z}_k(T):= \widetilde{Q}(k,k+1)= Q(n_k, n_{k+1}). 
\end{equation}
If $T$ has an Oseledets regular extension (in the sense of Definition~\ref{def:Oseledetsextension}), let $(\Gamma^{(n)}_{n\in\mathbb{N}})$ for $x\in \{s,c,u\} $ the the sequences of stable, central and unstable spaces provided by Definition~\ref{def:Oseledetsextension} and denote by 
\begin{equation}\label{projections}
\Pro _{x}^{(n)}: \mathbb{R}^d\to  \Gamma_x^{(n)}, \qquad \text{for} \ x\in \{s,c,u\}, \ n\in\mathbb{N}.
\end{equation}
the standard orthogonal projection (in $\mathbb{R}^d$) to the subspace $\Gamma_x^{(n)}$. Then, for the acceleration along the sequence $(n_k)_{k\in\mathbb{N}}$ we adopt the notation:
\begin{equation}\label{accelerationssu}
\widetilde{\Gamma}_x^{(k)}:= \Gamma_s^{(n_k)}, \qquad \widetilde\Pro _{x}^{(k)}:= \Pro _{s}^{(n_k)},\qquad \textrm{for}\ x\in \{s,c,u\}, \ n\in\mathbb{N},
\end{equation}
and refer to $\widetilde{\Gamma}_x^{(k)}$, $x\in\{s,c,u\} $ as the stable, central and unstable spaces respectively for the acceleration; the operators $\widetilde\Pro _{x}^{(k)}: \mathbb{R}^d \to \widetilde{\Gamma}_x^{(k)}$, for  $x\in\{s,c,u\} $, are the corresponding projections. 
\subsubsection{The Regular Diophantine Condition}\label{sec:regularDC}
We can now formulate the Diophantine condition, that we call \emph{Regular Diophantine Condition} (or $RDC$). The central Diophantine-type condition is expressed in terms of convergence of two series (see the forward and backward conditions $(F)$ and $(B) $ in Definition~\ref{def:RDC} below), describing the forward and backward growth of an acceleration of the cocycle along good returns.  
  It will be crucial for us to consider times when not only these sequences converge (we will show that these series always converge along a sequence of effective Oseledets acceleration times, see Definition~\ref{def:effectiveOsedeletsseq}), but they 
are \emph{uniformely} bounded. 
The accelerating sequence  is required to be not too sparse, namely has \emph{linear} growth (see $(ii)$) and that the matrices of the further acceleration grow subexponentially (see \ref{conditionS}). 

\begin{definition}[Regular Diophantine condition, or RDC]\label{def:RDC} We say that a (standard) IET $T$ and its rotation number $\gamma(T)$  satisfy the \emph{Regular Diophantine Condition}, or for short the $(RDC)$,  if:
\begin{itemize}
\item[(i)] $T$ has an Oseledets generic extension (in the sense of Definition~\ref{def:Oseledetsextension}),

  \item[(ii)] there exists a sequence $(n_k)_{k\in\mathbb{N}}$  of \emph{good return times} (see Definition~\ref{def:goodreturns}) growing at linear rate, i.e.~such that $\lim_{k\to \infty}\frac{n_k}{k}<+\infty$,  
\end{itemize}
and, if $\widetilde{Z}_k$ and $\widetilde{Q}(k,k')$, for $k<k'$,  denote the matrices of the Zorich cocycle acceleration along $(n_k)_{k\in\mathbb{N}}$ as defined in \eqref{Zorichaccelerations} and $\widetilde{\Gamma}_x^{(k)}$ and $\widetilde{\Pro }_{x}{(k)}$, for $x\in \{s,c,s\}$ denote the spaces and projections defined in \eqref{accelerationssu}, we also have that:
\begin{itemize}
\item[(iii)]
there exist  constants $K^\pm =K^\pm(T) >0$, $\delta>0$ and an increasing subsequence $(k_m)_{m\in \mathbb{N}}$ growing  at a linear rate, i.e.~such that $\limsup_{m\to\infty} \frac{k_m}{m} < +\infty$,  such that the following conditions hold:
\begin{align*}
 \tag{Condition {[B]}}&  \sum_{k=1}^{k_m}{ ||\widetilde{Q}(k,k_m)_{| \widetilde{\Gamma}_s^{(k)}}|| \,||\widetilde{\Pro }_{s}^{(k)}|| \,||{\color{black}\widetilde{Z}_{k-1}}||} \leq K^- , & \text{for \ all}\ m \in \mathbb{N}; \label{conditionB} \\
\tag{Condition {[F]}}&  \sum_{k=k_m +1}^{\infty}{||\widetilde{Q}(k_m,k)^{-1}_{| \widetilde{\Gamma}_u^{(k)}}|| \,||\widetilde{\Pro }_{u}^{(k)}|| \,||{\color{black}\widetilde{Z}_{k-1}}|| } \leq K^+ , & \text{for \ all}\ m \in \mathbb{N};\label{conditionF} \\ 
\tag{Condition {[S]}}&   \lim_{k \rightarrow +\infty}{\frac{\log ||\widetilde{Q}(k_m,k_{m+1})||}{m}} = 0 ;& \label{conditionS}\\
\tag{Condition {[A]}}& \angle(\widetilde{\Gamma}^{(k_m)}_{x_1},   \widetilde{\Gamma}^{(k_m)}_{x_2}) > \delta \ \text{for \ all } x_1 \neq x_2 \in \{ s, c ,u\}  & \text{for \ all}\ m \in \mathbb{N}; \label{conditionA}
\end{align*}
\end{itemize}
\end{definition}
\noindent [Here the letter {[S]} is chosen to remind of \emph{Subexponential}, {[A]} for \emph{Angle} condition, while {[B]} and {[F]}   stay respectively for  \emph{Backward} and \emph{Forward}  respectively since they impose a certain control of growth on the forward or respectively backward iterates of the (accelerated) cocycle.]
\medskip 

We will prove in Section~\ref{sec:fullmeasure} the following Theorem that shows that the $(RDC)$ condition is satisfied by the rotation numbers of a full measure set of (standard) IETs: 

\begin{thm}[full measure of the $(RDC)$]\label{thm:fullmeasure}
The set  of (standard) IETs in $ \mathcal{I}_d$ which satisfy the $(RDC)$ condition in Definition~\ref{def:RDC} has full measure with respect to the Lebesgue measure on $\mathcal{I}_d$.
\end{thm}
\noindent Section~\ref{sec:fullmeasure} is fully devoted to presenting the proof of the Theorem~\ref{thm:fullmeasure}. 
Recalling the Definition~\ref{def:fullmeasure} of full measure set of rotation numbers, we immediately have:
\begin{cor}[full measure $(RDC)$ rotation numbers]\label{fullmeasure_RDC_rotnumbers}
The set of rotation numbers which satisfy $(RDC)$ has full measure.
\end{cor}

\subsection{Affine Shadowing under the Regular Diophantine Condition}\label{sec:generalDC}
We will now state and prove the general case of the affine shadowing. The  Regular Diophantine Condition  for (irrational) GIETs, which is the condition we will need to prove affine shadowing, is defined through the standard IET conjugated to it:
\begin{definition}[$(RDC)$ for GIETs]\label{def:RDCforGIET} We say that an infinitely renormalizable generalized IET $T$ with irrational rotation number satisfy the Regular Diophantine Condition $(RDC)$ iff its rotation number $\gamma(T)$ satisfies the $(RDC)$ given by Definition~\ref{def:RDC}.
\end{definition}
\noindent This condition is satisfied by a full measure set of GIETs with irrational rotation number (in the sense of Definition~\ref{def:fullmeasure}) by Corollary~\ref{fullmeasure_RDC_rotnumbers}.

\smallskip
The main result is Theorem~\ref{shadowing} below,  formulated as a dichotomy, which shows that, if the evolution of a GIET under renormalization \emph{escapes} (i.e~does not stay bounded in the $\mathcal{C}^1$ sense, see the remarks after the statement), then the evolution of its shape slope vector   can be \emph{shadowed} by the orbit under renormalization of (the slope vector of) an  AIET (hence the name \emph{affine shadowing}). 
The dichotomy is expressed in terms of the evolution of the \emph{shape log-slope vector} $\vect{\omega}(T)$ which we recall is defined to be the log-slope vector $\vect{\omega}(A_T)$ of the \emph{shape} $A_T$ of $T$, see \S~\ref{sec:averageslope}.

\subsection{The affine shadowing dichotomy}
We can now state the main result. We will consider as renormalization $\mathcal{R}$ the acceleration of $\mathcal{Z}$ corresponding to the sequence $(n_{k_m})_{m\in\mathbb{N}}$ given by the $(RDC)$ in Definition~\ref{def:RDC} and separate two cases according to whether the shape log-slope vectors along the orbit of renormalization are bounded or diverge. Consider therefore
$$
\mathcal{R}^m(T):=  \widetilde{\mathcal{Z}}^{k_m}(T) = \mathcal{Z}^{n_{k_m}}(T), \qquad \omega(\mathcal{R}^mT)=\widetilde{\omega}^{(k_m)}={\omega}^{(n_{k_m})}, \qquad \text{for\ all}\ m \in\mathbb{N}.
$$
\begin{thm}[\bf Affine shadowing lemma]
\label{shadowing}
Let $T$ be a GIET which satisfies the Regular Diophantine Condition $(RDC)$. Let $(n_k)_{k\in\mathbb{N}}$ be the accelerating sequence of 
s given by the $(RDC)$ (see $(ii)$ in Definition~\ref{def:RDC}). Then we have the following dichotomy. 
Either we are in:\smallskip

\begin{itemize}
\item[{\bf Case 1}]{\bf (recurrence):} The sequence $\{\vect{\omega}_{n}\}_{n\in \mathbb{N}}$ is bounded along renormalization times, i.e.~there exists a $V>0$ such that
$$\Vert \omega(\mathcal{R}^mT)\Vert= || \vect{\omega}_{n_{k_m}} ||\leq V, \qquad \textrm{for\ all}\ m\in\mathbb{N},
$$
\end{itemize}
\smallskip
or, alternatively, we have:
\smallskip
\begin{itemize}
\item[\bf Case 2]{\bf (affine shadowing):} There exists $\vect{v} \in E_u$ such that $ \vect{\omega}_n = Q(n,0) \vect{v} + o(\vect{\omega}_n) $ for every $n\in\mathbb{N}$, i.e.
$$
\lim_{n\to \infty}\frac{\| \vect{\omega}_n - \vect{v}^{(n)}\| }{\| \vect{v}^{(n)}\|}=0, \qquad\textrm{where}\quad  \vect{v}^{(n)}:= Q(n,0) \vect{v} .
$$
\end{itemize}
\end{thm}

\noindent {\bf Case 1} is called \emph{recurrent} case since asking that the sequence $(|| \vect{\omega}_{n} ||)_{n\in\mathbb{N}}$ 
 is bounded along the subsequence $(n_{k_m})_{m\in\mathbb{N}}$ turns out to be equivalent to the fact that the iterates $\{ \mathcal{R}^m (T), \, m\in\mathbb{N}\}$ stay at  $\mathcal{C}^1$-bounded distance from $\mathcal{I}_d\subset \mathcal{X}_d$ together with their inverses; furthemore, in this case, one can show that the orbit $\{ \mathcal{Z}^{n}(T), \ n\in \mathbb{N}\}$ of $T$ under Zorich renormalization is \emph{recurrent} (along the subsequence $(n_{k_m})_{m\in\mathbb{N}}$) to a subset  which is \emph{bounded} in the space $\chi^r_d$ of GIETs (in view of the conditions imposed on the recurrence sequence $(n_{k_m})_{m\in\mathbb{N}}$, which are in particular  \emph{good return times} in the sense of Definition~\ref{def:goodreturns}, as requested by the $(RDC)$ in Definition~\ref{def:RDC}).
 
 \smallskip
\noindent {\bf Case 2}, on the other hand, shows that the orbit $\{ \vect{\omega}_n, \ {n\in\mathbb{N}}\}$ can be approximated, up to a lower order term, by the orbit $\{ Q(n,0) \,\vect{v}, \ {n\in\mathbb{N}}\}$ of the log-slope vector $\vect{v}$ under the Rauzy-Veech cocycle. If this case, if $\overline{T}$ is an AIET $T_0$  with the same rotation number $\gamma(\overline{T})=\gamma(T)$ and  log-slope vector $\vect{v}$,\footnote{One can show that such an AIET always exists. Indeed it is shown in \cite{MMY2} that the cone $Af\! f(\gamma,\vect{v})$ of AIET with rotation number $\gamma$ and log-slope vector $\vect{v}$  is not empty as long as $\vect{v}$ is orthogonal to the
  length vector $\underline{\lambda}$ of the standard IET with rotation number $\gamma$ (which is unique since one can show that IETs whose rotation number satisfies the $(RDC)$ are uniquely ergodic). Since by assumption $\vect{v}$ shadows $\vect{\omega}_n$ (see the statement of Case 2) and the growth of $\{\vect{\omega}_n\}_{n\in\mathbb{N}}$ is slower  than the growth of the norms of  $\{ Q(n,0)\}_{n\in\mathbb{N}}$, 
	it follows that $\vect{v}$ does not project to the leading Oseledets eigenspace, thus $Af\! f(\gamma,\vect{v})$ is not empty and any $T_0\in Af\! f(\gamma,\vect{v})$ is an affine shadow.}
the shape log-slope vectors of the orbit $\{ \mathcal{Z}^n(T)\}_{n\in\mathbb{N}}$ of the GIET $T$ under renormalization can be \emph{shadowed} (up to lower order terms, i.e.~can be approximated in the first order) by the shape log-slope vectors of the orbit $\{ \mathcal{Z}^n(\overline{T})\}_{n\in\mathbb{N}}$ of the affine IET $\overline{T}$. 
For this reason we call the vector $\vect{v}$ the \emph{affine shadow} of $T$.  

\medskip
The rest of this section is devoted to the proof of Theorem~\ref{shadowing}. 

\subsubsection{The shadowing lemma, general case}\label{sec:general}

We now turn to the proof of Theorem \ref{shadowing} in the general case, assuming the Diophantine-type condition $(RDC)$ namely that the rotation number of $T \in \mathcal{X}^3$ is a good rotation number (see Definition \ref{def:RDC}). With the definition of good rotation number comes a sequence $(n_k)_{k \in \mathbb{N}}$ and a subsequence $(n_{k_m})_{m \in \mathbb{N}}$ we will be working with throughout the proof of Theorem \ref{shadowing}.  As in the previous section, we first give an outline of the proof (see also the sketch of the proof of Proposition \ref{asymptotic} in the previous section about the periodic-type special case).

\medskip
\noindent {\it Outline}: We first prove (in \S~\ref{sec:error}) a \emph{linear approximation} result (Lemma~\ref{lemma3}, which is a generalization of Lemma~\ref{lemma1}  in the proof of Proposition \ref{asymptotic}) that shows that, thanks to the classical distortion bounds given by Lemma~\ref{bound1}, the error between $\vect{\omega}_n$ and the linear evolution of the  log-slope vector transform under the cocyle  is comparable to the norm of the cocycle matrices.  In \S~\ref{sec:error} we then define the candidate vector $\vect{v}$ to be the \emph{shadow} and show that it is well defined (see Lemma~\ref{convergence_shadow}).

There is a natural candidate for the shadow (what we call the shadow is the vector $\vect{v}$ in the statement if Theorem \ref{shadowing}). At each step of renormalization, when trying to approximate $\vect{\omega}_{n_{k+1}}$ by $\widetilde{Z}(k) \vect{\omega}_{n_k}$, an error is made in both the stable and unstable direction. Philosophically speaking we can ignore the error in the stable direction as it will be eaten away by further steps of renormalization. In the unstable direction, we get an error whose size is controlled by $||\widetilde{Z}(k)||$. Provided $||(\widetilde{Z}(k)||$ is not too big we can add a very small correction $v_k$ at the start which is going to be magnified by the renormalization (to reach a size of the order $\alpha^k ||v_k||$, where $\alpha = \exp(\lambda) > 1$ and $\lambda$ is the smallest positive Lyapunoff exponent). 

The heart of the proof is given by Proposition~\ref{unstable} in \S~\ref{sec:unstable}, which shows that the basic dichotomy we are trying to prove holds for the \emph{unstable} part. 
  We refer  to this result (i.e.~Lemma~\ref{lemma1}) as \emph{linear approximation}. Lemma~\ref{lemma1}  The projection $\Pro _s(\vect{\omega}_n)$ of $\vect{\omega}_n$ onto the stable space $E_s$ is controlled through Lemma~\ref{lemma2}, which is valid for any $T$ with periodic rotation number and shows that the part in the stable space always remains bounded.

We then consider  iterates of renormalizations of $T$ and consider separately two cases:  (1) if the log-slopes  are bounded, we are in Case 1; otherwise, (2) if the log-slopes are  not bounded, in virtue of the control of the stable part (given by Lemma~\ref{lemma2}), the component  in the unstable space is also unbounded. To prove that in this case we are in Case 2, namely we can build an \emph{affine shadow}, we  wait for a time when this compoment is large compared to the mistake that one makes when comparing the actual growth of the slopes with how it transforms linearly. If  one starts renormalizing from that moment, the slope change almost linearly up to a mistake that is more and more negligible as slopes in the unstable space grow exponentially fast. Thus,  adjusting using smaller and smaller corrections (see \eqref{def:shadow} and \eqref{eq:smallercorrections}) allows to find a vector shadowing the slopes. This is done rigorously through definition \eqref{def:shadow} of the \emph{shadow} and the proof of Proposition~\ref{asymptotic} presented below.

\smallskip
%
%
%
%
%
%
%
%
%
%
%

\subsubsection{Linear approximation error estimate}\label{sec:error}
We start with the following lemma which is valid for any infinitely renormalizable $T$ and is a direct generalisation of Lemma \ref{lemma1} to the non-periodic case. In this proof, though, we use the notation and decomposition for special Birkhoff sums introduced in \S~\ref{sec:SBS} and \S~\ref{decompBS}.

\smallskip
{\begin{lemma}[linear approximation in the general case]\label{lemma3}
For any infinitely renormalizable GIET $T$, 
for any $n_2 > n_1$ we have  
$$ || \vect{\omega}_{n_{2}} - Q(n_1,n_2) \vect{\omega}_{n_1} || \leq K(n_1) \Vert Q(n_1,n_2)\Vert \leq  K_T\, ||Q(n_1,n_2)||, \qquad \textrm{where} \quad K(n_1):=|N|\left(\mathcal{V}^{(n_1)}(T)\right) $$
and $K_T:=|N|(T)$, where $|N|(T)$ denotes the total non-linearity of $T$ (see Definition~\ref{def:Ns}).
\end{lemma}
\noindent As a special case of the above formula, setting $n_1=0$ and $n_2=n_k$ and recalling that $\widetilde{\vect{\omega}}_k= \vect{\omega}_{n_k}$ and $\widetilde{Z}_k:=Q(n_k,n_{k+1})$, we then have:
\begin{cor}
\label{corlemma3} For all $k$, we have $ || \widetilde{\vect{\omega}}_{{k+1}} - \widetilde{Z}_k \widetilde{\vect{\omega}}_{k} || \leq K_T ||\widetilde{Z}_k||.$
\end{cor}
Before proving the Lemma, we state and prove an intermediate step, which  will be used also later and connects the  shape log-average vectors with values of special Birkhoff sums of $f:=\log DT$.
\begin{lemma}[shape log-averages and special Birkhoff sums]\label{lemma:SBSomega}
For any irrational $T$, for every $n\in\mathbb{N}$ and $1\leq j\leq d$, there exists a point $x^{(n)}$
$$
x^{(n)}_j\in I^{(n)}_j,\quad \text{such that}\quad  (\omega_n)_j:= f^{(n)}_j (x^{(n)}_j),
$$
where $f^{(n)}_j$ is the $j^{th}$ branch $f^{(n)}_j$ of the special Birkhoff sum $f^{(n)}$ of the function $f:=\log DT$ and where $(\omega_n)_j$ is the $j^{th}$ entry of the shape log-slope vector $\vect{\omega}_n = \vect{\omega}(\mathcal{V}^n(T))$. Moreover, for any $0\leq m\leq n$,
$$
\Vert f^{(n)}_j - (\omega_n)_j\Vert_\infty:= \sup_{x\in I^{(n)}_j} \left| f^{(n)}_j (x)- (\omega_n)_j\right|\leq  |N| (\mathcal{Z}^m(T))\leq |N|(T).
$$ 
\end{lemma}
\begin{proof} For any $n\in\mathbb{R}$ and any $1\leq j\leq d$, 
 using the chain rule and recalling the definition of special Birkhoff sums (see \S~\ref{sec:SBS}) we have that
\begin{equation}\label{SBS_omega}
 \log \big( D T^{q^{(n)}_j}\big)(x) =S_{q^{(n)}_j}f(x) =f^{(n)} \big(x), \qquad \text{for\ all}\ x\in I^{(n)}_j. 
\end{equation} 
Thus, recalling that   $\vect{\omega}_{n}=\log \vect{\rho}_n$ (see Definition~\ref{def:rhoomega}), by mean value theorem and by Remark~\ref{rk:logvectorRn} (see in particular equation~\eqref{eq:logvectorTn}), 
for any $n\in\mathbb{R}$ and $1\leq j\leq d$,   there exists a point 
\begin{equation}\label{SBS_omegapt}
x^{(n)}_j\in I^{(n)}_j,\quad \text{such that}\quad  (\rho_n)_j:=  D( T^{q^{(n)}_j}  )(x^{(n)}_j), \qquad (\vect{\omega}_n)_j:= \log(\rho_n)_j=f^{(n)} \big(x^{(n)}_j\big) .
\end{equation}
Thus, 
 the classical distorsion bounds (Lemma~\ref{bound1}),  taking logarithms, shows that  
 special Birkhoff sums of each continuity interval have bounded fluctuations, namely 
$$
\left|f^{(n)} (x) -({\omega}_n)_j \right| =\left|f^{(n)}(x)-f^{(n)} \big(x^{(n)}_j\big) \right| \leq |N|(T), \qquad \textrm{for\ all}\  n\in\mathbb{N}, \quad 1\leq j\leq d, \ x \in I^{(n)}_j .
$$
This proves the estimate by $|N|(T)$. 
\smallskip

Fix now any $0\leq m<n$; to prove the estimate by $|N|(\mathcal{Z}^m(T))$, one can apply the estimate that we already proved to the GIET $\overline{T}:= \mathcal{Z}^{m} T$ and the function $g:=\log D (\mathcal{Z}^{m} T)$. Notice that for any $n\geq m$, since by the cocycle property of Birkhoff sums $f^{(n)}=f^{(n-m)}\circ f^{(m)}$  and $f^{(m)}=\log D (T_m)$, 
$f^{(n)}$ is a rescaled version (obtained by conjugating by a linear map) of $g^{(n-m)}$. 

\smallskip
Finally,  Property $(ii)$ in Proposition~\ref{prop:Nproperties} gives that $|N|(\mathcal{V}^{n}(T))\leq |N|(T)$. This concludes the proof. 
\end{proof}
We can now prove the linear approximation stated as Lemma~\ref{lemma3}.
\begin{proof}[Proof of Lemma~\ref{lemma3}] Fix any $n_2\in\mathbb{N}$ and $1\leq j\leq d$.  As in the proof of the previous Lemma, let $\overline{x}_j$ be a point in $I^{(n_2)}_j$ such that $({\omega}_{n_2})_j=f^{(n_2)}_j( \overline{x}_j)$ where $ f:=\log DT$ and 
and $f^{(n_2)}_j$ is the $j^{th}$ branch of the special Birkhoff sum $f^{(n_2)}$ (see \S~\ref{sec:SBS}). 
Thus, using one step of the decomposition of (special) Birkhoff sums introduced in \S~\ref{decompBS}, if for each $0\leq \ell < Q(n_1,n_2)_{ij}$, we let $ j_\ell\in \{1,\dots, d\}$ be the index such that $ {{T}_{n_1}}^{\ell} ( \overline{x}_j)$ belongs the interval $I^{(n)}_{j_\ell}$, 
\begin{equation}\label{decompn2} 
({\omega}_{n_2})_j = f^{{(n_2)}} ( \overline{x}_j) 
= \sum_{\ell=0}^{Q(n_1,n_2)_{j}-1} f_{j_\ell}^{{(n_1)}} \big({{T}_{n_1}}^{\ell} ( \overline{x}_j)\big), \qquad \text{where}\ Q(n_1,n_2)_{j}:= \sum_{1\leq i\leq j}Q(n_1,n_2)_{ij}.
\end{equation}
By the previous Lemma~\ref{lemma:SBSomega}, 
\begin{equation}\label{deviationsSBS} \Vert f_{j_\ell}^{{(n_1)}} - ({\omega}_{n_1})_{j_\ell}\Vert_\infty \leq K(n_1)\leq K,
\end{equation} where $K(n_1)$ and $K_T$ are defined as in the statement of the Lemma.  

\smallskip
Note now that, by the dynamical interpretation of the cocycle entries (see \S~\ref{sec:entries}) and matrix multiplication,
\begin{equation}\label{matrixproduct}
\begin{split}
\sum_{\ell=0}^{Q(n_1,n_2)_{j}-1}{(\omega_{n})_{j_\ell}}&=
 \sum_{1\leq i\leq d} Card\{1\leq \ell < Q(n_1,n_2)_{j}, \quad j_\ell=i\}\,  ({\omega}_{n_1})_i=\\ &=
 \sum_{1\leq i\leq d}  Q(n_1,n_2)_{ij} ({\omega}_{n_1})_i =\left(Q(n_1,n_2)\, \vect{\omega}_{n_1}\right)_j.\end{split}\end{equation}
Thus, combining \eqref{decompn2} and \eqref{matrixproduct} and estimating the difference through \eqref{deviationsSBS}, 
$$ 
\left| (\vect{\omega}_{n_2})_j -\left(Q(n_1,n_2)\, \vect{\omega}_{n_1}\right)_j\right|\leq \sum_{\ell=0}^{Q(n_1,n_2)_{j}-1}  \left| {f^{{(n_1)}_i}}( {{T}_{n_1}}^{\ell} ( \overline{x}_j)\big) -({\omega}_{n})_{j_\ell} \right|\leq \Vert Q(n_1,n_2)\Vert K(n_1)\leq \Vert Q(n_1,n_2)\Vert K_T.
$$
Since this holds for every $1\leq j\leq d$, this completes the proof.
\end{proof}}

\subsubsection{Building the shadow}\label{sec:shadow}
We can now construct the affine shadow $\vect{v}$ and show that it is well defined. For each $k\in\mathbb{N}$, consider the \emph{error}  between $\vect{\omega}_{n_{k}}$ and the linear evolution of $\vect{\omega}_{n_{k-1}}$, namely $\vect{\omega}_{n_{k}} - \widetilde{Z}_{k-1}\vect{\omega}_{n_{k-1}}$. Let $\Pro _u: \mathbb{R}^d\to \Gamma_u^{(k)}$  be  the projection on the unstable space  at stage $k$ (for the Zorich induction). Set 
$$e_k:= \Pro _u(\widetilde{\vect{\omega}}_{{k}} - \widetilde{Z}_{k-1}\widetilde{\vect{\omega}}_{{k-1}} )= \Pro _u(\vect{\omega}_{n_{k}} - \widetilde{Z}_{k-1}\vect{\omega}_{n_{k-1}} )$$
and formally define 
\begin{equation}\label{def:generalshadow} \vect{v}: = \sum_{k=1}^{+\infty}{v_k} + \Pro_u(\vect{\omega}_{0}) , \qquad  \text{where}\ v_k: = Q(0,n_k)^{-1}e_k = \widetilde{Q}(0,k)^{-1}e_k.
\end{equation}
 We just need to check that this series converges.

\begin{lemma}\label{convergence_shadow}
The series in \eqref{def:generalshadow} converges and hence $\vect{v}$ is well defined.
\end{lemma}
 \begin{proof}
Recall that by one of the assumptions in the $(RDC)$ (see Definition~\ref{def:RDC}), $T$ is Oseledets generic; thus there exists $C=C(T)>0$ such that 
$$Q(0,n_k)^{-1}_{|\Gamma_u(T^{(n_k)})}\leq Ce^{-\lambda n}, \qquad \textrm{for\ every}\ n\in \mathbb{N},$$ where 
we can take $\lambda>0$ to be $\lambda:=\theta_g/2$ where $\theta_g>0$ is the smallest positive Lyapunov exponent of the Zorich cocyle. Since $e_k\in \Gamma_u^{(n_k)}$, by  Lemma \ref{lemma3},  we can bound $||e_k||$ by $K_T ||\widetilde{\Pro }^k_u|| ||\widetilde{Z}_{k-1}||$, so that we can estimate
$$
||v|| - ||\Pro_u(\vect{\omega}_{0})|| \leq   \sum_{k=1}^{+\infty}||{v_k}|| = \sum_{k=1}^{+\infty}||{Q(0,n_k)^{-1}_{|\Gamma_u(T^{(n_k)})} e_k}|| \leq  C K_T \sum_{k=1}^{+\infty}e^{-\lambda n}  ||\widetilde{\Pro }^{(k)}_u|| \, ||\widetilde{Z}_{k-1}||,
$$
The assumption that $T$ is Oseledets generic also implies that $||\widetilde{\mathrm{P}}^{(k)}_u||$, which is comparable with $\angle (\Gamma_u^{(n_k)}, \Gamma_s^{(n_k)})$, grows subexponentially fast, i.e.~for every $\epsilon>0$ there exists $C_1=C_1(T,\epsilon)$ such that $||\widetilde{\Pro}^{(k)}_u||\leq C_1 e^{\epsilon n}$ for every $n\in\mathbb{N}$. Finally  Condition $(S)$ of the $(RDC)$ (see Definition~\ref{def:RDC}) shows that also $||\widetilde{Z}_{k-1}||\leq C_2 e^{\epsilon (k-1)}$ for some $C_2=C_2(T, \epsilon)>0$ and every $k\in\mathbb{N}$. Choosing $\epsilon<\lambda/2$, guarantees that the series converges and thus that $\vect{v}$ is well defined.
\end{proof}

\subsubsection{Control of the unstable part.}\label{sec:unstable}
The following Proposition is at the heart of the desired dichotomy.

\begin{proposition}[unstable part shadowing]\label{unstable}
We have the following dichotomy. Either $\vect{v} \neq 0$ and in this case
\begin{itemize}
\item[(1)]   for any $\epsilon > 0$,  $\Pro _u(\vect{\omega}_n) = Q(0,n) \vect{v} + o(||\Pro _u(\vect{\omega}_n)||^{\epsilon}) = Q(n,0) \vect{v} + o(||Q(0,n)\vect{v}||^{\epsilon}) $
\end{itemize}
or, otherwise, $\vect{v} = 0$ and in this case
\begin{itemize}
\item[(2)]  there exists $K(T) > 0$ such that for all $m \in \mathbb{N}$, $|\Pro _u(\vect{\omega}_{n_{k_m}}) | \leq K(T)$.
\end{itemize}

\end{proposition}

\begin{proof}

Assume to begin with that $\vect{v} \neq 0$. Because $T$ is Oseledets generic, $Q(0,n_k) \vect{v}$ grows exponentially fast. Recall that by definition, 
$$ \vect{v} =  \Pro_u(\vect{\omega}_0) + \sum_{j=1}^{\infty}{\widetilde{Q}(0,j)^{-1} \Pro _u(\widetilde{\vect{\omega}}_{j} - \widetilde{Z}_{j-1}\widetilde{\vect{\omega}}_{j-1})}.$$ For each given $k\in \mathbb{N}$, multiplying by $\widetilde{Q}(0,k) = Q(0,n_k)$ and recalling the cocycle relation \eqref{cocyclerel}, 
we get 
\begin{equation}
\label{eq1}
 \widetilde{Q}(0,k) \vect{v} = \widetilde{Q}(0,k) \Pro_u(\vect{\omega}_0) +  \sum_{j=1}^k{\widetilde{Q}(j,k) \Pro _u(\vect{\omega}_{n_{j}} - \widetilde{Z}_{j-1}\vect{\omega}_{n_{j-1}})} + \sum_{j = k+1}^{+\infty}{\widetilde{Q}(k,j)^{-1} \Pro _u(\vect{\omega}_{n_{j}} - \widetilde{Z}_{j-1}\vect{\omega}_{n_{j-1}})}.
\end{equation} Let us show that the first two terms of the RHS sum up to $\Pro _u(\vect{\omega}_{n_k}) $. Indeed, since $\Pro_u$ and $Q(j,k)$ commute for all $j < k + 1$,  and by the cocycle definitions $\widetilde{Q}(j,k)\widetilde{Z}_{j-1}=Q(n_{j},n_k) Q(n_{j-1}, n_{j}) = Q(n_{j-1},n_k)= \widetilde{Q}(j-1,k)$, we can write
 $$\sum_{j=1}^k{\widetilde{Q}(j,k) \Pro _u(\vect{\omega}_{n_{j}} - \widetilde{Z}_{j-1}\vect{\omega}_{n_{j-1}})} 
= \Pro_u \left(\sum_{j=1}^k{\widetilde{Q}(j,k)\vect{\omega}_{n_{j}} - \widetilde{Q}(j-1, k)\vect{\omega}_{n_{j-1}}}\right)= - \widetilde{Q}(0,k) \Pro_u(\vect{\omega}_0)+ \Pro_u(\vect{\omega}_k),$$ where the last equality exploits the telescopic nature of the sum. Thus, the sum of the first two terms of the RHS of \eqref{eq1} yields $$ \widetilde{Q}(0,k) \Pro_u(\vect{\omega}_0) +  \sum_{j=1}^k{\widetilde{Q}(j,k) \Pro _u(\vect{\omega}_{n_{j}} - \widetilde{Z}_{j-1}\vect{\omega}_{n_{j-1}})} = \Pro_u(\vect{\omega}_{n_k}).$$ We now show that the third term in the RHS of \eqref{eq1} grows subexponentially fast. Indeed $ ||\Pro _u(\vect{\omega}_{n_j} - \widetilde{Z}_{j-1}\vect{\omega}_{n_{j-1}})|| \leq K_T ||\Pro _u|| \,||\widetilde{Z}_{j-1}||$ (by Lemma \ref{lemma3}). Recall that by the condition \ref{conditionF} $$ \sum_{k=k_m +1}^{\infty}{||\widetilde{Q}(k_m,k)^{-1}_{| \widetilde{\Gamma}_u^{(k)}}|| \,||\widetilde{\Pro }_{u}{(k)}|| \,||\widetilde{Z}_{k-1}(T)|| } \leq K.$$ We thus obtain, for special times  $k_m$s $$ || \widetilde{Q}(k_m,0) \vect{v} - \Pro _u({\vect{\omega}}_{n_{k_m}}) || \leq  K_T \sum_{k=k_m +1}^{\infty}{||\widetilde{Q}(k_m,k)^{-1}_{| \widetilde{\Gamma}_u^{(k)}}|| \,||\widetilde{\Pro }_{u}{(k)}|| \,||\widetilde{Z}_{k-1}(T)|| } \leq K_T \,K. $$ 

\noindent We now interpolate. Consider arbitrary $n \in \mathbb{N}$ and let $m$ be such that $n_{k_m} < n \leq n_{k_{m+1}} $. We have
$$ Q(0,n) {v} - \Pro _u({ \vect{\omega}}_{n}) = Q(0,n) {v} - \Pro _u(\vect{\omega}_{n})  + Q(n_{k_m}, n)\Pro _u(\vect{\omega}_{n_{k_m}}) - Q(n_{k_m}, n)\Pro _u(\vect{\omega}_{n_{k_m}}). $$
\noindent One the one hand we have  
$$ Q(0,n) \vect{v} -  Q(n_{k_m}, n) \Pro _u(\vect{\omega}_{n_{k_m}}) = Q(n_{k_m}, n) \big(  \widetilde{Q}(k_m)v     - \Pro _u(\vect{\omega}_{n_{k_m}})\big)  $$ from which we get 
$$ || Q(0,n) \vect{v} -  Q(n_{k_m}, n) \Pro _u(\vect{\omega}_{n_{k_m}})  || \leq ||Q(n_{k_m}, n) || \,K K_T.$$
\noindent We have 
$$ Q(n_{k_m}, n)\Pro _u(\vect{\omega}_{n_{k_m}}) - \Pro _u(\vect{\omega}_{n})  = \Pro_u \big( Q(n_{k_m}, n) \vect{\omega}_{n_{k_m}} - \vect{\omega}_n \big) $$ and thus by Lemma \ref{lemma3}
$$|| Q(n_{k_m}, n)\Pro _u(\vect{\omega}_{n_{k_m}}) - \Pro _u(\vect{\omega}_{n})  || \leq K_T || \Pro_u(n) || \,||Q(n_{k_m}, n)|| $$ and putting the last two inequalities together we obtain 
$$ || Q(0,n) \vect{v} - \Pro _u(\vect{\omega}_{n}) ||  \leq K_T (|| \Pro_u(n) || + K) ||Q(n_{k_m}, n) ||.$$
\noindent Since $n\leq n_{k_{m+1}}$, $||Q(n_{k_m}, n) || \leq Q(n_{k_m}, n_{k_{m+1}}) = \widetilde{Q}(k_m, k_{m+1})$. The Diophantine-type condition \ref{conditionS} ensures that  $\frac{1}{m}\log || Q(n_{k_m}, n_{k_{m+1}}) ||$ tends to zero which implies that for any $\epsilon$ there exists $C_{\epsilon} > 0$ such that 
$$ || Q(n_{k_m}, n_{k_{m+1}}) || \leq C_{\epsilon} e^{\epsilon m}.$$ Similarly, $n \leq  n_{k_{m+1}}$.  Recall that the angles between $\Gamma_u^{(n)}, \Gamma_c^{(n)}$ and $\Gamma_s^{(n)}$ decrease at most subexponentially fast (by condition \ref{O-a}) for all $\epsilon > 0$, there exists $D_{\epsilon} > 0$ such that 
$$ || \Pro_u(n) || \leq D_{\epsilon}e^{n\epsilon}.$$
\noindent Since $n_{k_m}$ grows linearly in $k$ (by Condition $(RDC)$ in \ref{def:RDC}, item $(iii)$ ) we deduce the existence of $D'_{\epsilon} > 0$ such that 
 $$ || \Pro_u(n_{k_m}) || \leq D_{\epsilon}e^{m\epsilon}.$$ We conclude by showing the existence of $\lambda > 1$ such that for $m$ large enough, 
$$ || Q(0,n) \vect{v} || \geq \lambda^m.$$ Write $$ Q(n, n_{k_{m+1}}) \,Q(0,n) \vect{v} = \widetilde{Q}(0,k_{m+1}) \vect{v}. $$ 
\noindent Since $v$ belongs to the unstable space of the cocycle, there exists $\lambda' > 1$ such that for $m$ large enough $|| \widetilde{Q}(0,k_{m+1}) \vect{v} || \geq (\lambda')^m$. We thus get 
$$(\lambda')^m \leq ||  Q(n, n_{k_{m+1}}) \,Q(0,n) \vect{v} || \leq ||  Q(n, n_{k_{m+1}}) || \,|| Q(0,n) \vect{v} ||.$$ To conclude, observe that  $||  Q(n, n_{k_{m+1}})  \leq  || Q(n_{k_m}, n_{k_{m+1}}) || \leq C_{\epsilon} e^{\epsilon m}$ which implies 
$$ || Q(0,n) \vect{v} || \geq C_{\epsilon} \lambda'^m e^{-\epsilon m} = C_{\epsilon} \exp((\log \lambda' - \epsilon)m).$$ 
\noindent Take $\epsilon >0$ small enough so $\log \lambda' - \epsilon > 0$ and set $\lambda = \exp(\log \lambda' - \epsilon )$. For any sequence $b_n$ growing faster than $\lambda^n$ for $\lambda > 1$, a sequence $a_n$ growing subexponentially fast is $o(b_n^{\epsilon})$ for any $\epsilon > 0$. Applying this to $|| Q(0,n) \vect{v} || $ for the sequence growing at rate $\lambda^m$ and to $|| Q(0,n) \vect{v} - \Pro _u(\vect{\omega}_{n}) ||$ growing subexponentially fast we obtain that for any positive $\epsilon$,
$$ || Q(0,n) \vect{v} - \Pro _u(\vect{\omega}_{n}) || = o(||Q(0,n) \vect{v} ||^{\epsilon}).$$ A similar reasoning shows that $|| \Pro _u(\vect{\omega}_{n}) ||$ is larger than $(\lambda')^m$ for $m$ large enough and that we also have 
$$ || Q(0,n) \vect{v} - \Pro _u(\vect{\vect{\omega}}_{n}) || = o(|| \Pro _u(\vect{\vect{\omega}}_{n}) ||^{\epsilon}).$$ 
\vspace{2mm} \noindent We are left with the case $ \vect{v} = 0$.  In that case we still get 
$$ \Pro _u (\vect{\omega}_{n_k})  = \Pro _u(\vect{\omega}_{n_k}) - Q(n_{k},0) \vect{v} = - \sum_{j =k+1}^{+\infty}{\widetilde{Q}(j,k)^{-1} \Pro _u(\vect{\omega}_{n_{j}} - \widetilde{Z}_{j-1}\vect{\omega}_{n_{j-1}})} $$ (same calculation as in \ref{eq1} above). Thus we get, at special times $n_{k_m}$
$$ || \Pro _u(\vect{\omega}_{n_{k_m}}) || \leq K_T \sum_{j = k_m}^{+\infty}{||\widetilde{Q}(j,k_m)^{-1}_{|\widetilde{\Gamma}^j_u}|| \,||\widetilde{\Pro }_u^j|| \,||\widetilde{Z}_{j-1}||}.$$ By condition $(F)$ we get that for all $m\in\mathbb{N}$, 
$$|| \Pro _u(\vect{\omega}_{n_{k_m}}) || \leq K_T \,K.$$
\end{proof}
{ 
\subsection{Control of the central part}
We now turn to controlling the central component of $(\vect{\omega}_n)_{n\in\mathbb{N}}$ along the subsequence $({n_{k_m}})_{m\in\mathbb{N}}$. Here the control {\color{black}will exploit the invariance of the boundary operators defined in \S~\ref{sec:obs_boundary} and \S~\ref{sec:boundaryGIET}.}
\begin{proposition}
\label{central}
For any $T$ satisfying the $(RDC)$ there exists $C_c(T) > 0$ such that we have \begin{enumerate}
\item  if $\vect{v} = 0$, $ ||\Pro_c(\vect{\omega}_{n_{k_m}})|| \leq C_c(T), \qquad \text{for\ all}\ m \in \mathbb{N}$.
\item if $\vect{v} \neq 0$,  $|| \Pro_c(\vect{\omega}_{n}) || =  o(||\Pro_u(\vect{\omega}_n)||^{\epsilon})) $ for all $\epsilon > 0$.
\end{enumerate}
\end{proposition}
\noindent We need a couple of Lemmata. Recall that $B_n : \mathbb{R}^d \longrightarrow \mathbb{R}^s$ denotes the boundary operator  (see \S~\ref{boundary}) at $T^{(n)}$.
\begin{lemma}
\label{lemma5}
For all $n \in \mathbb{N}$,
$$(B_n)_{|\Gamma_s^{(n)}} = 0.$$
\end{lemma}
\begin{proof}
The $B_n$ are uniformly bounded (as the sequence $B_n$ takes finitely many values in $\mathcal{L}(\mathbb{R}^d, \mathbb{R}^\kappa)$. They are also invariant under the action of the Zorich cocycle thus for any $w \in \Gamma_s^n$,  $$ B_n(w) = B_{n+k}(Q(n,n+k)w).$$ By definition of $\Gamma_s^{(n)}$, $Q(n,n+k)w$ tends to $0$ when $k$ tends to infinity which implies $ B_n(w) =0$.
\end{proof}
\begin{lemma}
\label{lemma6} 
There exists constants $D_u$ and $\lambda_u < 1$ such that for all $n \in \mathbb{N}$,
$$ || (B_n)_{|\Gamma_u^{(n)}} || \leq D_u \lambda_u^{n}.$$
\end{lemma}
\begin{proof}
By invariance of the $\Gamma_u^{(n)}$'s, we have that for any $v \in \Gamma_u^n$, 
$$ B_n(v) = B_0(Q(0,n)^{-1}v).$$ Since $T$ is Oseledets generic there exists constants $D'_u > 0$ and $ \lambda_u < 1$ such that $||Q(0,n)^{-1}v|| \leq D_u' \lambda_u^n$. Thus $$ ||B_n(v) || \leq ||B_0|| D'_u \lambda_u^n.$$
\end{proof}
\begin{lemma}
\label{lemma7}
For all $\epsilon > 0$  there exists constants $D_c(\epsilon)$  such that for any $n \in \mathbb{N}$ and any $w \in \Gamma_c^{(n)}$,
$$  ||w||  \leq D_c\, \angle(\Gamma_c, \Gamma_u^{(n)} \oplus \Gamma_s^{(n)})  || B_n \, w || .$$
\end{lemma}
\begin{proof}
From the following three observations:
\begin{enumerate}
\item by Lemma \ref{lemma6}, $\angle(\Gamma_u^{(n)}, \Ker B_n) \leq D''_u \lambda_u^n$ for a certain $D''_U > 0$. 
\item we have that $\mathrm{dim}(\Gamma_u^{(n)} \oplus \Gamma_s^{(n)}) = \mathrm{dim }(\Ker B_n)$.
\item by Oseledets genericity for any $\epsilon$ there exists $D'_c (\epsilon) >0$ such that $\angle(\Gamma_c^{(n)}, \Gamma_u^{(n)} \oplus \Gamma_s^{(n)}) \geq D'_c(\epsilon) e^{-\epsilon n}$;
\end{enumerate}
one can deduce that there exists a constant $D_c(\epsilon) > 0$ such that $\angle(\Gamma_c, \Ker B_n) \leq D_c \angle(\Gamma_c, \Gamma_u^{(n)} \oplus \Gamma_s^{(n)})$.
\end{proof}
\noindent We are now ready to present the proof of Proposition \ref{central}.
\begin{proof}[Proof of Proposition \ref{central}]
Recall that $\mathcal{B}$ is a renormalization invariant (see Lemma~\ref{lemma:Bproperties}, property $(ii)$), therefore, by Remark~\ref{rk:Bexpression}, we can write
$$ \mathcal{B}(T^{(n)}) = B(\log D\varphi^n) + B(\vect{\omega}^u_n) +  B(\vect{\omega}^c_n) + B(\vect{\omega}^s_n) $$ where $\varphi^n \in \mathcal{P}$ is the profile of $T^{(n)}$, and $\vect{\omega}^u_n, \vect{\omega}^c_n$ and $\vect{\omega}^u_n$ are the projection of $\vect{\omega}_n$ on $\Gamma_u^{(n)}$,  $\Gamma_c^{(n)}$ and $\Gamma_s^{(n)}$ respectively. 
\noindent By the a priori bounds for the profile given by  Lemma~\ref{bound2}, there exists a constant $M = M(T)$ such that $|| \log D\varphi^n || \leq M$ for all $n \geq 0$, thus $B(\log D\varphi^n) $ is uniformly bounded by a constant $L = L(T) >0$. Also $B(\vect{\omega}_n^s) = 0$ by Lemma \ref{lemma5}. 
\vspace{2mm} \noindent Assume that $\vect{v} \neq 0$. We have in that case 
$$ || B(\vect{\omega}_n^c) || \leq || \mathcal{B}(T) || + L + || B(\vect{\omega}_n^u) ||.$$ By Proposition \ref{unstable}, $\vect{\omega}_n^u = Q(0,n) \vect{v} + o(||\vect{\omega}_n^u||^{\epsilon})$ for any $\epsilon > 0$. By invariance of $B$ we get $$ || B(\vect{\omega}_n^c) || \leq || \mathcal{B}(T) || + L + ||B(v)|| + o(||\vect{\omega}_n^u||^{\epsilon}).$$ Since $|| \vect{\omega}_n^c || \leq D_c(\epsilon) e^{\epsilon n} $  (by Lemma \ref{lemma7} we get that 
$$ ||  \vect{\omega}_n^c || =  o(||\vect{\omega}_n^u||^{\epsilon}) $$ for all $\epsilon$.
\vspace{3mm} \noindent Assume now that $\vect{v} = \epsilon$. The exact same reasoning as above gives 
$$|| B(\vect{\omega}_n^c) || \leq || \mathcal{B}(T) || + L + || B(\vect{\omega}_n^u) ||.$$ But by Proposition \ref{unstable} we have $\vect{\omega}_{n_{k_m}}^u \leq K(T)$ for all $m \geq 0$, and by \ref{conditionA} combined with Lemma \ref{lemma7} we obtain $ ||\vect{\omega}_{n_{k_m}}^c || \leq D_c \delta  || B(\vect{\omega}_n^c) || $ for $m \geq 0$. Putting everything together we get the existence of $K_c = K_c(T)$ such that for all $m \geq 0$
$$ ||\vect{\omega}_{n_{k_m}}^c || \leq C_c(T).$$
\end{proof}
}

\subsection{Control of the stable part}
We now \ref{shadowing} by establish bounds for the projection of $(\omega_n)_{n\in\mathbb{N}}$ on the stable part of the splitting. We have the following. 
\begin{proposition}
\label{stable}
For any $T$ satisfying the $(RDC)$ there exists $C_s(T) > 0$ such that we have 
$$ ||\Pro_s(\omega_{n_{k_m}})|| \leq C_s(T), \qquad \text{for\ all}\ m \in \mathbb{N}$$
\end{proposition}

\begin{proof}
The proof is a rather straightforward application of Lemma \ref{lemma3}.  Writing 
$$ \omega_{n_k} = \omega_{n_k} - \widetilde{Z}_{k-1}\omega_{n_{k-1}} + \widetilde{Z}_{k-1}(\omega_{n_{k-1}}-\widetilde{Z}_{k-2}\omega_{n_{k-2}}) + \cdots + \widetilde{Z}_{k-1} \cdots \widetilde{Z}_{1} (\omega_{n_1} - \widetilde{Z}_{0}\omega_{n_0}) + \widetilde{Z}_{k-1} \cdots \widetilde{Z}_{0} \omega_{n_0} $$ which can be rewritten as
$$ \omega_{n_k} = \sum_{j=1}^k{Q(j,k)(\omega_{n_{j}} - \widetilde{Z}_{j-1}\omega_{n_{j-1}} ) },$$
 we get by projecting on stable spaces and applying Lemma \ref{lemma3}
$$||\Pro _s(\omega_{n_k}) || \leq K_T \sum_{j=1}^k{|| \widetilde{Q}(k,j)_{|\widetilde{\Gamma}_s^j}|| \,||\Pro _s^{(j)}  || \,||\widetilde{Z}_{j-1}||}.$$ We then see that since $T$ satisfies the $(RDC)$, by \eqref{conditionB} in Definition~\ref{def:RDC} the right-hand side of the inequation is bounded by $K\cdots K_T$ at special times $n_{k_m}$ which concludes the proof of the Lemma.
\end{proof}

\subsection{Proof of Theorem \ref{shadowing}}
{
The proof of Theorem \ref{shadowing} ensues easily from Propositions \ref{unstable}, \ref{central} and \ref{stable}. Indeed assume that $\vect{v}$ of Proposition \ref{unstable} is non-zero. Condition $(S)$ and $(B)$ imply that $\Pro _s(\omega_n)$ grows subexponentially fast in which case $\Pro _s(\omega_n) = o(\Pro _u(\omega_n))$, and by Proposition \ref{central}  $|| \Pro_s(\omega_{n}) || =  o(||\Pro_u(\omega_n)||^{\epsilon}))$. We thus have $\omega_n \sim Q(n,0)\vect{v}$. Otherwise, $\vect{v} = 0$. In this case Propositions \ref{unstable} , \ref{central} and \ref{stable} imply that at times $n_{k_m}$ $\Pro _s(\omega_n)$, $\Pro _c(\omega_n)$ and $\Pro _u(\omega_n)$ are uniformly bounded by a constant $\kappa(T) >0$ which implies the theorem.
}

\section{Convergence of renormalization in the recurrent case}\label{sec:convergence}
{In this section we prove exponential convergence of renormalization for irrational GIETs in the recurrent case, i.e.~Case 1, of Theorem~\ref{shadowing}. The key steps of the proof follow conceptually the main steps in the work of Herman \cite{Herman} on circle diffeomorphisms, thus generalizing his results to GIETs. Although a lot of  the material contained in this section is similar to Herman's work or further extensions of his theory, we believe it does not appear under this form (to the best knowledge of the authors) anywhere in the literature about GIETs.   
Some steps in particular require a careful treatment to be generalized to $d>2$, see for example \S~\ref{sec:sizePcontrol} or Lemma~\ref{combinatoric}. 

\smallskip
We will show more precisely that, under suitable assumptions that we now comment upon, the orbit $\{\mathcal{R}^m(T), \ m\in\mathbb{N}\}$ of an acceleration of Rauzy-Veech induction (that, a posteriori, can be taken to be simply Zorich acceleration) converge at an exponential rate to the space of IETs. This will then allows to show in the next section~\ref{sec:conjugacy} that  the GIETs for which we have this form of exponential convergence of renormalization are $\mathcal{C}^1$-conjugated to a standard IET.  We will first show  that if $T$ is a $GIET$ satisfying the Regular Diophantine Condition and that Case 1 of Theorem~\ref{shadowing}) holds, the $\mathcal{C}^1$-distance between $\mathcal{Z}^{n}(T)$ and the subspace $\mathcal{M}_d$ of \emph{Moebius IETs} (see Definition~\ref{def:MIET} and \S~\ref{parameters}) decrease exponentially. 
To show convergence to the space of linear IETs (affine first and standard then), we {\color{black}exploit} the \emph{boundary operator} $\mathcal{B}(T) $ of a GIET (see Definition~\ref{GIET:boundary} in \S~\ref{sec:boundaryGIET}), which is a renormalization invariant based on the boundary operator defined by Marmi, Moussa and Yoccoz in \cite{MMY3}. The boundary gives an obvious obstruction to the existence of a $C^1$ conjugacy, so it is necessary to ask that $\mathcal{B}(T) = 0$ to prove convergence to the subspace $\mathcal{I}_d$ of standard IETs.  

\smallskip
The main result of this section  is thefore the following theorem.
\begin{thm}
[{\bf Exponential convergence of renormalization}]\label{convergencethm}
Assume that $T\in \mathcal{X}^3_d$  satisfy the $(RDC)$ Diophantine condition. There exists $K_1 =K_1(T) > 0$ and $\alpha = \alpha(T)< 1$ such that if $\mathcal{B}(T) = 0$, we have
$$ d_{\mathcal{C}^1}(\mathcal{Z}^{n}(T), \mathcal{I}_d) \leq K_1 \, \alpha^n .$$
\end{thm}
\noindent The $d_{\mathcal{C}^1}$ which appears in the statement of the theorem, defined in \S~\ref{sec:distances} below, is the $\mathcal{C}^1$-distance with respect to the  shape-profile parametrisation $\mathcal{X}^3_d = \mathcal{A}_d \times \mathcal{P}_d$ introduced \S~\ref{coordinates} (see Definition~\ref{sec:distances}). Although our goal is to get a control this $\mathcal{C}^1$-distance, we are going to work with the a distance $d_{\eta}$ defined using \emph{total non-linearity} on the profile coordinates, see \S~\ref{sec:distances}, since (as we hinted when describing the properties of total non-linearity, see the comments after Proposition~\ref{prop:Nproperties}) this quantity does not increase  under renormalization (and, as we will show, decreases strictly along the subsequence of good return times $(n_k)_{k\in\mathbb{N}}$ given by the $(RDC)$, see \S~\ref{sec:decrease}). We relate $d_{\mathcal{C}^1} $ and $ d_{\eta}$ in \S~\ref{sec:Ckbounded}.

{
\smallskip The boundary $\mathcal{B}(T) $ is a vector ${b}= (b_s)_{s} \in\mathbb{R}^\kappa$ (where we recall that $\kappa$ is the cardinality of singularities of any surface suspension of $T$, see \S~\ref{gietandfoliations}), which encodes information about geometric obstructions given by each singularity. It is philosophically important to distinguish on three cases:
\begin{enumerate}
\item The case ${b} = 0$, which contains all \textit{standard} IETs which we call the \textit{linear regime}.
\item  The case $\sum_{1\leq s\leq \kappa}{b_s} = 0$ (which contains all \textit{affine} IETs) which we call the \textit{affine regime}.
\item The case where ${b}$ is arbitrary, which is the \textit{non-linear} regime.
\end{enumerate}
Thus, asking that $\sum_{1\leq s\leq \kappa}{b_s} = 0$ is a necessary assumption to prove convergence to $\mathcal{A}_d$ and the request  that    $\mathcal{B}(T) ={b}=0$, 
 is a necessary assumption to prove convergence to $\mathcal{I}_d$. 
While article is concerned with establishing a rigidity theory for the linear regime, we stress that some of the results we prove in this section  apply to the other cases too (as explained already in the Outline in \S~\ref{sec:outline} below). The non-linear regime, in particular, is of independent interest and provides a natural higher genus framework which generalizes the much studied space of circle diffeomorphisms \emph{with break points}.}
  

\subsection{Outline of the proof.}\label{sec:outline}  Let us give an outline of the main steps of the proof of Theorem \ref{convergencethm} and describe the organization of the section.
\begin{enumerate}
\item 
\smallskip
\noindent {\it A priori bounds.} We first show (in \S~\ref{sec:apriori}) that the uniform bound on $(\omega_{n_{k_m}})_{m \in \mathbb{N}}$ implies that the iterates of accelerated renormaliation $\mathcal{R}^m(T):=\mathcal{Z}^{n_{k_m}}(T)$ (corresponding to the special sequence $(n_{k_m})_{m\in\mathbb{N}}$ given by the $(RDC)$), as well as their inverses $\mathcal{R}^m(T)^{-1}$, $m\in \mathbb{N}$,   remain in a bounded set for the $\mathcal{C}^1$-topology. This is what is often called an \textit{a priori bound}, and in our case replaces the Denjoy-Koksma inequality for circle diffeomorphisms. 

\smallskip
\item {\it Exponential decay of the partitions mesh.} 
In \S~\ref{sec:sizePcontrol} we then show that such a priori bounds, combined with the fact the times $(n_{k_m})_{m\in\mathbb{N}}$ are \textit{good return times} (see Definition \ref{def:goodreturns}) imply that the size of the dynamical partition associated with $T$ at step $n$ decreases at an exponential rate (with respect to $n$). { While in the study of  circle diffeomorphisms this is an easy step, as no particular arithmetic hypothesis is needed, it is an important step  to deal with in the treatment of GIETs, which requires new ideas.  
Indeed, when renormalization has more than two dynamical towers ($d$ in this case), these are not a priori related as in the case of circle diffeomorphisms. In order to compare different towers, we exploit a quite subtle geometric argument based on renormalization, which exploits good return times and a priori bounds to prove first balance of some \emph{relative} dynamical partitions and then infer, through distorsion bounds, the needed decay of the mesh size (see in particular \S~\ref{sec:keylemma} for details). }

\smallskip
\item {\it Convergence to Moebius \textrm{IETs}.} The exponential decay of the size of the dynamical partition is easily shown to imply convergence of $\mathcal{Z}^n(T)$ to the space of Moebius IETs with respect to the $\mathcal{C}^3$-norm, as shown in \S~\ref{sec:Moebius}. This part is completely standard, and is where is made use of the \textit{Schwarzian derivative}. This step is exactly the same as in the case of circle maps, and it is the reformulation of Herman's theory in renormalization terms due to Khanin and Sinai (see \cite{KS:Her, KT:Her}, which generalize \cite{VulKhanin, KhaninKhmelev}). 

The intuition behind  be true is the following: convergence to Moebius allows for a simplification of the discussion and at this point the total non-linearity will be a good enough measure of the complexity of the maps we are dealing with. As we have seen in \S~\ref{nonlinearity}, the total non-linearity is a decreasing function, which is the average of the absolute value of a \emph{mean-zero} function. Renormalization operates enough cancellation between positive and negative values of $\eta_T$ to get it to cancel altogether at the limit.

\smallskip
\item \noindent {\it Convergence to \textrm{AIET}s.} While the first steps do not require any particular hypothesis on the value of $\mathcal{B}(T)$,  under the additional hypothesis thast the \emph{sum of the components} of $\mathcal{B}(T) \in \mathbb{R}^d$ vanishes (namely, that we are in the \emph{affine regime} listed below) or, \emph{equivalently}, that $\int{\eta_T} = 0$), in \S~\ref{AIETs} we show that  $\mathcal{Z}^n(T)$  actually converges (exponentially fast) to the space of \textit{affine IETs}. This step makes use of the fact the times $(n_{k_m})_{m\in\mathbb{N}}$ are \textit{good return times}.

\smallskip
\item \noindent {\it Convergence to \textrm{IET}s.} 
Finally, under the extra hypothesis that $\mathcal{B}(T)$ vanishes altogether (which annihilates any potential contribution of the central part of $\omega_n$), \S~ref{sec:IETs} we show convergence at an exponential rate to the space of \textit{standard IETs}. A technical (but important) tool for this part of the proof is some partial differentiability properties of $\mathcal{Z}$ which are proved in Appendix~\ref{app:regularity}.
\end{enumerate}

\medskip
\noindent For the rest of the Section, we will assume that $T$ is a GIET such that:
\begin{enumerate}
\item the rotation number $\gamma(T)$ satisfies the Regular Diophantine Condition ($RDC$) in Definition \ref{def:RDC}; 
\item the sequences denoted by $(n_k)_{k\in \mathbb{N}}$ and $(k_m)_{m\in \mathbb{N}}$ are the sequences of good recurrence times given by Definition \ref{def:RDC};
\item the conclusion of Case 1 of Theorem \ref{shadowing} holds true, \textit{i.e.} the sequence $(||\omega_{n_{k_m}}||)_{m \in \mathbb{N}}$ is bounded.
\end{enumerate} 
\subsection{Preliminaries: distances and a priori bounds}\label{sec:distances_bounds}
In this first section we\ define the distances which we will use (see \S~\ref{sec:distances} and \S~\ref{sec:Ckbounded}) and the boundary of a GIET (in \S~\ref{sec:boundaryGIET} and \S~\ref{sec:Bstrata}) and then show 
show that being in Case $1$ of Theorem~\ref{shadowing} ensures a priori bounds, see \S~\ref{sec:apriori}.

{\subsubsection{Distances on parameter space}\label{sec:distances}
To define the distances $d_{\mathcal{C}^1}$ and $d_\eta$ on $\mathcal{X}_d^r$, for any $r\geq 2$, let us consider for each $\pi$ the shape-profile coordinates decomposition $\mathcal{X}^r_\pi= \mathcal{A}_{\pi}\times \mathcal{P}^r $ where $\mathcal{P}^r = \mathrm{Diff}^r([0,1])^d$ with $r\geq 2$:
\begin{itemize}
\item[-] since $\mathcal{A}_{\pi}$ identifies with a subset of $\mathbb{R}^{d} \times \mathbb{R}^{d-2}$ an is endowed with a distance $d_{\mathcal{A}}$ induced by the Euclidean distance of  $\mathbb{R}^{d} \times \mathbb{R}^{d-2}$;
\item[-] on $\mathcal{P}^r = \mathrm{Diff}^r([0,1])^d$, for $r\geq 2$, we can endow each of the coordinates of $\mathcal{P}^r$ with either the distance $d_{\mathcal{C}^1}$ or 
or the distance  $d_{\eta}$ on $ \mathrm{Diff}^2([0,1])\supset \mathrm{Diff}^r([0,1])$, namely
\begin{align*}
d_{\mathcal{C}^1}(\varphi_1,\varphi_2)&:=||\varphi_1-\varphi_2 ||_{\infty}+||(\varphi_1-\varphi_2)' ||_{\infty}\\ &\phantom{:} = \sup_{0\leq x\leq 1} |(\varphi_1-\varphi_2)'(x)| +\sup_{0\leq x\leq 1} |\varphi_1(x)-\varphi_2(x)|,\\ {\color{black} d_{\eta}(\varphi_1,\varphi_2)}&{\color{black}:
=\int_{0}^1 |\eta_{\varphi_1}-\eta_{\varphi_2}|\mathrm{d}x },
\end{align*}
where $\eta_\varphi$ denotes the non-linearity (see \S~\ref{sec:nl}).
\end{itemize}
\noindent It is well know that $d_{\mathcal{C}^1}$ is a distance and one can show that $d_\eta$ is also a distance (see Appendix~\ref{distance}); it is the distance induced by the $L^1$-norm on $ \mathcal{C}^0([0,1], \mathbb{R})$ via the homeomorphism between $\mathrm{Diff}^2([0,1])$ and $ \mathcal{C}^0([0,1], \mathbb{R})$ given by $f  \longmapsto  \eta_f$. 

\smallskip
\noindent We can then endow $\mathcal{P}$ with the distances $d^{\mathcal{P}}_{\mathcal{C}^1}$ and 
$d^{\mathcal{P}}_{\eta}$ defined taking the sums of the corresponding distance on each coordinates, namely, if $\varphi_{T_i}\in \mathcal{P}$ have coordinates $(\varphi^1_{T_i},\dots,\varphi^d_{T_i})$ for $i=1,2$, setting
\begin{equation*}
d^{\mathcal{P}}_{\mathcal{C}^1}(\varphi_{T_1},\varphi_{T_2}) :=\sum_{j=1}^d  d_{\mathcal{C}^1}(\varphi^j_{T_1},\varphi^j_{T_2}),\qquad 
d^{\mathcal{P}}_{\eta}(\varphi_{T_1},\varphi_{T_2}):=\sum_{j=1}^d  d_{\eta}(\varphi^j_{T_1},\varphi^j_{T_2}).
\end{equation*}
\noindent Through the shape-profile coordinates  identification 
$\mathcal{X}_\pi=\mathcal{A}_{\pi}\times \mathcal{P}_d$ (introduced in \S~\ref{coordinates}) 
 this hence defines also two product distances on $\mathcal{X}^3_d$, namely $d_{\mathcal{A}}\times d^{\mathcal{P}}_{\mathcal{C}^1}$ and  $d_{\mathcal{A}}\times d^{\mathcal{P}}_{\eta}$, 
which can then be extended to $\mathcal{X}_d=\sqcup_{\pi\in\mathfrak{S}^0}\mathcal{X}_\pi$ using the discrete distance\footnote{{\color{black}The \emph{discrete} distance $d_0$ on $\mathfrak{S}^0$ is simply the distance given by $d( \pi_1, \pi_2)=1$ for any $\pi_1,\pi_2\in\mathfrak{S}^0$ unless $\pi_1=\pi_0$ (in which case  $d( \pi_1, \pi_2)=1$).}} on the combinatorial data $\mathfrak{S}^0$.  Abusing the notation, we will still denote by $d_{\mathcal{C}^1}$ and $d_{\eta}$ the distances on $\mathcal{X}_d$ obtained in this way. 

\smallskip
\noindent {In analogous way to  $d_{\mathcal{C}^1}$, one can also define for any $k\in \mathbb{N}$  distances $d_{\mathcal{C}^k}$ on $\mathcal{X}^r$ for any $r\geq k$.} 

\smallskip
{\color{black}
We conclude the subsection with a useful intepretation of the total non-linearity as distance from the (sub)space of AIETs: 
\begin{remark}[interpretation of total non-linearity as distance]\label{rk:nonlinearityasd}
Notice that, for every $f\in \mathrm{Diff}^2([0,1])$, if we denote by $I$ the identity map  $I(x)=x$, since $\eta_I\equiv 0$, we can write $|N|(f) = d_{\eta}(f, I)$ where $|N|(f)$ denotes the total non-linearity (see Definition~\ref{def:Ns}). Thus,  if $A_T$ is the shape of $T$, which has shape-profile decomposition $A_T=(A_T, (I,\dots, I))$, since $d_\eta$ on $\mathcal{X}^r_d$ for $r\geq 2$ is defined as a product distance\footnote{{\color{black}Indeed, \label{minimal} given any $A\in \mathcal{A}_d$, $d_\eta (T, A)$ depends on $d_\mathcal{A} (A_T, A)$ and  $d^\mathcal{P}_\eta (\varphi_T, P_\mathcal{P}(A))$, where  $\varphi_T= P_\mathcal{P}(T)$ is the profile of $T$, but since $P_\mathcal{P}(A)=(I,\dots, I)$ where $I(x)=x$ is the identity in Diff$^r([0,1]$ for any $A\in\mathcal{A}_d$, the profile component is independent on $A\in\mathcal{A}_d$, while the first component, namely $d_\mathcal{A} (A_T, A)$, is clearly minimized by 
the shape $A=A_T$ of $T$, for which  it is zero.}},
$$
d_\eta(T, \mathcal{A}_d)= d_{\eta}(T, A_T) = |N| (T), \qquad \text{for\ all}\ T\in \mathcal{X}^r_d.
$$
\end{remark}}

\subsubsection{$\mathcal{C}^k$-bounded sets}\label{sec:Ckbounded}
Since we are working with invertible maps which are piecewise diffeomorphisms, when we describe a \emph{bounded} set we also need to have \emph{lower} bounds on derivatives, or, equivalently, upper bounds on the derivative of the \emph{inverse}. 
For fixed  $k$, on $\mathrm{Diff}^2([0,1]) $, with $r\geq k$ it is customary to introduce the distance 
$d^{\pm}_{\mathcal{C}^k}(f,g)= d_{\mathcal{C}^k}(f,g)+d_{\mathcal{C}^k}(f^{-1},g^{-1})$.
We then say that a set $\mathcal{K}\subset \mathrm{Diff}^r([0,1]) $ is $\mathcal{C}^k$-\emph{bounded} if it has bounded diameter with respect to the distance $d^{\pm}_{\mathcal{C}^k}$. 

\smallskip
Similarly, we therefore define, at the level of GIETs, 
$$
d_{\mathcal{C}^k}^\pm (T_1, T_2) := d_{\mathcal{C}^k} (T_1, T_2)+ d_{\mathcal{C}^k} (T^{-1}_1, T^{-1}_2), \qquad \text{for\ all}\ T_1,T_2\in \mathcal{X}^k.
$$
\begin{definition}[$\mathcal{C}^k$-bounded sets]\label{def:Ckbounded}
 We will say that a set $\mathcal{K}\subset \mathcal{X}^k$ is $\mathcal{C}^k$-\emph{bounded} iff it is bounded with respect to $d_{\mathcal{C}^k}^\pm $,  i.e.~contained in a ball with respect to $d_{\mathcal{C}^k}^\pm $.
\end{definition}
\begin{lemma}[Equivalent characterizations of $\mathcal{C}^1$-bounded sets]\label{lemma:equivbounded}
For $k=1$,  $\mathcal{K}\subset \mathcal{X}^1$ is $\mathcal{C}^1$-bounded in the sense of Definition~\ref{def:Ckbounded} iff, equivalently, one of the following conditions hold:
\begin{enumerate}
\item  there exists a constant $\nu_\mathcal{K}>1$ such that $(\nu_{\mathcal{K}})^{-1}<|| DT||_{\infty}< \nu_\mathcal{K}$ for every $T\in \mathcal{K}$.
\item there exists a constant $C_\mathcal{K}>0$ such that $||\log DT||_{\infty}< C_\mathcal{K}$ for every $T\in \mathcal{K}$;
\end{enumerate}
\end{lemma}
\begin{proof} 
The lemma follows from the explicit expression for $DT$ in shape-profile coordinates, given by \eqref{Dincoordinates}, which shows that
if $T = (A_T ,\varphi_T)$ and $\rho=\rho(T)$ is the average slope vector of $T$ (see Definition~\ref{def:rhoomega}), 
\begin{equation}\label{Dbound}
||\mathrm{D}T||_{\infty}  \leq ||{\vect{\rho}}|| \max_{1\leq i\leq d} || D {\varphi}^i_{T} ||_\infty .
\end{equation}
\noindent Asking that $\mathcal{K}$ is $\mathcal{C}^1$-bounded (i.e.~that $\mathcal{K}$ has bounded diameter with respect to $d^\pm_{\mathcal{C}^1}$, see Definition~\ref{def:Ckbounded}), in view of the definition of $d_{\mathcal{C}^1}$ in shape-profile coordinates, is equivalent to asking that $||\rho(T)||$ and 
$||\rho(T^{-1})||$, as well as $||D\varphi_{T}^i||_{\infty}$ and $||D\varphi_{T^{-1}}^i||_{\infty}$ for $1\leq i\leq d$, are bounded above by a constant depending on $\mathcal{K}$ only (notice that the other parameters describing $\mathcal{A}_d$, as well as the sup norm of the profile coordinates, are always bounded). In view of \eqref{Dbound} (applied to $T$ and its inverse), this shows that there exists a constant $\nu_{\mathcal{K}}>0$ such that $||\mathrm{D}T||_{\infty} , ||\mathrm{D}(T^{-1})||_{\infty} \leq \nu_{\mathcal{K}}$.

\smallskip
The equivalence with $(1)$ now follows simply by the formula for the derivative of the inverse, which shows that a \emph{lower} bound on $|DT(x)|$ for all $x\in I$ is equivalent to an \emph{upper} bound for   $||DT^{-1}||_{\infty}$. The equivalence between $(1)$ and $(2)$ is clear.
\end{proof}
{\color{black}
When studying convergence of renormalization, using $d_{\mathcal{C}^k}^\pm $ or $d_{\mathcal{C}^k}$ is equivalent, as shown by the following remark. The use of $d_{\mathcal{C}^k}^\pm $ on the other hand is important for us since we study the \emph{global} dynamics or renormalization and recurrence to  \emph{bounded} (but not shrinking) sets. 
\begin{remark}\label{comparisontozero}
On each ${\mathcal{C}^k}$-bounded set,  $d_{\mathcal{C}^k}^\pm $ and $d_{\mathcal{C}^k}$ are \emph{comparable}: in particular, for any subset $\mathcal{Y}\subset \mathcal{X}^k_d$ and any infinitely renormalizable $T\in \mathcal{X}^k_d$, $d_{\mathcal{C}^k} (\mathcal{R}^n(T),\mathcal{Y} )$ converges to zero (exponentially) if and only if $d_{\mathcal{C}^k}^\pm (\mathcal{R}^n(f),\mathcal{Y} )$ 
converges to zero (exponentially).
\end{remark}

\subsubsection{Distances comparision}\label{sec:distancesineq}
Let us consider and compare the two distances $d_\eta$ and $d_{\mathcal{C}^1}$ on each \emph{profile} coordinate, namely on $\mathrm{Diff}^r([0,1])$, where $r$ is an integer $r\geq 2$. Recall that the definition of $\mathcal{K}\subset \mathrm{Diff}^r([0,1]) $ is $\mathcal{C}^1$-\emph{bounded} was given in the previous \S~\ref{sec:Ckbounded}.

\begin{lemma}[$d_{\mathcal{C}^1}$ and $d_{\eta}$ comparision]\label{lemma:distancesrelDiff}
For any  $\mathcal{K}\subset \mathrm{Diff}^r([0,1]) $ which is $\mathcal{C}^1$-\emph{bounded}, there exists a constant $L = L(\mathcal{K})>0$ such that for $f_1, f_2 \in \mathrm{Diff}^2([0,1]) $
$$ d_{\mathcal{C}^1}(f_1,f_2) \leq d^\pm_{\mathcal{C}^1}(f_1,f_2)  \leq L \,d_{\eta}(f_1,f_2).$$
\end{lemma}
\noindent The proof of this lemma is included for completelenss in Appendix~\ref{sec:distancesAppendix}, together with the proof of the next lemma (consequence of the definition of distances on GIETs and classical distorsion bounds), that provides a comparison of distances from AIETs, which will be useful later:
\begin{cor}[$d_{\mathcal{C}^1}$ and $d_{\eta}$ distance from AIETs]\label{lemma:distancesrel}
For any  $d\geq 2$ and any $T \in \mathcal{X}^r_d$ with $r\geq 2$, there exists $L=L(T)$ such that  
$$ d_{\mathcal{C}^1}(\mathcal{V}^n(T),\mathcal{A}_d) \leq  d^\pm_{\mathcal{C}^1}(\mathcal{V}^n(T),\mathcal{A}_d) \leq L \,d_{\eta}(\mathcal{V}^n(T),\mathcal{A}_d), \qquad \textrm{for\ all}\ n\in\mathbb{N}.$$ 
\end{cor}}

{\color{black}
\subsubsection{Boundary stratification and regimes}\label{sec:Bstrata}
 Consider the \emph{boundary} $\mathcal{B}(T)$ of a GIET $T$ (see Definition~\ref{boundary} of a GIET} in \S~\ref{sec:boundaryGIET}),
which is given by $\mathcal{B}(T) := B(\log \D T) $ (where $B$ is the Marmi-Moussa-Yoccoz boundary operator for observable $\mathcal{C}_0\big(\sqcup_i{I_i^t(T)} \big)$, see \S~\ref{sec:obs_boundary}).} 
For each value  ${b} \in \R^\kappa$, we can define the following subspaces 
$$ \mathcal{X}({b}) := \big\{  T \in \mathcal{X} \ | \ \mathcal{B}(T) = {b} \big\}.$$ 
\noindent In view of property $(ii)$ in Lemma~\ref{prop2}, since $ \mathcal{X}({b})$ is invariant under the action of $\mathcal{V}$ (and consequently under that of $\mathcal{Z}$), these subspaces are invariant under renormalization. Moreover, $\mathcal{I}_d$ (resp.~$\mathcal{A}_d$) is a subspace of $\mathcal{X}({b})$ for ${b}$ in the linear-regime  ${b}=0$ (resp.~in the affine regime $\sum_{1\leq s\leq \kappa} b_s=0$). 
 As remarked in the introduction of this section \S~\ref{sec:convergence}, in order for a GIET $T$ to converge  to IETs (resp.~AIETs) under renormalization, $T$ needs therefore to already belong to $\mathcal{X}({0})$ (resp.~$\mathcal{X}({b})$ with ${b}$ in the affine regime).}

\begin{lemma}[Affine regime and vanishing of non-linearity]\label{boundaryaffinereduction}
The \emph{affine} regime  corresponds to the assumption that the \emph{mean non-linearity} vanishes, i.e.~
$$\sum_{1\leq s\leq \kappa}b_s=0, \quad \text{where}\ (b_s)_{s=1}^\kappa=\mathcal{B}(T)\qquad \Leftrightarrow \qquad \overline{N}(T)=\int_0^1 \eta_T(x)\mathrm{d}x=0.$$
\end{lemma}
\begin{proof} On one hand, by definition of non-linearity (see \S~\ref{sec:nl} and in particular Definition~\ref{def:Ns}), on  each continuity interval $I^t_j=(u^t_j, u^t_{j+1})$ for $1\leq j\leq d$,  we have that $\eta_T(x)=D \log D T_j(x)$ so 
$$\overline{N}(T)=\sum_{j=1}^d\int_{I^t_j}\eta_T(x)\ \mathrm{d}x= \sum_{j=0}^{d-1} \left( \lim_{x\to (u^t_{j+1})^-} D T_j(x) -\lim_{x\to (u^t_{j})^+} D T_j(x)\right).
$$
One can then check that this is the same than $\sum_{s=1}^\kappa b_s$ simply by recalling the definition of $\mathcal{B}(T) =B(\log DT)$ and boundary of an observable (see \S~\ref{sec:obs_boundary}) and remarking that summing over all possible values of $s(u_i)\in \{1,\dots, \kappa\}$ gives a rearrangement of the above sum over singularities $u^t_j$. \end{proof}

\begin{remark}\label{rk:Bbreaks}If $T$ is a circle diffeomorphisms with \emph{breaks} (i.e.~a piecewise differentiable homemorphism, with $d-1$ \emph{breaks} (i.~e.~$d\!-\! 1$ points of discontinuity of the derivative), $T$ can be seen as a $d$-GIET in $\mathcal{X}_d^1$ (with a \emph{rotational} combinatorics). In this case   
 $\kappa=d\!-\! $ (since $g=1$ and $d=2g+\kappa-1$) and the values $e^{b_s}$, where $b_s$ are  the entries of $\mathcal{B}(T)$ for $1\leq s\leq d-1$, encode the \emph{breaks}, 
  which are well-known $\mathcal{C}^1$-invariants in the theory of circle diffeos with break points. In this case, the assumption that $\mathcal{B}(T)$ is zero, i.e.~that each entry $b_s$ is zero, is equivalent to asking that $T$ is indeed induced from a circle diffeomorphism (i.e.~there are no breaks).  
\end{remark}
\smallskip
\subsubsection{A priori bounds.}\label{sec:apriori}
Let us first show that in the recurrent case (Case $2$ of Theorem~\ref{shadowing}), we have \textit{a priori bounds} which  hold along the orbit $\{\mathcal{R}^m(T)\}_{m\in\mathbb{M}}$ where $\mathcal{R}$ is the acceleration of $\mathcal{V}$ along the subsequence $\{n_{k_m}\}_{m\in\mathbb{N}}$ given by the $(RDC)$. 

\smallskip
\noindent {\it Notation:} To lighten the notation, we denote by $||f ||_{\infty}$ the sup-norm on the domain where $f$ is defined, so if $f:I\to \mathbb{R}$, $||f ||_{\infty}:=||f||_{L^\infty(I)}=\sup_{x\in I}|f(x)|$.
\begin{proposition}[{a priori bounds}]\label{apriori}
The iterates $\{ \D\mathcal{R}^m(T), \ m\in \mathbb{N} \}$ belong to a $\mathcal{C}^1$-bounded set (in the sense of Definition~\ref{def:Ckbounded}), i.e.~there exists a constant $K_2 = K_2(T) > 0$ such that
$$  K_2(T)^{-1} \leq ||\D\mathcal{R}^m(T)||_\infty := ||\mathrm{D}T^{(n_{k_m})}||_\infty = || \mathrm{D} T_{n_{k_m}} ||\leq K_2(T) , \qquad \text{for\ all}\ m \in \mathbb{N} .$$
\end{proposition}
\noindent The Proposition can be easily proved using the shape-profile decomposion $\mathcal{X}^2_d= \mathcal{A}_d\times \mathcal{P}^2$ (see \S~\ref{coordinates}). We will first show (in Lemma~\ref{bound2} here below) that the \emph{profile} coordinates always satisfy a priori bounds along the orbit of renormalization, simply as a consequence of the classical distorsion bounds (given by Lemma~\ref{bound1}). The assumption of being in the recurrent case $2$ of Theorem~\ref{shadowing}) provides the required bounds for  the \textit{shape} coordinates.

\smallskip

\noindent Let $\Pro _{\mathcal{P}}: \mathcal{X}^2_d\to \mathcal{P}_d^2=\big( \mathrm{Diff}^2([0,1]) \big)^d$ be the projection on the \emph{profile} coordinates $\mathcal{P}^2_d$ (see \S~\ref{coordinates}).
\begin{lemma}[bounded distorsion for the profile]
\label{bound2}
For any GIET $T\in \mathcal{X}^2_{\pi} $ there exists a constant  $M=M(T)$ (which depends only on the $\mathcal{C}^2$-norm of $T$ and $T^{-1}$ and hence is uniform on $\mathcal{C}^2$-bounded sets) 
 such that for any $T$ GIET renormalizable under Rauzy-Veech induction $n$ times
 we have
$$ || \Pro_{\mathcal{P}}\big( \mathcal{V}^{n}(T) \big ) - (\mathrm{Id})^d ||_{\mathcal{C}^1} \leq  M, \qquad || \Pro_{\mathcal{P}}\big( \mathcal{V}^{n}(T)^{-1} \big ) - (\mathrm{Id})^d ||_{\mathcal{C}^1} \leq  M , 
 $$  where $(\mathrm{Id})^d  = (\mathrm{Id}, \cdots, \mathrm{Id}) \in \mathcal{P} = \big( \mathrm{Diff}^2([0,1]) \big)^d.$
\end{lemma}
\begin{proof}
Let $\varphi^j_n$ be a coordinate of $\Pro _{\mathcal{P}}\big( \mathcal{V}^{n}(T) \big )$. By definition of profile, $\varphi^j_n$ is obtained by composing the restrictions of $T$ to pairwise disjoint intervals and then rescaling. More precisely,  if we denote  by  $q_j:= q^{(n)}_j$ the height of the Rohlin tower $\mathcal{P}^j_n$ and  by $f_j^k$, for  $0\leq k< q_j$,   the restriction of $T$ to the floor $T^k(I_j^{(n)})$ of $\mathcal{P}^j_n$, then $\varphi^j_n=\mathcal{N}(T^{(n)}_j)$ where 
 $T^{(n)}_j = f_j^{q_j-1} \circ f_j^{q_j-2} \circ \cdots \circ f_2 \circ f_1$ and $\mathcal{N}(\cdot )$ is the normalisation operator which produces a diffeo of $[0,1]$. 

Since $\varphi^j_n$ is a diffeomorphism of $[0,1]$,  by chain rule and mean value, choosing $y\in [0,1]$ such that $D\varphi^j_n(y)=1$,
$$\sup_{x\in [0,1]}  \mathrm{D} \varphi^j_n(x) = \sup_{x\in [0,1]} \frac{\mathrm{D} \varphi^j_n(x)}{ \mathrm{D} \varphi^j_n(y)} = \sup_{x,y\in[0,1]} \frac{\mathrm{D}  \mathcal{N}(f^{q_j-1} \circ \cdot \circ f^0_k))(x)}{\mathrm{D}  \mathcal{N}(f^{q_j-1} \circ \cdot \circ f^0_k)(y)}
= \sup_{x,y\in I_j^{(n)}} \frac{\mathrm{D}  (f^{q_j-1} \circ \cdot \circ f^0_k)(x)}{\mathrm{D}  (f^{q_j-1} \circ \cdot \circ f^0_k)(y)},$$
so that we can now apply the distorsion bound given by Lemma \ref{bound1}. Since we can reverse the role of $x$ and $y$, recalling the formula for the derivative of the inverse function, we can then  deduce from Lemma \ref{bound1} that, for all $x \in [0,1]$, 
$$\max\left\{  \mathrm{D} \varphi^j_n(x), \left(\mathrm{D} \varphi^j_n(x)\right)^{-1}\right\} \leq \exp\left(\int_0^1{|\eta_T(x)|\mathrm{d}x}\right)= \exp\left(\int_0^1{\left|{D^2T(x)}/{DT(x)}\right|\mathrm{d}x}\right),$$
where the last equality uses simply the definition $\eta_T= \frac{D^2T}{DT}$ and shows that the RHS depends on the $\mathcal{C}^2$ norm of $T$ and $T^{-1}$ only. 
This, recalling the definition of total non-linearity (see Definition~\ref{nonlinearity}), shows that 
$$\max_{1\leq j\leq 1} \sup_{x\in [0,1]}|\log  D\varphi^j_n(x) |\leq |N|(T).$$
Since the exponential is Lipschitz on bounded sets of $\mathbb{R}$, there exists a constant $L>0$ (which depends on the $\mathcal{C}^2$-norm of $T$ and $T^{-1}$ only) such that
$$ 
\sup_{x\in [0,1]}|(D\varphi^j_n(x))^{\pm 1}-1| \leq L \sup_{0<x<1} |\log {\D \varphi^j_n(x)}| \leq L |N|(T), \qquad \text{for\ all}\ 1\leq j\leq d.$$ 
Since these  inequality holds for all the components of the profile, this proves the lemma. 
\end{proof}

\begin{proof}[Proof of Proposition~\ref{apriori}]
Let us consider $\mathcal{R}^m(T):=\widetilde{\mathcal{Z}}^{{k_m}}( T)=T^{(n_{k_m})} $. 
 Denote by  $\vect{{\underline{{\omega}}}}^{m}:=\vect{\omega}_{n_{k_m}}$  the shape log-slope vector of $\mathcal{R}^m(T)$ and by $\vect{\underline{\rho}}^{m}:= \vect{\rho}_{n_{k_m}}$ be the slope vector, so  $\vect{\underline{\rho}}^{m} = \exp(\vect{\underline{\omega}}^{m})$. 

By the chain rule, since the induced map $T^{(n_{k_m})}$ is related to $\mathcal{R}^{m}T$ through conjugation by a linear map, see \eqref{renormalizationeq}, and by the explict expression for $D\mathcal{R}^{m}T$ in shape-profile coordinates  (see in particular \eqref{Dincoordinates}), we have that, denoting by $(\underline{\varphi}^1_{m}, \dots, \underline{\varphi}^d_{m})$ the profile coordinate of $\mathcal{R}^m(T)$,
\begin{equation}\label{Dbound2}
||\mathrm{D}T^{(n_{k_m})}||_{L^\infty \left(I^{(n_{k_m})}\right)}  =
||\mathrm{D}\mathcal{R}^{m}T||_{L^\infty(0,1)} \leq \max_{1\leq i\leq d} |\underline{\vect{\rho}}^m_{i}| ||\underline{\varphi}_i^{m} ||_\infty .
\end{equation}
\noindent We remark now that: 
\begin{enumerate}
\item by Lemma~\ref{bound2}, all coordinates of ${\mathcal{P}}\left(\mathcal{R}^m(T)\right)$, and therefore $\max_{1\leq i\leq d} || \underline{\varphi}_m^i||_{\infty}$, are uniformly bounded;
\item  by the assumption that we have made on $T$, the sequence 
$||\vect{\omega}_m||=|| \omega_{n_{k_m}}||$ is bounded, 
so also $\max_{1\leq i\leq d} |\underline{\vect{\rho}}^m_i|=\max_{1\leq i\leq d} |e^{\underline{\vect{\omega}}^m_i}|$ is uniformly bounded. 
\end{enumerate}
Using these two facts to estimate \eqref{Dbound2} we get the desired upper bounds. For the lower bounds, it suffices to estimate similarly the inverse $(\mathrm{D}\mathcal{R}^m T)^{-1}$ 
(which has slope vector $e^{-\underline{\vect{\omega}}^m}$, which is also bounded) and recall the formula for the inverse function (see also the proof of Lemma~\ref{lemma:equivbounded}). 
\end{proof}

{
\subsection{Exponential decay of the  dynamical partitions mesh.}\label{sec:sizePcontrol}
We will now prove exponential estimates on the decay of the size of the sequence of dynamical partitions $\{ \mathcal{P}_n, n\in \mathcal{N}\}$ along the sequence $(k_m)_m$ given by the $(RDC)$. Since the sequence $(k_m)_m$ grows linearly, we can then deduce a posteriori that the mesh decay exponentially (see Corollary~\ref{expmeshdecay}). 

\smallskip
Throughout this section, $\{ \mathcal{P}_n, n\in \mathcal{N}\}$ denotes the sequence of dynamical partitions (as defined in \S~\ref{dynamicalpartitions}) associated to the orbit $T^{(n)}:=\mathcal{Z}^nT$, $n\in\mathbb{N}$ of $T$ under the Zorich acceleration $\mathcal{Z}$. Let us measure their \emph{size} by  $\mesh (\mathcal{P}_n)$, given  by definition $\mesh( \mathcal{P}_n) := \sup_{I \in \mathcal{P}_n}{| I |}$. Then 
\begin{proposition}[Partition  mesh decay]\label{partition}
There exists $0 < \alpha_1(T) = \alpha_1 <1$ such that for all $m \in \mathbb{N}$
$$\mesh(\mathcal{P}_{n_{k_m}}) \leq  \alpha_1^m .$$
\end{proposition}
\noindent To prove this Proposition we will crucially exploit \emph{both} that $(n_{k_m})_{m\in\mathbb{N}}$ are \emph{good return times}, \emph{and} that, at the same time, there  are \emph{a priori bounds}. More precisely, we will show that if the double occurrence of a positive matrix occurr at a time where also the shape log-slope vector is \emph{bounded}, then this produces enough geometric control on ratios of floors in the dynamical partitions of the generalized IET to in particular produce a controllable decay of the mesh (see Lemma~\ref{lemma:key}). We encode here in the definition of \emph{good bounded distorsion sequence}
 the simultaneous presence of good return time with a priori bounds.

\begin{definition}[good $\mathcal{C}^1$-\emph{recurrence sequence}]\label{def:C1recurrence}
Let us say that the sequence $(n_m)_{m\in\mathbb{N}}$ is a \emph{good} $\mathcal{C}^1$-\emph{recurrence sequence} if $(n_m)_{i\in\mathbb{N}}$ is a sequence of $p$-good returns (in the sense of Definition~\ref{def:goodreturns}) for some $p>0$ and  $K>0$ such that
$$
\frac{1}{K}\leq \Vert D T_{n_m} \Vert \leq K, \qquad \forall \ m\in \mathbb{N}.
$$
We call each time $n_m$ in a  good $\mathcal{C}^1$-recurrence sequence $(n_m)_{m\in \mathbb{N}}$ a \emph{good} $\mathcal{C}^1$-\emph{recurrence time}.
\end{definition}
\noindent 
} A  good $\mathcal{C}^1$-recurrence describes iterates of renormalization that are are \emph{recurrent} to certain ${\mathcal{C}^1}$-bounded sets in the space $\mathcal{X}_d$ (in the sense of Definition~\ref{def:Ckbounded}, see Lemma~\ref{lemma:equivbounded}) and at the same time are good returns,\footnote{We remark that recurring to a ${\mathcal{C}^1}$-bounded sets (in the sense of Definition~\ref{def:Ckbounded}  only controls the profile coordinates as well as the log-slope vector of the shape, but does not control the lengths coordinates, which are always bounded. The request that the $\mathcal{C}^1$-recurrence is also a sequence of good returns plays an essential role in controlling ratios of dynamical partition elements, see Lemma~\ref{lemma:key}.} from which the choice of the name.
\smallskip

\noindent The crucial step in proving exponential decay of the size of dynamical partitions is the following.
\begin{lemma}[Key lemma for mesh decay]\label{lemma:key}
For every  good $\mathcal{C}^1$-recurrence sequence $\{n_m\}_{m\in\mathbb{N}}$ (which is in particular a $p$-good recurrence sequence for some $p>0$) there exists a constant $0<{\alpha}_1<1$ such that
$$
\mesh (\mathcal{P}_{n_m+p})\leq {\alpha}_1 \, \mesh (\mathcal{P}_{n_m}).
$$
\end{lemma}
\noindent The proof of this key Lemma will take all of \S~\ref{sec:keylemma}. 
Let us first show that the Lemma allows to finish by induction the proof of exponential decay of the partitions mesh.

\begin{proof}[Proof of Proposition~\ref{partition} (from Lemma~\ref{lemma:key})]
Notice that, if we are in Case $1$ of the conclusion of Theorem~\ref{shadowing}, the sequence $(n_{k_m})_{m\in\mathbb{N}}$ is a  good $\mathcal{C}^1$-recurrence sequence, since the sequence $(n_k)_{k\in\mathbb{N}}$ (and therefore any of its subsequences) is, by the $(RDC)$ (recall Definition~\ref{def:RDC}) a sequence of good-returns and  $\{ \Vert DT_{n_{k_m}} \Vert , m\in\mathbb{N}\}$ are controlled above and below by the a-priori bounds in Proposition~\ref{apriori}. Without loss of generality, we can also assume (disregarding some good times if needed) that $n_{k_{m+1}}- n_{k_m}\geq p$ (notice that this new subsequence still grows linearly). Thus, 
iterating the key Lemma~\ref{lemma:key}, we get that, for any $m\geq 1$,
$$
\mesh (\mathcal{P}^{n_{k_m}})\leq  \mesh (\mathcal{P}^{n_{k_{m-1}+p}})\leq  {\alpha}_1 \mesh (\mathcal{P}^{n_{k_{m-1}}})\leq\cdots  \leq ({\alpha}_1)^{m} \mesh (\mathcal{P}^{n_{k_0}})\leq  ({\alpha}_1)^{m},
$$
where the last inequality holds trivially since $\mesh({\mathcal{P}})\leq 1$ for any partition of $[0,1]$.
\end{proof}

\noindent Let us also deduce that the \emph{whole} sequence $\left(\mesh(\mathcal{P}^n)\right)_{n\in\mathbb{N}}$ decay exponentially. We record separately the following elementary servation since it will be used again in this section.
\begin{remark}\label{rk:linearexp} If a decreasing (i.e.~non increasing) sequence $(a_k)_{k\in\mathbb{N}}$ \emph{decays exponentially along a subsequence  with linear growth}, then the \emph{whole sequence} $(a_k)_{k\in\mathbb{N}}$ \emph{decays exponentially}. To see this,  assume that there exist a subsequence $(k_m)_{m\in\mathbb{N}}$ such that $k_m/m$ has a finite limit and $0<\theta_0<1$ and $K>0$ such that $a_{k_m} \leq K (\theta_0)^m$ for every $m\in\mathbb{N}$. Then, if $\ell>0$ is such that $m\geq k_{m+1}/\ell$ for all $m\in\mathbb{N}$, setting $\theta_1:=({\theta}_0)^{1/\ell}$, we still have $0<\theta_1<1$ and, for each $k\in\mathbb{N}$, choosing $m$ such that  $k_m\leq k< k_{m+1}$ and using that $(a_k)_{k\in\mathbb{N}}$ is not increasing, we see that 
 $a_k \leq a_{k_m}\leq  K ({\theta}_0)^{m}\leq  K ({\theta}_0)^{k_{m+1}/\ell} = K \theta_1^{k_{m+1}}\leq K \theta_1^k$. 
\end{remark}
\noindent Since $\mesh(\mathcal{P}^n)$ is descreasing in $n$ and both $(k_m)_{m\in\mathbb{N}}$ and $(n_k)_{k\in\mathbb{N}}$ grow  linearly (see Property $(ii)$ and $(iii)$ in Definition~\ref{def:RDC}), Proposition~\ref{partition}  combined with the above Remark gives the following stronger conclusion.

\begin{cor}[exponential decay of the mesh]\label{expmeshdecay}
There exists $0 < \alpha_2(T) = \alpha_2 <1$ such that  $\mesh(\mathcal{P}_{n}) \leq  \alpha_2^n$ for all $n \in \mathbb{N}$.
\end{cor}

\noindent We will now focus on proving the key Lemma. We first isolate and prove, in the next \S~\ref{sec:balance}, an intermediate technical step in the proof. The proof of the key Lemma is then given in the following \S~\ref{sec:keylemma}.

\subsubsection{Balance of continuity intervals via  good $\mathcal{C}^1$-recurrence sequence}\label{sec:balance}
Let $\lambda^{(n)}$ be the lenght vector, whose entries  $\lambda^{(n)}_j=|I^{(n)}_j|$, $1\leq j\leq d$ be the lenghts of the continuity intervals of the induced map $T_n$ on $I^{(n)}$. We say that $\lambda^{(n)}$ is $\nu$-\emph{balanced} for some constant $\nu>1$, iff 
$$
\frac{1}{\nu} \leq \max_{1\leq i,j\leq d} \frac{\lambda^{(n)}_j}{\lambda^{(n)}_i} \leq \nu.
$$
It is well known, in the study of standard IET, that the occurrence of a positive matrix produces balanced lenghts vectors  (a fact which has been exploited since the seminal work by Veech \cite{Ve:gau}). It turns out that the same is also true for GIETs, as long as one has a priori bounds. More precisely, we will prove in this section the following:
\begin{lemma}[Balance of continuity intervals]\label{lemma:balancecont}
Given a positive Zorich matrix $A\in SL(d,\mathbb{Z})$ of lenght $p$ (see \S~\ref{sec:goodreturns} for terminology) and $K>0$, there exists a constant $\nu=\nu(A,K)$ which depends only on $K$ and the \emph{norm} $\| A\|$ such that,  if $n$ is such that $Q(n,n+p)=A$ and $K^{-1}\leq \Vert \D T_n\Vert\leq K$, then the lenghts vector $\lambda^{(n)}$ at time $n$ is $\nu$-balanced.
\end{lemma}
\begin{proof}
Without loss of generality, by replacing $T$ with $T^{(n)}$, we can assume that $n=0$. Consider the dynamical partition $\mathcal{P}_p$. By construction, since $Q(0,p)=A$ this partition consists of $\|A\|$ intervals. Let $F_0$ be the largest, so that $|F_0|\geq 1/\|A\|$. Let $j_0$ be the index of the Rohlin tower $\mathcal{P}^{j_0}_p$ which contains $F_0$ as a floor. Since every  floor $F$ in $\mathcal{P}^{j_0}_p$ can be written in the form $F=T^{\pm i}F_0$ for some $0\leq i<\|A\|$, by mean value  we have that $|F_0|\leq \| \D T^{\mp i}\| |F|$ and thus, from 
 the distorsion bound on $\D T_n$ and the analogous one for $\D T_n^{-1}$ (which holds by the formula for the derivative of the inverse function) we get that 
$$
|F|\geq |F_0| \frac{|F|}{|F_0|}\geq \frac{|F_0| }{K^{\|A\|}}\geq \frac{1}{\| A\| K^{\|A\|}} , \qquad \text{for\ all}\ F\in \mathcal{P}^{j_0}_p.
$$
(where the last inequality is by the choice of $|F_0|$). 

Now, since $A$ is positive, each continuity interval $\lambda^{(0)}_j$, for any $1\leq j\leq d$ contains  at least one floor of the tower $\mathcal{P}^{j_0}_p$.  The previous lower bound hence shows that $\min_{1\leq j\leq 1}\lambda^{(0)}_j\geq {(\| A\| K^{\|A\|})}^{-1}$. Since $\max_{1\leq j\leq 1}\lambda^{(0)}_j\leq 1$, this proves that   $\lambda^{(0)}$ is $\nu$-balanced for $\nu:=\|A\|  K^{\|A\|} $ and concludes the proof.
\end{proof}
{
\subsubsection{Decay of the mesh at  good $\mathcal{C}^1$-recurrence times.}\label{sec:keylemma}
The goal of this section is now to prove the key Lemma~\ref{lemma:key}. The main idea behind the proof, which is split in several intermediate steps, is that a  good $\mathcal{C}^1$-recurrence times produces \emph{relative balance} of the dynamical partitions, i.e.~if one studies the dynamical partitions of the induced map $T_{n_m}$ corresponding to a  good $\mathcal{C}^1$-recurrence time $n_m$ (that we call \emph{relative} partitions), one can show that after $p$ (Zorich) renormalization  steps (hence at the time $n_m+p$ in the 'middle' of the double occurrence of the positive matrix $A$) one can have a good control on the ratios of all partition elements (i.e.~we show that this relative partition is \emph{balanced}, see Step $2$ of the proof for the precise statement). Moreover, this partition is made by a subset of the intervals of the dynamical partition $\mathcal{P}_{n_m}$, sufficiently well spaced to be able to infer (via the classical distorsion bound) the desired decay of the mesh.
\begin{proof}[Proof of the key Lemma~\ref{lemma:key}]
Let $n_k$ be a given  good  time; let $A$ be the positive matrix whose double occurrence gives the corresponding good return time and let $p$ be its Rauzy-Veech length. In the proof we will work with four different renormalization times, namely $n_k, n_k+p$ and $n_k+{2p}$ (which are, informally, the times just before, in the middle and just after the double occurrence of $A$) an the initial time $n=0$. For brevity of notation, let us use the notation
$$
\ell_0:= n_k, \qquad \ell_1:= n_{k+p}, \quad   \ell_2:= n_{k+2p}
$$
\noindent Thus, $\ell_0$ is the time just \emph{before} the occurrence of $AA$, $\ell_1$ is the time in the middle, and $\ell_2$ just \emph{after}. We split the proof in several steps.

\smallskip
\noindent \emph{Step 0: Persistence of a priori bounds up to time $\ell_2$.} By assumption, since $\ell_0:=n_k$ is a  $\mathcal{C}^1$- recurrence time, $||\D T_{\ell_0}||\leq K$. We claim that since $\ell_2-\ell_0=2p$, for some $K_1>0$, we actually have a that
\begin{equation}\label{boundedderivatives}
\frac{1}{K_1}\leq ||\D T_{n}||\leq K_1, \qquad \text{for\ all}\ \ell_0\leq n\leq \ell_2.
\end{equation} 
To see this, remark first that (by definition of $A$-good return times and the cocycle relations, see~\ref{sec:Zorichcocycle}), we have that $q^{(\ell_2)}=A^2 \, q^{(\ell_0)}$ and therefore, for all $\ell_0\leq n\leq \ell_2$, writing $\D T_{n}(x)$ (for any $x$ in the domain $I^{(n)}$ of $T_n$) by the chain rule as a product of at most $\|A\|^2$ terms of the form $\D T_{\ell_0}(x_i)$ (or equivalently, considering the logarithm and decomposing the special Birkhoff sums $f^{(n)}(x)$ as sum of special Birkhoff sums  $f^{(\ell_0)}(x_i)$, see \S~\ref{sec:SBS}), we have that
$$||\D T_{n}||\leq  ||\D T_{\ell_0}||^{||A||^2}\leq K_1:= K^{||A||^2}.$$ 
To prove the lower bound, we can use $\D (T_n^{-1})$. By decomposing it in the same way as a product of at most $\|A\|^2$ terms of the form $\D T_{\ell_0}^{-1}(x_i)$ and exploiting twice the formula for the derivative of the inverse function as well as the lower bound  $\| \D T_{\ell_0}\|\geq K^{-1}$,
$$
\frac{1}{||\D T_{n}||}= \| \D (T_n^{-1})\|\leq  ||\D (T_{\ell_0}^{-1})||^{||A||^2} = \left(\frac{1}{||\D T_{\ell_0}||}\right)^{||A||^2}\leq  K^{||A||^2}=K_1
$$
which is the desired lower bound $||\D T_{n}||\geq K_1^{-1}$.
\medskip

\noindent \emph{Step 1: Balance of continuity intervals at time $\ell_1$.} Consider now the induced map $T_{\ell_1}$ at time $\ell_1$. 
Since $\ell_1$ is the time \emph{before} an occurrence of $A$, i.e.~$Q(\ell_1,\ell_2)=A$, and, by  Step 0, we have the derivatives bounds \eqref{boundedderivatives} for $n=\ell_1$, Lemma~\ref{lemma:balancecont} gives that the lenghts $\lambda^{(\ell_1)}_j=|I^{(\ell_1)}_j|$ of the continuity intervals  of  $T_{\ell_1}$ are $\nu_A$-balanced for a constant $\nu>0$ (depending on the matrix $A$ but independent on the choice of the  $\mathcal{C}^1$- recurrence time), i.e.
$$
\frac{1}{\nu_A} \leq \max_{1\leq i,j\leq d} \frac{\lambda^{(\ell_1)}_j}{\lambda^{(\ell_1)}_i} \leq \nu_A.
$$

\medskip
\noindent \emph{Step 2: Balance of the relative towers of step $\ell_0$ over $\ell_1$.} The induced map $T_{\ell_1}$ on $I^{(\ell_1)}$, since $\ell_1> \ell_0$, can also be seen as first return map of  $T_{\ell_0}$ on the subinterval $I^{(\ell_1)}\subset I^{(\ell_0)}$. Consider hence the corresponding dynamical partition of $I^{(\ell_0)}$ into subintervals which are floors of Rohlin towers (see \S~\ref{dynamicalpartitions}) for the map $T_{\ell_0}$ seen as acting on a skyscraper with base $I^{(\ell_1)}$ (see Figure~\ref{relativeP}). We will call this dynamical partition (and resp.~its Rohlin towers) the \emph{relative} dynamical partitions (resp.~the \emph{relative} towers) of step $\ell_0$ with respect to step $\ell_1$ and denote it $\mathcal{P}(\ell_0,\ell_1)$ (resp.~$\mathcal{P}^j(\ell_0,\ell_1)$, $1\leq j\leq d$).  We claim that this relative dynamical partition is \emph{balanced}, in the sense that there exists a constant $\nu_1>0$ (which depends on $A$ and $K$ but not on the initial choice of  $\mathcal{C}^1$- recurrence time $\ell_0$) such that all floors $F_1,F_2$ of towers $\mathcal{P}(\ell_0,\ell_1)$  have comparable lenghts, i.e.
$$
\frac{1}{\nu_1} \leq  \frac{|F_1|}{|F_2|} \leq \nu_1, \qquad \mathrm{for\ all}\ F_i=(T_{\ell_0})^{k_i} I^{(\ell_1)}_j,\ 1\leq j\leq d, \ 0\leq k_1,k_2< q_j(\ell_0,\ell_1),$$
where $q_j^{(\ell_0,\ell_1)}$ is the height of the relative Rohlin tower $\mathcal{P}^j(\ell_0,\ell_1)$. Since $\ell_1$ is just \emph{after} an occurrence of $A$ and therefore $q^{(\ell_1)} = A\, q^{(\ell_0)}$, the heights of each of these relative towers $\mathcal{P}^j(\ell_0,\ell_1)$ is at most $\|A\|$. Since $K_1^{-1}\leq ||\D T_{\ell_1}||\leq K_1$ by Step 0, it thefore follows that, for every floor $F$ of the relative tower $\mathcal{P}^j(\ell_0,\ell_1)$, 
$$
\frac{1}{K_1^{\|A\|}}\leq \frac{|F|}{|I^{(n_0)}_j|}\leq K_1^{\|A\|}.
$$
 Therefore, since the base intervals $|I^{(\ell_1)}_j|$ are $\nu_A$-balanced by Step 1, this shows that floors are $\nu_1$-balanced for $\nu_1:= K_1^{\|A\|} \nu_A$.

\begin{figure}
\centering
\begin{minipage}{.5\textwidth}
  \centering
			\def\svgwidth{ 0.8\columnwidth}
  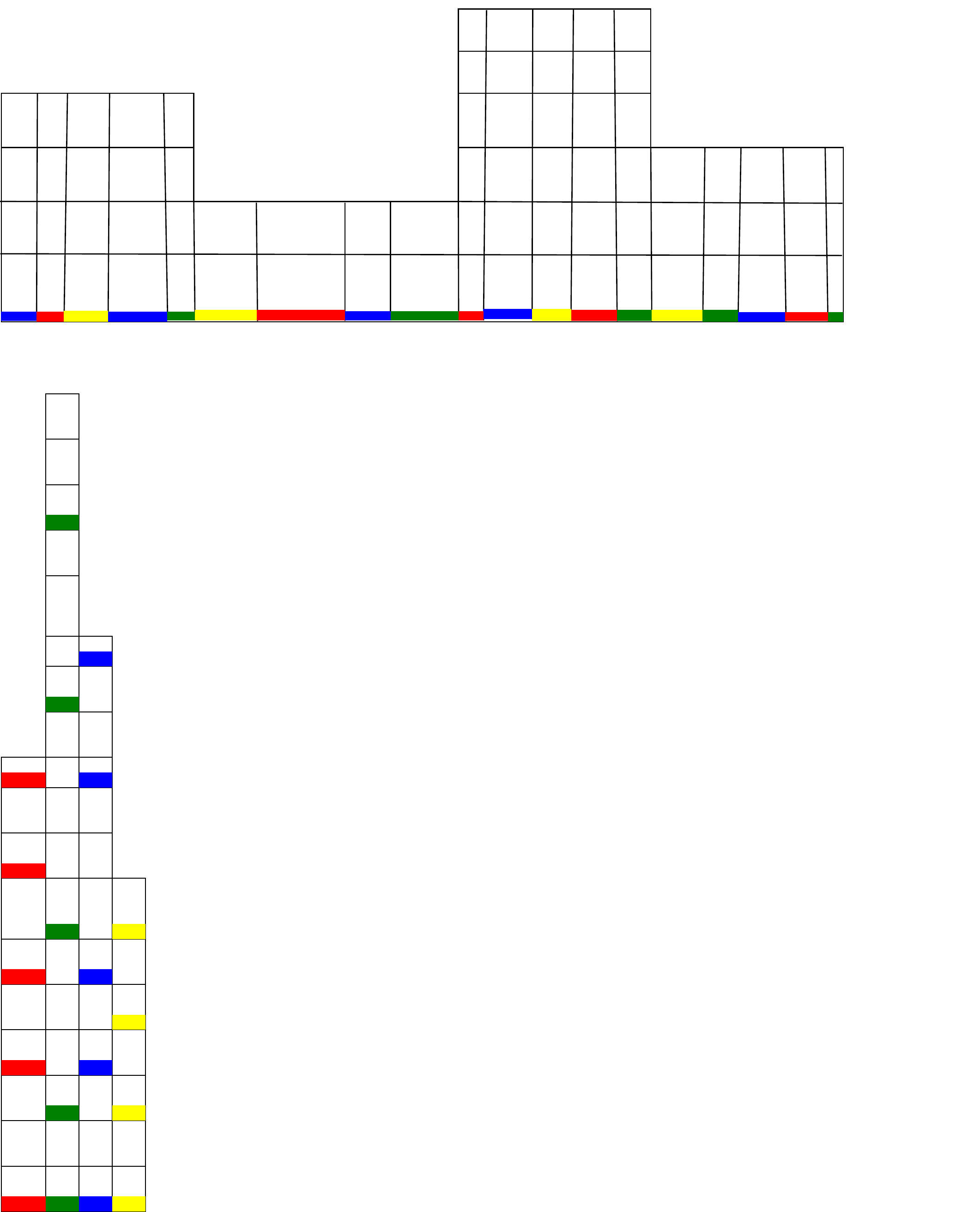
	 \caption{ The partitions $\mathcal{P}_{\ell_0}$ and $\mathcal{P}_{\ell_1}$. \label{relativeP}}
\end{minipage}%
\begin{minipage}{.5\textwidth}
  \centering
			\def\svgwidth{ 0.8\columnwidth}
			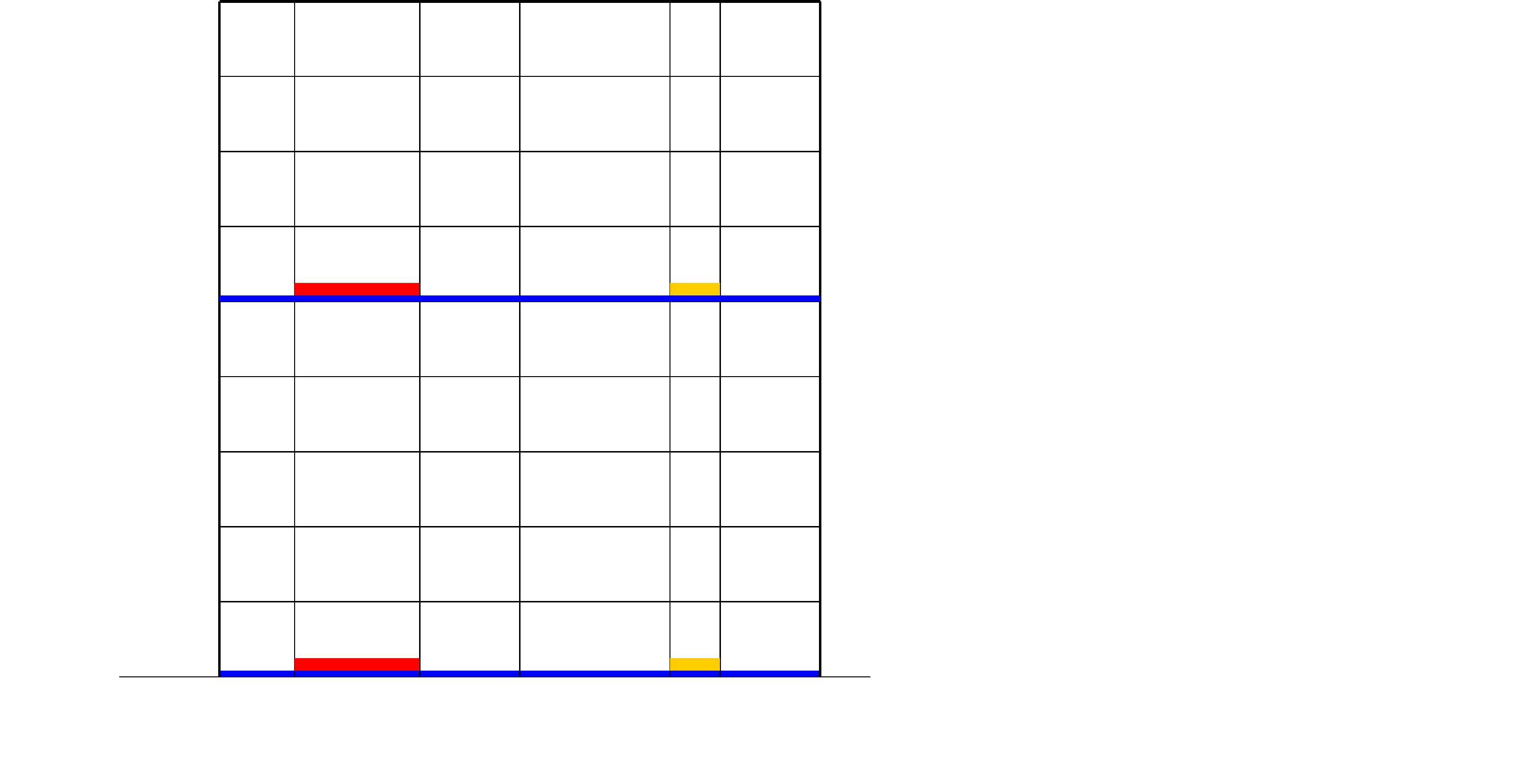
 \caption{ The tower subintervals used in Step 4 of the proof of the key Lemma~\ref{lemma:key}.\label{towernames}}
\end{minipage}
\end{figure}

\medskip
\noindent \emph{Step 3: Decay of mesh in the bases of step $\ell_0$ relative partition.} Let us now think of the same relative partition $\mathcal{P}(\ell_0,\ell_1)$ not as Rohlin towers, but as a partition of $I^{(\ell_0)}$ (as shown in the lower row of Figure~\ref{relativeP}). Notice that each of the continuity intervals $I^{(\ell_0)}_j$, $1\leq j\leq d$, is partition into a union of at most $\|A\|$ elements of $I^{(\ell_0)}$. Therefore,  
\begin{equation}\label{lowerboundfloor}
\frac{\min_{I\in \mathcal{P}(\ell_0,\ell_1)}|I|}{|I^{(\ell_0)}_j|}\geq \frac{\min_{I\in \mathcal{P}(\ell_0,\ell_1)}|I|}{\|A\| \max_{I\in \mathcal{P}(\ell_0,\ell_1)}|I|}\geq \frac{1}{\nu_1\, \|A\|},
\end{equation}
where the last inequality follows from the balance of the partition $\mathcal{P}(\ell_0,\ell_1)$ proved in Step 1. 


\medskip
\noindent \emph{Step 4: Propagating the decay of mesh in the base through distorsion bounds.} 
In this final step, we infer the decay of the mesh by an argument very similar to Step 3, only not at the base, but in the floor of the Rohlin tower of $\mathcal{P}(\ell_0)$ which contains the interval of $\mathcal{P}(\ell_1)$ which relized the mesh. The classical distorsion lemma will allow us to \emph{propagate} and repeat the estimates of Step 3 to other floors of  $\mathcal{P}(\ell_0)$.

Notice first of all that the elements of the relative  partition  $\mathcal{P}(\ell_0,\ell_1)$ are a \emph{subset} of the elements of the partition  $\mathcal{P}_{\ell_1}$, consisting exactly of all elements of $\mathcal{P}_{\ell_1}$ which are contained in the interval $I^{(\ell_0)}$, as illustrated in Figure~\ref{relativeP} (this is because $T_{\ell_0}$ is by definition the first return of $T$ to $I^{(\ell_0)}$). 

 Let $F_1$ be an interval of $\mathcal{P}_{\ell_1}$ such that $\mesh (\mathcal{P}_{\ell_1})=|F_1|$. Then, since the partition  $\mathcal{P}_{\ell_1}$ is a refinement of $\mathcal{P}_{\ell_0}$,  $F_1$ belongs to a floor of one of the Rohlin towers $\mathcal{P}^j_{\ell_0}$ of $\mathcal{P}_{\ell_0}$, say the $j_0^{th}$ one. Let us call $F_0$ this floor. 
Then we can write (referring the reader to the schematic representation in Figure~\ref{towernames} for a picture)
$$
 F_0:=T^{k_0} I^{(\ell_0)}_{j_0}, \quad \text{and}\quad F_1= T^{k_0} (I_1), \quad \text{for\ some}\ I_1\subset I^{(\ell_0)}_{j_0}.
$$
By construction $I_1$ is an element of  $\mathcal{P}_{\ell_1}$ and contained in $ I^{(\ell_0)}_{j_0}$ and hence, by the initial remark of this step, it is also an element of  $\mathcal{P}(\ell_0,\ell_1)$. As in Step 3, by positivity of the matrix $A$, there is at least one (and actually at least $d-1$) other element of $\mathcal{P}_{\ell_1}$, that we will call $I_2$, such that $I_2\subset  I^{(\ell_0)}_{j_0}$. Let $F_2=T^{k_0} (I_2)$, so that $F_2\subset F_0$, i.e.~it belongs to the same floor than contains $F_1$. Since $F_1$ and $F_2$ are by construction different, $F_1\subset F_0\backslash F_2$ and therefore $|F_1|\leq |F_0|-|F_2|$ (refer again to Figure~\ref{towernames}).  
Thus, recalling the choice of $F_1$ and using that $|F_0|\leq \mesh (\mathcal{P}_{\ell_0})$ (simply since $F$ is an element of $\mathcal{P}_{\ell_0}$), we get
\begin{equation}\label{meshratio}
\frac{\mesh ({\mathcal{P}_{\ell_1}})}{\mesh ({\mathcal{P}_{\ell_0}})} = \frac{|F_1|}{\mesh ({\mathcal{P}_{\ell_0}})}\leq \frac{|F_0|-|F_2|}{|F_0|} = 1-\frac{|F_2|}{|F_0|}.  
\end{equation}
We therefore now want to estimate the ratio $|F_2|/|F|$ which appears in the last estimate. Remark first that, by mean value theorem and then the classical distorsion bounds (namely Lemma~\ref{bound1}), for some $x,y\in I^{(\ell_0)}_{j_0}$ 
$$
\frac{|I_2|}{|I^{(\ell_0)}_{j_0}|} \frac{|F_0|}{{|F_2|}} =\frac{|F_0|/|I^{(\ell_0)}_{j_0}|}{|F_2|/|I_2|} =\frac{|T^{k_0}I^{(\ell_0)}_{j_0}||/|I^{(\ell_0)}_{j_0}|}{|T^{k_0}I_2|/|I_2|}= \frac{|\D T^{k_0}(x)|}{|\D T^{k_0}(y)|}\leq \exp{|N|(T)}.
$$
Using this estimate and then Step $3$ (in particular \eqref{lowerboundfloor}) to estimate   $I_2/|I^{(\ell_0)}_{j_0}| $ from below), we get
$$
\frac{|F_2|}{|F_0|}\geq \frac{1}{\exp{|N|(T)}}\frac{|I_2|}{|I^{(\ell_0)}_{j_0}|}\geq \frac{1}{\exp{|N|(T) \nu_1 \|A\|}}.
$$
Thus, using this estimate in \eqref{meshratio} and setting ${\alpha}_1:=\exp{(|N|(T) \nu_1 \|A\|)}^{-1}$, we get that
$\mesh ({\mathcal{P}_{\ell_1}})\leq (1-{\alpha}_1)\mesh({\mathcal{P}_{\ell_0}})$. Recalling that $\ell_1=n_{m+p}$ and $\ell_0=n_m$, this proves the key lemma. 
\end{proof}}

\subsection{Convergence to Moebius maps}\label{sec:Moebius}
 Moebius interval exchange transformations were defined in \S~\ref{affine} (see Definition~\ref{def:MIET}); we recall that $\mathcal{M}_d$ denotes the space of Moebius IETs (see \S~\ref{parameters}). In this section we show that the decay of the mesh of the dynamical partition given by Proposition~\ref{partition} implies fast convergence of $\mathcal{Z}^{n}(T)$ to the subspace $\mathcal{M}_d$ of Moebius IETs. These are by now classical arguments, well known in the study of circle diffeomorphisms and circle diffeomorphisms with break points.
\vspace{2mm}

\subsubsection{Estimates of the distance from MIETs.}
Let $T \in \mathcal{X}^3$ be infinitely renormalizable.  Recall that $\{ \mathcal{P}_n, \ n\in\mathbb{N}\}$ denotes the sequence of dynamical partitions (as defined in \S~\ref{dynamicalpartitions}) associated to Zorich renormalization orbit $\{\mathcal{Z}^nT, \ n\in\mathbb{N}\}$. 
\begin{proposition}[{distance to Moebius via the mesh}]
\label{controlMoebius}
There exists a constant $L(T) > 0$ such that we have 
$$ \mathrm{d}_{\mathcal{C}^3}(\mathcal{Z}^n T, \mathcal{M}) \leq L(T)\, \mesh ( \mathcal{P}_n ), \qquad \text{for\ all}\ n \in \mathbb{N}.$$
\end{proposition}
\noindent The proof of this statement, which we give below, uses the Schwarzian derivative (whose definition was recalled in \S~\ref{Schwarzian}). It is a modern reformulation of the miraculous cancellations that Herman brought to light in his celebrated thesis \cite{Herman}. 

\smallskip
Consider the \emph{shape-profile} coordinates $\mathcal{A}\times \mathcal{P}$ introduced in \S~\ref{coordinates}. We first state and prove a Lemma which relates the $\mathcal{C}^3$ distance from $\mathcal{M}$  to the Schwarzian derivatives of the profile coordinates. Recall that $d_\mathcal{C}^3$ is defined analogously to $d_\mathcal{C}^1$ replacing $||\cdot ||_{\mathcal{C}_1}$ with   $||\cdot ||_{\mathcal{C}_3}$ (see \S~\ref{sec:distances}).  
\begin{proposition}[{distance to Moebius via the Schwarzian}]\label{lemma:distanceMoebius}
For any $\mathcal{K}\subset \mathcal{X}^3$ which is $\mathcal{C}^3$-bounded, there exists a constant $C=C(\mathcal{K})$ such that for any  $T \in \mathcal{K}$
$$
d_{\mathcal{C}^3}(T, \mathcal{M}) \leq C(\mathcal{K})\, \sum_{i=1}^d|| \mathrm{S}(\varphi_T^i) ||_{\mathcal{C}^0},
$$
where, for each $i$, $S(\varphi_T^i)$ is the Schwarzian derivative of the coordinate $\varphi_T^i$ of the profile of $T$.
\end{proposition}
\noindent The proof of this Proposition is given in the Appendix~\ref{app:C3S}.  We  now prove Proposition~\ref{controlMoebius}.
\begin{proof}[Proof of Proposition~\ref{controlMoebius}]
For any $n\in\mathbb{N}$, let  $\varphi_n^j \in \mathrm{Diff}^3([0,1])$, for $1\leq j\leq d$, denote the $j$-th coordinate of the \emph{profile} of $\mathcal{Z}^n T$ (in the \emph{shape}-\emph{profile} coordinates of \S~\ref{coordinates}). By definition of profile, 
$\varphi_n^j = \Norm (T^{(n)}_j)$, where $T^{(n)}_j$ is the $j^{th}$ branch of $T^{(n)}=\mathcal{Z}^n T$ and $\mathcal{N}(\cdot)$ denotes the renormalization operator which, to any diffeomorphism $f : I \rightarrow J$ of a connected interval $I$ onto $J$, associates  $ \Norm(f) := b \circ f \circ a $, where $a$ and $b$ are respectively the only orientation-preserving affine map mapping $[0,1]$ onto $I$ and $J$ onto $[0,1]$.

For brevity, let us denote by  $q_j:= q^{(n)}_j$ the height of the Rohlin tower $\mathcal{P}^j_n$. Let   $f_j^k$, for  $0\leq k< q_j$,  denote the restriction of $T$ to the floor $T^k(I_j^{(n)})$ of the Rohlin tower $\mathcal{P}^j_n$. Then, since by definition of renormalization $T^{(n)}_j = f_j^{q_j-1} \circ f_j^{q_j-2} \circ \cdots \circ f_2 \circ f_1$, we can write 
$$  \varphi_n^j = \Norm( T^{(n)}_j ) = \Norm(f_j^{q_j-1}) \circ \Norm(f_j^{q_j-2}) \circ \cdots \circ \Norm (f_j^{1}) \circ \Norm(f_j^{0}).$$
\noindent 
Thus, by the chain rule for the Schwarzian derivative (see $(S1)$ in \S~\ref{Schwarzian}), if we introduce the following notation for partial products:
$$ \phi_j^k :=  \begin{cases} \Norm(f_j^{k}) \circ \cdots \circ \Norm(f_j^{1}) \circ \Norm(f_j^{0}), & \text{for}\ k=0,\dots, q_j-1, \\  0 & \text{for} \ k=-1, \end{cases}
$$
one can verify by induction  that
\begin{equation}\label{chainruleSch}\mathrm{S}(\varphi_n^j) = \sum_{k=0}^{q_j-1}{ \mathrm{S}(\Norm(f_j^{k})) \circ  \phi_j^{k-1} \,(\mathrm{D}(\phi_j^{k-1}))^2}.\end{equation}
We now use two observations. The first, which follows by Property $(S3)$ of the Schwarzian derivative (see \S~\ref{Schwarzian}) since the domain of $f_j^k$  is the interval $T^k(I_j^{(n)})$, is that the Schwarzian derivatives of each $f_j^k$ satisfies
\begin{equation}\label{Sch_rescaling} ||\mathrm{S}(\Norm(f_j^k))||_\infty
=  |T^k(I_j^{(n)})|^2 ||\mathrm{S}(f_j^k)||_\infty,
\end{equation}
where $|| \cdot ||_\infty$ denotes the sup norm on the domain where the function is defined (so the $L_\infty([0,1])$ norm for $\Norm(f_j^k)$ and the ${L_\infty(T^k(I_j^{(n)}))}$ norm  for $\mathrm{S}(f_j^k)$.

\smallskip
The second important observation is the claim, that follows from the classical distortion bounds (see Lemma~\ref{bound1}), that 
the derivatives of  $ \phi_j^k$ are uniformly bounded above and below, \textit{i.e} there exists $D_1(T) = D_1$ such that for all $j \leq d$ and $k \leq l_j^n$, 
$$ D_1^{-1} \leq \mathrm{D} \big(\Norm(f_j^{k}) \circ \cdots \circ \Norm(f_j^{1}) \circ \Norm(f_j^{0}) \big) \leq D_1.$$ 
The  proof is the same than the proof of the profile a priori bounds in Lemma~\ref{bound2}:   by chain rule and mean value, choosing $y\in [0,1]$ such that $D\phi_j^k(y)=1$,
$$\sup_{x\in [0,1]}  \mathrm{D} \big(\Norm(f_j^{k}) \circ \cdots \circ \Norm(f_j^{1}) \circ \Norm(f_j^{0}) \big) = \sup_{x\in [0,1]} \frac{{D} \phi_j^k (x)}{{D} \phi_j^k (y)} = \sup_{x,y\in I_j^{(n)}} \frac{D (f^k_j \circ \cdot \circ f^0_k)(x)}{D (f^k_j \circ \cdot \circ f^0_k)(y)},$$
so we can then conclude by applying the a priori bounds given by Lemma~\ref{lemma1}. 
\smallskip

We can now estimate \eqref{chainruleSch} using these two observations, together with the remark that  $||\mathrm{S}(f_j^k)||_{\infty} \leq || \mathrm{S}(T) ||_{\infty}$ since $f_j^k$ is a restriction of $T$, thus getting that 
$$ || \mathrm{S}(\varphi_n^j) ||_{\infty} \leq ||  \mathrm{S}(T) ||_{\infty} \,D_1^{\color{black}2} \,\sum_{k=0}^{q_j-1}{|T^k(I_j^{(n)})|^2 }.$$ Since for any $0\leq k<q_j$, we can estimate $ | T^k(I_j^{(n)})|^2 \leq |T^k(I_j^{(n)})| \,\mesh( \mathcal{P}_n) $, we obtain that 
$$ || \mathrm{S}(\varphi_n^j) ||_{\infty} \leq ||\mathrm{S}(T)||_{\infty} \,D_1^{\color{black}2} \, \mesh(\mathcal{P}_n) .$$ 
To conclude, it suffices now to apply Proposition~\ref{lemma:distanceMoebius}, which shows that the distance $ \mathrm{d}_{\mathcal{C}^3}(\mathcal{Z}^n T, \mathcal{M}) $ is controlled by the sum over $j=1,\dots, d$ of the Schwatzian derivatives above and hence gives the desired estimate in terms of $\mesh( \mathcal{P}_n)$. 
\end{proof}

\subsubsection{Exponential convergence to MIETs.}
Combining Proposition~\ref{controlMoebius} with the decay of the partitions mesh given by Proposition~\ref{partition}, 
since the sequence $(n_{k_m})_{m\in\mathbb{N}}$ grows linearly with $m$, we get (reasoning as in the proof of Proposition~\ref{partition} in \S~\ref{sec:sizePcontrol}) that iterated renormalizations of $T$ converge exponentially fast to Moebius IETs: 

\begin{cor}[Exponential convergence to Moebius IETs]
\label{convergenceMoebius}
Let $\alpha_2 = \alpha_2(T) $ be as in Corollary~\ref{expmeshdecay}. There exists $K_3(T) > 0$ such that for any $n\in \mathbb{N}$ 
$$\mathrm{d}_{\mathcal{C}^3}(\mathcal{Z}^{n}T, \mathcal{M} ) \leq K_3(T) \, \alpha_2^n.$$
\end{cor}

{\color{black}
\subsubsection{Parameters of MIETs}\label{MIETparameters}
Let us gather here some basic  properties of Moebius diffeomorphisms and, as a consequence, of Moebius IETs which will be useful in the following sections.

\smallskip
\noindent Consider first the group $\mathcal{M}([0,1])$ of {orientation preserving} \emph{Moebius diffeomorphism} of $[0,1]$. If $m\in\mathcal{M}([0,1])$, then one can check that:
\begin{itemize}
\item[(m1)] {the sign of $D^2m$ is constant (i.e.~$m$ is either convex or concave)} and $\log Dm$ is monotone;
\item[(m2)] The mean non-linearity $\overline{N}(m)=\log \,\! \mathrm{D} \,\! m\, (1)- \log\,\! \mathrm{D}\,\! m\, (0)$;
\item[(m3)] Given $u\in \mathbb{R}$, there exists a {unique}\footnote{The (unique) Moebius diffeomorphism $m_u$ with $m_u(0)=0$, $m_u(1)=1$ and mean non-linearity $u$ is indeed explicitely given by $$m_u(x)=\frac{x e^{-\frac{u}{2}}}{1+x(e^{-\frac{u}{2}}-1)}.$$ It can be found for  example as a special case of the formula in Appendix~\ref{distance}.} $m_u\in \mathcal{M}([0,1])$ with mean non-linearity $\overline{N}(m_u)=u$.

\end{itemize}
These observations translate into the following properties Moebius IETs. Recall that if $M$ is a MIET, then each branch $M_i$ is a Moebius diffeomorphism of $I^t_i$ into $I^b_i$ (see Definition~\ref{def:MIET}). Since by definition of non-linearity { $\eta_{M_i}=D M_i^2/ D M_i$} (see (see \S~\ref{sec:nl}) and $D M_i$ is continuous and non-zero on each $I^t_i$, the following remark then follows immediately from  
 $(\mathrm{m}1)$.
\noindent 

\begin{remark}[Sign-coherence of non-linearity for MIETs]\label{Moebius_constantsign}
If $T$ is a Moebius IET, the \emph{sign} of the non-linearity $\eta_M$ is \emph{constant} on each of the continuity intervals $I^t_j$, $1\leq j\leq d$, of $T$.
\end{remark}  
\noindent We now deduce from $(\mathrm{m}2)$ and $(\mathrm{m}3)$ the finite-dimensionality of $\mathcal{M}_d$ and characterize subsets which are $d_{\mathcal{C}^k}$-bounded in the sense of Definition~\ref{def:Ckbounded} in terms of a priori bounds.

\begin{lemma}[Finite-dimensionality and bounded subsets of $\mathcal{M}_d$]\label{Mfindim}
The space  $\mathcal{M}_d$  of MIETs on $d$ intervals \emph{finite-dimensional} space of dimension $3d-2$. More precisely:
\begin{itemize}
\item[(M1)]  Given $M\in\mathcal{M}_d$, $M$ is completely determined by its \emph{shape} $A_M$ and the mean non-linearities of each branch, i.e.~by a vector $\underline{\eta}=(\eta_1,\dots, \eta_d)\in\mathbb{R}^d$ such that $\eta_i= \overline{N}({M}_i)$ for $1\leq i\leq d$, which fully determine the profile coordinate $\varphi_M$;\smallskip
\item[(M2)] The  vector $\underline{\eta}$  of mean non-linearities is fully determined by  the right and left limits of $\log\,\! D\,\! M$ at endpoints of the top partition; in particular $||\underline{\eta}||\leq {2}||\log\,\! D\,\! M||_\infty$;\smallskip
\item[(M3)]  If $\mathcal{K}\subset \mathcal{M}_d$ is such that MIETs in $\mathcal{K}$ satisfy an a priori bound, i.e.~there exists a $K>0$ such that $K^{-1}\leq || DM ||_{\infty} \leq K$ for each $M\in \mathcal{K}$, then $\mathcal{K}$ is $\mathcal{C}^k$-bounded for every $k\in \mathbb{N}$.
\end{itemize}
\end{lemma}

\begin{proof}
Since $\mathcal{A}_d$ is a finite dimensional space of dimension $2d-2$ (see~\S~\ref{AIET}), to show Property $(M3)$ and the finite-dimensionality, 
it is enough to show that the vector $\eta=(\eta_1,\dots, \eta_d)\in\mathbb{R}^d$, which gives $d$ additional parameters, determines the profile $\varphi_M=(\varphi_M^1, \dots, \varphi_M^d)$. For each $1\leq i\leq d$, since $\varphi_M^i$ is obtained from the branch $M_i$ by rescaling,  by definition of $\eta_i=\overline{N}(M_i)$ and the rescaling property $(vi)$ of non-linearity (see Lemma~\ref{lemma:nonlinearity}), we should have $\overline{N}(\varphi_M^i)=\eta_i$ and this determines $\varphi_M^i$ by $(\mathrm{m}3)$.
\smallskip 
Property $(M2)$ now follows from $(\mathrm{m}2)$, that shows that $\underline{\eta}$ is fully determined by the values of $\log\,\! D\,\! M$. 

\smallskip 
Finally, if we want to show that a subset of  $\mathcal{M}\in\mathcal{M}_d$  is $d_{\mathcal{C}^k}$-bounded (in the sense of Definition~\ref{def:Ckbounded}, i.e.~contained in a ball for the distance $d^\pm_{\mathcal{C}^k}$, by definition of the distance $d_{\mathcal{C}^k}$ as product (see \S~\ref{sec:distances}), the  shape parameters do not matter since they are always contained in a ball of fixed radius for $d_\mathcal{A}$. Thus, the only potentially unbounded parameters are the coordinates of $u$. Now, by $(M2)$, these are bounded if are controlled by
 $||\log DM ||_{\infty}$, so $\mathcal{K}$ is contained in a ball for $d^\pm_{\mathcal{C}^1}$
 property $(M3)$ therefore follows. Finally, since $\mathcal{M}_d$ is finite dimensional and the distances $d_{\mathcal{C}^k}$, $k\in\mathbb{N}$, are all induced by a norm, they are all equivalent. Thus,  $(M3)$ follows.
 \end{proof}

}
\subsection{Convergence to AIETs}\label{AIETs}
We now turn to improving the convergence to Moebius IETs to a convergence to AIETs, under the additional hypothesis that we are in the \emph{affine regime}, namely that $\sum_{s=1}^\kappa b_s=0$ where $(b_s)_s=\mathcal{B}(T)$. We recall that this assumption is equivalent to asking that the mean non-linearity $\overline{N}(T)= \int_0^1 \eta_T(x) dx$ vanishes (see Lemma~\ref{boundaryaffinereduction}). 

Our approach to convergence to AIETs is to study the total non-linearity $|N|(\mathcal{Z}^n(T))$ (see Definition~\ref{def:Ns}) and to show that it  converges to zero as $n $ grows. 

\subsubsection{A combinatorial lemma}
We first need an easy (but crucial) combinatorial Lemma.

\begin{lemma}
\label{combinatoric} Given  any   $a_{ij} \in \mathbb{R}$, for $i,j \in \{1, \cdots, d \}$, denote by 
$$
A_i:=\sum_{j=1}^d {a_{ij}} \, \qquad A^j := \sum_{i=1}^d{a_{ij}}.
$$
For any given $0<c<1/d$, there exists $c'<1$ such that, if
\vspace{1mm}
\begin{enumerate}
\item $\sum_{i=1}^d{A_i} = 0$, or equivalenly, $ \sum_{i,j}{a_{ij}} = 0$; (\text{zero-average assumption})\vspace{1.3mm}
\item ${a_{ij}}/{A_i}>c$ for any $i,j\in {1,\dots, d}$ (\text{balance assumption}),

\vspace{1.3mm}
 \noindent (which implies in particular that for every $i$,  all $a_{ij}$, for $1\leq j\leq d$, have the \emph{same sign} of $A_i$); \vspace{1.3mm}
\end{enumerate} 
Then we have that $$ \sum_{j=1}^d{|A^j|} \leq c' \sum_{i=1}^d{|A_i|}. $$
\end{lemma}
\begin{proof}
We first prove the Lemma for $d=2$. By assumption $(2)$ and definition of $A_i$, 
we can write:
\begin{align*}
a_{11} = c_1 \,A_1,  \ \ \qquad &  a_{12} = (1-c_1) \,A_1,&  \text{where} \ c\leq c_1 \leq 1-c,\\
a_{12} = c_2 \,A_2, \ \ \qquad &  a_{12} = (1-c_2) \,A_2,  &  \text{where} \ c\leq c_2 \leq 1-c.
\end{align*} 
Thus, by definition of $A^j$ and since $A_2=-A_1$ by $(1)$, we now have
$$A^1 = c_1 A_1 + c_2 A_2 = (c_1 - c_2) A_1, \qquad A^2 = (1 -c_1) A_1 + (1 - c_2) A_2 = (c_1 - c_2)A_2.$$ 
Now, since $|c_1 -c_2| \leq 1 - 2c$, we thus get 
$$ |A^1| + |A^2| \leq (1-2c) |A_1| + |A_2|.$$

\noindent The general case can be reduced to the the case $d=2$ by grouping together positive $A_i$ and  negative $A_i$ as follows. Let us define 
$$B_+ := \sum_{A_i > 0}{A_i}, \qquad  B_- := \sum_{A_i < 0}{A_i}.$$
We can further write $B_+ = b_{++} + b_{+-}$ and  $B_- = b_{-+} + b_{--}$ where: 
\begin{align*}
b_{++} &:= \sum_{A_i > 0, A_j > 0 }{a_{ij}}, & b_{+-} := \sum_{A_i > 0, A_j < 0 }{a_{ij}}, \\
b_{-+} &:= \sum_{A_i < 0, A_j > 0 }{a_{ij}}, & b_{--} := \sum_{A_i < 0, A_j < 0 }{a_{ij}}
\end{align*}
We claim now that we can apply the case $d=2$ of the lemma $\left\{ b_{ij}, \ i,j\in \{ +,-\}\right\}$. Indeed, assumption $(1)$ holds since their sum is $B_+ + B_-=\sum_{i=1}^d{A_i}=0$, while  
$| b_{++} | \geq \sum_{A_i > 0}{c|A_i|} = c|B_+|, $
and similar estimates hold for the other coefficients, so also the balance assumption $(2)$ holds. Thus, denoting by $B^+ := b_{++} + b_{-+}$ and $B^- := b_{+-} + b_{--}$ the conclusion of the lemma for $d=2$ proved above together with the trivial remark that if $b_1,\dots , b_k$ have the \emph{same sign} then  $\sum_{j=1}^k|b_k|=|\sum_{j=1}^k b_k|$ gives
$$  \sum_{j=1}^d{|A^j|} = |B^+| + |B^-| \leq c' ( |B_+| + |B_-| )= \sum_{i=1}^d{|A_i|},$$ 
which is the result for $d>2$.
\end{proof}

\subsubsection{Non-linearity decrease at  $\mathcal{C}^1$- recurrence times}\label{sec:decrease}
 The following key Lemma, which is based on the linear algebra Lemma \ref{combinatoric} above, will be used to show that every  $\mathcal{C}^1$- recurrence time $n_{k_m}$ (which corresponds to a double occurrence of a positive matrix $A$ together with a priori bounds, see Definition~\ref{def:C1recurrence}),   the total non-linearity \emph{decreases} by a \emph{definite factor} after renormalizing.

\begin{lemma}[Non-linearity decrease at  $\mathcal{C}^1$- recurrence times]
\label{decreaseAN}
Let $L>0$, let $p$ be positive integers and let $A$ be positive Zorich matrices of lengths respectively $p$.   
 There exists a constant $\alpha_3 =\alpha_3(L,p, || A||)<1$, such that the following holds.
Let $M$ be a Moebius IET such that:
\begin{enumerate}
\item  the total non-linearity vanishes, i.e.~$|N|(M)=\int_0^1{\eta_M}(x)\, \mathrm{d}x = 0$;\vspace{1.3mm}
\item the distorsion bound $L^{-1}\leq || \mathrm{D}M ||_{\infty} \leq L$ holds;\vspace{1.3mm}
\item  one has  $Q(0,2p)=A A$ where $Q(0,n)=Z_1(M)\cdots Z_{n-1}(M)$ is the Zorich cocycle associated to $M$.\vspace{1.3mm}
\end{enumerate}
Then $ |N|(\mathcal{Z}^{p}(M))  \leq \alpha_3 |N|(M)$.
\end{lemma}

\begin{proof}
Consider the dynamical partitions $\mathcal{P}_n$, $n\in \mathbb{N}$, associated to $\mathcal{Z}^n M$, $n\in\mathbb{N}$ (see \S~\ref{dynamicalpartitions}). Let $\ell_1:=p$ and let  $M^1:= \mathcal{Z}^{\ell_1} (M)$. 

\smallskip
\noindent {\it Step 1: Partition balance at time $\ell_1$.} We first claim that the assumptions $(2)$ and $(3)$ imply that, at time $\ell_1=p$,  all floors $F^k_i$ of the partition $\mathcal{P}_{\ell_1}$ are balanced, i.e.~there exists $\nu_1>1$ such that
$$
\nu_1^{-1} < \frac{F^{k_1}_{i_1}}{F^{k_2}_{i_2}}
<  \nu_1, \qquad \text{for\ all}\  1\leq k_i \leq q^{(\ell_1)}_i, \ 1\leq i_1,i_2 \leq d, \qquad \text{where}\ F^k_i:=M^{k}\left( I^{(\ell_1)}_i\right).
$$
This can be seen (as in the proof of mesh decay key Lemma~\ref{lemma:key} in \S~\ref{sec:keylemma}) in two steps, by considering the three times $\ell_0:=0, \ell_1:=p$ and $\ell_2:=2p$. First, since $Q(\ell_1, \ell_1+p)=Q(p,2p)=A>0$, from Lemma~\ref{lemma:key} we get the continuity intervals $I^{(\ell_1)}_j$ in the base $I^{(\ell_1)}$ are all $\nu$-balanced for some $\nu=\nu(A)$. Then, since the number of floors $q^{(\ell_1)}_j$ in each Rohlin tower $\mathcal{P}_{\ell_1}^j$ is bounded by $\| A\|$ (using here that $Q(0,\ell_1)=Q(0,p)=A$), by the  $\mathcal{C}^1$- recurrence assumption $(2)$, the balance  on the base can be \emph{transported} to show $\nu_1:= \nu K^{||A||}$ balance for all floors of the towers, as desired (we refer the reader to  the proof of Lemma~\ref{lemma:key} for more details).

\medskip  \noindent {\it Step 2: decompositions of non-linearities.} Let us now consider the two non-linearities that we want to compare, namely $N(M)$ and $N(M^1)$ for $M^1=\mathcal{Z}^p M$ and decompose them as follows. On one hand, for each continuity  interval $I_i:=I^t_i$ for $M$, since we can write $I_i=\cup_{j=1}^d ( I_i\cap \mathcal{P}^{\ell_1}_j)$, we can write the mean non linearity of the branch $M_i$ of $M$, that we wil denote by $N_i$, as
\begin{equation}\label{NbranchM}
N_i:= \int_{I_i}\eta_{M_i}(x)\, \mathrm{d}x = \sum_{j=1}^{d}n_{ij}, 
 \qquad \text{where}  \ n_{ij}:= \int_{I_i\cap \mathcal{P}^{\ell_1}_j}\eta_{M_i}(x)\, \mathrm{d}x.
\end{equation}
On the other hand, by definition of $M^1$ as induced map, on the continuity interval $I^1_j:=I^{(\ell_1)}_j$, since the branch $M^1_j$ is the composition of the restrictions of $M$ to the floors $F^k_j$ of the Rohlin tower $\mathcal{P}_{\ell_1}^j$ of height $q^1_j:=q^{(\ell_1)}_j$, exploiting the distribution property of non-linearity (i.e.~property $(ii)$ in Lemma~\ref{lemma:nonlinearity}) we can write 
\begin{equation}\label{NbranchM1}
 \int_{I^1_j}\eta_{M^1_j}(x)\, \mathrm{d}x = \sum_{k=0}^{q^1_j-1}\int_{F^k_j}\eta_{M}(x)\, \mathrm{d}x = \int_{\mathcal{P}^{\ell_1}_j}\eta_T(x)\, \mathrm{d}x = \sum_{i=1}^d\int_{\mathcal{P}^{\ell_1}_j\cap I_i}\eta_M(x)\, \mathrm{d}x= \sum_{i=1}^d n_{ij}.
\end{equation}

\medskip  \noindent {\it Step 3: Linear algebra lemma assumptions.} We now want to apply the linear algebra Lemma~\ref{combinatoric}, to the quantities
$$
n_{ij}:= \int_{I_i\cap \mathcal{P}^{\ell_1}_j}\eta_{M_i}(x)\, \mathrm{d}x, \qquad  i,j\in \{1,\dots,d\}.
$$
In Step $2$ we have already shown both that $N_i=\sum_{j=1}^d n_{ij}$, by \eqref{NbranchM}, and that the non-linearity of $M^1_j$, that we will denote $N^j$, satisfies $N^j=\sum_{i=1}^d n_{ij}$, by \eqref{NbranchM1}.  
Moreover, by the assumption $\overline{N}(M)=0$, we have that $\sum_{i=1}^d N_i= \sum_{i=1}^d \int_{I_i}\eta_{M_i}(x)\, \mathrm{d}x=0$, which shows that assumption $(1)$ is satisfied. We will now show that also assumption $(2)$ holds for  $a_{ij}:=n_{ij}$, $1\leq i,j\leq d$. We use here that $M$ is a Moebius IET.


\smallskip
\noindent 
Remark first of all that since $M$ is a Moebius functions (see Remark~\ref{rk:Moebius})
\begin{enumerate} 
\item for each $1\leq i\leq d$, $\eta_{M_i}(x)$  has \emph{constant sign} on $I_i$ and therefore $\eta_{ij}$, which are obtained integrating $\eta_{M_i}(x)$ on subintervals of $I_i$, have the same sign that $\eta_i$ for all $j=1,\dots, d$;
\item $|| \log \mathrm{D}M ||_{\infty} \leq L$ (by the assumption $(2)$ of the Lemma), so $e^{1/L}\leq || \mathrm{D}M_i||_{\infty} \leq e^L$;
\item  $\eta_{M}$  is uniformly bounded above and below on each branch by $(1)$ above;
\item the floors of $\mathcal{P}_{n_1}$ are balanced by Step 1.
\end{enumerate} 
These  remarks can be used to conclude that   also assumptions $(2)$ of  Lemma~\ref{combinatoric} holds for $\{ \eta_{ij}, \ 1\leq i,j\leq d\}$. 

\medskip \noindent {\it Step 4: Conclusions.}  By Step 3, we can apply the linear algebra Lemma~\ref{combinatoric} and we get the existence of a constant $c'>0$ such that 
\begin{equation}\label{combinatoricconclusion} \sum_{i=1}^d{|\eta_{M_i}|} \leq c' \sum_{j=1}^d{|\int_{I_j}{\eta_{M_1}}|}.
\end{equation} 
Since both $M$ and $\mathcal{Z}^p(M)$ are MIETs, the signs of $\eta_M$ and $\eta_{M_1}$ are constant on each branches of $M$ and $M_1$ respectively (by the considerations in the previous Step), so that have 
$$|N^j| = |\int_{I^1_j}{\eta_{M^1}}=\int_{I^1_j}{|\eta_{M^1}|}, \qquad |N_i|=|\int_{I_i}{\eta_M}| = \int_{I_i}{|\eta_M|}.$$ 
\noindent With this observation, since  $N^{j}$ is the total non-linearity of the $j$-th branch of $M_1=\mathcal{Z}^p(M)$ (see Step 1, \eqref{NbranchM1}), the inequality \eqref{combinatoricconclusion} given by  Lemma~\ref{combinatoric} can be rewritten as
$$|N|(\mathcal{Z}^p(M))=  \sum_{j=1}^d \int_{I^1_j}{|\eta_{M^1}|} = \sum_{j=1}^d{|N^j|} \leq c' \sum_{i=1}^d |N_i|= c' \sum_{i=1}^d{|\int_{I_j}{|\eta_M}|}=c' |N|(T), $$ 
which proves the Proposition for $\alpha_3:= c'$.
\end{proof}

\subsubsection{Exponential decay of the total non-linearity}
We now have everything we need to show that the total non-linearity decreases exponentially along the orbit under renormalization and therefore conclude exponential convergence to the subspace $\mathcal{A}_d$ of AIETs:
\begin{proposition}[exponential decay of the total non-linearity]\label{nLdecay}
There exists constants $K_5 = K_5(T) > 0$ and $0 < \alpha_5 = \alpha_5(T) < 1$ such that for all $n \geq 0$
$$ |N|(\mathcal{Z}^{n}T) \leq K_5 \,\alpha_5^n.$$
\end{proposition}
\noindent We prove this Proposition below. This will then be shown to also imply convergence to AIETs:
\begin{cor}[exponential convergence to AIETs]\label{convergenceaffine}
For $K_5>0$ and $\alpha_5<1$ as in Proposition~\ref{nLdecay}, 
$$ \mathrm{d}_{\mathcal{C}^3}(\mathcal{Z}^{n}T, \mathcal{A}_d) \leq K_5 \,\alpha_5^n, \qquad \text{for \ all}\ n\in\mathbb{N}.$$
\end{cor}
\noindent In the proof of Proposition~\ref{nLdecay}, we use the distance  $\mathrm{d}_{\eta}$ which was defined in \S~\ref{sec:distances} and is used here as a technical tool for this step of the proof. We will also use the following proposition:
\begin{proposition}[see \cite{Selim:loc} and Appendix~\ref{sec:LipR}]\label{prop:LipR}
\label{lipschitz2}
Let $\mathcal{K} \subset \mathcal{X}^3$ be a ${\mathcal{C}^3}$-bounded set. Then there exists a constant $K = K(\mathcal{K})$ such that $\mathcal{V}$ is $K$-Lipschitz on $\mathcal{K}$ with respect to $d_{\eta}$.
\end{proposition}
\noindent The Proposition was proved by the first author in \cite{Selim:loc}. We include a proof in Appendix~\ref{sec:LipR}.

\begin{proof}[Proof of Proposition~\ref{nLdecay}]
For simplicity of notation let us denote in this proof $N_m:=|N|({\mathcal{R}^m T})$ the total-non linearity of $\widetilde{R}^{m}T:=\mathcal{Z}^{n_{k_m}}(T)$. Since by Corollary~\ref{convergenceMoebius} $d_{\eta}(\mathcal{R}^m T, \mathcal{M})  \leq 
 K_3 \alpha_2^k$, by definition of distance from a set, for every $m\in\mathbb{N}$ we can find an Moebius IET $M_m$ in $\mathcal{M}_d$ such that 
\begin{equation}\label{expdec} d_{\eta}(\mathcal{R}^{m}T, M_m) \leq (K_3 + 1) \alpha_2^m.\end{equation}
Thus, since  $|N|(\mathcal{Z}^{n}T)$ is descreasing in $n$ (see $(iii)$ of Lemma~\ref{prop:Nproperties}) and we can assume without loss of generality that $n_{k_{m+1}}\geq n_{k_m}+p$,  writing $|N|(\mathcal{Z}^p(\mathcal{R}^m T))$ as a $d_\eta$ distance 
(see Remark~\ref{rk:nonlinearityasd}) and using the triangle inequality for $d_\eta$ 
\begin{equation}\label{Nm1}
N_{m+1} = |N|(\mathcal{R}^{m+1} T)) \leq  |N|(\mathcal{Z}^p(\mathcal{R}^m T))\leq |N|(\mathcal{Z}^{p} M_m) + d_{\eta}(\mathcal{Z}^{p}\mathcal{R}^m T, \mathcal{Z}^{p}M_m).
\end{equation}
Since the times $(n_{k_m})_{n\in\mathbb{N}}$ (by the $(RDC)$ condition) are $p$-good return times, for every $m$ we know that $Q(n_{k_m} ,n_{k_m}+2p)=AA $ for a fixed positive matrix $A$ and 
 furthermore the a-priori distorsion bounds given by Proposition~\ref{apriori} holds (so that $n_{k_m}$ are $\mathcal{C}^1$ recurrence times in the sense of Definition~\ref{def:C1recurrence}). We can therefore apply the non-linearity decrease Lemma~\ref{decreaseAN} to the Moebius IET $M_n$, followed by once more the triangle inequality for $d_\eta$ (recalling Remark~\ref{rk:nonlinearityasd}) and by \eqref{expdec}  to get 
$$ |N|(\mathcal{Z}^{p} M_m) \leq  \alpha_3 |N|( M_m) \leq  \alpha_3 \left( |N|(\mathcal{R}^m T) + d_{\eta}(\mathcal{R}^m T,M_m)\right)\leq \alpha_3 \left(N_m + (K_3 + 1) \alpha_2^m\right).
$$ 
Combining this with \eqref{Nm1}, we get
\begin{equation}\label{Nm2}
N_{m+1}  \leq \alpha_3 ( N_m + (K_3 + 1) \alpha_2^m ) + d_{\eta}(\mathcal{Z}^{p}\mathcal{R}^m T, \mathcal{Z}^{p}M_m) .
\end{equation}
\noindent {\it Bounded set for Lipschitz control.} In order to estimate $d_{\eta}(\mathcal{Z}^{p}\mathcal{R}^m T, \mathcal{Z}^{p}M_m)$ in the next step by using the  Lipschitz property on bounded sets given by  Proposition~\ref{lipschitz}, let us now show that there exists a $\mathcal{C}^3$-bounded set $\mathcal{K}$ such that $\mathcal{R}^m T$, as well as the iterates  $\mathcal{V}^n(\mathcal{R}^m T)$ with $0\leq n\leq n_p$, where $n_p$ is such that  
$\mathcal{V}^{n_p}(\mathcal{R}^m T) = \mathcal{Z}^{p}(\mathcal{R}^m T)$, all belong to $\mathcal{K}$.  

\noindent We already know that the a priori bounds given by Proposition~\ref{apriori} hold for $\mathcal{R}^m T =\mathcal{Z}^{n_{k_m}}(T)$, namely $K_2^{-1} \leq ||\D\mathcal{R}^m(T)||_\infty \leq K_2(T)$. Since  $n_{k_m}$ is also a $p$-good time and therefore $Q(n_{k_m}, n_{k_m}+p)=A$, we get that  all the branches of $\mathcal{V}^n(\mathcal{R}^m T)$ for $0\leq n <n_p$ are obtained composing at most $||A||$ branches of $\mathcal{R}^m T$. This shows that a priori bounds also hold for all $\mathcal{V}^n(\mathcal{R}^m T)$ for $0\leq n <n_p$, when the constant $K_2$ is replaced by $K_2^{||A||}$. 

\noindent Since we have already proved that $(\mathcal{Z}^{n}T)_{n\in\mathbb{N}}$ converges exponentially fast to $\mathcal{M}$ with respect to the $\mathcal{C}^3$-distance (by Corollary~\ref{convergenceMoebius}) and Moebius IETs which satisfy a priori bounds, by $(M3)$ in Lemma~\ref{Mfindim}, belong to a $\mathcal{C}^3$-bounded set (in the sense of Definition~\ref{def:Ckbounded}), this implies that, for any $m$ which is large enough, 
the GIETs $\{ \mathcal{V}^n(\mathcal{R}^m T), \ m\in\mathbb{N}, 0\leq n <n_p\}$ as well as the corresponding Moebius maps  $\{ \mathcal{V}^n(M_m), \ m\in\mathbb{N}, 0\leq n <n_p\}$ are contained in a set, that we will denote $\mathcal{K}$, which is $\mathcal{C}^3$-bounded subset of $\mathcal{X}^3$ (see Lemma~\ref{lemma:equivbounded}). 

\smallskip
\noindent {\it Lipschitz estimate and final arguments.}
By Proposition~\ref{lipschitz}, $\mathcal{Z}$ is Lipschitz on the bounded set $\mathcal{K}$ constructed in the previous step. 
 Let $K$ be the corresponding Lipschitz constant. Since by the previous step we can apply  the Lipschitz property $n_p\leq ||A||$ times to \eqref{Nm2} so that, recalling \eqref{expdec}, we  get 
$$ N_{m+1} \leq \alpha_3 N_m + \left(\alpha_3(K_3+1)+ K^{||A||}(K_3 + 1) \right) \alpha_2^m ,
$$ from which on can derive the existence of $K_4$ and $\alpha_4$ such that 
$N_m \leq K_4 \,\alpha_4^m$ for every $m\in\mathbb{N}$. 

\smallskip
\noindent Since $|N|(\mathcal{Z}^{n}T)$ is decreasing (see Lemma~\ref{prop:Nproperties}, Property $(iii)$) and  
 $(n_{k_m})_{m\in\mathbb{N}}$ grows linearly (recall Properties $(ii)$ and $(iii)$ in the Definition~\ref{def:RDC} of the  $(RDC)$),  Remark~\ref{rk:linearexp} now allows to find constants $K_5>0,0< \alpha_5<1$ to concludes the proof of Proposition~\ref{nLdecay}.
\end{proof}

\noindent We can now prove Corollary~\ref{convergenceaffine}. 
\begin{proof}[Proof of Corollary~\ref{convergenceaffine}]
{\color{black}
In view of Remark~\ref{rk:nonlinearityasd}, $ \mathrm{d}_{\eta}(\mathcal{Z}^{n}T, \mathcal{A}_d) = |N|(\mathcal{Z}^{n}T)$  goes to zero at an exponential rate by  Proposition~\ref{nLdecay}. 
Therefore, since by Corollary~\ref{lemma:distancesrel} there exists a constant $L=L(T)>0$ such that $d_{\mathcal{C}^1}(\mathcal{Z}^{n}T, \mathcal{A}_d)  \leq L\,  \mathrm{d}_{\eta}(\mathcal{Z}^{n}T, \mathcal{A}_d)$,  we deduce  that also  $ d_{\mathcal{C}^1}(\mathcal{Z}^{n}T, \mathcal{A}_d)$  goes to zero at an exponential rate.

Furthemore, we have also shown that $(\mathcal{Z}^{n}T)_{n\in\mathbb{N}}$ converge exponentially to $\mathcal{M}_d$ with respect to $d_{\mathcal{C}^3}$ (see Corollary~\ref{convergenceMoebius}), thus for every $n$, we can find a MIET $M_n$ such $d_{\mathcal{C}^3}( \mathcal{Z}^{n}T, M_n)$ is exponentially small. 
Since $\mathcal{C}^3$-convergence implies in particular $\mathcal{C}^1$-convergence, it follows that $d_{\mathcal{C}^1}( M_m,\mathcal{A}_d)$ is exponentially small. Since $\mathcal{M}_d$ is finite dimensional (see Lemma~\ref{Mfindim}), it follows that $d_{\mathcal{C}^1}$ and $d_{\mathcal{C}^3}$, restricted to $\mathcal{M}_d$ (which contains $\mathcal{A}_d$) are comparable. Thus, 
we conclude that also  $d_{\mathcal{C}^3}( M_m,\mathcal{A}_d)$ is exponentially small and finally 
that  $d_{\mathcal{C}^3}( \mathcal{R}^{n}T, \mathcal{A}_d)$ decay exponentially.
}\end{proof}

\subsection{Convergence to IETs}\label{IETsconvergence}
We now want to get convergence to the set of standard IETs (under the assumption $\mathcal{B}(T) = 0$. The idea behind this last step is the following : the logarithm of the slopes of an AIET transform under Rauzy-Veech induction under the action of Zorich-Kontsevich cocycle. Either this vector belongs to the stable space in which case the logarithm of the slopes converge exponentially fast to zero, which is equivalent to convergence to IETs; or it belongs to the unstable space in which case it grows exponentially fast and iterated renormalizations are unbounded in $\mathcal{C}^1$-norm which in our case is impossible. The only thing we need to show is that we can follow this argument when the GIET we are starting with is not an AIET but only exponentially asymptotic to the set of AIETs.

\vspace{3mm}

\noindent Recall that we have a special sequence $n_{k_m}$ (as $T$ satisfies the $(RDC)$ of Definition~\ref{def:RDC})    for which we know that: 

\begin{enumerate}

\item $ || \omega_{n_{k_m}} || \leq K $ for a uniform constant $K>0$;

\item  since $(n_{k_m})_{m\in\mathbb{N}}$ grows linearly in $m$, the difference $({n_{k_{m+1}}} - n_{k_m})/m$ tends to $0$;

\item {there exists a constant $C_1 >0$  and $\theta > 1$ such that for all $k$ and $i \geq$ we have for all $v \in \Gamma_u^{(n_k)}$
$$ || Q(n_k, n_k + i) \,v ||  \geq C_1 \theta^i ||v||.$$
}
\end{enumerate}
\noindent For any $n \geq 0$, recall that we write 
$$\omega_n =\omega_{n}^s + \omega_{n}^c + \omega_{n}^u \in \Gamma_s^{(n)}  \oplus \Gamma_c^{(n)}\oplus \Gamma_u^{(n)},$$
where $\omega_n^a\in \Gamma_s^{(n)} $ for $a\in\{s,c,u\}$ are the components of $\omega_n$  with respect to the the decomposition  of $\mathbb{R}^d$ of (the extension of) $T$ given by Definition~\ref{def:Oseledetsextension}.  

 Consider the errors $e_n := \omega_{n+1} - Z(n)\omega_n$ (similar to those used in \S~\ref{shadowing}, but here defined using the whole sequence of Zorich renormalization times, and not only the special times given by the $(RDC)$). Let us decompose also those according to the invariant splitting, writing, for each $n\in\mathbb{N}$,
$$e_n : = \omega_{n+1} - Z_n\omega_n= e_{n}^s +e_{n}^c + e_{n}^u \in \Gamma_s^{(n)} \in \oplus \Gamma_c^{(n)}\oplus \Gamma_u^{(n)}.$$ 
By  Proposition \ref{convergenceaffine}, we know that $|N|(\mathcal{Z}^{n}T) \leq K_5 \,\alpha_5^n$ for 
 $\alpha_5 = \alpha_5(T) < 1$. Thus, by Lemma \ref{lemma3}, we get that  $|| e_n || \leq K_5 \alpha_5^n $. Because the angle between  $\Gamma_s^{(n)}$,   $\Gamma_c^{(n)}$ and  $\Gamma_u^{(n)}$ decays at worst subexponentially fast,  
 we can find $K_6$ and $\alpha_6 < 1$ such that 
\begin{equation}\label{componentsexpdecay}
\max\{ ||e_{n}^s||,  ||e_{n}^c||  , || e_{n}^u ||\} \leq K_6 \,\alpha_6^n.\
\end{equation}
\noindent Because the action of the cocycle preserves the decompositions $\mathbb{R}^d = \Gamma_s^{(n)} \oplus \Gamma_u^{(n)}$ we have 
\begin{equation}
\label{e2} 
 \omega_{n+1}^a = Z_n\omega_{n}^a + e_{n}^a,\qquad \text{for\ all}\ a\in \{s,c,u\}. 
\end{equation} We deal with the decay of $\omega_{n}^u$,  $\omega_{n}^c$,  and $\omega_{n}^s$ separately, that of $\omega_{n}^u$ being the most delicate.

{
\begin{lemma}[decay of the unstable component]
\label{convunstable}
There exists $K_7 = K_7(T) > 0$ and $\alpha_7 = \alpha_7(T) <1$ such that 
$$ ||\omega_{n}^u|| \leq K_7 \,\alpha_7^n.$$
\end{lemma}
\begin{proof} One can first observe that, for all $n,k \geq 0$, by the definition of $e_n$ and by \eqref{e2}, we get the telescopic sum identity
$$ \omega_{n+j}^u = Q(n,n+j)\omega_{n}^u + \sum_{i=0}^{j-1}{Q(n+i+1,n+j)e_{n+i}^u},$$ which we re-write, solving for $\omega_{n}^u$ and using cocycle identities (see   \eqref{cocycleid}), as
\begin{equation}
\label{eq2} 
\omega_{n}^u = Q(n,n+j)^{-1}\omega_{n+j}^u - \sum_{i=0}^{j-1}{Q(n,n+i+1)^{-1}e_{n+i}^u}.
\end{equation}

\smallskip
\noindent {\it Step 1: control at special times.} We first show that the sequence $(\omega_{n}^u)_{n\in\mathbb{N}}$ is bounded along the subsequence $(n_{k_m})_{m\in\mathbb{N}}$ given by the $(RDC)$. 
 We recall that:
\begin{enumerate}
\item $ ||e_{n+i}^u|| \leq K_6 \,\alpha_6^{n+i} $, by \eqref{componentsexpdecay};
\item there exist constants $C_1 >0$ and $\theta > 1$ $||Q(n_{k_m},n_{k_m}+i)^{-1} \omega_{n_{k_m+i}}^u|| \geq C_1^{-1} \theta^{-i}$.
\end{enumerate}
\noindent We thus get from \eqref{eq2} at $n = n_{k_m}$ and $j=j(m,m'):= n_{k_{m'}} - n_{k_m}$
$$|| \omega_{n_{k_m}}^u || \leq C_1^{-1} \theta^{-j(m,m')} K+ \sum_{i=0}^{j(m,m')-1}{C_1^{-1} \theta^{-i} \alpha_6^{n_{k_m} + i} } $$  from which get 
$$ || \omega_{n_{k_m}}^u || \leq C_1^{-1} \theta^{-j(m,m')} K + C_1^{-1} \alpha_6^{n_{k_m}} \left(\sum_{i=0}^{j(m,m')-1}{\theta^{-i}\alpha_6^i}\right) \leq C_1^{-1} \theta^{-j(m,m')} K + C_1^{-1}C_2 \alpha_6^{n_{k_m}}, $$
where in the last inequality we used that, since  $\theta > 0$ and $\alpha_6 <1$, the series $\sum_{i=0}^\infty{\theta^{-i}\alpha_6^i}$ converges and denoted by $C_2 = C_2(T)$ its value.  
Since we can take $j(m,m') = n_{k_{m'}} - n_{k_m}$ arbitrarily large by letting $m'$ go to infinity,  we can infer the existence of $C_3 = C_3(T)$ such that 
$$ || \omega_{n_{k_m}}^u || \leq C_3 \alpha_6^{n_{k_m}}.$$ 

\noindent {\it Step 2: interpolation to all times.} We can now estimate all $n\in\mathbb{N}$ by interpolation as follows. Given $n\in\mathbb{N}$, Consider $m$ such that $  n_{k_m} \leq n \leq n_{k_{m+1}}$.  Then, by the linear approximation Lemma~\ref{lemma3},
$$ || \omega_n - Q(n_{k_m},n)  \omega_{n_{k_m}} || \leq |N|(T^{(n_{k_m})}) \, || Q(n_{k_m},n) || $$  and hence, by invariance of the splitting,  
\begin{equation}\label{toestimate}
 || \omega_n^u || \leq || Q(n_{k_m},n) ||\left(  |N|(T^{(n_{k_m})}) +|| \omega_{n_{k_m}}^u ||\right).
\end{equation} 
To estimate \eqref{toestimate} and conclude, we now use that: 
\begin{enumerate}
\item  by Step 1, we  have that $ || \omega_{n_{k_m}}^u || \leq C_3 \alpha_6^{n_{k_m}}$;
\item by the $(RDC)$, $ || Q(n_{k_m},n) || \leq Q(n_{k_m},n_{k_{m+1}}) || = o(e^{\epsilon m})$ for all $\epsilon > 0$;
\item By Corollary~\ref{nLdecay} along the subsequence $(n_{k_m})_m$,  we know that $|N|(T^{(n_{k_m})}) \leq K_4 \alpha_4^{n_{k_m}}$.
\end{enumerate}
{\color{black}Since $n_{k_m}$ grows linearly with $m$, we have that $m\leq C n_{k_m}\leq C n$ for some $C>0$ and also that $( n_{k_{m+1}} - n_{k_m})/m$ tends to $0$. Thus, for $m$ sufficiently large, for $i=4$ and $i=5$, since $n\leq n_{k_{m+1}}$, we can estimate
 $\alpha_i^{n_{k_m}} \leq \alpha_i^{n_{k_{m+1}}} \alpha_i^{-\epsilon m} \leq \alpha_i^{n}(\alpha_i^{-1})^{\epsilon m} $.  
This final remarks together with $(1)-(3)$ above allow to deduce the claimed exponential decay of \eqref{toestimate}.
} \end{proof}

\smallskip
\noindent We now turn to showing that also $(\omega_{n}^s)_{n\in\mathbb{N}}$ decays at an exponential rate.
{
\begin{lemma}[decay of the stable component]
\label{convstable}
There exists $K_8 = K_8(T) > 0$ and $\alpha_8 = \alpha_8(T) <1$ such that 
$$ ||\omega_{n}^s || \leq K_8 \,\alpha_8^n.$$
\end{lemma}
\begin{proof}
Again, because $T$ is assumed to satisfies the $(RDC)$ we know that there exists a constant $C_2 = C_2(T) > 0$ and a constant $\theta = \theta(T) < 1$ such that for all $k \in \mathbb{N}$ and $m \in \mathbb{N}$ and all $v \in \Gamma_s(T^{(n_{k_m})})$
$$ ||Q(n_{k_m},n_{k_m} + j)v|| \leq C_2 \theta^j ||v||.$$
\noindent We also have  for all $n \geq 0$ 
$$ \omega_{n}^s = Q(0,n)\omega_{0}^s + \sum_{i=0}^{n-1}{Q(i,n)e_{i}^s}.$$  We now group terms by \emph{paquets} of terms between $n_{k_m}$ and $n_{k_{m+1}}.$ Let $m_n$ be such that    $n_{k_{m_n}} \leq n \leq n_{k_{m_n+1}}$. We thus have 
$$ \omega_{n}^s = Q(0,n)\omega_{0}^s + \sum_{j=0}^{m_n}{\sum_{i=0}^{n_{k_{j+1}} - n_{k_{j}}}{Q(n_{k_j} + i,n_{k_{j+1}}) Q(n_{k_{j+1}},n) e_{n_{k_j} + i}^s}}$$.  We recall that 
\begin{enumerate}
\item $|| Q(n_{k_j} + i,n_{k_{j+1}})  || \leq || Q(n_{k_j},n_{k_{j+1}}) || = o(e^{\epsilon\, j})$ for any $\epsilon>0$;
\item  $|le_{n_{k_j} + i}^s|| \leq K_6 \alpha_6^{n_{k_j} + i }$;
\item $ ||Q(n_{k_m},n_{k_m} + j)v|| \leq C_2 \theta^j ||v||.$
\end{enumerate}
Putting all this together we get 
$$ || Q(n_{k_j} + i,n_{k_{j+1}}) Q(n_{k_{j+1}},n) e_{n_{k_j} + i}^s||  \leq C(\epsilon) e^{\epsilon j} C_2 \theta^{n - n_{k_{j+1}}} K_6 \alpha_6^{n_{k_j} + i}$$ for some constant $C_{\epsilon}> 0$ depending on $\epsilon$. Since $n_{k_j}$ grows linearly with $j$ we get the existence of a constant $C_4 >0$ such that  
$$ || Q(n_{k_j} + i,n_{k_{j+1}}) Q(n_{k_{j+1}},n) e_{n_{k_j} + i}^s||  \leq C_4 \theta'^n$$ for some $\theta' < \max(\alpha_6, \theta)$. From this we obtain 
$$ || \omega_{n}^s || \leq \max(C_4 ||\omega_0|| ,C_2)  (\theta')^n $$ and we get the result for any constant $\alpha_8$ such that $\theta' <\alpha_8 < 1$.
\end{proof}
}
\noindent Finally, we  show  now how to control the central part. Here the boundary assumption is crucial.
{
\begin{lemma}[decay of the central component]
\label{convcentral}
There exists $K_9 = K_9(T) > 0$ and $\alpha_9 = \alpha_9(T) <1$ such that 
$$|| \vect{\omega}_{n}^c || \leq K_9 \,\alpha_9^n.$$
\end{lemma}
\begin{proof}
Recall that $\mathcal{B}(T) = \mathcal{B}^{(n)} = 0$ for all $n \geq 0$. But we also have $\mathcal{B}^{(n)} = B(\vect{\omega}_n^u) + B(\vect{\omega}_n^c) + B(\vect{\omega}_n^s) + B(\log \varphi^n)$ where $\varphi^n \in \mathcal{P}$ is the profile of $T^{(n)}$. By Proposition \ref{convergenceaffine}, we have $|| B(\log \varphi^n) || \leq C_5 || \log \mathrm{D}(\varphi^n)|| \leq C_5 K_4 \alpha_4^n $ where $C_5 = || B ||$. By Propositions  \ref{convunstable} and \ref{convstable} we finally get that 
$$ || B(\vect{\omega}_n^c) || \leq C_5 K_4 \, \alpha_4^n + K_7\,  \alpha_7^n + K_8\,  \alpha_8^n.$$
\noindent By Lemma \ref{lemma7}, $ ||\vect{\omega}_n^c||  \leq D_c\, \angle(\Gamma_c, \Gamma_u^n \oplus \Gamma_s^n)  || B(\vect{\omega}_n^c) ||$, which leads to 
$$ || \vect{\omega}_n^c) || \leq D_c \, \angle(\Gamma_c, \Gamma_u^n \oplus \Gamma_s^n) C_5\, K_4 \alpha_4^n + K_7 \, \alpha_7^n + K_8\,  \alpha_8^n.$$ But since $T$ satisfies $(RDC)$, we know that $\angle(\Gamma_c, \Gamma_u^n \oplus \Gamma_s^n) $ decreases subexponentially fast, from which we get the conclusion of Proposition \ref{convcentral}.
\end{proof}
}
\noindent Combining the Lemmas which give individual control of the components of  $(\vect{\omega}_n)_{n\in \mathbb{N}}$, we can now prove exponential convergence to IETs:
\begin{proof}[Proof of Theorem~\ref{convergencethm}]
 The control on the unstable, stable and central components of $\vect{\omega}_n$ respectively given by Lemmas \ref{convunstable}, \ref{convstable} and \ref{convcentral},  imply that $||\vect{\omega}_n||$  converges to zero as $n$ grows at an exponential rate.  {\color{black}We already proved that $ \mathrm{d}_{\eta}(\mathcal{R}^{n}T, \mathcal{A}_d) $  goes to zero at an exponential rate (see Corollary~\ref{convergenceaffine}) and that therefore, by Corollary~\ref{lemma:distancesrel}, also $d_{\mathcal{C}^1}(\mathcal{R}^{n}T, \mathcal{A}_d) $ does. 
 Since if $A\in \mathcal{A}_d$ is such that $\omega(A)=0$, then $\rho(A)=1$ and therefore $A\in \mathcal{I}_d$, we conclude that $ \mathrm{d}_{\mathcal{C}^1}(\mathcal{R}^{n}T, \mathcal{I}_d) $   goes to zero at an exponential rate. }
\end{proof}

\section{Rigidity for GIETs}\label{sec:rigidity}
\noindent In this section, we prove our main rigidity result for GIETs.
\begin{thm}[GIETs rigidity in genus $2$]
\label{maintheorem}
Let $T \in \mathcal{X}_d^3$ be an irrational GIET with $d=4$ or $d=5$ continuity intervals and with zero boundary $\mathcal{B}(T) = 0$. For a full measure set of rotation numbers, 
if $T$ is $\mathcal{C}^0$-conjugate to a standard IET $T_0$, then the conjugacy is in fact a diffeomorphism of class $\mathcal{C}^1$.
\end{thm}
Remark that the case of GIET with $d=4$ or $d=5$ and $\pi $ minimal correspond to genus two, i.e.~any \emph{minimal} flow on a (compact, orientable) surface with genus two has as a Poincar{\'e} section (for a suitable chosen transverse arc) which is given by such a IET. This result will hence imply our foliation rigidity result, which is essentially only a reformulation in geometric language of the $d=5$ case, see Section~\ref{foliations}. 

In this section will actually prove a more general result (see Proposition~\ref{reduction} below) which is valid for minimal IET with any $d\geq 2$ and yields a partial result also for IETs with any $d> 5$ (i.e.~a rigidity statement \emph{conditional} to an assumption on the position of the shadow in the Oseledets filtration). This technical condition is \emph{automatically} satisfied when $d=4,5$. 

\medskip
\noindent {\it Strategy outline}: Recall that the main result of Section~\ref{sec:affineshadowing}, namely Theorem \ref{shadowing}, tells us that if $T$ satisfies the $(RDC)$, two scenarios can occur : either the orbit of $T$ under renormalization is \emph{recurrent} to a certain bounded set, or the orbit of $T$ is somewhat shadowed, in the first order of approximation, by an affine IET. The proof then splits into two steps.

\smallskip
\noindent {\it Step 1.} The first step, in \S~\ref{sec:conjugacy}, is to show that in the recurrent case, applying the results of Section \ref{sec:convergence} about convergence of renormalization, $T$ is indeed $\mathcal{C}^1$-conjugate to $T_0$. This step is simply based by interpolation and Gottschalk-Hedlund theorem and is by now quite standard also for GIETs.

\smallskip
\noindent {\it Step 2.} In \ref{sec:wandering}, the second step is then to show that in the divergent case (where there is an \emph{affine shadow}), the map $T$ must have a wandering interval, and therefore was not $\mathcal{C}^0$-conjugate to $T_0$ in the first place. We first of all show that the conclusion of shadowing allows us to compare the Rohlin towers for $T$ to that of an AIET; we then exploit a result Marmi-Moussa-Yoccoz \cite{MMY2} giving the existence of wandering intervals for AIETs to conclude.

\smallskip
\noindent This two points together imply Theorem \ref{maintheorem}, as summarized in \S~\ref{sec:proofGIETrigidity}. The only place where the genus $2$ (i.e.~$d=4,5$) assumption is needed is in Step $2$, in the use of Marmi-Moussa-Yoccoz result \cite{MMY2} for AIETs.

\subsection{Regularity of the conjugacy}\label{sec:conjugacy}
In this section we show that convergence of renormalization in the $\mathcal{C}^1$-norm implies $\mathcal{C}^1$-conjugacy to the linear model. We prove the following 
\begin{proposition}[exponential convergence gives a.s.~$\mathcal{C}^1$-conjugacy]
\label{conjugacy}
Let $T$ be a $\mathcal{C}^1$-GIET of $d$ intervals satisfying $(RDC)$ and assume that $\{ \mathcal{Z}^n(T)\}_{n\in\mathbb{N}}$ converges exponentially fast to the set of $IET$s with respect to $\mathcal{C}^1$-distance, namely there exists $K_1>0$ and {\color{black} $0<\alpha_1<1$} such that
$$ d_{\mathcal{C}^1}(\mathcal{Z}^{n}(T), \mathcal{I}_d ) \leq K_1 \, \alpha_1^n.$$
 Then $T$ is $\mathcal{C}^1$-conjugate to an IET.
\end{proposition}

 First we show how a statement on the Birkhoff sums of $\log \mathrm{D}T$ implies Proposition \ref{conjugacy}. This is a classical result for diffeomorphisms of the circle and also, by now, for GIETs in view of the work by Marmi, Moussa and Yoccoz (see \cite{MMY, MMY3, Yoc:Clay}). 
\begin{lemma}
\label{Birkhoff}
Let $T$ be a GIET of class $\mathcal{C}^1$ with irrational rotation number. Assume that there exists $K > 0$ such that for all $x \in [0,1]$ and for all $n \in \mathbb{N}$ 
$$ \big| S_n {\log \mathrm{D}T (x)} \big| = \left| \sum_{i=0}^{n-1}{\log \mathrm{D}T( T^i(x)) } \right| \leq K.$$ Then there exists an IET $T_0$ such that $T$ is conjugate to $T_0$ via a $\mathcal{C}^1$ diffeomorphism of $[0,1]$.
\end{lemma}
\begin{proof} The proof follows from an application of Gottschalk-Hedlund theorem. The map $T$ is not a homeomorphism, but by following the arguments by Marmi-Moussa-Yoccoz (see for example \cite{MMY}, Corollary 3.6), one can extend $T$ to a homeomorphism of a Cantor space and therefore apply Gottschalk-Hedlund theorem, which gives that there exists a continuous function $\varphi$ which solves the cohomological equation  
$$ \varphi \circ T - \varphi = \log \mathrm{D}T.$$ 
We deduce from this cohomological equation that the measure $ m := e^{\varphi} \mathrm{Leb} $ is invariant under the action of $T$. Up to normalising $m$ so it has total mass $1$, we get that  $$\psi := x \mapsto \int_0^x{e^{\varphi(t)} dt} $$ conjugates $T$ to a GIET which preserve the Lebesgue measure, in other words a standard IET. The map $\psi$ being of class $\mathcal{C}^1$ with $\mathrm{D}\psi(x) = e^{\varphi(x)}$, the lemma is proven.
\end{proof}

{
Now that we have this Lemma, the proof is reduced to showing that, assuming convergence of renormalizations, Birkhoff sums are uniformly bounded, i.e.~the assumptions of Lemma~\ref{Birkhoff} hold. Let us first isolate in a Lemma the relation between convergence of renormalization and convergence of special Birkhoff sums of $f:=\log D T$.

\begin{lemma}[Special Birkhoff sums of $\log D T $ via renormaliation]\label{SBSviaR}{\color{black}
Let $T$ be an infinitely renormalizable GIET and let $f:=\log DT$. If $d_{\mathcal{C}^1} (\mathcal{Z}^k(T), \mathcal{I}_d)$ converges to zero exponentially, then there exists $K>0, \alpha<1$ such that $$\| f^{(k)}\|_{\infty}\leq K \alpha^k, $$
i.e.~the sup-norm $\| f^{(k)}\|_{\infty}$ of the special Birkhoff sums $f^{(k)}$ on their domain $I^{(k)}$ also converges to zero exponentially.}
\end{lemma}
\noindent The Lemma shows in particular that exponential convergence of renormalization to the space of IETs $\mathcal{I}_d$ gives exponential decay of the sup norm of special Birkhoff sums of $f=\log DT$.
\begin{proof} 
For every $k\in\mathbb{N}$, the $k^{th}$ image by renormalization $\mathcal{Z}^k(T)$ and the induced map $T_k$ are conjugated by an affine map (see \eqref{eq:renormalizedmap}), 
\begin{equation}\label{supBS} 
\sup_{x\in [0,1]} D\mathcal{Z}^kT(x)=\sup_{x\in I^{(k)} } DT_k (x), \qquad \sup_{x\in [0,1]} D(\mathcal{Z}^kT )^{-1}(x)=\sup_{x\in I^{(k)} } D(T_k^{-1}) (x).
\end{equation}
Furthermore, since $f=\log DT$, if we consider a point $x\in I^{(k)}_j$, taking logarithms and applying the chain rule,
\begin{equation}\label{logBS}
\log D T_k (x) = \log D\big( T^{q^{(k)}_j}\big) (x) = S_{q^{(k)}_j }(\log DT )(x)=
S_{q^{(k)}_j }f(x) = f^{(k)}(x), \qquad \text{for\ all}\ x\in  I^{(k)}_j.
\end{equation}
{\color{black}These two equations show that $||f^{(k)}||_{\infty}$ is controlled by  $||\log D\mathcal{R}^k(T) ||_{\infty}$,  which in turn is controlled by $ d^{\pm}_{\mathcal{C}^1} (D\mathcal{R}^k(T), \mathcal{I}_d)$ (see Lemma~\ref{lemma:equivbounded}). Since 
 the assumption that  $d_{\mathcal{C}^1} (\mathcal{Z}^k(T), \mathcal{I}_d)$ converges to zero exponentially implies, by Remark~\ref{comparisontozero}, that the same type of convergence also with respect to  $ d^{\pm}_{\mathcal{C}^1} $, this concludes the proof.}
\end{proof}
\noindent We can now proceed with the proof of rigidity, i.e.~of Proposition \ref{conjugacy}.
\begin{proof}[Proof of Proposition \ref{conjugacy}] In order to verify the assumption of Lemma~\ref{Birkhoff}, 
consider $x \in [0,1]$ and arbitrary $n \in \mathbb{N}$ and let us estimate the Birkhoff sums  $S_nf$ for $f:=\log DT$. 
By the geometric decomposition of Birkhoff sums described in \S~\ref{decompBS}, if $k_n$ is defined to be the largest $k$ such that the orbit $\{x,\dots, T^{n-1}x\}$ visits $I^{(k_n)}$ at least twice, then we have 
\begin{equation}
\label{decomplogD}
|S_{n} \log DT (x)|=|S_{n} f (x)|\leq 2 \sum_{k=0}^{k_m}  \Vert Z_k \Vert \, \Vert f^{(k)} \Vert, \qquad \text{for any}\ x\in [0,1].
\end{equation}
Since by assumption $\mathcal{R}^k(T)$ converges exponentially fast to the space of IETs with respect to the $\mathcal{C}^1$-distance,  
by Lemma~\ref{SBSviaR}, $\Vert f^{(k)}\Vert \leq K \alpha^k$ for some $K>0$ and $\alpha<1$. 
Thus,  using this estimate in the decomposition~\eqref{decomplogD} and recalling that by the $(RDC)$ (see in particular Condition $(C)$ in Definition~\ref{def:RDC} that implies that $\Vert Z_k\Vert $ also grows subexponentially) we have that,  for a chosen  $\epsilon>0$ such that $e^{\epsilon}\alpha<\alpha_2<1$, there exists $K_1'$ such that 
$$
\Vert S_{n} \log DT \Vert_{\infty} \leq K_1' \sum_{k=0}^{k_m} e^{\epsilon k} \alpha_1^k  < K:= K_1' \sum_{k=0}^{\infty}\alpha_2^k<\infty, \qquad \text{for\ all}\ n\in\mathbb{N}.
$$
Thus, we can apply  Lemma~\ref{Birkhoff} to conclude that $T$ is $C^1$-conjugate to an IET $T_0$. 
\end{proof}
}

{
\subsection{Wandering intervals and distorted towers}\label{sec:wandering}
{In this section we state the main result (namely Proposition~\ref{reduction} below) that we will use to prove the existence of wandering intervals in Case~2 of Theorem~\ref{shadowing}. We recall that in this case  the sequence $\{\vect{\omega}_n(T), n\in\mathbb{N}\}$ of shape log-slope vectors  of the orbit under renormalization of the GIET $T$   is  shadowed by the orbit of the the log-slope vector $\vect{v}:=\omega(T_0)$ of an AIET $T_0$ with  $v $ in the unstable space. We show in this case that  the presence of wandering intervals for $T$ can be reduced to the existence of wandering intervals for $T_0$. We then  exploit the result by Marmi-Moussa-Yoccoz \cite{MMY2} 
that shows that, if $\vect{v}$ has a non-zero projection on the second, positive, Lyapunov exponent, then one can show the existence of wandering intervals. This allows us to conclude that in genus two (i.e~for irreducible IETs with $d=4$ or $d=5$ intervals),  
 where every log-slope vector as above automatically projects second Lyapunov exponent, one can conclude that Case 2 cannot occur if we assume that $T$ is topologically conjugated to its linear model $T_0$. Notice in particular that to extend the rigidity result in Theorem~\ref{maintheorem} to any genus is therefore reduced, by the results in this paper, to extending the work of Marmi-Moussa-Yoccoz \cite{MMY2} to treat $\vect{v}$ in the Oseledets eigenspace of the other non-zero Lypaunov exponents.
}

\subsubsection{Distorted towers}\label{sec:distortedtowers} Let $\mathcal{P}_n$, $n\in\mathbb{N}$, be the sequence of dynamical partitions defined in \S~\ref{dynamicalpartitions} and let  $\mathcal{P}^j_n$ for $1\leq j\leq d$ be the corresponding Rohlin towers (refer to \S~\ref{dynamicalpartitions} for definitions). Recall that each $\mathcal{P}^j_n$ is disjoint union of the $q_j^{(n)}$ intervals $T^k(I^{(n)}_j)$, for $0\leq k< q_j^{(n)}$.
 
Let us recall that $J\subset [0,1]$ is a \emph{wandering interval} for $T$ if its images, i.e.~the elements of the orbit $\{T^i(J), i\in\mathbb{N}\}$, are all disjoint. In this case, one has in particular $\sum_{i=-\infty}^{\infty}|T^i(J)|<1$ (where we recall $|T^i(J)|$ denotes the Lebesgue measure). Notice also that, since $T$ is continuous on $T^i(J)$ for every $i\in\mathbb{N}$,  for every $n\in\mathbb{N}$, $J$ (as well as any of its images), should be fully contained in a floor of a Rohlin tower.  
The presence of a wandering interval then forces the dynamical towers a very degenerate geometry, that we now describe introducing the notion of \emph{distorted towers}. 

\begin{definition}[distorted towers]\label{def:distorted}
We say that $T$ admits a \emph{sequence of distorted towers} if there exists a constant $C>0$ and infinitely many $n\in\mathbb{N}$ such that
\begin{equation}\label{eq:fattowers}\mathrm{Leb}(\mathcal{P}^j_n) \leq  C \max_{0\leq k< q^{(n)}_j} \left|T^k(I^{(n)}_j)\right| = C \max\,\{\, \mathrm{Leb}(T^k(I^{(n)}_j) , \quad 0\leq k< q^{(n)}_j\, \} , \qquad \text{for\ all}\ 1\leq j\leq d.\end{equation}
\end{definition}
\noindent Thus, if the towers are distorted, the \emph{size} of \emph{each} tower is comparable to the size of its largest floor. 

\smallskip
Let us recall that if $T$ is minimal, the sequence of dynamical partitions has to converge to the trivial partition into points, i.e.~the \emph{mesh} of the partitions $\mathcal{P}_n$, denoted by $\mesh (\mathcal{P}_n)$  and defined as the maximum length of intervals in $\mathcal{P}_n$, has to go to zero as $n$ grows. The existence of distorted towers is therefore incompatible with minimality and can be used to prove the existence of wandering intervals, through the following Lemma:

\begin{lemma}[sufficient condition for wandering intervals]\label{wilemma}
If $T$ admits a sequence of distorted towers, then $T$ has a wandering interval. 
\end{lemma}
\begin{proof}
Let us recall that minimality of a GIET $T$ is equivalent to the non-existence of wandering intervals. Thus it is sufficient to show that if $T$ admits distorted towers, it cannot be minimal. Let $(n_\ell)_{\ell\in\mathbb{N}}$ be an increasing sequence of $n$ for which \eqref{eq:fattowers} holds. Since $\mathcal{P}_{n_\ell}$ is a partition of $[0,1]$, $Leb (\mathcal{P}_{n_\ell})=1 $. Thus, since $\mathcal{P}_{n_\ell}=\cup_{j=1}^d \mathcal{P}^j_{n_\ell}$, for every $\ell\in\mathbb{N}$, at least one tower should be \emph{large}, i.e.~there exists $j(\ell) $ such that $Leb (\mathcal{P}_{n_\ell}^{j_\ell})\geq 1/d$. By   the distortedassumption \eqref{eq:fattowers}, this implies that
$$
\mesh \left({\mathcal{P}_{n_\ell}}\right): = \max_{1\leq j\leq d} \max_{0\leq k<q^{(n_\ell)}_j}  \left|T^k(I^{(n_\ell)}_j)\right| \geq  \max_{\ 0\leq k<{q^{(n_\ell)}_{j_\ell}}} |T^k(I^{(n_\ell)}_{j_\ell})|  \geq \frac{Leb (\mathcal{P}_{n_\ell}^{j_\ell})}{C} \geq \frac{1}{dC}>0
$$
for every $\ell\in\mathbb{N}$. This shows that $\mesh \left({\mathcal{P}_{n_\ell}}\right)$ does not go to zero as $\ell$ grows and hence contradicts minimality and proves the Lemma.
\end{proof}

\begin{remark} One can show furthermore, following arguments analogous to those used at the end of the paper \cite{MMY2} by Marmi, Moussa and Yoccoz for affine interval exchange transformations, that if $T$ has a sequence of distorted towers, the complement of  the union of the orbits of the wandering intervals has zero Lebesgue measure. 
\end{remark}

\subsubsection{Exponentially distorted towers}
For AIETs with wandering intervals and, as we will show, also for GIETs which are shadowed by them, one can prove the existence of distorted towers by proving quantitative estimates on the size of the towers floors and by showing that in each tower they achieve a maximum and then decrease with a stretched exponential rate. Therefore, let us give the following definition:

\begin{definition}[exponentially distorted towers]\label{def:expdistorted}
We say that $T$ has a sequence of \emph{exponentially distorted towers} if  for some constants $C>0, c>0$ and $\gamma>0$ such that for infinitely many $n\in \mathbb{N}$ in each Rohlin tower   $(\mathcal{P})^j_n$, $1\leq j\leq d$, there is a floor $F_{0}=F_0(j)$ of the form $F_0=T^{k_0}(I^{(n)}_j)$, where $k_0=k_0(j)$ is an integer with  $0\leq k_0<q^{(n)}_j$,   such that for every $x_0\in F_0$, 
$$
|T^i F_0 |=|T^{k_0+i} I^{(n)}_j | \leq C \exp ({-c|i|^{\gamma}})\, |F_0 |, \qquad \mathrm{for \ every}\ -k_0\leq i<q^{(n)}_j - k_0 . 
$$
\end{definition}
\begin{remark}\label{rk:expdist} If $T$ has a sequence of exponentially distorted towers, in particular it has a sequence of distorted towers (in the sense of Definition~\ref{def:distorted} above), because  $\sum_{i=-\infty}^\infty \exp({-c|i|^\gamma})$ is convergent, so  $|\mathcal{P}^j_n|$ and the size $|F_0|$ of the corresponding floor $F_0=F_0(j)$ are comparable for every $1\leq j\leq d$ and every $n$ with exponentially distorted towers. 
\end{remark}

\subsubsection{Reduction to the affine shadow.}\label{sec:reductionstatement}
The main result that we prove in this section is the following.
\begin{proposition}[Reduction to affine distorted towers]\label{reduction}
Let $T$ be an irrational GIET with a rotation number $\gamma$ that satisfies the $(RDC)$. Assume that we are in the affine shadowing (Case $2$) of Theorem~\ref{shadowing} let $\vect{v}$ be the shadow of $T$. Then, if an AIET with rotation number $\gamma$ and log-slope vector $\vect{v}$ has exponentially distorted towers, then $T$ also has exponentially distorted towers and has a wandering interval.   
\end{proposition}
\noindent Thus, Proposition~\ref{reduction} shows that one can reduce the proof of existence of wandering intervals (which follow from the existence of distorted towers by Lemma~\ref{wilemma}) to the study of AIETs. The proof of this Proposition will take all of \S~\ref{sec:reductionproof}.  The work by \cite{MMY2} by Marmi, Moussa and Yoccoz in \S~\ref{sec:AIETwi} shows that  exponential distorsion of towers holds indeed for many AIETs. Their results together with this Proposition will then be used in the proof of the rigidity result for GIETs with $d=4,5$ and boundary zero.  

\subsection{Towers distorsion via Birkhoff sums}
In this section we will control distorsion of towers via Birkhoff sums and then prove Proposition~\ref{reduction}. 
We fist explain how the size of Rohlin towers floors is related to Birkhoff sums, see \S~\ref{sec:partitionBS}. We then recall, in \S~\ref{sec:AIETwi}, the results by Marmi, Moussa and Yoccoz in \cite{MMY2}.  The proof  of Proposition~\ref{reduction} is given in \S~\ref{sec:reductionproof}.

\subsubsection{Partition size estimates via Birkhoff sums.}\label{sec:partitionBS}
The following simple Lemma, based on the distorsion bounds in Lemma~\ref{bound1},  show how control on the size of the floors of dynamical partitions for $T$ can be obtained by estimating \emph{Birkhoff sums} for the function $f:=\log D T$. 
\begin{lemma}[reduction to Birkhoff sums]\label{towersviaBS}
Given an infinitely renormalizable $T$, there exists a constant $C_T>1$ such that, for each $n\in\mathbb{N}$ and $1\leq j\leq d$, for any two floors $F_1, F_2$ of the Rohlin tower $\mathcal{P}^{(n)}_j$ of the form $F_1=T^{k_1}(I^{(n)}_j)$ and $F_2=T^{k_1}(I^{(n)}_j)$ for some $0\leq k_1<k_2<q^{(n)}_j$, for any point $x\in F_1$ we have 
$$
\frac{1}{C_T}  \exp({S_{k_2-k_1} \log D T (x)})\leq  \frac{|F_1|}{|F_2|}\leq D_T \exp({S_{k_2-k_1} \log C T (x)}).
$$
\end{lemma}
\begin{proof} By definition, $F_2=T^{k_2-k_1}(F_1)$. For short, let $k:=k_2-k_1$ and write $F_2=T^k F_1$. By mean value theorem, there exists $\overline{x}\in F_1$ such that $|F_2|= |F_1| D(T^{k})(\overline{x})$. Thus, by the classical distorsion bound in Lemma~\ref{bound1}, we have that for any other $x\in F_1$, $D(T^{k})({x})/C_T\leq  |F_1|/ |F_2| \leq C_T D(T^{k})({x})$, where $C_T:=\int |\eta_T| dx$. Thus the result follows from the chain rule that relates $\log (D(T^k)(x))$ with $S_k \log DT (x)$.
\end{proof}


\subsubsection{Wandering intervals in affine IETs.}\label{sec:AIETwi}
We now recall the estimates proved by Marmi, Moussa and Yoccoz in \cite{MMY2} to show the existence of wandering intervals for  AIETS and  that will be also the starting point for our proof of existence of wandering intervals for GIETs.
The type of estimates that they prove give \emph{stretched exponential decay} of the size of floors in each Rohlin towers, which in particular implies that towers are exponentially distorted.

\smallskip
Let $T_0$ be a standard IET with irrational rotation number $\gamma$ which is Oseledets generic. Let 
\begin{equation}\label{positiveOseledets} \theta_1\geq \theta_2\geq \dots \geq \theta_g\geq  0, \qquad \mathbb{R}^d=E_1(T_0)\supset E_2(T_0)\supset \cdots \supset  E_g(T_0)\supset E_{g+1}(T_0):=E^{cs}(T_0)
\end{equation} 
be the $g$ positive Lyapunov  exponents and the corresponding Oseledets filtration for $T_0$, which we completed with the central stable space $E^{cs}(T_0)$ which corresponds to zero and negative exponents,  so that if 
$$v\in E_{i}(T_0)\backslash E_{i-1}(T_0),  \qquad  \lim_{n\to \infty}\log \Vert v^{(n)}\Vert /\log n=\theta_i, \quad \text{where}\ v^{(n)}:= Z^{(n)} v. $$
Given a vector $v\in\mathbb{R}^d$, we can identify it as usual with a piecewise constant function in $\mathcal{C}(T_0)$, that we will denote by $v_0(x)$. We will denote by $S^0_n v_0(x)$ the Birkhoff sums of the function $v_0$ over $T_0$ (see \S~\ref{sec:BS}) where we added the apex $0$ to recall that the Birkhoff sums are with respect to $T_0$. Similarly, let $(\mathcal{P}^0)_n$ with the apex $0$, for $n\in\mathbb{N}$, be the dynamical partitions for $T_0$. The following estimates are proved\footnote{Proposition~\ref{AIETwi} is not explicitely stated in this form in \cite{MMY2} but can be deduced from the results in the paper, in particular from the estimates in \S~3.7. The floor $F_0$ in Proposition~\ref{AIETwi}  in the Rohlin tower over $I^{(n)}_\alpha$  (in \cite{MMY2} indexing of intervals is by letters $\alpha\in\mathcal{A}$ of an alphabet of cardinality $d$) which they call $I^{(\textrm{max})}_\alpha(n)$. Estimates of the form  \eqref{eq:powerdecay} are explicited stated only for a point $x^{\star}$ in any non-empty intersection of intervals $I^{(\textrm{max})}_{\alpha_n}(n)$ (as stated in Proposition {\color{black} add ref}) but from the arguments in the proof it is clear that they hold for any point in any $I^{(\textrm{max})}_{\alpha_n}(n)$. 
 The interested reader may notice also that the estimates in \S~3.7 are stated for a specific vector $v$ (chosen to generate the $1$-dimensional space associated to the second positive Lyapunov exponent $\theta_2$ in the Oseledets \emph{splitting}, which is determined once a past is given). As the authors remark at the beginning of section \S~3.7.1 though, the same estimates also hold for any other vector $v\in E_2(T_0)$ in virtue of Zorich's estimates on deviations of Birkhoff averages in \cite{Zo:dev}.}
 in \cite{MMY2}.

\begin{proposition}[Marmi, Moussa, Yoccoz, \cite{MMY2}]\label{AIETwi}
For almost every Oseledets generic IET $T_0$ and any $v_0\in E_2(T_0)\backslash E_3(T_0)$, there exists $C_0>0$ and $0<\gamma_0<1$ such that, for every $n\in\mathbb{N}$ and $1\leq j\leq d$, there is a floor $F_0$  of the Rohlin tower $(\mathcal{P}^0)^j_n$ with the property that for every $x_0\in F_0$, the Birkhoff sums $S^0_n v_0(x_0)$ of the function $v_0(x)$ over $T_0$ satisfy
\begin{equation}\label{eq:powerdecay}
S^0_i v_0(x_0) \leq C_0 - |i|^{\gamma_0}, \qquad \mathrm{for \ every}\ i\in\mathbb{N}. 
\end{equation}
\end{proposition}
Notice that the estimate \eqref{eq:powerdecay} implies in particular (by Lemma~\ref{towersviaBS}) that the dynamical towers of $T$ are exponentially distorted (in the sense of Definition~\ref{def:expdistorted}).
We remark that the assumption $v_0\in E_2(T_0)\backslash E_3(T_0)$ plays an important role in their result: while conjecturally, an analogous result should hold for any $v_0\in E_2(T_0)$, the proof in 
\cite{MMY2} uses this assumption crucially\footnote{In the proof, in order to control Birkhoff sums, the author introduce and exploit an object called \emph{limit shapes}, which is used to describe fluctuations of Birkhoff sums. Limit shapes of a full measure set of IETs are in turn controlled exploiting  returns to a set $\mathcal{Y}_\delta$ which gives quantitative control on the location of local maxima of Birkhoff sums at various scales.  The assumption that $v_0\in E_2(T_0)\backslash E_3(T_0)$ plays an important role in the proof that $\mathcal{Y}_\delta$ has positive measure, since it provides an  explicit, smooth dependence of the limit shape on $\lambda$. This dependence is not explicit in the case of other Lypaunov exponents other than $\theta_2$, which makes the generalization not straighforward.}. In the  case of $d=4,5$, though, the  assumption $v_0\in E_2(T_0)\backslash E_3(T_0)$ 
 is automatically satisfied since there are no other positive exponents (see the proof of Theorem~\ref{maintheorem} in \S~\ref{sec:proofGIETrigidity} for details).  In the special case of rotational GIETs, this also provides a generalization to almost every GIET of a result by Cunha and Smania for bounded type GIETs (see \cite{CS:ren}). 

\subsubsection{Proof of Proposition~\ref{reduction}}\label{sec:reductionproof}
Throughout this section we assume that  $T$ is an irrational GIET which satisfies the $(RDC)$ and denote by $(n_k)_{k\in\mathbb{N}}$ the sequence of renormalization times given by the $(RDC)$ (see Definition~\ref{def:RDC}). We assume furthermore that we are in case $2$ of Theorem~\ref{shadow}, so that one can define a shadow $\vect{v}$ for $T$. Recall that to the vector ${\bf v}$ we can associate a piecewise constant function $v(x)\in\mathcal{C}(T)$ given by $v(x)=v_j$ for $x\in I^{t}_j$. 
To simplify the indexing, in analogy with the notation $\widetilde{Q}(k,k'):=Q(n_k, n_{k'})$ 
 already introduced in \S~\ref{sec:acceleration}, we will use the notation
$$\widetilde{w}_k=w_{n_k}, \quad 
\widetilde{v}^{(k)}:= v^{(n_k)}, \quad \widetilde{v}^{(k)} := v^{(n_k)}, \quad  \widetilde{v}^{(k)} :=f^{(n_k)},
$$
to denote respectively the vectors $w_n$ and $v^{(n)}$ and  the special Birkhoff sums of the functions $v(x)$ and $f(x)$ along the subsequence. 

\smallskip
As we saw in \S~\ref{sec:partitionBS}, we can estimate ratios of floors in a tower estimating Birkhoff sums. The key estimate 
to reduce the existence of exponentially distorted towers for the GIET to the one for the shadow is given by the following Lemma. Recall that we denote by $S_n f$ the Birkhoff sums of a function $f$ over $T$, see \S~\ref{sec:BS}. 

\begin{lemma}[shadowing interpolation]\label{lemmadeviations}
Let $f:= \log D T$. Then for any $\epsilon_0>0$ there exists a constant $C=C(\epsilon_0)>0$ such that, for any $n\in\mathbb{N}$, any $1\leq j\leq d$ and any floor $F_0=T^{i_0}(I^{(n)}_j)$, 
$$
\left| S_i f (x) - S_i v (x) \right| \leq C  |i|^\epsilon, \qquad \text{for\ any}\ x\in F_0, \,\text{for\ any}\ i\in \mathbb{N}.
$$
\end{lemma}
\noindent Before giving the proof, we remark that, when  $S_n \vect{v}(x)$ is a special Birkhoff sum, i.e.~$x\in I^{(m)}_j$ for some $m$ and $n=q^{(m)}_j$, one has that $S_n v(x)= v^{(m)}_j$ (see~\ref{sec:SBS})
and furthemore, from the definition of $\vect{\omega}_m$, one can show that 
there exists an $\overline{x}$  such that $S_n f(\overline{x})=(\vect{\omega}_m )_j$. Thus, in this special case, the Lemma follows from the \emph{shadowing} given by Theorem~\ref{shadowing}, which gives that $\Vert v^{(m)}- w_m \Vert\leq \Vert v^{(m)} \Vert^\epsilon$. The general case will  be obtained by \emph{interpolation} (from which the name \emph{shadowing interpolation}), exploiting the geometric decomposition of Birkhoff sums into special Birkhoff sums described in \S~\ref{decompBS}.

\begin{proof}[Proof of Lemma~\ref{lemmadeviations}] 
Fix any $i\in\mathbb{Z}$. We will consider the case $i\geq 0$. The case $i<0$ can be treated analogously replacing Birkhoff sums for $T$ with Birkhoff sums for  $T^{-1}$. Let $k_i\in\mathbb{N}$ be the largest $k\in \mathbb{N}$ such that the orbit  segment $\{ x, \dots T^{i}(x)\}$ intersects $I^{(n_{k})}$ \emph{twice}. Then, by the geometric decomposition of Birkhoff sums in \S~\ref{decompBS}  (see \eqref{geometricestimate}), we can estimate, for any $x\in F_0$,
\begin{equation}\label{interpolationbound}
\left| S_i f (x) - S_i v (x) \right|\leq 2 \sum_{k=0}^{k_i} \Vert \widetilde{Z}_k \Vert \, \Vert \widetilde{f}^{(k)}(x)- \widetilde{v}^{(k)}(x) \Vert_{\infty}.
\end{equation}
Notice that here $v\in\mathcal{C}(T)$ is piecewise constant, the special Birkhoff sums 
$\widetilde{v}^{(k)}(x)$ (which are piecewise constant on the continuity intervals of $I^{(n)}$) 
can be identified with the vector $\widetilde{v}^{(k)}=(\widetilde{v}^{(k)}_j)_{j}\in\mathbb{R}^d$. 
To estimate the  sup norm of the difference of special Birkhoff sums $\widetilde{f}^{(k)} (x)- \widetilde{v}^{(k)}$, when $x\in I^{(n)}_j$ we add and subtract the constant $(\widetilde{\omega}_k)_j$, i.e.~the $j^{th}$ entry of the vector $\widetilde{\omega}_k=\vect{\omega}_{n_k}$. 

For any $n\in\mathbb{R}$ and any $1\leq j\leq d$,  by mean value theorem and by Remark~\ref{rk:logvectorRn} (see in particular equation~\eqref{eq:logvectorTn}) there exists a point $x^{(n)}_j$ in $I^{(n)}_j$ such that 
$(\rho_n)_j=  D( T^{q^{(n)}_j}  )(x^{(n)}_j)$. 
Thus,  recalling that   $\vect{\omega}_{n}=\log \rho_n$ (see Definition~\ref{def:rhoomega}) and $f:=\log D T$, 
 using the chain rule and recalling the definition of special Birkhoff sums (see \S~\ref{sec:SBS}) we get that
$$
(\vect{\omega}_n)_j:= \log(\rho_n)_j = \log \big( D T^{q^{(n)}_j}\big)\big(x^{(n)}_j\big) =S_{q^{(n)}_j}f\big(x^{(n)}_j\big) =f^{(n)} \big(x^{(n)}_j\big) . 
$$
 Moreover, by 
another simple consequence of the classical distorsion bounds (Lemma~\ref{bound1}), is that special Birkhoff sums of each continuity interval have bounded fluctuations, namely 
there exists $C_T>0$ such that for any $n\in\mathbb{N}$ 
$$
\left|f^{(n)} (x) -({\omega}_n)_j \right| =\left|f^{(n)}(x)-f^{(n)} \big(x^{(n)}_j\big) \right| \leq C_T, \qquad \textrm{for\ all}\ x \in I^{(n)}_j, \quad 1\leq j\leq d.
$$
Thus,  using this estimate for a time $n$ of the form $n_k$ and the property of the shadow given by the conclusion of Theorem~\ref{shadowing}, we get that, for any $\varepsilon>0$, for some $c>0$, 
\begin{eqnarray}\nonumber
\Vert \widetilde{f}^{(k)} (x)- \widetilde{v}^{(k)}\Vert_{\infty} & \leq& \sup_{1\leq j\leq d} \sup_{x\in I^{(n)}_j}| {f}^{(n_k)} (x)- ({\omega}_{n_k})_j| + \Vert \widetilde{\vect{\omega}_{n}}- \widetilde{v}^{(k)}\Vert \\ &\leq  & C_T+ \Vert {\vect{\omega}_{n_k}}- {v}^{(n_k)}\Vert \leq  C_T+ c  \Vert v^{(n_k)}\Vert^\epsilon  \leq  C_T' \Vert v^{(n_k)}\Vert^\epsilon \nonumber
\end{eqnarray}
for some $C_T'>0$. 
Inserting this estimate in \eqref{interpolationbound} and using that, by assumption of the $(RDC)$, there exists $C_\epsilon\geq 0$ such that  $\Vert \widetilde{Z}_k(T)\Vert \leq C_\epsilon e^{\epsilon k}$ for any $k$ and that, for any vector in $\mathbb{R}^n$ and hence in particular for $\vect{v}$ we have that $\Vert v^{(n)} \Vert\leq C_2 e^{\theta n} $ for every $n\in\mathbb{N}$  where  $\theta>0$ is any exponent $\theta>\theta_1$ (and actually, for $\vect{v}$, one can actually choose any $\theta>\theta_2$) and $n_k$ grow linearly (see Definition~\ref{def:RDC}), there exists $C'>0$ such that 
\begin{align*}
\sup_{x\in F_0}\left| S_i f (x) - S_i v (x) \right|& \leq 2C_T' \sum_{k=0}^{k_i} \Vert \widetilde{Z}_k \Vert \Vert v^{(n_k)}\Vert^\epsilon \leq C' \sum_{k=0}^{k_i} e^{k\epsilon} e^{\epsilon 2 \theta k } \\ & = C' e^{\epsilon k_i(1+2\theta)} \sum_{k=0}^{k_i} e^{-\epsilon(1+ 2\theta) (k_i-k) } \leq  C' e^{\epsilon k_i(1+2\theta)}    \sum_{j=0}^{\infty}e^{-\epsilon (1+2\theta)j}<C'' (e^{k_i})^{\epsilon_1}, 
\end{align*}
for some $C''>0$ independent on $k_i$ and $\epsilon_1:=\epsilon (1+2\theta)$. To conclude we will now show that $|i|\geq c e^{\theta_1' k_i}$ for some $c>0$ and any $\theta_1'<\theta_1$, so that the above estimate can be written in the desired form $C_0 |i|^{\epsilon_0}$ for some $\epsilon_0$ going to zero as $\epsilon$ goes to zero and $C_0$ depending on $\epsilon$ (and hence $\epsilon_0$).

Since by definition of $n_i$ the orbit  segment $\{ x, \dots T^{i}(x)\}$ (or, respectively, in the case  $i<0$, the orbit segment $\{ x, T^{-1}(x) ,\dots T^{-i}(x)\}$) intersects $I^{(n_{k_i})}$ at least twice, $|i|$ is greater than the height of one Rohlin tower over $I^{(n_{k_1})}$. Since the sequence $(n_k)_{k\in\mathbb{N}}$ is a sequence of $p$-positive times (see Definition~\ref{def:RDC} and Definition~\ref{def:goodreturns}), the matrices ${Q}(n_k,n_k+p)$ are  positive matrices,  it is now a standard argument 
to see that,  since each $Z_{n}$ increase subexponentially, for any $\theta_1'<\theta_1$,
\begin{align*}
|i| \geq \min_{1\leq j\leq d} q^{(n_{k_i})}_j\geq {\max_{1\leq j\leq d} q^{(n_{k_i}-p)}_j} &\geq\frac{\Vert {Q}({0,n_{k_{i}}-p)}\Vert}{d} \\ & \geq \frac{\Vert {Q}(0,n_{k_{i}})\Vert }{d {\Vert  {Q}(n_{k_i}-p,n_{k_i})\Vert }} \geq \frac{\Vert {Z}^{(n_{k_i})} \Vert }{d \Vert {Z}_{n_{k_i-p}}\Vert \Vert {Z}_{n_{k_i-p+1}}\Vert \cdots \Vert {Z}_{n_{k_i}-1}  \Vert }
\geq c_0 e^{ \theta_1'n_{k_i}}
\end{align*}
for some $c_0>0$. Since $\{n_k\}_{k\in\mathbb{N}}$ grow linearly, this also shows, as claimed, that $|i|\geq c e^{ \theta_1'{k_i} }$ and hence concludes the proof.
\end{proof}

We can now prove Proposition~\ref{reduction}.  We isolate first a remark which will be used also later.
\begin{remark}\label{BSpc} Birkhoff sums of \emph{piecewise-constant functions} in $\mathcal{C}(T)$ transform well under semi-conjugacy, in the following sense.
Assume that $T_1$ and $T_2$ are two semi-conjugated GIET, i.e.~there exists a surjective $h:[0,1]\to [0,1]$ such that $h\circ T_2=T_1\circ h$. Let $v$
 Given a vector $v$, if for $i=1,2$, $v_i(x)$ denotes  the piecewise constant functions $v_i(x)\in \mathcal{C}(T_i)$ associated to $v$ and $S_i v_i$ the Birkhoff sum of $v_i$ under $T_i$, we claim that we have that
$$
S^2_n v_2(x) = S^1_n v_1 (h(x))
\qquad \text{for\ all}\ n\in \mathbb{N}
$$
and for all $x$ which belong to a continuity interval for $S^2_n v_2$. The equality follows indeed from conjugacy relation $h\circ T^k_2=T^k_1\circ h$ for $k\in\mathbb{N}$ and the observation that, since $v_2$ is piecewise constant on continuity intervals, $v_2(x)=v_1(h(x))$. 
\end{remark}
\begin{proof}[Proof of Proposition~\ref{reduction}.]
Let us call $\overline{T}$ any AIET with the same rotation number of $T$ and log-slope vector given by the  shadow $v$. Since the GIET $T$ and the AIET $\overline{T}$ have the same irrational rotation number, they are both semi-conjugated  to a common standard IET $T_0$, via semi-conjugacies that we will call respectively $h$ and $\overline{h}$.  
In particular, since semi-conjugacies map (floors of) Rohlin towers to (floors of) Rohlin towers, the  floors of Rohlin towers $\mathcal{P}^j_n$ for $T$ can be put in one to one correspondence  with corresponding floors of the Rohlin towers for $\overline{T}$, that we will denote by $\overline{\mathcal{P}}^j_n$. 

We will now show that if $\overline{T} $ has exponentially distorted towers, also $T$ has them. 
Fix $n\in\mathbb{N}$ and $1\leq j\leq d$ and let $\overline{F}_0$ be the floor given by Definition~\ref{def:expdistorted} of exponentially distorted towers for $\overline{T}$ and ${F}_0$ 
the corresponding floor for $T$. We want to show that ${F}_0$ satisfies the estimates in Definition~\ref{def:expdistorted}. Consider therefore the floors $F_i:=T^i F_0$ which belong to the same Rohlin tower ${\mathcal{P}}^j_n$ and correspondigly the floors $\overline{F}_i:=\overline{T}^i \overline{F}_0$ in $\overline{\mathcal{P}}^j_n$.

\smallskip
Since $\overline{T}$ is an AIET with log-slopes vector $v$, $\log D\overline{T}(x)$ coincides with the piecewise constant function $\overline{v}(x)$ in $\mathcal{C}(\overline{T})$ given by $v$. It follows also that the Birkhoff sums $\overline{S}_i\overline{v}(x)$ of $\overline{v}$ under $\overline{T}$ are constant on $F_0$. Therefore, exploiting Remark~\ref{BSpc} twice (for the semi-conjugacies $h$ and $\overline{h}$), we have that  
$$
\overline{S}_i \overline{v}(\overline{x})= {S}_i {v}(x), \qquad \text{for\ every}\ x\in F_0, \ \text{and\ every}\ \overline{x}\in \overline{F}_0.
$$ 
Thus,  by Lemma~\ref{towersviaBS} we have that, for every $i$ such that $\overline{F}_i$ is a floor of the Rohlin tower $\overline{\mathcal{P}}^j_n$,
$$
\frac{|\overline{F}_i|}{|\overline{F}_0|} 
\geq \frac{1}{C_{\overline{T}}} \, \exp({S_{i}  v (x)}) , \qquad \text{for\ all}\ x\in F_0.
$$
Thus, together with another application of Lemma~\ref{towersviaBS}, this time to the floors of $\mathcal{P}^j_n$, we get
$$
\frac{|{F}_i|}{|F_0|} \frac{|\overline{F}_0|}{|\overline{F}_i|} \leq \frac{C_T}{C_{\overline{T}}}  \exp( S_i f (x)-{S_{i}  v (x)}),\qquad \text{where}\ f=\log DT. 
$$
Thus, applying Lemma~\ref{lemmadeviations} and then the exponentially distorted estimates given by Definition~\ref{def:expdistorted} for $\overline{T}$, we get that
$$
\frac{|T^i{F}_0|}{|F_0|}=\frac{|{F}_i|}{|F_0|}\leq \frac{C_T}{C_{\overline{T}}}   \exp( C_0 |i|^\epsilon) \frac{|\overline{F}_i|}{|\overline{F}_0|} \leq  \frac{C_T \, C}{C_{\overline{T}}} \exp (-c |i|^\gamma +C_0 |i|^\epsilon )\leq C' \exp(-c' |i|^\gamma)
$$
for some $C', c'>0$, thus showing that also $T$ has exponentially distorted towers. It then follows by Remark~\ref{rk:expdist} and Lemma~\ref{wilemma} that $T$ has wandering intervals.
 \end{proof}

\subsection{Final arguments in the proof of Theorem~\ref{maintheorem}}\label{sec:proofGIETrigidity}
We have now all the ingredients to finalize the proof of Theorem~\ref{maintheorem} according to the Outline shown in the initial subsection of this section.

\begin{proof}[Proof of Theorem~\ref{maintheorem}]
Let $T$ be an irrational IET in $\mathcal{X}^3_4\cup \mathcal{X}^3_5$ and assume that its rotation number $\gamma(T)$ belong to the full measure set obtained intersecting the $(RDC)$ with the full measure condition on rotation numbers in Proposition~\ref{AIETwi}.
Assume furthermore $T$ is \emph{conjugated} to  an IET $T_0$ (which hence has rotation number $\gamma$). This implies in particular that $T$ is \emph{minimal}.
\smallskip

\noindent {\it Strategy.} Let $( n_k)_{k\in\mathbb{N}}$ be the sequence given by the $(RDC)$ and consider the shape log-slope vectors $\widetilde{\omega}_k:=\omega(\mathcal{Z}^{n_k}(T))$.   
By Theorem~\ref{shadowing}, either $(\widetilde{w}_{k})_{k\in\mathbb{N}}$ is bounded (i.e.~we are in Case 1), or we are in Case 2 and it shadows a vector $\vect{v}$ in the unstable space $E^u$ (for the Oseledets regular extension of $T_0$, see Definition~\ref{def:Oseledetsextension}). We will show now that by the results in this section minimality is incompatible with Case $2$, so we are forced to be in Case $1$ and in this case will conclude that the conjugacy is smooth.
 
\medskip
\noindent {\it  Case 2 implies the existence of wandering intervals.} Let us assume first that we are in Case $2$ and show that this contradicts minimality. Let $v$ be the shadow given by Theorem~\ref{shadowing} in Case $2$ and consider the Oseledets filtration for $T_0$ given in \eqref{positiveOseledets}. Since we are assuming that $d=4$ or $d=5$ and  that  $T_0$ is Oseledets generic (which is part of the $(RDU)$), we have that $g=2$ and there are exactly two simple positive exponents $\theta_1>\theta_2>0$. Thus $E_3(T_0)=E^{cs}(T_0)$.  
Since as part of the conclusion of Case $2$ of Theorem~\ref{shadowing} we also know that $v\in E^u$ (unstable space for the extension of $T_0$ used in the Definition~\ref{def:RDC} of the RDC), we know that $v\notin E^{cs}(T_0)=E_3(T_0)$, so $v\in E_1(T_0)\backslash E_3(T_0)$. 

\smallskip
We will now show that  $v\notin E_1(T_0)\backslash E_2(T_0)$, so that we can conclude that $v\in E_2(T_0)\backslash E_3(T_0)$. 
 Let us argue by contradiction: if  $v$ were in $ E_1(T)\backslash E_2(T_0)$, 
it had a projection  on the largest Oseledets exponent (which has a positive Oseledets eigenvector). In this case  there would exist  a time $n\in\mathbb{N}$ for which the entries of 
 $v^{(n)}=Z^{(n)}{v}$ were all positive (or all negative, in which case we can replace $\vect{v}$ with -$\vect{v}$ and reduce to the previous case) and as large as we like.  By the properties of the shadow (see the conclusion of Theorem~\ref{shadowing} in Case $2$), this would imply that the same is true for $\vect{\omega}_n$. This would in turn imply that $\vect{\rho}_n=\exp({\vect{\omega}_n})$ has all entries  (strictly) larger than $1$,  which is impossible since by definition the entries of $\rho_n$ are the shape log-slopes of $\mathcal{R}^n(T)$ (see \eqref{eq:renormalizedmap}) and $\mathcal{R}^n(T)$ cannot be either everywhere expanding.
Thus, we conclude that $v\notin E_1(T_0) \backslash E_2(T_0)$. 

\smallskip
Since we now have that $ \textbf{v}\in E_2(T_0)\backslash E_3(T_0)$ and we are assuming that $\gamma$ belongs to the full measure set of rotation numbers under which Proposition~\ref{AIETwi} holds, the estimates \eqref{eq:powerdecay} of Proposition~\ref{AIETwi} hold for the Birkhoff sums $S_iv_0$. If we now transport these estimates through the conjugacy between $T$ and $T_0$, since the conjugacy maps $v_0(x)\in \mathcal{C}(T_0)$ to the function $v(x)\in \mathcal{C}(T)$ given as usual by $v(x)=v_j$ if $x\in I^t_j$, we also have that the Birkhoff sums $S_iv $ of $v(x)$ under $T$ satisfy the same estimates, i.e.~for every $n\in\mathbb{N}$ and $1\leq j\leq d$ there exists a floor $F_0$ of the Rohlin tower $\mathcal{P}^j_n$ such that 
$$
S^0_i v_0(x_0) \leq C - |i|^{\gamma_0}, \qquad \mathrm{for \ every}\ x\in F_0, \   i\in\mathbb{N}. 
$$
This, by Lemma~\ref{towersviaBS}, shows that the AIET with rotation number $\gamma$ and log-slope $v$ (namely the affine shadow of $T$) has exponentially distorted towers and wandering intervals. Proposition~\ref{reduction} therefore implies that also $T$ has wandering intervals, which contradicts minimality. We conclude that Case $2$ cannot happen when $T$ is minimal.

\medskip
\noindent {\it  Case 1 implies rigidity.} We now  assume to be in Case 1. Since by assumption we also have that $\mathcal{B}(T)=0$, by Theorem~\ref{convergencethm} there is exponential convergence of renormalization, namely $ d_{\mathcal{C}^1}(\mathcal{Z}^{n}(T), \mathrm{IET}) \leq K_1  \alpha_1^n$. Thus, we can apply Proposition~\ref{conjugacy} and conclude that $T$ is $\mathcal{C}^1$-conjugate to an IET. 

Notice now that  IETs with rotation number satisfying the $(RDC)$ are uniquely ergodic (since by construction the times $(n_{k})_k$ are good return times, so they correspond to returns to a compact subset in the space of IETs and this implies ergodicity by the seminal results by Veech, see \cite{Ve:gau} or \cite{Yoc:CF}) and therefore there is a unique IET  with rotation number $\gamma$. Thus, we conclude that the IET has to be $T_0$  and that the conjugacy between $T$ and $T_0$ is $\mathcal{C}^1$.
\end{proof}

\subsection{A priori bounds in genus two.}\label{pf:thmD}
We can now prove  Theorem D in the introduction, namely a priori bounds in genus two,  which follows from Theorem~\ref{shadowing} and Proposition~\ref{AIETwi} using the same reasoning than in the proof of  Theorem~\ref{maintheorem}.
\begin{proof}[Proof of Theorem D]
Consider the same full measure set of rotation numbers defined at the beginning of the proof of Theorem~\ref{maintheorem}. If Case $2$ of  Theorem~\ref{shadowing} holds, then, as in  the proof of Theorem~\ref{maintheorem},  we get a contradiction to minimality. Thus, Case $1$ of Theorem~\ref{maintheorem} holds. The a priori bounds then follow by Theorem~\ref{apriori}. 
\end{proof}


\section{Rigidity of foliations.}
\label{foliations}
In this section we translate our rigidity result on GIETs in the language of  foliations on surfaces of genus two. We first give some preliminary definitions on foliations (see \S~\ref{sec:foliationsprelim}). 
{We then deduce Theorem A from Theorem B in \S~\ref{sec:foliationsmain} (see also Proposition \ref{thm:rigidityfoliation}).}

\subsection{Preliminaries on singular foliations and holonomy}\label{sec:foliationsprelim}
This section we define regularity of  singular foliations and holonomy around a singular point. The section follows partly \cite{Arnoux} and \cite{Levitt1,Levitt2,Levitt3}.
\subsubsection{Foliations singularities and regularity}\label{sec:leavesregularity}
Throughout this section, $S_g$ is closed orientable smooth surface {and all foliations are \emph{orientable}.}  We consider foliations on $S_g$ with a finite number of singularities, and we further ask that those singularities are of \emph{saddle type}, as in Figure~\ref{saddle} below. Formally:
\begin{definition}[saddle-type singularity]\label{def:saddle}
A singular point $p$ of a foliation is of \emph{saddle-type} if, locally, in a neighbourhood of $p$, there are charts for which the  {\it topological} model of the foliation is given, equivalently, by either: 
\begin{itemize}
\item[(i)] the level sets in $\R^2$ of the function $(x,y) \longmapsto xy$ around $0$;
\item[(ii)] the integral curves of the vector field $y \, \partial x + x\, \partial y$.
\end{itemize}
\end{definition}
\noindent {More generally, one can allow \emph{degenerate} saddles,\footnote{{Degenerate saddles are also called  \emph{multi}-saddles, since they are saddles with more prongs. Since we are considering only orientable foliations, the number of prongs needs always be even.}} which are defined by level sets of a smooth function with a zero of order two or higher\footnote{{Condition $(i)$ describes the level sets of the foliation given by level sets of the function $(x,y)\mapsto \mathrm{Im} (z^2) :=\mathrm{Im} ((x+iy)^2)$. More generally, one can for example consider the foliation whose leaves are level sets for  $\mathrm{Im} (z^n)$ (or $\mathrm{Re} (z^n)$) for some $n\geq 2$.}}.} 
Let us denote by $\mathcal{F}$ the singular foliation on $S_g$ and by $Sing_\mathcal{F} \subset S_g$ be the finite set of (saddle-like) singular points of $\mathcal{F}$. We define now what it means for $\mathcal{F}$ to be of class $\mathcal{C}^r$. This definition is due to Levitt \cite{Levitt3}.

\begin{definition}[Foliation of class $\mathcal{C}^r$]\label{def:Crfol}
We say that the foliation $\mathcal{F}$ is of class $\mathcal{C}^r$ iff: 
\begin{enumerate}
\item[(r1)] the  leaves of $\mathcal{F}$ in $S_g \setminus Sing_\mathcal{F}$ are locally embedded $\mathcal{C}^r$-curves;
\item[(r2)] for any two smooth open transverse arcs $I$ and $J$ which are joined by leaves of $\mathcal{F}$, the holonomy map $I \longrightarrow J$ is a $\mathcal{C}^r$ diffeomorphism on its image and extends to the boundary of $I$ to a $\mathcal{C}^r$-diffeomorphism.
\end{enumerate}
\end{definition}
\noindent The  subtlety covered by this definition is the following: it can happen that a foliation of $S_g$, when restricted to $S_g \setminus Sing_\mathcal{F}$, {\color{black}is} as regular as desired in the standard sense, but when considering open transverse arc based at a singular point, holonomies from this arc have a critical point at the singularity, or be much less well-behaved altogether. The above definition excludes such cases. 

 {We remark that, as special case,  foliations defined by $\mathcal{C}^r$ vector fields with {non-degenerate critical points} (which equivalently implies   exactly that the leaves in a neighbourhood of each critical points are locally defined by $\mathcal{C}^r$ \emph{Morse} functions)} are of class $\mathcal{C}^{r}$  in the above sense. It is however not the case that every $\mathcal{C}^r$  (or even \emph{smooth}) vector field  gives rise to a $\mathcal{C}^r$ (or smooth) foliation, see Appendix~\ref{saddles}.

{Moreover,} it is also not always the case that the differentiable structure of a  $\mathcal{C}^r$-foliation {\emph{near}  a singularity $p$ (i.e.~in a \emph{punctured} neighbourhood of $p$)} is defined by a $\mathcal{C}^r$-Morse function. The obstruction for a foliation to be $\mathcal{C}^r$-smooth {in the sense of Definition~\ref{def:Crfol}, i.e.~for the $\mathcal{C}^r$-smooth structure to extend \emph{at} the singularity,} can be encoded through \emph{holonomies around singular points}, as we explain in the next subsection.}

\subsubsection{Holonomy around singular points}\label{sec:holonomies}  We describe in this paragraph how to construct an invariant of a saddle $p$ of a $\mathcal{C}^r$-foliation which essentially encodes the obstruction for $\mathcal{F}$ to be defined as the level sets of a regular function. {We give the definition of simple saddles ($4$-prongs), but  construction straightforwardly generalises to the case of saddles with an arbitrary \emph{even} number of prongs.}

\smallskip Consider a 4-pronged saddle $p$ of a $\mathcal{C}^r$-foliation as in Figure~\ref{saddle} below, together with smooth four transverse arcs, each of which intersecting one of the four separatrices {and whose endpoints pairwise belong to the same leave, as shown in Figure~\ref{saddle}.  
We call these arcs {$I_d$}, $I_r$, {$I_u$} and   $I_l$  (for  {down}, right,  { up} and left respectively) and, correspondingly, we  call $v_d, v_r, v_u$ and $v_l $ the point of intersections of these intervals with  the separatrices emanating from $p$. } {\color{black} We  identify each of these arcs with the interval  $(-\epsilon, \epsilon)$ and assume that $0$ is the point of intersection with the separatrix {(i.e.~$v_d, v_r, v_u$ and $v_l $ respectively.} }}
 { By definition of foliation, there exists a  function $f$ in a neighbourhood $\mathcal{U}$ of $p$ which contains $I_l \cup I_u \cup I_r\cup I_d$  such that: 
\begin{itemize}
\item[(i)] $f$ is a \emph{continuous} function on the whole $\mathcal{U}$ and is equal to $0$ on the separatrices;
\item[(ii)] $f$ is  $\mathcal{C}^r$ on $\mathcal{U}\, \backslash \{ $separatrices$\}$, i.e. on the complement in $\mathcal{U}$ of the separatrices containing $p$.
\end{itemize}
Let us call \emph{quadrants} the connected components of $\mathcal{U}\, \backslash \{ $separatrices$\}$.  
Starting from the lower right quadrant (which contains the right endpoint of $I_d$) and modifying $f$ in successive adjacent \emph{quadrants}, one can assume, without loss of generality, that 
\begin{itemize}
\item[(iii)] $f$ is a $\mathcal{C}^r$-function also on the separatrices through $v_r, v_u$ and $v_l$, so on all of $\mathcal{U} $, possibly with the exception of the separatrix through $v_d$. 
\end{itemize}
We stress that in general $f$ is \emph{not} $\mathcal{C}^r$ in the whole neighbourhood $\mathcal{U}$ and we \emph{cannot} modify $f$ any further (since it was already chosen on the initial quadrant).

\begin{figure}[!h]
\label{saddle}
	\begin{center}
		\def\svgwidth{ 0.4 \columnwidth}
			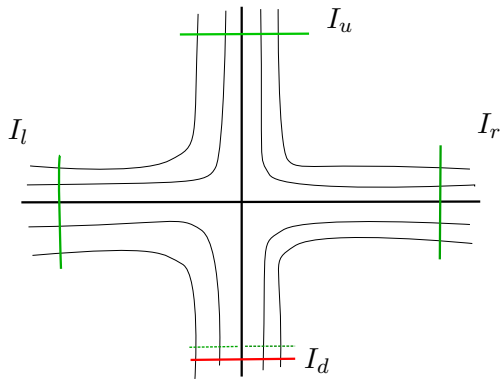
	\end{center}
	\caption{A saddle}\label{saddle}
\end{figure}
\noindent  Recalling that $I_d$ is identified with $(-\epsilon, \epsilon)$ and $0$ is the coordinate of the point $v_d$ intersecting the separatrix, by construction, $f$ \emph{restricted} to $I_d$ defines a continuous function $f : (-\epsilon, \epsilon)\to \mathbb{R}$  such that: 
\begin{itemize}
\item[(i)] $f$ is $\mathcal{C}^r$ on $(-\epsilon, \epsilon) \setminus \{0\}$;
\item[(ii)] $f$ extends to a $\mathcal{C}^r$-function to both $[-\epsilon, 0]$ and $[0, \epsilon]$.
\end{itemize}
\noindent We define the $\mathcal{C}^r$-\textit{honolomy} of $\mathcal{F}$ around this saddle as the $r$-jet 
defined by the difference between the value of the extension of (the restriction to $I_d$ of) $f$    to  the \emph{right} interval $I_d^+$, identified with $[0,\epsilon]$, at $0$ and the value at $0$ of its extension to the \emph{left} interval $I_d^-$, identified  with $[-\epsilon, 0]$. This is a measure of the \emph{obstruction} for $f$ to extend to a $\mathcal{C}^r$-function to the whole neighbourhood $\mathcal{U}$ of the saddle.

\subsubsection{Minimality (or quasi-minimality) for a foliation}\label{sec:minimalityfol}
The notion of minimality of a foliation  (which is sometimes known as  \emph{quasi-minimality} in the literature) is the following.
\begin{definition}[Minimality of a foliation]\label{def:minimal}{We say that a (singular) foliation is \emph{quasi-minimal}, or simply \emph{minimal}, iff  every \emph{regular} leaf (i.e.~every leaf which is not a point or a separatrix) is \emph{dense}.\footnote{{Notice that if the foliation comes from a flow, this is not the usual notion of minimality: the orbits of fixed points, which correspond to singularities, are indeed not dense. Moreover, if there are \emph{saddle connections}, these are also not dense.  One can in addition ask that also the orbits of separatrices (i.e.~leaves which emanate from a singular points and therefore are not \emph{regular} leaves) are dense; in this case, though, these orbits are only dense in the past or in the future only.}}}
\end{definition}}
\subsubsection{{\color{black} Relation with  GIETs}}
Let us recall that any $\mathcal{C}^r$-generalized interval exchange transformation can be suspended {\color{black}to a $\mathcal{C}^r$ foliation on a surface $S_g$ of  a certain genus $g \geq 1$} (see \S~\ref{gietandfoliations} and Appendix~\ref{sec:boundarycomb} for details). 
This operation can, for a large class of foliations ({which includes all minimal foliations}), be inverted, see \cite{Arnoux} or \cite{Levitt1}. In particular, we have the following Lemma:
\begin{lemma}
\label{prop:folgiet}
Let $\mathcal{F}$ be a minimal $\mathcal{C}^r$-foliation on $S_g$. Then there exists a smooth arc $J$ in $S_g$, everywhere transverse to $\mathcal{F}$ and joining {two} saddles, such that, if we identify  $J $ with $[0,1]$, the first-return map of $\mathcal{F}$ on $J $, when identified with $[0,1]$ is a minimal $\mathcal{C}^r$-GIET.
\end{lemma}
\noindent We refer to \cite{Arnoux} or \cite{Levitt1,Levitt2,Levitt3} for a proof of this Lemma.

\smallskip
\noindent An important observation, especially to the purpose of our reduction of rigidity of foliations (Theorem A) to rigidity of GIETs (Theorem B) is the following relation between holomomies around singularities as defined above and the boundary  $\mathcal{B}(T)$ of the GIET as defined in Definition~\ref{GIET:boundary} through the boundary operator defined by Marmi-Moussa-Yoccoz in \cite{MMY3}:


\begin{lemma}[Boundary as holomomy]\label{Basholonomy}
If  $\mathcal{F}$ is minimal and $T$ is a GIET obtained as first-return map of $\mathcal{F}$ on  an arc $J$ identified with $ [0,1]$  as in Lemma~\ref{prop:folgiet}, then the { exponentials $(\exp (b_s))_{s=1}^d$ of the} entries of the boundary $\mathcal{B}(T)=(b_s)_{s=1}^d$ are exactly the holonomy of $\mathcal{F}$ around singularities of $\mathcal{F}$.
\end{lemma}
{ \begin{proof}
Let $J$ be a standard arc such that $T$ can be identified to the first-return of $\mathcal{F}$ on $J$. We can isotope $J$ to a transverse curve that contains the arcs $I_b, I_t, I_r $ and $I_l$ (represented in Figure \ref{saddle})   as subarcs  and such that the intersection of the separatrices with $J$, which we called $v_b,v_t,v_r$ and $v_l$,  are exactly the discontinuities of $T$ and $T^{-1}$ that are involved in the computation of the value of the boundary $\mathcal{B}(T)$ at the singularity $p$.  We can  chose the parametrization of $J$ by $[0,1]$ (for which the first return of $\mathcal{F}$ is $T$) so that $T$ is a GIET of class $\mathcal{C}^r$. This parametrization  induces parametrizations of the subarcs $I_b, I_t, I_r $ and $I_l$.

To compute the holonomy around the saddle $p$ using $T$, we can start from the parametrization of $I_d^+$ by $(0,\epsilon)$, extending it and \emph{transporting it around}, to define a function $f$ on $\mathcal{U}$ which is \emph{constant on leaves} of each \emph{quadrant}\footnote{{Recall that a  \emph{quadrant}, as defined in \S~\ref{sec:holonomies}}, is a connected component of the complement of the separatrices in the neighbourhood $\mathcal{U}$ of $p$.}  and whose values on $I_d^+$ are given by the parametrization.  The request that $f$ is constant on leaves means in particular that we want that $f(x)=f(T(x))$ for every $x\in J$. We can easily define such $f$ in the lower quadrant, containing $I_d^+$. 
To extend $f$ to adjiacent quadrants (going around $p$ in counterclockwise direction), starting from the adjacent quadrant containing $I_r$, first $I_r$ and then each time the new subarc, should be reparametrized  
using\footnote{{For example, to reparametrize $I_r$, one should use the change of variables $y=T(u_d+x)-u_r$. More generally, if $u$ is the previous discontinuity and $v$ the next, we either have $v=T(u)$ or $v=T^{-1}(u)$ and should use, respectively, a change of variable of the form $y=T(u+x)-v$ or $y=T^{-1}(u+x)-v$.}}
 either $T$ or $T^{-1}$,  depending on the parity of the step (since to reach the successive subarc we move along leaves in a quadrant either in the same, or with opposite orientation, according to the parity). 

The  (right or left) \emph{germ} of such a reparametrisation at each  intersections  of $J$ with each separatrix (i.e.~at a singularity $v\in  \{v_d, v_r, v_u, v_l\}$ of $T$ of those involved in the computation of the boundary value at $p$) is the (right or left) derivative $DT^\pm(v)$ of $T$ or of $T^{-1}$ at $v$. 
We thus realise the \emph{germ} of the holonomy around the saddle as a product of values of derivatives of $T$ or $T^{-1}$ at the singular points which correspond to separatrices at $p$; recalling the definition of the boundary operator $B$ acting on  observables (see \S~\ref{sec:obs_boundary}), one can see that the \emph{logarithm} of such a product is exactly the alternating sum of right/left derivatives of $D\log DT$ at singularities which gives the value $b_p$ of the boundary $B(\log D T)$ at the singularity $p$. Thus, the holonomy around $p$ is exactly $\exp(b_p)$.
\end{proof}}

\subsection{Conjugacies of foliations and rigidity}\label{sec:foliationsmain}
{ We now define \emph{linear} foliations and restate the existence of topological conjugacy to a linear model in terms of \emph{minimality}. 

\subsubsection{Linear foliations} {A special class of foliations are those are given by closed $1$-forms,  which we will call \emph{linear} (since their holonomies belong to the linear group). Examples of linear foliations include foliations whose leaves are  trajectories of linear flows on translation surfaces (see Remark~\ref{Calabi}).}
\begin{definition}[linear foliations]\label{def:linearfol}
A \emph{linear} foliation $\mathcal{L}$ is a foliation on $S_g$ defined by a smooth, closed $1$-form $\omega$ such that $\omega$ vanishes at only finitely many points which are (multi)saddles, described by level sets of smooth functions near a zero of finite multiplicity.
\end{definition}
\noindent  The local integration of $\omega$ defines {a (non atomic, smooth) transverse measure to the foliation $\mathcal{L}$ as well as an Euclidean structure of the space of leaves of the foliation. One can then show that} the first return map of a linear foliation on a transverse curve is a \textit{standard interval exchange transformation} with respect to the Euclidean structure induced by $\omega$ on this transverse curve. This shows in particular that the holonomies are \emph{linear}, from which the name \emph{linear} foliation. 
\begin{remark}\label{Calabi}
We remark as an aside that a result of Calabi 
 \cite{Ca:fol} (see also \cite{Zo:how}) shows that under a technical condition\footnote{{The assumption of Calabi~\cite{Ca:fol} is
 equivalent (as remarked in \cite{Zo:how}) to asking that any cycle obtained as a union of closed paths following in the positive direction a sequence of saddle connections is not homologous to zero. 
This is in particular the case when there are no saddle connections. In this special case the result was proved independently also by Katok in \cite{Ka:inv}.}}  
 linear foliations in the sense of Definition~\ref{def:linearfol} are actually given\footnote{{Calabi's theorem  in \cite{Ca:fol} gives a condition under which a given a closed $1$-form is \emph{harmonic}. In the language of foliations, this means that the linear foliation 
is the vertical foliation of a holomorphic differential in some complex structure.}} by linear flows on translation surfaces. 
\end{remark}


\subsubsection{Linearization of minimal foliations}  The following important result { classifies topological conjugacy classes of minimal foliations} (see also \cite{Ka:inv}):
\begin{proposition}[Topological conjugacy of minimal foliations]\label{topcfol}
A \emph{minimal} foliation $\mathcal{F}$ on $S_g$ is \emph{topologically} conjugate to a \emph{linear} one.  Furthermore, if the linear foliation is \emph{uniquely ergodic}, it is the \emph{unique} linear representative in the topological conjugacy class of $\mathcal{F}$.
\end{proposition}
 \begin{proof}
The first part of the statement is equivalent to the fact that a first-return map of $\mathcal{F}$ is topologically conjugated to a standard IET.  By Lemma~\ref{prop:folgiet}, there exists a transverse arc $J$ to $\mathcal{F}$ { and a smooth identification of  $J$ with $[0,1]$ such that the first return map upon $J$ under this identification is a minimal GIET which we call $T$.}  
 Consider any invariant probability measure\footnote{{The existence of such a measure follows for example by Krylov–Bogolyubov theorem: even if $T$ is not continuous, it can indeed be extended to a homeomorphism of a Cantor space (see \cite{MMY}, Corollary 3.6 or the lecture notes \cite{Yoc:Clay}). A direct proof of the existence of an invariant probability measure can also be found in Katok's work \cite{Ka:inv}.}} $\mu$ for $T$. 
By minimality of $T$,  $\mu$ has no atoms and  gives mass to any open subset of $[0,1]$. 
Thus, the map $\varphi : x \mapsto \int_0^x{d\mu}$ conjugate{s} $T$ to a GIET which preserves the Lebesgue measure (as $\varphi$ maps $\mu$ onto the Lebesgue measure), which is by definition a {(standard) IET}. 

Finally, if the IET $T$ is uniquely ergodic, but there were two linear foliations topologically conjugate to $\mathcal{F}$, one could find a different IET $T_1$ which is topologically conjugate to $T$. The pull-back of the Lebesgue measure via the conjugacy would then  produce an invariant measure for $T$ (in addition to the Lebesgue measure), contradicting unique ergodicity.
 \end{proof}

\subsubsection{Measure (class) on minimal foliations}\label{sec:foliationmeasure} A linear foliation, up to smooth isotopy fixing the set of singular points $Sing_\mathcal{F}$, is locally determined by the class defined by $\omega$ in $\mathrm{H}^1(S_g, Sing_\mathcal{F}, \mathbb{R})$. We can therefore endow the space of linear foliations with fixed singularities, up to isotopy, with the affine structure of $\mathrm{H}^1(S_g, Sing_\mathcal{F}, \mathbb{R}) =  \mathbb{R}^d$. 
{ In view of Proposition~\ref{topcfol}, topological classes of (singular) minimal foliations on $S_g$ are therefore parametrized by (relative) cohomology classes in $\mathrm{H}^1(S_g, Sing_\mathcal{F}, \mathbb{R})$.  The cohomology class associated to $\mathcal{F}$ is known in the literature as \emph{Katok fundamental class}.
\footnote{{Katok also showed in \cite{Ka:inv} that the fundamental class is a \emph{local} smooth (and topological) conjugacy invariant for foliations with only Morse saddles  with a non-atomic invariant measure which gives positive measure to open sets. 
}}

\smallskip
The Lebesgue measure on $\mathbb{R}^d$ induces a \emph{measure class} (i.e.~a notion of measure zero sets) on  linear foliations. Notice that a full measure set of such foliations are \emph{uniquely ergodic} by a classical result of Masur \cite{Ma:int} and Veech \cite{Ve:gau}. Therefore,  through Proposition~\ref{topcfol}, we also have a measure class on (topological conjugacy classes of) minimal foliations. The notion of \emph{full measure} in Theorem A is defined with respect to this measure class. It is well know that this notion of full measure is related to the notion of \emph{almost every} (standard) IET, by following remark (see e.g.~\cite{Ul:abs}):

\begin{remark}\label{rk:fullmeasfol}
To show that a result holds for minimal foliations on surfaces of genus $g\geq 1$ under a \emph{full measure} condition in the sense above, it is sufficent to prove that it holds for Lebesgue-almost every (standard) IET (in the sense of \S~\ref{sec:measures}).
\end{remark}
}

\subsubsection{Rigidity of foliations in genus two.}
Our rigidity Theorem \ref{maintheorem} can be reformulated in the language of foliation the following way, by first extending the definition of the Diophantine-type condition to foliations:
{
\begin{definition}[(RDC) for linear foliations]\label{def:DCfol}
 A linear foliation is said to satisfy the \textit{Regular Diophantine Condition} $(RDC)$ if there \emph{exists} a normal transverse arc such that the IET which arise as Poincar{\'e} section  satisfies the $(RDC)$ (given by Definition~\ref{def:RDC}).
\end{definition}}
\noindent {It is likely that if the $(RDC)$ holds for one choice of section, then it actually holds for \emph{any} IETs which arise from any other choice of normal sections (similarly to what one can show for example for the Roth-type condition for IETs, see the Appendix of \cite{MMY}), but we do not dwell into this, since is not needed for our purposes.} 

\begin{proposition}
\label{thm:rigidityfoliation}
Let $\mathcal{F}$ be a $\mathcal{C}^3$, {orientable, minimal} foliation on a surface $S_2$ of genus two. Assume:
\begin{itemize}
\item[(i)] $\mathcal{F}$ is topologically conjugate to a linear foliation $\mathcal{L}$ satisfying the $(RDC)$;
\item[(ii)] the $\mathcal{C}^1$-holonomies of $\mathcal{F}$ all vanish;
\end{itemize}
\noindent then $\mathcal{F}$ is actually $\mathcal{C}^1$-conjugate to $\mathcal{L}$.
\end{proposition}
{\begin{remark}\label{rk:H2} We remark that the statement of Proposition~\ref{thm:rigidityfoliation}  concerns not only the foliations in Theorem A, which have only (two) \emph{simple} saddles, but also foliations on $S_2$ with one degenerate saddle with $6$-prongs  (whose linear models are linear flows on translation surfaces in the  stratum $\mathcal{H}(2)$ of Abelian differentials  with a double zero), as long as the holomomy around the singularity vanishes (i.e.~$(ii)$ holds). This is the case when the leaves in a neighbourhood of a singularity are given by a level set of a $\mathcal{C}^1$ function with a zero of order two.
 \end{remark} }
\begin{proof}[Proof of Proposition~\ref{thm:rigidityfoliation}]
By definition of the $(RDC)$ for a (linear) foliation (Definition~\ref{def:DCfol}),  there exists an arc $J\subset S_2$ such that the Poincar{\'e} map of $\mathcal{L}$ to $J$, identified with $[0,1]$, is an IET $T_0$ which satisfies the $(RDC)$. The Poincar{\'e} map of the foliation $\mathcal{F}$ on $J$ is a GIET by Lemma~\ref{prop:folgiet}. By assumption $(i)$, $T$ and $T_0$ are topologically conjugate; the conclusion is equivalent to showing that the conjugacy is $\mathcal{C}^1$. Since the saddles of $\mathcal{F}$ are Morse, by  Proposition~\ref{prop:} the boundary $\mathcal{B}(T)=B(\log D T)=0$. Therefore we can apply Theorem B to conclude that the conjugacy is differentiable.
\end{proof}
\noindent The proof of Theorem A now also follows:
\begin{proof}[Proof of Theorem A]
By Theorem~\ref{thm:fullmeasure} and Remark~\ref{rk:fullmeasfol}, the $(RDC) $ is satisfied by a full measure set of minimal foliations (in the sense of \S~\ref{sec:foliationmeasure}) in genus two (and in any other genus). Moreover for a Morse saddle,  all holonomies vanish (see \S~\ref{sec:leavesregularity} and \S~\ref{sec:holonomies}). Thefore Theorem A follows from Proposition~\ref{thm:rigidityfoliation}.
\end{proof}}

\section{Full measure of the Regular Diophantine Condition}\label{sec:fullmeasureproof}
This section is fully devoted to the proof that the Regular Diophantine Condition has full measure, i.e.~Theorem~\ref{thm:fullmeasure}. 
The condition has three parts (see Definition~\ref{def:RDC}): the existence of an Oseledets regular extension, i.e.~condition $(i)$, follows simply from Oseledets theorem applied to the natural extension and is proved for a full measure set of IETs in \S~\ref{sec:fullmeasurei}, while the existence of good return times (condition $(ii)$ is easy to prove using ergodicity of Rauzy-Veech induction and is treated in \S~\ref{sec:existencegoodreturns}. The harder part to verify is Condition $(iii)$, in particular  the convergence of the series $(S)$ and $(B)$. 

As a key intermediate step towards the proof of convergence of the series in Condition $(iii)$, we define  in \S~\ref{sec:effectiveseqs} an acceleration of Zorich induction $\mathcal{Z}$ that we call \emph{effective Oseledets acceleration} and denote by $\widetilde{\mathcal{Z}}$. The accelerating times are given by a sequence $(n_k)_{k\in\mathbb{N}}$ where Oseledets theorem (for the natural extension) can be made \emph{effective}, i.e.~the hyperbolicity control (both in the future and in the past) can be made quantitative (see \S~\ref{sec:effectiveseqs} and Definition~\ref{def:effectiveOsedeletsseq}). We show that such sequences exist for almost every IET and, the same time, we can also guarantee that the accelerating sequence $(n_k)_{k\in\mathbb{N}}$ along which one has the effective Oseledets control is a sequence of good return times (see Proposition~\ref{prop:existenceeffective}). For the times $(n_k)_{k\in\mathbb{N}}$, one can prove that the series given by $(B)$ and $(S)$ both converge. The subsequence $(k_m)_{m\in\mathbb{N}}$ in the RDC (see Definition~\ref{def:RDC}, Condition $(iii)$) is then chosen to provide a further acceleration given by returns to a set which allows to ensure that the bounds on the series $(B)$ and $(S)$ are uniform, see \S~\ref{sec:prooffullmeasure}, which contains the proof of Theorem~\ref{thm:fullmeasure}, for details.

\subsection{Full measure of Conditions $(i)$ and $(ii)$}\label{sec:iandii}
Let us first verify that almost every IET has a Oseledets generic extension, i.e.~Condition $(i)$ of Definition~\ref{def:RDC} is satisfied (see \S~\ref{sec:fullmeasurei}) and has a sequence of good return times as required by Condition $(ii)$, see \S~\ref{sec:existencegoodreturns}.

\subsubsection{Full measure of generic Oseledets extensions}\label{sec:fullmeasurei}
Let us first prove that a.e.~IET has a generic Oseledets extension, in the sense of Definition~\ref{def:Oseledetsextension}. Full measure is here defined in the sense of the Lebesgue measure on $\mathcal{I}_d$  defined in \S~\ref{sec:measures}. The proof is simply an application of Oseledets theorem for the \emph{natural extension} of the Zorich accel

\begin{lemma}[Full measure of Condition $(i)$]\label{lemma:Oseledetsgeneric}
Lebesgue almost-every IET $T\in \mathcal{I}_d$ has a Oseledets generic extension.
\end{lemma}
\begin{proof}
By Remark~\ref{rk:ac} it is enough to show that for every fixed irreducibe $\pi\in\mathfrak{S}^0_d$, Lebesgue-almost every $T\in \mathcal{I}_\pi$ has a Oseledets generic extension (as defined in Definition~\ref{def:Oseledetsextension}). Consider the natural extension $\hat{\Zo}:\hat{\mathcal{I}_\pi}\to \hat{\mathcal{I}}_\pi$ of the Zorich acceleration ${\Zo}:{\mathcal{I}_\pi}\to {\mathcal{I}}_\pi$ defined in  \S~\ref{sec:natextension} 
 and let $p:\hat{\mathcal{I}_\pi}\to {\mathcal{I}_\pi}$ be the natural projection and ${\mu}_{\hat{\Zo}}$  the invariant measure  preserved by $\hat{\Zo}$, which gives $\mu_{\Zo}$ as pull back by $p$.  Recall from \S~\ref{sec:Oseledets} that the (extented) Zorich cocycle $Z: \hat{\mathcal{I}_\pi}\to SL(d,\mathbb{Z})$ is a cocycle over $\hat{\Zo}$ (see \S~\ref{sec:natextension}), is integrable w.r.t.~$\mu_{\hat{\mathcal{Z}}}$ and has Lypaunov exponents which, by the symplectic nature of the cocycle and the results by Forni \cite{Fo:ann} and Avila-Viana \cite{AV:sim} are:
\begin{equation}\label{exponents}
-\theta_1\leq -\theta_2\leq \cdots \leq -\theta_g < 0 < \theta_g \leq \theta_1 \leq \cdots\leq \theta_1.
\end{equation}
where $d=2g+\kappa-1$. 

Since $\hat{\Zo}:\hat{\mathcal{I}_\pi}\to \hat{\mathcal{I}}_\pi$ is ergodic w.r.t.~${\mu}_{\hat{\Zo}}$ and $Z$ is integrable w.r.t.~${\mu}_{\hat{\Zo}}$, the conclusions of Oseledets theorem for invertible maps hold for ${\mu}_{\hat{\Zo}}$-almost every $\hat{T}\in $ and gives the existence of an invariant splitting as in \eqref{splittingsextension}. 
Since $\mu_{\hat{\mathcal{Z}}}$ is the pull-back by $p$ of the measure $\mu_{{\mathcal{Z}}}$, it follows by Fubini theorem that for $\mu_{{\mathcal{Z}}}$-almost every $T$ there exists an extension $\hat{T}$ such that $p(\hat{T})=T$. Let  $E^{(n)}_i(\hat{T})$ be the subspaces given by Oseledets for $\hat{T}$. 
 For $x\in\{s,c,u\}$, if $E^{(n)}_x (\hat{T})$   are respectively the stable, central and unstable spaces for $\hat{T}$ given by Oseledets, we set $\Gamma^{(n)}_x:=E^{(n)}_x(\hat{T})$. Invariance and property $(H)$ therefore follow (the latter since $\lambda_g>0$ so both $\Gamma^{(n)}_s$ and $\Gamma^{(n)}_xu$ have dimension $g$. 
 
Since $Z_n(T)=Z_n(\hat{T})$ for every $n\in\mathbb{N}$ (by definition of the cocycle  extension, see \S~\ref{sec:Oseledets}), the remaining properties in the Definition~\ref{def:Oseledetsextension} then from the conclusions of Oseledets theorem for $\hat{T}$: (\ref{O-c}) and  (\ref{O-a}) are immediate from \eqref{eq:Oseledets} and \eqref{Oseledetsangles}; both (\ref{O-s}) and  (\ref{O-u})  hold for any choice of $0<\theta<\theta_g$ (with a constant $C>0$ depending on $\hat{T}$ and the choice of $\theta$: (\ref{O-s})  follows directly from  \eqref{eq:Oseledets} for $i<0$, to verify  (\ref{O-u}) 
 notice first that  \eqref{eq:Oseledets} for $i>0$, since $\theta<\theta_g$,  implies that 
$$\lim_{n\to \infty} \frac{\log \left\|\left. Q(0,n)\right|_{\Gamma^{(0)}_u}\right\|}{n} > \theta>0,$$ where $ \left. Q(0,n)\right|_{\Gamma^{(0)}_u}$ denotes the restriction of $Q(0,n)$ to the invariant space $\Gamma^{(0)}_u$. Therefore, there exists $c>0$ such that
$$
\| Q(0,n)\,  w\|\geq c\,  e^{n\theta} \| w\|, \qquad \text{for\ all}\ w\in \Gamma^{(0)}_u, \  \text{for\ all}\  n\in\mathbb{N}.
$$
Given any $v\in \Gamma^{(0)}_u$, consider $w:=  Q(0,n)^{-1}v$ which by invariance of the splitting belongs to $\Gamma^{(0)}_u$. Then, by the previous inequality applied to $w$, we get (\ref{O-u}).
\end{proof}
\subsubsection{Construction of good return times}\label{sec:existencegoodreturns}
Condition $(ii)$ in Definition~\ref{def:RDC} concerns the existence of \emph{good return times}. Let us show that sequences of good return times can be easily constructed exploiting visits to certain sets.  Recall from \S~\ref{sec:goodreturns} that we say that $A\in SL(d,\mathbb{Z})$ is a Zorich cocycle matrix of (Zorich) length $p$ if it can be obtained as product of $p$ matrices of the Zorich cocycle, i.e.~$A=Q(0,p)(T_A)$ for some $T_A\in \mathcal{I}_d$.
\begin{lemma}[Construction of good return times]\label{lemma:existencegoodreturns}
There exists a positive Zorich matrix $A\in SL(d,\mathbb{Z})$ and a set ${\hat{G}}_{A}\subset \hat{\mathcal{I}}_d$  with $\mu_{\hat{\Zo}}({\hat{G}}_{A})>0$ such that if $(n_k)_{k\in\mathbb{N}}$ is such that $\mathcal{Z}^{n_k}(\hat{T}) \in {\hat{G}}_A$ for all $k\in\mathbb{N}$, then $(n_k)_{k\in\mathbb{N}}$ is a sequence of good return times for $T=p(\hat{T})$. 
\end{lemma}
\begin{proof}
Let $A$ be a positive $d\times d$ matrix which can be obtained as product of $p$ matrices of the Zorich cocycle, i.e.~such that $A=Q(0,p)(T_A)$ for some $T_A\in \mathcal{I}_d$. Let $T_A$ have combinatorial datum $\pi$.  
Consider the set $G_A\subset\mathcal{I}_d $ given by 
$$
G_{A}:= \Delta_{A^2}\times \{\pi\}, \qquad \text{where}\ \Delta_{A^2}:= \left\{ \lambda = \frac{A^\dag A^\dag \lambda'}{\|A^\dag A^\dag  \lambda' \|}, \ \lambda'\in \Delta_{d-1} \right\}.
$$
\noindent Then, if we consider $T$ given by $\lambda$ and $\pi$ with $\lambda\in \Delta_{A^2}$, then 
$Z^{(2p)} (T) = A A $, i.e.~the Zorich cocycle matrices at $T$ start with a double occurence of $A$ (see \S~\ref{sec:lengthcocycle}).
It follows  if $\mathcal{Z}^{n_k}(T)\in G_A$, then $Q(n_k, n_k+2p)=AA$ (since $\mathcal{Z}$ acts as a shift on the sequence of matrices $(Z_n)_{n\in\mathbb{N}}$ associated to $T$, i.e.~the sequence of cocycle matrices associated to $\mathcal{Z}^{n_k}(T)$ is $(Z_{n_k+n})_{n\in\mathbb{N}}$). 

Let us now define ${\hat{G}}_{A}:= p^{-1}(G_{A})$ where $p$ is the projection $p:\hat{\mathcal{I}}_\pi\to \mathcal{I}_\pi$. Recall that if $\hat{T}$ is such that $p(\hat{T})= T$, then $Z^{(n)}(\hat{T}) = Z^{(n)}(T)$ for every $n\in\mathbb{N}$. Therefore, if 
$Z^{(n_k)}(\hat{T})\in {\hat{G}}_{A}$, we have that $Q(n_k, n_k+2p)= A A$. This shows that visits to ${\hat{G}}_{A}$ produce sequences of good returns as desired.
\end{proof}
\subsection{Effective Oseledets}\label{sec:effectiveOseledets}
We are going to consider sequences where the estimates given by Oseledets theorem can be quantified in an \emph{effective} with uniform  constants along the sequence. 

\subsubsection{Effective Oseledets return times}\label{sec:effectiveseqs}
 Let $\hat{T}\in \hat{\mathcal{I}_\pi}$ be Oseledets generic for the (extension of the) Zorich cocycle $Z$ over the Zorich natural extension $\hat{\mathcal{Z}}$.  Let $\Gamma_x^{(n)}$ for $x\in \{s,c,u\}$, $n\in\mathbb{Z}$ be the spaces given by the Oseledets splittings for $\hat{T}$. Recall that, for any pairs of non negative integers $m<n$,  $Q(m,n)$  denotes the matrices of the Zorich cocycle (see \S~\ref{sec:Zorichcocycle}) and that $Q(m,n)$ maps $\Gamma_x^{(m)}$ 
 to  $\Gamma_x^{(n)}$  for any $x\in \{s,c,u\}$.

\begin{definition}[Effective Oseledets sequence]\label{def:effectiveOsedeletsseq} Given $C_1>0$ and $\epsilon>0$, a sequence $(k_m)_{m\in\mathbb{N}}$ is a $(C_1,\epsilon)-$\emph{effective Oseledets sequence} for $\hat{T}$, if for some $\theta>0$ we have:
\begin{align}\label{estimateEO1}
\tag{EO1} ||Q(n_k,n)\vert_{\Gamma_s^{(n_k)}}||_{\infty} \leq C_1 e^{-\theta(n-n_k)} & \qquad \textrm{for\ every} \ n\geq n_k,  \\
\tag{EO2} ||Q(n,n_k)^{-1}\vert_{\Gamma_u^{(n_k)}}||_{\infty}  \leq C_1 e^{-\theta(n_k-n)} & \qquad \textrm{for\ every}\   n\leq n_k, \label{estimateEO2}
\end{align}
and furthermore, for some $c_2(\epsilon)>0$, the angle $\angle (\Gamma^{(n)}_x, \Gamma^{(n)}_y)$ for \emph{distinct} $x,y\in\{s,c,u\}$ between $\Gamma^{(n)}_x $ and $ \Gamma^{(n)}_y$ (defined as in \S~\ref{sec:Oseledets}) satisfies 
\begin{align}\label{expsmall_fromki}
\tag{EO3} |\angle (\Gamma^{(n)}_x, \Gamma^{(n)}_y)|\geq c_2(\epsilon)\ e^{-\epsilon|n-n_k|}, \qquad \textrm{for\ all\ } n\in\mathbb{Z},
\end{align}
and the Zorich cocycle grows subexponentially along the subsequence, i.e.
\begin{equation*}
\tag{EO4} \lim_{k \rightarrow +\infty}{\frac{\log ||Q(n_k,n_{k+1})||}{k}} = 0  \label{S}
\end{equation*}
We say that $(n_k)_{k\in\mathbb{N}}$ is an \emph{effective Oseledets sequence} for $\hat{T}$ if it is $(C_1,\epsilon)-$\emph{effective Oseledets acceleration sequence} for some $\epsilon>0, C_1>0$. 
 \end{definition}

\subsubsection{Construction of effective Oseledets sequences}\label{sec:existenceeffective}
 Effective Oseledets sequences (see Definition~\ref{def:effectiveOsedeletsseq} above) can be obtained (using Oseledets and Lusin's theorems) as \emph{return times} to certain \emph{good sets} sequences for the natural extension $\hat{\mathcal{Z}}$ (see \S~\ref{sec:fullmeasure} for details). We stress that working with the natural extension is an essential technical tool to impose \emph{backward} conditions like \eqref{estimateEO2} and \eqref{expsmall_fromki} for $n\leq n_k$ (see \S~\ref{sec:fullmeasure} for details). Exploiting ergodicity of $\hat{\mathcal{Z}}$, let us show that:

\begin{proposition}[Effective Oseledets good returns]\label{prop:existenceeffective}
For any irreducible $\pi\in\mathfrak{S}^0_d$, there exists $\epsilon_0>0$ such that for every $0<\epsilon<\epsilon_0$, there exists a constant $C>0$ such that 
 $\mu_{\hat{\Zo}}$-almost every $\hat{T}\in \hat{\mathcal{I}}_\pi$ 
which	admits a $(C,\epsilon)$-effective Oseledets sequence $(n_k)_{k\in\mathbb{Z}}$ which is also a sequence of good returns for $T=p(\hat{T})$. Furthemore, the sequence is given by considering returns of the forward orbit $\{\mathcal{Z}^n\hat{T}, \, n\in\mathbb{N}\}$ to a set ${\hat{G}}\subset \hat{\mathcal{I}}_\pi$.
\end{proposition}
To prove Proposition~\ref{prop:existenceeffective}, we will construct a 
 \emph{good} set in $\hat{\mathcal{I}}_d$, denoted by  $\hat{G}$ for \emph{Good} (the \emph{hat} is to stress that it is a set in the domain $\hat{I}_d$ of the natural extension)  such that visits to $\hat{G}$ produce effective Oseledets sequences. In addition, intersecting with the set ${\hat{G}}_A$ given by Lemma~\ref{lemma:existencegoodreturns} in \S~\ref{sec:goodreturns}, we can get sequences of good returns (see Definition~\ref{def:goodreturns}) where the Oseledets growth is effective. These will provide  the accelerating sequences $(n_k)_{k\in\mathbb{N}}$ which appear in the RDC (see Definition~\ref{def:RDC}).

\begin{proof}[Proof of Proposition~\ref{prop:existenceeffective}] Fix $\pi\in\mathfrak{S}^0_d$ irreducible  and let $\mu_{\mathcal{Z}}$ be the invariant measure for the natural extension $\hat{\mathcal{Z}}:\hat{\mathcal{I}}_\pi \to \hat{\mathcal{I}}_\pi$ (recall \S~\ref{sec:natextension} and \S~\ref{sec:measures}).   
Let us first of all construct the  \emph{good set} ${\hat{G}}\subset \hat{\mathcal{I}_\pi}$ of $\mu_{\hat{\mathcal{Z}}}$-positive measure where the control given by Oseledets theorem holds uniformely.

\smallskip
\noindent \paragraph{{\it Construction of the good set ${\hat{G}}$.}}  Let $\hat{T}\in \hat{\mathcal{I}}_\pi$ be Oseledets generic. Recall that $Z^{(n)}:=Z^{(n)}(\hat{T})$,  for $n\in\mathbb{Z}$, and $Q(m,n)$, for $m<n$, denote its Zorich cocycle matrices, as defined as in \S~\ref{sec:Zorichcocycle} . 
Since the cocyle ${Z}$ has the Lyapunov exponents in \eqref{exponents},  if we denote by $\Gamma_x^{(n)}:= E_x^{(n)}(\hat{T})$ for $x\in\{x,s,u\} $ respectively the stable, central and unstable space given by Oseledets (see \S~\ref{sec:Oseledets}), 
for every $\epsilon>0$ there exists a constant $C_1(\epsilon, \hat{T})>0$ such that, for all $n \geq 0$, for all $v_s$ in $\Gamma_s^{(n)}$, 
\begin{equation}\label{controlfuture}||Q(0,n) v_s||_{\infty} \leq C_1(\epsilon, \hat{T}) e^{-(\theta_g -\epsilon) n} ||v_s||_{\infty},\end{equation}
where $-\theta_g<0$ is the largest negative exponent (smallest in absolute value) with respect to $\mu_{\hat{\mu}}$, see \eqref{exponents}. Moreover, by the symmetry of exponents recalled in \eqref{exponents} above,  $\theta_g>0$ is also the smallest positive exponent (see again \eqref{exponents}), so that for all $n>0$ and all  
$v_u$ in $\Gamma_u^{(n)}$,  we also have
\begin{equation}\label{controlpast}||(Q(-n,0))^{-1} v_u||_{\infty} \leq C_1(\epsilon, \hat{T}) e^{-(\theta_g -\epsilon )n} ||v_u||_{\infty}\end{equation}
(where we can assume without loss of generality that the constant $C_1=C_1(\epsilon, \hat{T})$ is the same than above).

Fix a positive $\epsilon$ such that $\epsilon<\theta_g/2$. Let now ${\hat{G}}_A$ be the set given by Lemma~\ref{lemma:existencegoodreturns} and consider its measure $0<\mu_{\hat{\Zo}}({\hat{G}}_A)<1$.
Since the constant $C_1(\epsilon,\hat{T})$ depends measurably on $\hat{T}$, by Lusin theorem, for some fixed $C_1=C_1(\epsilon,\mu_{\hat{\Zo}}({\hat{G}}_A))>0$ sufficiently large, we can find a set ${\hat{G}}_1={\hat{G}}_1(\epsilon, A)\subset \hat{\mathcal{I}}_\pi$ of measure $\hat{\mu}_{\mathcal{Z}}({\hat{G}}_1)> 1-\mu_{\hat{\Zo}}({\hat{G}}_A)/2$ such that, for every $\hat{T}'$ which belongs to ${\hat{G}}_1$, \eqref{controlfuture} and \eqref{controlpast} hold uniformely, i.e.~one has $C_1(\epsilon,\hat{T}')\leq C_1$. Thus, for every $n_k\in\mathbb{N}$ such that $\mathcal{Z}^{n_k}(T)\in {\hat{G}}_1$, we have (recalling that  $0<\epsilon<\theta_g/2$), 
\begin{align}\label{estimateQF}
||Q(n_k,n)\vert_{\Gamma^{(s)}(\mathcal{Z}^{n}(T)}||_{\infty} \leq C_1 e^{-(\theta_g -\epsilon)(n-n_k)}\leq C_1 e^{-(\theta_g/2)(n-n_k)} & \qquad \textrm{for\ every}\ n\geq n_k, \\ \label{estimateQB}
||Q(n,n_k)^{-1}\vert_{\Gamma^{(u)}(\mathcal{R}^{n}(T)}||_{\infty} \leq C_1 e^{-(\theta_g -\epsilon)(n_k-n)} \leq C_1 e^{-(\theta_g/2)(n_k-n)} & \qquad \textrm{for\ every}\  n\leq n_k.
\end{align}
{Moreover, Oseledets theorem 
 applied respectively to the cocycle $Z$ over $\hat{\Zo}$ and to the inverse cocycle $Z^{-1}$ over $\hat{\Zo}^{-1}$ (see \S~\ref{sec:Oseledets}), 
gives that, for almost every $\hat{T}\in \hat{\mathcal{I}}_\pi$, 
if we consider the $\angle (\Gamma^{(m)}_x, \Gamma^{(m)}_y)$ the angle between any two distinct pairs of spaces $\Gamma^{(m)}_x$ and $\Gamma^{(m)}_y$ with $x,y\in\{s,c,u\}$ (see \S~\ref{sec:Oseledets}), we have 
\begin{equation}
\label{suexpgrowth}
\lim_{m\to \pm \infty}\frac{ \log |\angle (\Gamma^{(m)}_x, \Gamma^{(m)}_y)|}{|m|}
= 0, \qquad \text{for\ all}\ x,y\in\{s,c,u\},\ x\neq y.
\end{equation}
 By Egoroff theorem, this pointwise almost everywhere convergence can be upgraded to uniform convergence, i.e. there exists a set ${\hat{G}}_2={\hat{G}}_2(\epsilon,A)\subset \hat{\mathcal{I}}_\pi$, with ${\mu}_{\hat{\Zo}}({\hat{G}}_2)>1-\mu_{\hat{\Zo}({\hat{G}}_A)}/2$, such that,   for all $\epsilon>0$ exists a constant $c_2(\epsilon)>0$ such that, for all 
 $\hat{T}\in {\hat{G}}_2$
and for all $m\in\mathbb{Z}$  
$$
|\angle (\Gamma^{(m)}_x, \Gamma^{(m)}_y)|\geq c_2 (\epsilon)\ e^{-\epsilon|m|}\qquad \text{for\ all}\ x,y\in\{s,c,u\},\ x\neq y.
$$
Notice that, for any $n_k\in\mathbb{Z}$, 
 $$\Gamma_x^{(n_k+m)} =E_x\big(\Zo^{m+n_k}(\hat{T})\big)\big) 
=E_x^{(m)}\big(\Zo^{n_k}(\hat{T})\big) , \qquad \text{for\ all\ } x\in \{s,c,u\}, \ \text{for\ all\ }  m\in \mathbb{Z},$$
so that, if $\hat{\Zo}^{n_k}(\hat{T})\in {\hat{G}}_2$, (taking as index $m:=n-n_k$ and applying the above subexponential growth estimate \eqref{suexpgrowth} to $\hat{\Zo}^{n_k}(\hat{T})$), we have, for any distinct $x,y\in\{s,c,u\}$,
\begin{equation}\label{expsmall_fromki}
|\angle (\Gamma^{(n)}_x, \Gamma^{(n)}_y)|\geq c_2(\epsilon)\ e^{-\epsilon|n-n_k|}, \qquad \textrm{for\ all\ } n\in\mathbb{Z}. 
\end{equation}
Notice in particular that, for $n=n_k$, this gives that
\begin{equation}\label{lowerboundA}
|\angle (\Gamma^{(n_k)}_x, \Gamma^{(n_k)}_y)|\geq c_2(\epsilon), \qquad \textrm{for\ all\ distinct}\ x,y\in\{s,c,u\},\ \text{for\ all\ }  n\in\mathbb{Z}. 
\end{equation}}
\noindent Define the good set ${\hat{G}}={\hat{G}}(\epsilon)$ to be ${\hat{G}}:= {\hat{G}}_1(\epsilon)\cap{\hat{G}}_2(\epsilon)\cap {\hat{G}}_A$. Notice that since by construction $\mu_{\hat{\Zo}}({\hat{G}}_i)>1-\mu_{\hat{\Zo}}({\hat{G}}_A)$ for $1=1, 2$, and ${\hat{G}}$ has positive measure ${\mu}_{\hat{\Zo}}({\hat{G}})>0$.

\smallskip
\noindent \textit{Final arguments.} Set $\epsilon_0:=\theta_g/2$. For $0<\epsilon<\epsilon_0$, let ${\hat{G}}:={\hat{G}}(\epsilon)$ be the good set constructed above and define $\theta:=\theta_g/2$ and  $C:=C_1(\epsilon)$ to be the constant for which \eqref{estimateEO1} and \eqref{estimateEO2} hold.  
By ergodicity of $\hat{\mathcal{Z}}$, since ${\mu}_{\hat{\Zo}}({\hat{G}})>0$, it follows that ${\mu}_{\hat{\Zo}}$-almost every $\hat{T}\in \mathcal{I}_\pi$ will visit ${\hat{G}}$
 infinitely often. 
Set $(n_k)_{k\in\mathbb{N}}$ to be the successive visits of the forward orbit $\{\hat{\mathcal{Z}}^n(\hat{T}), n\in\mathbb{N}\}$ to ${\hat{G}}_A$, i.e.~we set by convention $n_0:=0$ and, given $n_k$ for $k\in\mathbb{N}$, we let $n_{k+1}$ to be the minimum $n>n_k$ such that $\hat{\mathcal{Z}}^n(\hat{T})\in {\hat{G}}$. Then,  $(n_k)_{k\in\mathbb{N}}$ is by construction a $(C_1,\epsilon)$-Oseledets effective sequence for $\hat{T}$ and, by Lemma~\ref{lemma:existencegoodreturns}, since ${\hat{G}}\subset {\hat{G}}_A$, $(n_k)_{k\in\mathbb{N}}$ is also a sequence of good returns for $T:=p(\hat{T})$. 
\end{proof}

\subsection{Control of the series in Condition $(iii)$ and proof of full measure.}\label{sec:prooffullmeasure}
We can now use the partial results proved so far to give the proof the the RDC has full measure. 

\begin{proof}[Proof of Theorem~\ref{thm:fullmeasure}]
Fix $\pi\in\mathfrak{S}^0_d$ irreducible. Let $\hat{G}$ be the good set given by Proposition~\ref{prop:existenceeffective} constructed in \ref{sec:existenceeffective}. By Proposition~\ref{prop:existenceeffective}, $\mu_{\hat{\mathcal{\Zo}}}$-almost every $\hat{T}$  the forward orbit of $\hat{T}$ under $\hat{\mathcal{Z}}$ visits $\hat{G}$ infinitely often along a sequence 
 $(n_k)_{k\in\mathbb{N}}$  of return times, which
 is  an effective Oseledets sequence (and also a sequence of $A$-good return times for some positive $A$ for $T=p(\hat{T})$). We now want to impose a better control on the frequency of visits.

\smallskip
\noindent \paragraph{\it Frequency of recurrence times to ${\hat{G}}$ and good set of IETs.} 
Since both $\hat{\Zo}$ and $\hat{\Zo}^{-1}$ are ergodic and ${\hat{G}}$ has positive measure i.e.~${\mu}_{\hat{\mathcal{Z}}}({\hat{G}})>0$, almost every $\hat{T}\in \hat{\mathcal{I}}_\pi$  is Birkhoff generic for the characteristic function $\chi_{\hat{G}}$ of the good set ${\hat{G}}$, so that its orbit visits infinitely often the set ${\hat{G}}$ and with the expected frequency both in the past and in the future,~i.e.
\begin{equation}\label{Birkhoffgood}
\lim_{n\to+\infty}\frac{1}{n}\sum_{k=0}^{n-1}\chi_{{\hat{G}}} \left(\hat{\Zo}^k(\hat{T})\right) =\lim_{n\to-\infty}\frac{1}{n}\sum_{k=0}^{n-1}\chi_{\hat{G}} \left(\hat{\Zo}^{-k}(\hat{T})\right) =\mu_{\Zo}({\hat{G}}).
\end{equation}
 Furthermore, it follows from Fubini theorem for the measure ${\mu}_{\hat{\Zo}}$ and the foliation into fibers $\{ p^{-1}(T), T\in {\mathcal{I}}_\pi\}$ of the natural projection map $p:\hat{\mathcal{I}}_\pi\to {\mathcal{I}}_\pi$, that for set $\mathcal{G}^0_\pi\subset \mathcal{I}_\pi$ of $T\in {\mathcal{I}}_\pi$  of  full measure w.r.t.~$\mu_{\mathcal{Z}}$ 
 we have that $p^{-1}(T)$ contains at least one $\hat{T}$ for which \eqref{Birkhoffgood} holds. By Remark~\ref{rk:ac}, the set $\mathcal{G}_0:=\bigcup_{\pi\in\mathfrak{S}^0_d}\mathcal{G}^0_\pi\subset \mathcal{I}_d$ has full Lebesgue measure.

\smallskip
We will show that for any IET in the set $\mathcal{G}_0$ we can find a sequence $(n_k)_{k\in\mathbb{N}}$  which can be used verify the properties in the  Definition~\ref{def:RDC} of RDC for a IET. We will later need to refine further this set (keeping it still of full measure) to also guarantee the existence of a subsequence $(n_{k_m})_{m\in\mathbb{N}}$ on which $(iii)$ holds. 

Given any $T\in\mathcal{G}_0$, pick a $\hat{T}\in p^{-1}(T)$ (which exists by definition of $\mathcal{G}_0$) and  consider the (infinite) sequence $(n_k)_{k\in\mathbb{N}}$ of successive visits  of the forward orbit of $\hat{T}$ under the iterations of $\hat{\Zo}$  to the set ${\hat{G}}$ (which again exists by definition of $\mathcal{{\hat{G}}}_0$, since \eqref{Birkhoffgood} holds): more precisely, we set $n_0:=0$ and we let $n_1$ be the first entrance time in ${\hat{G}}$, i.e.~the minimum $n\geq 0$ such that  $\hat{\Zo}^{n}(\hat{T})\in {\hat{G}}$  and for any $k>0$, let $n_{k+1}$ be the first return time of   $\hat{\Zo}^{n_k}(\hat{T})$ to $ {\hat{G}}$ under $\Zo$. 

\subsubsection{Induced cocycle}
Denote by $\widetilde{Z}_k$, $k\in\mathbb{Z}$, and $\widetilde{Q}(k,k')$ for $k'>k$ the matrices of the cocycle accelerations along  $(n_k)_{k\in\mathbb{N}}$, namely
\begin{equation}\label{accelerations}
\widetilde{Q}(k,k'):= Q(n_k,n_{k'}), \qquad  \widetilde{Z}_k=\widetilde{Z}_k(T):= \widetilde{Q}(k,k+1)= Q(n_k, n_{k+1}).
\end{equation}
  By definition of the recurrence sequence $(n_k)_{k\in\mathbb{N}}$, $$\hat{T}_0:=\hat{\Zo}^{n_0}(\hat{T})\in {\hat{G}}, \qquad \hat{\Zo}^{n_k}(\hat{T})=\hat{\Zo}_{\hat{G}}^k(\hat{T}_0)  \quad \text{for\ all}\ k>0,$$
so that the matrices $\widetilde{Z}_k, k\geq 1$ can be seen as the iterates of the cocycle  $\widetilde{Z}$ over the induced map $\hat{\mathcal{Z}}_{\hat{G}}$ of $\hat{\mathcal{Z}}$ over $\hat{G}$ (see \S~\ref{sec:inducing})  starting from $\hat{T}_0$. 
It follows from integrability of the Zorich cocycle $Z$ over $\mathcal{Z}$ and \S~\ref{sec:integrability}) that $\widetilde{Z}$ is still an integrable cocycle over  $\mathcal{Z}_{\mathcal{G}}$ and therefore admits Oseledets splittings (see \ref{sec:Oseledets}).  Moreover, if   $\Gamma^{(n)}_{x}$, for $x\in \{s,c,u\}$ are the stable, central and unstable spaces given by Oseledets theorem for $\hat{T}$ and $P^{(n)}_x$ the respective orthogonal projections defined in \eqref{projections}, 
the stable, central and unstable spaces $\widetilde{\Pro}^{(k)}_x$, $x\in\{s,c,u\}$, $k\in\mathbb{Z}$, for the accelerated cocycle $\widetilde{Z}$  and the corresponding orthogonal projections $\widetilde{\Pro}^{(k)}_x$ are given by
\begin{equation*}
\widetilde{\Gamma}_x^{(k)}:= \Gamma_x^{(n_k)}, \qquad \widetilde{\Pro}_{x}^{(k)}: \mathbb{R}^d\to  \widetilde{\Gamma}_x^{(k)}, \qquad \text{for\ all}\ x\in\{s,c,u\}, \ k\in\mathbb{Z}.
\end{equation*}
 We will show at the end that the sequence $(n_k)_{k\in\mathbb{N}}$ verifies condition $(i)$ and $(ii)$ in the Definition~\ref{def:RDC} of the RCD for $T$ and that, up to restricting to a smaller full measure set, we can extract a subsequence $(n_{k_m})_{m\in\mathbb{N}}$ along which also $(iii)$ holds.   

\smallskip
\noindent{\it Linearity of returns and uniform convergence times.}
From the Birkhoff genericity statement in \eqref{Birkhoffgood}, it follows that $(n_k)_{k\in\mathbb{N}}$ grows linearly. Indeed, since $n_k$ is by definition the time of the $n^{th}$ visit to ${\hat{G}}$, for all $\hat{T}\in\mathcal{{\hat{G}}}$, \eqref{Birkhoffgood} gives that $\lim_{k\to \pm \infty} k/n_k=\mu_{\mathcal{Z}({\hat{G}})}$. Moreover, by Egoroff theorem, there exists sets ${\hat{G}}_{B^\pm}\subset {\hat{G}}$ with $\mu_{\mathcal{Z}}({\hat{G}}_{B^\pm})>5/6$ on which this convergence (in the past and in the future respectively) is uniform, i.e.~there exists  a constant $c_{B}>0$ such that if $n_m$ is such that $\Zo^{n_m}(\hat{T})\in {\hat{G}}_B$, then for every $n_k>n_m$
$$0<c_B\leq \frac{\sum_{i=0}^{n_k-n_m}\chi_{\hat{G}} \left(\hat{\Zo}^{ i}(\Zo^{n_m}(\hat{T}))\right) }{n_k-n_m}\leq 1$$
and, similarly, for any  $n_k<n_m$,
$$0<c_B \leq \frac{\sum_{i=0}^{n_m-n_k}\chi_{\hat{G}} \left(\hat{\Zo}^{- i}(\Zo^{n_m}(\hat{T}))\right) }{n_m-n_k}\leq 1.$$
Thus, setting ${\hat{G}}_B:={\hat{G}}_{B^+}\cap {\hat{G}}_{B^-}$, $\mu({\hat{G}}_B)>4/6$ if $n_m$ is such that $\Zo^{n_m}(\hat{T})\in {\hat{G}}_B$, 
since $|m-k|$ is exactly the number of visits to ${\hat{G}}$ in the orbit segment considered, we have that%
\begin{equation}\label{uniform linearity}
c_B  \, |n_{m}-n_k| \leq  |m-k|  \leq |n_{m}-n_k|,  \qquad \text{for\ all} \ k\in\mathbb{Z}.
\end{equation}

\smallskip
\noindent {\it Uniform subexponential growth.} 
 Let us now estimate the growth of the accelerated matrices $\widetilde{Z}_k:= Q(n_k, n_{k+1})$, $k\in\mathbb{N}$ and show that it is subexponential. This fact will be used later in the proof of the convergence of the series $(B)$ and $(F)$. Remark first of all that, since $\hat{T}\in p^{-1}(T)$ and the (forward) iterates of the cocycle $R$ (or $Z$) depend only on $p(\hat{T})$, these matrices  are the same for $T$ and $\hat{T}$. 
Since, as already remarked at the beginning, $\widetilde{Z}$ is an integrable cocycle over the invertible map $\hat{\widetilde{Z}}_{\hat{G}}$,   
it follows from Oseledets (see in particular \eqref{subexpgrowth} in \S~\ref{sec:Oseledets}) and ergodicity (which guarantees that $\mu_{\hat{\Zo}}$-almost every $\hat{T}$ will enter the full measure set of IETs in $\hat{G}$ which is Oseledets generic for $\hat{\Zo}_{\hat{G}}$), 
that, for almost every $\hat{T}\in \hat{\mathcal{I}}_\pi$, 
\begin{equation}\label{Sgrowth}
\lim_{\ell \to \pm \infty}\frac{ \log ||\widetilde{Z}_\ell(\hat{T}_1)||}{|\ell|} = 0.
\end{equation}
  Furthermore,  once again by Egoroff theorem,   there exists a set ${\hat{G}}_S\subset {\hat{G}}\subset \hat{\mathcal{I}}_\pi$ with ${\mu}_{\hat{\Zo}}({\hat{G}}_S)>0$ (which actually can be made arbitrarily close to $\hat{\mu}_{\Zo}({\hat{G}})$) such that, for every $\epsilon>0$ there exists a constant $C_3(\epsilon)>0$ 
such that, for all $\hat{T}_1\in {\hat{G}}_S$ 
$$
 ||\widetilde{Z}_\ell(\hat{T}_1 )||\leq C_3 (\epsilon) e^{\epsilon|\ell |}, \qquad \text{for\ all} \ \ell \in\mathbb{Z}.
$$
Notice that, if $k_m\in\mathbb{Z}$ is  such that $\Zo_{\hat{G}}^{k_m}(\hat{T})\in {\hat{G}}_S$, then (since the cocycle matrices for $\hat{T}^{(k_m)}:=\hat{\Zo}_{\hat{G}}^{k_m}(\hat{T})$ are a shifted copy of the matrices for $\hat{T}$, namely $\widetilde{Z}_\ell(\hat{T}^{(k_m)})= \widetilde{Z}_{k_m+\ell}(\hat{T})$ for all $\ell\in\mathbb{Z}$) this implies (choosing $\ell=k-k_m+1$ and applying the previous estimate to $\hat{T}^{(k_m)}$) that  we have
\begin{equation}\label{expsmall_fromki}
 ||\widetilde{Z}_{k-1}(\hat{T})||\leq C_4  (\epsilon) e^{\epsilon|k-k_m|}, \qquad \textrm{for\ all\ } k\in\mathbb{Z}. 
\end{equation}

\smallskip
\noindent{\it The subsequence $(k_m)_{m\in\mathbb{N}}$.} 
Let ${\hat{G}}_{SB}\subset {\hat{G}}$ be by ${\hat{G}}_{SB}:={\hat{G}}_B\cap {\hat{G}}_S$ where ${\hat{G}}_S$ and ${\hat{G}}_B$ where the sets for uniform Birkhoff convergence and subexponential growth defined in the previous paragraphs. Remark that $\mu_{\Zo}({\hat{G}}_{SB})>0$. 
Define finally the sequence $(k_m)_{m\in\mathbb{N}}$ to be the subsequence of indexes $k\in\mathbb{N}$ which corresponds to visits of the orbit of $\hat{T}_0$ under $\hat{\Zo}_{\hat{G}}$ to the subset ${\hat{G}}_{SB}$ defined above, i.e.~$k_0$ is the first entrance time of  $\hat{T}_0$ to ${\hat{G}}_{SB}$ while, for every $m>0$, $k_{m+1}$ is defined to be the smallest $k>k_{m}$ such that $\hat{\Zo}_{{\hat{G}}}^k(\hat{T})\in {\hat{G}}_{SB}$.  

\smallskip \noindent \textit{Linear growth in $\mathcal{G}_L$.}
Linear growth of the sequence  $(k_m)_{m\in\mathbb{N}}$ for a full measure set of $\hat{T}$ can then be deduced from ergodicity as follows. 
Since the set ${\hat{G}}$ has positive measure with respect to $\mu_{\Zo}$ 
and hence for the induced invariant measure $\mu_{\Zo_{\hat{G}}}$ for the Poincar{\'e} map $\hat{\Zo}_{\hat{G}}$ and $\hat{\Zo}_{\hat{G}}$ (being the induced map of an ergodic map) is ergodic with respect to $\mu_{\Zo_{\hat{G}}}$, it follows that for all $\hat{T}$ in a subset 
${\hat{{G}}}_L\subset {\hat{{G}}}$ with $\mu_{{\Zo}_{\hat{{G}}}}({\hat{{G}}}_L)=\mu_{{\Zo}_{\hat{G}}}({\hat{{G}}})=1$, the orbit of $\hat{T}$ under $\Zo_{\hat{G}}$ vists ${\hat{G}}_{SB}$ with the expected frequency, namely, since, for any $m \in \mathbb{N}$,  
 $m$ is exactly the number of visits to ${\hat{G}}_{SB}$  in the piece of orbit $\{ \hat{\mathcal{Z}}_{\hat{G}}^k(T),\ 0\leq k< k_{m}\}$ (since  $k_0\geq 0$ is the time of first visit to ${\hat{G}}_{SB}$ and hence $k_{m-1}$ corresponds to the $m^{th}$ visit), we have that
\begin{equation}\label{linearity0}
\lim_{m\to\infty}\frac{m}{k_{m}}= \lim_{m\to\infty} \frac{Card \{\, k:\, \Zo_{\hat{G}}^k(\hat{T})\in {\hat{G}}_{SB},\ 0\leq k < k_{m}\} }{k_m} = \mu_{\mathcal{Z}_{\hat{G}}}({\hat{G}}_{SB})>0. 
\end{equation}
Thus, if $\hat{T}$ is such that the first return  $\hat{T}_0:=\hat{\Zo}^{n_0}(\hat{T})\in \hat{G}$ belongs to $\hat{G}_L$, the corresponding subsequence $(k_m)_{m\in\mathbb{N}}$  has linear growth.

\smallskip

\noindent{\it The set of IETs which satisfy the RDC.} 
We can now define the set $\mathcal{{{G}}}_\pi\subset \mathcal{{{G}}}^0_\pi$ of (standard) IETs in $\mathcal{I}_\pi$ which satisfy the RDC to be the set of IETs $T\in \mathcal{{{G}}}_\pi^0$ such that:
\begin{itemize}
\item[(a)] $T$ has an Oseledets generic extension $\hat{T}\in p^{-1}(T)$;
\item[(b)] the forward orbit  under $\Zo$ of $\hat{T}$ in part $(a)$ enters the set ${\hat{G}}_{L}$ defined above;
\item[(c)] the first visit  $\hat{T}_0 = \Zo^{n_0}({\hat{T}})$ to $\hat{G}$ (whose existence is guaranted by $(b)$) where~$n_0$ is the smallest $n\geq 0$ such that $ \Zo^{n_0}{\hat{T}}\in \hat{G}$, is Oseledets generic for the induced cocycle over the induced map $\hat{\mathcal{Z}}_{\hat{G}}$.  
\end{itemize}

\smallskip
\noindent \textit{Full measure of $\mathcal{G}_\pi$}. 
Let us show that $\mathcal{{\hat{G}}}_\pi$ has full measure with respect to $\mu_{\Zo_{\hat{G}}}$.  
Full measure of condition $(a)$ is given by Lemma~\ref{lemma:Oseledetsgeneric}, but since we want to verify also $(b)$ and $(c)$, let us consider the full measure set $\hat{I}_O\subset \hat{I}_\pi$ which are Oseledets generic, i.e.~the conclusion of Oseledets theorem holds for the cocycle $Z$ over $\hat{\mathcal{Z}}$. 
 Since ${Z}$ is integrable, the induced cocycle ${Z}_{\hat{G}}=\widetilde{Z}_{\hat{G}}$ over the induced map $\hat{\mathcal{Z}}_{\hat{G}}$ (defined as in \S~\ref{sec:inducing}) is again integrable (see \S~\ref{sec:integrability}). Therefore, by Oseledets theorem, 
 $\mu_{\Zo_{\hat{G}}}$-almost every $\hat{T}$ in $\hat{G}$ is Oseledets generic (see \S~\ref{sec:Oseledets}). We denote by $\hat{G}_O$ the full measure subset   of  $\hat{G}$ which consiste of Oseledets generic IETs for the cocycle $\widetilde{Z}_{\hat{G}}$. Since, as shown before, also $\hat{G}_L\subset\hat{G}$ has full measure in $\hat{G}$ the intersection $\hat{I}_O\cap \hat{G}_O\cap \hat{G}_L$ has full measure in $\hat{G}$.  
Reasoning again   by ergodicity and Fubini theorem for the measure $\mu_{\hat{\mathcal{G}}}$, the set of IETs $T$ which have an ergodic extension $\hat{T}$ whose orbit under $\Zo$ enters the intersection $\hat{I}_O\cap \hat{G}_O\cap \hat{G}_L$ (and hence verify all three assumptions $(a)$, $(b)$ and $(c)$ 
has full measure (with respect to the corresponding $\mu_{\Zo_{\hat{G}}})$ for every irreducible $\pi$.
Since this is true for every irreducible $\pi$, if we set $\mathcal{{{G}}}:=\bigcup_{\pi\in\mathfrak{S}^0_d} \mathcal{G}_\pi$, then $\mathcal{{{G}}}$ has Lebesgue full measure (see Remark~\ref{rk:ac}). 

\smallskip
\noindent \textit{Verifications of the RDC conditions.}
 Let us now verify that all $T\in \mathcal{{\hat{G}}} $ satisfy the RDC. Given any $T\in \mathcal{{\hat{G}}} $, by definition there exists $\hat{T}\in p^{-1}(T)$ which is recurrent to ${\hat{G}}$  along a subsequence sequences $(n_k)_{k\in\mathbb{N}}$ of iterates $\mathcal{Z}$, and recurrent to ${\hat{G}}'$ along a subsequence $(k_m)_{m\in\mathbb{N}}$  of iterates of the induced map $\mathcal{Z}_{\hat{G}}$.
 We claim that all conditions of Definition~\ref{def:RDC} hold for $T$ along the  sequences $(n_k)_{k\in\mathbb{N}}$ and $(k_m)_{m\in\mathbb{N}}$. 

\smallskip
\noindent \paragraph{{\it Conditions $(i)$ and $(ii)$.}} By definition of $\mathcal{G}$, the extension $\hat{T}$ is Oseledets generic; therefore Condition $(i)$ holds. Furthermore, by Proposition~\ref{prop:existenceeffective} and the construction of the set ${\hat{G}}$, the sequence $(n_k)_{k\in\mathbb{N}}$ is a sequence of good return times for $T=p(\hat{T})$. Furthermore, we showed earlier that, since $\hat{\Zo}^{n_0}(\hat{T})\in\hat{G}'$, $(n_k)_{k\in\mathbb{N}}$ has linear growth. Therefore, also Condition $(ii)$ is satisfied.

\smallskip
\noindent \paragraph{{\it Conditions $(S)$.}}
{To check Condition $(S)$ in Condition $(iii)$,  consider the cocycle obtained accelerating $\widetilde{\Zo}$ along the sequence $(k_m)_{m\in\mathbb{N}}$.  This is by construction an induced cocycle (see~\S~\ref{sec:cocycles}), over the map obtained inducing $\widetilde{\Zo}$ to returns to ${\hat{G}}_L$. Since we assumed in the definition of $\hat{\mathcal{G}}$ (see condition $(c)$) that the first visit $T_1:={\mathcal{Z}}^{(n_1)}(\hat{T})$ to $\hat{G}$ of the chosen extension $\hat{T}$ of $T$  is Oseledets generic for the induced cocycle corresponding to returns to $\hat{G}$, 
Condition $(S)$ follows from an application of Oseledets theorem for the accelererated cocycle, in particular from  \eqref{subexpgrowth} for the cocycle whose $(k-1)^{th}$ matrix is $\widetilde{Q}(k_m,k_{m+1})$.}

\smallskip{
\noindent \paragraph{{\it Conditions $(A)$.}} From the definition of good return times (and since $(n_{k_m})_{m\in\mathbb{N}}$ is a subsequence of the sequence $(n_k)_{k\in\mathbb{N}}$ of good returns), we also get that, for any $n_k$, Condition $(A)$ on the angle holds: this reduces indeed simply the lower bound on angles given by \eqref{lowerboundA}, specialized to the subsequence $(n_{k_m})_{m\in\mathbb{N}}$ when  recalling the notation $\widetilde{\Gamma}_x^{(k)}={\Gamma}_x^{(n_k)}$.
}

\smallskip
\noindent \paragraph{{\it Conditions $(B)$ and $(F)$.}} Consider any fixed $k_m$ in the subsequence $(k_m)_{m\in\mathbb{N}}$. Let us finally show that Conditions $(B)$ and $(F)$ hold for this $k_m$. 
Recalling the definition of $\widetilde{Q}(k,k')$, we hence get from \eqref{estimateEO1} and \eqref{uniform linearity} that, for every $k\geq k_m$, 
\begin{equation}\label{estimateQtildeF}
||\widetilde{Q}(k_m,k)\vert_{\widetilde{\Gamma}_s^{(k)}}||_{\infty} 
=||Q(n_{k_m},n_k)\vert_{\Gamma^{(n_k)}(T)}||_{\infty} 
\leq C_1 e^{-(\theta_g +\epsilon)(n_k-n_{k_m})} 
\leq C_1 e^{-(\theta_g +\epsilon)(k-k_m)},  
\end{equation}
and similarly, using \eqref{estimateEO2} this time, we also get that, {for every} $1 \leq k\leq k_m$,
\begin{equation}
\label{estimateQtildeB}
||\widetilde{Q}(k,k_m)^{-1}\vert_{\widetilde{\Gamma}_u^{(k)}}||_{\infty} 
\leq C_1 e^{-(\theta_g +\epsilon)(k_m-k)}.
\end{equation}
 Recall now that $\widetilde{\Pro }^{(k)}_s$ and $\widetilde{\Pro }^{(k)}_u$ denote  respectively the projection operators to the stable and unstable spaces $\widetilde{\Gamma}_s^{(k)}=\Gamma^{(n_k)}(T), \widetilde{\Gamma}_u^{(k)}= \Gamma^{(n_k)}(T)$. Estimate the norms $||\widetilde{\Pro }^{(k)}_s||$ and $||\widetilde{\Pro }^{(k)}_u||$ of these projections through the angle $\angle(\widetilde{\Gamma}^{(k)}_s, \widetilde{\Gamma}^{(k)}_u)$ between $\widetilde{\Gamma}_s^{(k)}$ and $\widetilde{\Gamma}_u^{(k)}$ and using \eqref{expsmall_fromki} for a time $n_{k_m}$ which corresponds to a visit to ${\hat{G}}$ and again the linear growth of $(n_k)_{k\in\mathbb{N}}$, we get, for some universal $c>0$, that, for any $m\in\mathbb{N}$ and for any $k\geq 0$,
\begin{equation}\label{projections}||\widetilde{\Pro }^{(k)}_s||, ||\widetilde{\Pro }^{(k)}_u||\leq \frac{c}{\angle(\widetilde{\Gamma}^{(k)}_s, \widetilde{\Gamma}^{(k)}_s))}= \frac{c}{\angle({\Gamma}^{(n_k)}_s, {\Gamma}^{(n_k)}_u))} \leq  \frac{c}{c_2(\epsilon) e^{-\epsilon |n_k-n_{k_m}|}}\leq \frac{c}{c_2(\epsilon)} e^{ \epsilon|k-k_m|/c_B}.
\end{equation}
We can now prove the convergence of the series $(B)$ and $(F)$. Fix now $\va>0$ such that $\va<\theta_g/ 2(1+c_B^{-1})$ and let $c_2:=c_2(\va)$ and $C_4:=C_4(\va)$.  Then, combining all the estimates proved so far, namely \eqref{estimateQtildeF}, \eqref{projections} and \eqref{expsmall_fromki}, which give, setting $C:=c C_1C_4/c_2$
\begin{align*}
\sum_{k=k_m +1}^{\infty}||\widetilde{Q}(k,k_m)^{-1}_{| \widetilde{\Gamma}_u^{(k)}}||   \,||\widetilde{\Pro }_{u}^{(k)}|| & \,||\widetilde{Z}_{k-1}||  
\leq \sum_{k=k_m +1}^{\infty}\left(C_1 e^{-(\theta_g)(k-k_m)/2}\right)\left( c{e^{\va(k-k_m)/c_B}}/{C_2}\right) \left(C_4 e^{ \va(k-k_m)} \right)\\ 
& \qquad \leq \sum_{k=k_m +1}^{\infty} C e^{-(\theta_g /2-\va (c_B^{-1}+1) )(k-k_m)} \leq K^+:= \sum_{\ell=1}^{\infty}C e^{-(\theta_g/2 -\va (c_B^{-1}+1) )\ell} ,
\end{align*}
where $K<+\infty$ since $\theta_g/2 -\va (c_B^{-1}+1)>0$ by choice of $\va$. This proves Condition $(F)$. Similarly, for the series in Condition $(B)$, we get
\begin{equation*}
  \sum_{k=1}^{k_m}{ ||\widetilde{Q}(k_m,k)_{| \widetilde{\Gamma}_s^{(k)}}|| \,||\widetilde{\Pro }_{s}^{(k)}|| \,||\widetilde{Z}_{k-1}||} 
	\leq  \sum_{k=1}^{k_m} 
	C e^{-(\theta_g/2 -\va (c_B^{-1}+1) )(k_m-k)}=
\sum_{\ell=1}^{k_m} C e^{-(\theta_g/2 -\va (c_B^{-1}+1) )\ell} \leq K^-.\end{equation*}
This proves Condition $(B)$ and thus concludes the proof that any $T\in\mathcal{G}$ satisfy the RDC.
\end{proof}

\subsubsection{Acknowledgements}
We would like to thank Michael Bromberg, Charles Fougeron and Liviana Palmisano for many  discussions at the very initial learning stages of this project as well as Giovanni Forni, Pascal Hubert and Stefano Marmi for useful discussions. {\color{black}We also thank Liviana Palmisano and Marco Martens for sharing their first draft of a nice modern exposition of the theory of classical theory of circle diffeos and 
Giovanni Forni as well as Frank Trujillo for their feedback and comments on a preliminary draft of this paper.} 
 This collaboration has started with many informal discussions and scientific visits throught a span of several years; we are grateful during this period for the financial support by the {\it ERC Starting grant} {\it ChaParDyn} and the {\it Swiss National Science Foundation}  Grant $200021\_188617/1$ which made these scientific visits possible and for the hospitality of the {\it University of Bristol} and the {\it University of Z{\"urich}}.  C.U.~acknowledges as well the {\it Leverhulme Trust}, the {\it Wolfson foundation}  and {\it SwissMap}.  {\color{black}The
research leading to these results has received funding from the {\it European Research Council} under
the European Union Seventh Framework Programme (FP/$2007$-$2013$)/ERC Grant Agreement n.~$335989$}.

\appendix

\section{}\label{appendix}
\noindent We include here, for completeness and for convenience of the reader, some results which were used in the paper and are either variations of those present in the literature or folklore.

{
\subsection{Boundary and suspensions.}\label{sec:boundarycomb}
We include in this section an explict construction of a standard suspension as well as the combinatorial definition of the boundary operator purely in terms of the combinatorial datum, following \cite{MarmiYoccoz}.

\smallskip  
Let $T$ be a GIET  with combinatorial data $\pi=(\pi_t,\pi_b)$. Recall that $u_i^t$ (resp.~$u_i^b$) denote the endpoints of the top (resp.~bottom) partition (see \S~\ref{sing}). For $0\leq i\leq d$ 
consider $\lambda_j := |I^t_j|= |I^b_{j}|$ and $\tau_j:=\pi_b^{-1}(j)-\pi_t^{-1}(j)$ and 
define the complex numbers 
$$
U_i:=u_0^t+\sum_{j\leq i} \lambda_{\pi_t(j)} + \sqrt{-1} \, \tau_{\pi_t(j)}, \qquad V_i:=u_0^b+\sum_{j\leq i} \lambda_{\pi_b(j)} + \sqrt{-1} \, \tau_{\pi_b(j)}.
$$
One has $U_0= u_ 0 ^t=u_0^b =V_ 0$ and $U_d = u_d^t=u_c^b= V_d $. 
Moreover, $\mathrm{Im} \, U_i>0$ and $\mathrm{Im}\, V_i>0$ for $1\leq i <d$. 
The $2d$ segments $L^t_{\pi_t(i)}:=[U_{i-1} , U_i ]$, $L^b_{\pi_b(i)}:=[V_{i-1} , V_i ]$ for $1 \leq  i \leq d$  form the boundary of a polygon. 
Gluing the pairs of parallel top and bottom sides $L^t_{i}$ and $L^b_{i}$ of this polygon produces a translation surface
$M_T$, in which the vertices of the polygon define a set  of marked points.
 
\smallskip 
\noindent Consider now the $2d$-elelents set 
$\mathcal{V}:=\{ U_0=V_0, U_1, V_1, \dots, U_{d-1}, V_{d-1}, U_d=V_d\}$, which is in bijection with the vertices of the polygon. The identifications of elements of $\mathcal{V}$ induced by the glueings of parallel sides is encored by the following permutation $\sigma$:
\begin{align*}
\sigma(U_i):= V_j & \qquad \text{if}\ \pi_b(j+1)=\pi_t(i+1), & \text{for}\ 0\leq i<d;\\
\sigma(V_k):= U_\ell & \qquad \text{if}\ \pi_t(\ell)=\pi_b(k), & \text{for}\ 0< j\leq d.
\end{align*}
Thus, cycles of $\sigma$  are in bijection with the the singularities $\Sing(M_T)$ of $M_T$. This shows that $\pi$ determines $\kappa$, the number of singularities (which is exactly equal to the number of cycles of $\sigma$) and therfore, from the formula $d=2g+\kappa-1$, it determines also the genus $g$ of any suspension.

\smallskip
We give now the definition of the (observable) boundary operator $B:\Sing(M_T) \to \mathbb{R}^\kappa$, following \cite{MarmiYoccoz}. Given a function $f \in \mathcal{C}(T)$ (recall~\ref{sec:obs}),  $B(f)=(b_s)_{1\leq s\leq \kappa}$ is defined as follows. For each $1\leq s\leq \kappa$, if $C_s$ is the cycle of $\sigma$ corresponding to the singuarity labeled $s$, we have
$$ b_s := \sum_{0\leq i\leq  d , \, U_i\in C_s}\left ( {f^r(u_i^t) }- {f^l(u_i^t)}\right),$$ 
where $f^l(u_i)$ and $f^r(u_i)$ denote the left and right limits at the discontinuity point $u_i$ (see \S~\ref{sec:obs_boundary}) and, by convention, $f^l(U_0)=0$ and $f^r(U_d)=0$.}}
\subsection{Distorsion bounds for GIETs}\label{estimates}
We present here for completeness the proof of the classical distorsion bound stated as Lemma~\ref{bound1}. The proof for GIETs is the same than the classical proof for circle diffeomorphisms.
\begin{proof}[Proof of Lemma~\ref{bound1}]
Assume without loss of generality that $x<y$. 
 We have that, by chain rule, $$ \log \mathrm{D}T^{n}(x) = \sum_{i=0}^{n-1}{\log \mathrm{D}T (T^i(x) )},$$ and therefore we also have that 
$$ | \log \mathrm{D}T^{n}(x) - \log \mathrm{D}T^{n}(y) | \leq  \sum_{i=0}^{n-1}{|\log \mathrm{D}T (T^i(x) ) -  \log \mathrm{D}T (T^i(x) ) |}  \leq  \sum_{i=0}^{n-1}{|\int_{T^i(y)}^{T^i(x)}{\eta_{T}}|}.$$  Notice now that, since by assumption, $T^i(J)$ do not contain singularities of $T$ for any $0\leq i<n$, each $T^i(J)$ is again an interval and, since $T$ is an isometry, $T^i(x) < T^i(y)$. It hence follows from the assumptions that  the intervals $[T^i(y), T^i(x)]$ are pairwise disjoint and
$$ | \log \mathrm{D}T^{n}(x) - \log \mathrm{D}T^{n}(y) | \leq \int_0^1{|\eta_T|\mathrm{dLeb}}. $$ Exponentiating this bound we get the desired result.
\end{proof}


\subsection{Distances comparisons.
}\label{sec:distancesAppendix}
 In \S~\ref{sec:distances} we defined two distances, namely $d_\eta$ and $d_{\mathcal{C}^1}$ on $\mathrm{Diff}^3([0,1])$ and, by abusing the notation, also extended their definition to distances  $d_\eta$ and $d_{\mathcal{C}^1}$  on  the space $\mathcal{X}^r_d$ of GIETs with $r\geq 2$. Here we first show that $d_\eta$ is a distance and then 
 prove the comparisons given by Lemma~\ref{lemma:distancesrelDiff} and Corollary~\ref{lemma:distancesrel}. 
{\color{black}
\subsubsection{The semi-distance $d_\eta$ is a distance.}\label{distance}
Consider  $d_\eta$ on $\mathrm{Diff}^1([0,1])$ defined in \S~\ref{sec:distances}.  Symmetry and triangle inequality are obvious, so to see that it is a distance, we only have to check that if $d_{\eta}(\varphi_1,\varphi_2) = 0$, and therefore $\eta_{\varphi_1} (x)= \eta_{\varphi_2}(x)$ for every $0<x<1$, then $\varphi_1=\varphi_2$.  This can be seen, for example, by showing  that the non-linearity $\eta_\varphi$ completely determines $\varphi \in$ Diff$^2([0,1])$, namely given a continuos function $\eta:[0,1]\to \mathbb{R}$ there exists a unique orientation-preserving diffeomorphism $\varphi \in$ Diff$^2([0,1])$ such that $\eta_\varphi= \eta$, which is explicitely given (see for example \cite{Martens}) 
 by the formula
$$
\varphi(x)=\frac{\int_{0}^x \exp \left(\int_0^z \eta(y)\mathrm{d}y \right)\mathrm{d}z}{\int_{0}^1 \exp \left(\int_0^z \eta(y)\mathrm{d}y \right)\mathrm{d}z}.
$$
}


\subsubsection{Comparison of $d_\eta$ and $d_{\mathcal{C}^1}$ on $\mathrm{Diff}^1([0,1])$.}
 Consider first $\Diff^r([0,1])$, where $r$ is an integer $r\geq 2$. Notice that it  is  an open subset of 
$$\mathcal{C}^r_{\partial}([0,1],\mathbb{R}) := \{ f \in \mathcal{C}^r([0,1]) \ | \ f(0) = 0 \ \text{and} \  f(1) = 1 \}.$$
\begin{proof}[Proof of Lemma~\ref{lemma:distancesrelDiff}] 
First note that there exists $x_0 \in [0,1]$ such that $f_1'(x_0) = f_2'(x_0)$ (otherwise $f_1' > f_2'$ or $f_1' < f_2'$ which is incompatible with the fact that $\int_0^1{f'_1} = \int_0^1{f'_2} = 1$). In which case we have 
$$ | \log f_1'(x) -  \log f_2'(x) | = | \int_{x_0}^x{(\eta_{f_1} - \eta_{f_2})} | \leq  \int_{x_0}^x{|\eta_{f_1} - \eta_{f_2}|} \leq d_{\eta}(f_1,f_2).$$ The exponential function being Lipschitz on bounded sets, {\color{black}we can find a constant $L>0$ such that $|f_1'(x)-f_2'(x)| \leq L |\log f_1'(x) -  \log f_2'(x) |$ and hence control $|| f_1'-f_2'||_\infty$. From this control and $f_1(0)=f_2(0)=0$, we can then control also $|| f_1-f_2||_{\infty}$ and hence $d_{\mathcal{C}^1}(f_1,f_2)$.}
\end{proof}

\subsubsection{Comparison of $d_\eta$ and $d_{\mathcal{C}^1}$ distances from AIETs.}
We can now deduce Corollary~\ref{lemma:distancesrelDiff}, i.e.~the comparison of the $ d_{\mathcal{C}^1}$ and $d_{\eta}$ on $\mathcal{X}^2$ from the locus of AIETs.
\begin{proof}[Proof of Corollary~\ref{lemma:distancesrel}]{\color{black} For any $n\in\mathbb{N}$, denote by $\mathcal{V}^n(T)=(A_n,\varphi_n)$ are the shape-profile coordinates of $\mathcal{V}^n(T)$. Recall first of all that the infimum in $ d_{\eta}(\mathcal{V}^n(T), \mathcal{A}_d):=\inf_{A\in \mathcal{A}_d} d_{\eta}(\mathcal{V}^n(T), A)$ is realized by the \emph{shape} $A_n$ of $\mathcal{V}^n(T)$ (by Remark~\ref{rk:nonlinearityasd}, see also footnote~\ref{minimal}). 
It is sufficient to find $L=L(T)$ such that  $d_{\mathcal{C}^1}(\mathcal{V}^n(T), {A}_n)\leq L \, d_{\eta}(\mathcal{V}^n(T), {A}_n)$ for every $n\in\mathbb{N}$, since this then gives that
$$ d_{\mathcal{C}^1}(\mathcal{V}^n(T), \mathcal{A}_d)\leq d_{\mathcal{C}^1}(\mathcal{V}^n(T), {A}_n) \leq L \, d_{\eta}(\mathcal{V}^n(T), {A}_n)= L \, d_{\eta}(\mathcal{V}^n(T), \mathcal{A}_n).$$ 
Since the distances $ d_{\mathcal{C}^1}$ and $ \,d_{\eta}$ on $\mathcal{X}^2$ are both  products of distances in the shape-profile coordinates  (see \S~\ref{sec:distances}) and $d_\mathcal{A}(\mathcal{V}^n(T), A_n)=0$ by definition of $A_n$, 
the distances 
 $ d_{\mathcal{C}^1}(\mathcal{V}^n(T), A_n)$ and $ d_{\eta}(\mathcal{V}^n(T), A_n)$ depend only on $d^\mathcal{P}_{\mathcal{C}^1}(\varphi_n, (I,\dots, I))$ and  $d^\mathcal{P}_{\eta}(\varphi_n, (I,\dots, I))$ respectively\footnote{{\color{black}Here $(I,\dots, I)\in \Diff^r([0,1])^d$ denotes the identity vector with $I(x)=x$ identify function in every coordinate. Notice that $P_\mathcal{P}(A)=(I,\dots, I)$ for every $A\in\mathcal{A}_d$, i.e.~in particular $(I,\dots, I)$ is the profile of $A_T$.}}.
\smallskip
Notice now that,  as a consequence of the classical distorsion bounds, the coordinates $\{ \varphi_n^i, \ n\in\mathbb{N}\}$ of the profiles 
 $\varphi_n =(\varphi_n^1,\dots, \varphi_n^d)$ of the orbit $(\mathcal{V}^nT)_n$, by Lemma~\ref{bound2},  are $\mathcal{C}^1$-bounded (in the sense of Definition~\ref{def:Ckbounded}) and therefore there exists a $\mathcal{C}^1$-bounded $\mathcal{K}= \mathcal{K}(T)\subset$ Diff$^r([0,1])$ which contains all coordinates $\varphi^i_n$, for every $1\leq i\leq d$ and every $n\in\mathbb{N}$.  
 Therefore the conclusion follows from the comparision given by Lemma~\ref{lemma:distancesrelDiff} for each profile coordinate, applied to the bounded set $\mathcal{K}(T)$.
}\end{proof}

\subsubsection{The Schwarzian derivative and the $\mathcal{C}^3$-distance}\label{app:C3S}
In this Appendix we give a proof of Lemma~\ref{lemma:distanceMoebius}, which shows that the $\mathcal{C}^3$-distance $d_{\mathcal{C}^3}(T, \mathcal{M})$ of a $T\in\mathcal{X}^r_d$ from the  the subspace $\mathcal{M}$ of Moebius IETs   can be controlled, on a bounded set,  by the Schwarzian derivative $\mathrm{S}(f)$ of $T$ (see \S~\ref{Schwarzian}).
We first prove the analogous statement in $\mathrm{Diff}^3([0,1])$ (namely the following Lemma~\ref{lemma:C3andS}), which will give control on each of the \emph{profile} coordinates of $T$. 

\smallskip
Let  $\mathcal{M}[0,1]$ denote the subspace of $\mathrm{Diff}^3([0,1])$ consisting of (restrictions of) Moebius maps. Recall that $\mathrm{S}(f)$ is the Schwarzian derivative of a diffeomorphism $f\in \mathrm{Diff}^3([0,1])$ (see \S~\ref{Schwarzian}).

\begin{lemma}\label{lemma:C3andS}
Let $\mathcal{K} \subset \mathrm{Diff}^3([0,1])$ a $\mathcal{C}^2$-bounded set, meaning that there exists a constant $K > 0$ such that for all $f \in \mathcal{K}$, $||\log \mathrm{D}f|| \leq K$ and $||\mathrm{D}^2 f|| \leq K$. Then there exists a constant $L = L(\mathcal{K}) > 0$ such that for $f  \in \mathcal{K}$,
$$ d_{\mathcal{C}^3}(f, \mathcal{M}([0,1])) \leq L \cdot \mathrm{S}(f).$$
\end{lemma}
\noindent Let us first prove an auxiliary technical lemma.
\begin{lemma}
\label{lemma8}
Let $g \in \mathcal{C}^1([0,1], \mathbb{R})$ such that $g(x_0) = 0$ for some $x_0 \in [0,1]$, and let $f \in \mathcal{C}^0([0,1], \R)$. Assume that there exists $\epsilon >0$ such that $|\D g - f \cdot g|\leq  + \epsilon$. Then 
$$ ||g|| \leq \epsilon e^{||f||}.$$
\end{lemma}
\begin{proof}[Proof of Lemma~\ref{lemma:C3andS}]
Define $F := x \mapsto \int_{0}^x{f'(t)dt}$, so that $\D F = f$ and $F(0) =0$. Thus $||F|| \leq ||f||$. Consider now the auxiliary function $ \psi = ge^{-F}$. We have $\D \psi = (Dg - fg)e^{-F}$. We therefore have $||D \psi|| \leq \epsilon e^{||f||}$ and $\psi(x_0) = 0 $. We obtain this way that for all $x \in [0,1]$, $|\psi(x)| \leq \epsilon e^{||f||}$.
\end{proof}

\begin{proof}
We consider $f \in \mathcal{K}$, and let $a := \int_0^1{\eta_f}$. Let $m_a \in \mathcal{M}([0,1])$ be the unique Moebius diffeomorphism of  $\mathcal{K} \subset \mathrm{Diff}^3([0,1])$ which is such that $\int_0^1{\eta_{m_a}} = a$.   Since $m_a$ is a Moebius map, $S(m_a) = 0$ (see Property $(S2)$ in \S~\ref{Schwarzian}) and we thus, using the expression \eqref{Sviaeta} for $\mathrm{S}(f)$ in terms of $\eta_f$, have 
$$ \mathrm{S}(f) = \mathrm{S}(f) -\mathrm{S}(m_a) = \D(\eta_f) - \D(\eta_{m_a}) - \frac{1}{2}(\eta_f^2 - \eta_{m_a}^2).$$ Set $\epsilon := ||\mathrm{S}(f)||$. We have 
$$ |\D(\eta_f -\eta_{m_a}) - \frac{1}{2}(\eta_f + \eta_{m_a})(\eta_f - \eta_{m_a})| \leq \epsilon.$$  Since $f$ is assumed to belong a set $\mathcal{K}$ that is $K$-bounded in the $\mathcal{C}^2$-topology, there exists a constant $M = M(\mathcal{K}) > 0$ such that $||\frac{1}{2}(\eta_f + \eta_{m_a})|| \leq M$. We can thus apply Lemma \ref{lemma8} to obtain 
$$ || \eta_f -\eta_{m_a} || \leq \epsilon e^M.$$ Since $|\D(\eta_f -\eta_{m_a}) - \frac{1}{2}(\eta_f + \eta_{m_a})(\eta_f - \eta_{m_a})| \leq \epsilon$, we also get that 
$$ ||\D(\eta_f -\eta_{m_a})|| \leq  (M e^M +1) \epsilon.$$

\vspace{2mm}
\noindent {\it Control of the $\mathcal{C}^1$-norm.} Since $ || \eta_f -\eta_{m_a} || \leq \epsilon e^M$, we have in particular $d_{\eta}(f,m_a) \leq \epsilon e^M$. By {\color{black}Lemma~\ref{lemma8}}
, there exists a constant $L$ (which only depends on $\mathcal{K}$) such that 
$$d_{\mathcal{C}^1}(f,m_a) \leq L_1e^M \epsilon.$$

\vspace{2mm}
\noindent {\it Control of the $\mathcal{C}^2$-norm.} We have $$\eta_f - \eta_{m_a} = \frac{f''}{f'} - \frac{m_a''}{m_a'} =  (\frac{f''}{f'} - \frac{m_a''}{f'}) + (\frac{m_a''}{f'} - \frac{m_a''}{m_a'}).$$ We obtain this way
$$ ||f'' - m_a''|| \leq ||f'|| \cdot ||\eta_f - \eta_{m_a}|| + ||f'|| \cdot ||m_a''|| \cdot || \frac{1}{f' m_a'}|| \cdot ||f' - m_a'||.$$ Since $f$ (and consequently $m_a$) belong to a $\mathcal{C}^2$-bounded set, the terms $||f'|| $ and $||f'|| \cdot ||m_a''|| \cdot || \frac{1}{f' m_a'}|| $ are bounded by constants depending only on $\mathcal{K}$. Together with the fact that $d_{\mathcal{C}^1}(f,m_a) \leq L_1 e^M \epsilon$, we get the existence of a constant $L_2 = L_2(\mathcal{K})$ such that 
$$ ||f'' - m_a''|| \leq L_2 \epsilon.$$

\vspace{2mm}
\noindent {\it Control of the $\mathcal{C}^3$-norm.} Using the fact that 
$$ \D \eta_f = \frac{f'''f' - (f'')^2}{(f')^2} $$ and the control  $ ||\D(\eta_f -\eta_{m_a})|| \leq  (M e^M +1) \epsilon$ we obtain, via a calculation similar to that done for the control of the $\mathcal{C}^2$-norm, the existence of $L_3 = L_3(\mathcal{K}) >0$ such that 
$$|| f''' - m_a'''|| \leq L_3 \epsilon.$$ This concludes the proof.
\end{proof}

\begin{proof}[Proof of Proposition~\ref{controlMoebius}]
The proof follows  recalling the definition of $d_{\mathcal{C}^3}$ on $\mathcal{X}^3$ (see \S~\ref{sec:distances}), in particular that the \emph{profile} coordinate $d_{\mathcal{C}^3}^\mathcal{P}$ is obtained summing up the distance $d_{\mathcal{C}^3}$ on each profile component, and  applying Lemma~\ref{lemma:C3andS} to each component of the profile.
\end{proof}

\subsection{Lipschitz regularity of composition and renormalization}\label{app:regularity}
Consider the composition map 
$$
\begin{array}{ccc}
\mathrm{Diff}^3([0,1]) \times  \mathrm{Diff}^3([0,1])  & \longrightarrow & \mathrm{Diff}^3([0,1])  \\
(f,g) & \longmapsto & f \circ g
\end{array}$$
A well known difficulty in the theory of renormalization of circle diffeomorphisms is that 
composition is \emph{not} differentiable with respect with the natural structure of a Banach space on $\mathrm{Diff}^3([0,1])$, which is inherited\footnote{Recall that $\mathrm{Diff}^3([0,1])$  is an open subset of $\mathcal{C}^3_{\partial}([0,1],\mathbb{R})$, defined in \S~\ref{sec:distancesAppendix}. The latter is a codimension $1$ affine subspace $\mathcal{C}^3([0,1])$ of tangent space $\mathcal{C}^3_{0}([0,1],\mathbb{R}) := \{ f \in \mathcal{C}^3([0,1]) \ | \ f(0) = 0 \ \text{and} \  f(1) = 0 \}$. It is endowed of the structure of a Banach space inherited from the $\mathcal{C}^3$.} 
from $( \mathcal{C}^3([0,1]),\mathbb{R}) $.

A way around this difficulty is to show that the composition, when restricted to bounded sets of $\mathrm{Diff}^3([0,1])$,  is on the other hand \emph{Lipschitz} with respect to the distance $d_\eta$ (see Proposition~\ref{lipschitz} below). From this, one can then show also that the renormalization operator given by Rauzy-Veech induction is Lipschitz (see \S~\ref{sec:LipR} and Propostion~\ref{prop:LipR}).

\subsubsection{Lipschitz regularity of composition}\label{app:regularity}
The following proposition crucially exploits the good behaviour of 
of non-linearity $\eta$  under composition (see in particular the preservation of mean non-linearity $(ii)$ and the triangle inequality $(iii)$ for non-linearity in Lemma~\ref{lemma:nonlinearity}).
\begin{proposition}
\label{lipschitz}
The composition map $\mathrm{Diff}^3([0,1]) \times  \mathrm{Diff}^3([0,1])  \longrightarrow  \mathrm{Diff}^3([0,1])$ is Lipschtiz with respect to the distance $d_{\eta}$ on $\mathcal{C}^3$-bounded sets of $\mathrm{Diff}^3([0,1]) $. 
\end{proposition}
\begin{proof}
Let $f_1,g_1,f_2,g_2$ belong to a fixed bounded $\mathcal{K} \subset \mathrm{Diff}^3([0,1])$, in the sense of Definition~\ref{def:Ckbounded}. Recall that by Lemma~\ref{lemma:equivbounded} {\color{black}(and recalling that the non-linearity is $\eta_f=\log Df /(\D^2 f)^2$, see \S~\ref{sec:nl})} 
this implies the existence of a constant $K>0$ such that 
$$ | \log \mathrm{D}f_i | , |\mathrm{D}^2f_i|,  | \log \mathrm{D}g_i | , |\mathrm{D}^2g_i|, |\eta_{f_i}|,  |\eta_{g_i}| \leq K $$ for $i = 1,2$. We want to estimate of $\int_0^1{|\eta_{f_1 \circ g_1} - \eta_{f_2 \circ g_2}| }$. Using the chain rule for non-linearity (see property $(i)$ in Lemma~\ref{lemma:nonlinearity}), we get 
\begin{align*}  \int_0^1{|\eta_{f_1 \circ g_1} - \eta_{f_2 \circ g_2} |}& = \int_0^1{| \eta_{f_1}| \circ g_1 \,\mathrm{D}g_1 + \eta_{g_1} - \eta_{f_2} \circ g_2 \,\mathrm{D}g_2 + \eta_{g_2}     | } ,  \\
\int_0^1{|\eta_{f_1 \circ g_1} - \eta_{f_2 \circ g_2}| } & \leq  \int_0^1{|  \eta_{g_1} - \eta_{g_2}| }  +  \int_0^1{| \eta_{f_1} \circ g_1 \,\mathrm{D}g_1  - \eta_{f_2} \circ g_2 \,\mathrm{D}g_2    | }  . 
\end{align*} 
The first term on the right hand side is $d_{\eta}(g_1,g_2)$, we then just focus on the second term. 
We now rewrite 
\begin{align*} 
\eta_{f_1} \circ g_1 \,\mathrm{D}g_1  - \eta_{f_2} \circ g_2 \,\mathrm{D}g_2  = (\eta_{f_1} \circ g_1 \,\mathrm{D}g_1 - \eta_{f_1} \circ g_1 \,\mathrm{D}g_2) 
& +  (\eta_{f_1} \circ g_1 \,\mathrm{D}g_2 - \eta_{f_1} \circ g_2 \,\mathrm{D}g_2)\\&  +  (\eta_{f_1} \circ g_2 \,\mathrm{D}g_2 -  \eta_{f_2} \circ g_2 \,\mathrm{D}g_2).
\end{align*}
 We deal with each of the three terms on the right hand side individually.

\medskip
 \noindent {\it First term.} To estimate the first term, let us use that 
 $$| \eta_{f_1} \circ g_1 \,\mathrm{D}g_1 - \eta_{f_1} \circ g_1 \,\mathrm{D}g_2 | \leq  || \mathrm{D} (\eta_{f_1} \circ g_1) || \,|| \mathrm{D}g_1  - \mathrm{D}g_2 || .$$ Since $(f_1,g_1)$ are assumed to belong to a bounded set $\mathcal{K}$ of $\mathrm{Diff}^3([0,1])$  there exists a constant $C_1 = C_1(\mathcal{K})$ such that $ || \mathrm{D} (\eta_{f_1} \circ g_1) || \leq C_1$.

 \noindent {\it Second term.} To estimate it, we use that $$| \eta_{f_1} \circ g_1 \,\mathrm{D}g_2 - \eta_{f_1} \circ g_2 \,\mathrm{D}g_2)| \leq ||\mathrm{D}g_2|| \,||\eta_{f_1}|| \,||g_1 - g_2||.$$
 
 \noindent {\it Last term.}
Finally, for the last term, by integrating $\eta_{f_1} \circ g_2 \,\mathrm{D}g_2 -  \eta_{f_2} \circ g_2 \,\mathrm{D}g_2$ and changing variable using $g_2$ we get $\int_0^1{| \eta_{f_1} \circ g_2 \, \mathrm{D}g_2 -  \eta_{f_2} \circ g_2 \, \mathrm{D}g_2| } = \int_0^1{|\eta_{f_1} - \eta_{f_2}|}$. 
Because we are on a bounded set $\mathcal{K}$ of $\mathrm{Diff}^3([0,1])$, there exists by Lemma~\ref{lemma:distancesrel} a constant $L = L(\mathcal{K})$ such that $d_{\mathcal{C}^1}(g_1,g_2) \leq L\,d_{\eta}(g_1,g_2)$ and we can conclude that 
 $$\int_0^1{|\eta_{f_1} \circ g_1 - \eta_{f_2} \circ g_2| } \leq d_{\eta}(g_1,g_2) +  C_1 d_{\eta}(g_1,g_2) + K e^K L  d_{\eta}(g_1,g_2) + d_{\eta}(f_1,f_2) $$ and thus that the composition is Lispchitz for the constant $(1 + C_1 + Ke^K)$ on $\mathcal{K}$.
\end{proof}

\subsubsection{Lipschitz regularity of renormalization}\label{sec:LipR}
Let us recall that $\mathcal{V}$ denotes the renormalization map on GIETs defined (almost everywhere) by Rauzy-Veech induction, see~\S~\ref{Rauzysec}. We can now prove Proposition~\ref{prop:LipR}, namely that $\mathcal{V}$ is $K$-Lipschitz  with respect to $d_{\eta}$ on any set $\mathcal{K} \subset \mathcal{X}^3$ bounded  with respect to the $d_{\mathcal{C}^3}$-topology.

\begin{proof}[Proof of Proposition~\ref{prop:LipR}]
Recall that the Banach structure that we are using  on $\mathcal{X}^3$ is given by the identification $\mathcal{X}^3 = \mathcal{A} \times \mathcal{P}$ given by the affine profile coordinates in \S~\ref{coordinates} (see the definition of the distance $d_\eta$ \S~\ref{sec:distances}, the norm is the norm which induces this distance)
 we can therefore write $\mathcal{V} = (\mathcal{V}_{\mathcal{A}}, \mathcal{V}_{\mathcal{P}})$ and it is enough to show that the restrictin on both coordinates, namely $\mathcal{V}_{\mathcal{A}}$ and $\mathcal{V}_{\mathcal{P}}$ are Lipschitz. 

\smallskip
In \cite{Selim:loc}, it is proven that $\mathcal{V}_{\mathcal{A}}$ is actually differentiable and therefore Lispchitz on bounded sets (see Appendix A in \cite{Selim:loc}). 
Therefore, since  $\mathcal{V}_{\mathcal{P}}$  is obtained by modifying on component by pre or post-composing it by the restriction of another component, one get the result by  applying Proposition \ref{lipschitz}.
\end{proof}

\subsection{Saddles defined by a non-degenerate vector field}\label{saddles}{
In this Appendix we show that even if a vector field is \emph{smooth}, it does not have to define a \emph{smooth} foliation in the sense of Definition~\ref{def:Crfol}.} 

\smallskip
\noindent Let $\vec{X}$ be a vector field on $\mathbb{R}^2$ vanishing at $0$ and assume that $0$ is not a critical point, \textit{i.e.} $\D \vec{X}_0$ is invertible. Assume furthermore that the matrix of $\D \vec{X}_0$ is of the form $$\begin{pmatrix}
\lambda & 0\\
0 & \mu
\end{pmatrix}$$ with $\lambda \mu < 0$, in such a way that the foliation induced by $\vec{X}$ in a neighbourhood of $0$ is saddle-like. We treat the case where $\vec{X}$ is actually equal to $\lambda x \partial_x + \mu y \partial_y$. For such an $\vec{X}$, solutions to the differential equation
$$ \frac{\mathrm{d}}{\mathrm{d}t}f = \vec{X}(f) $$ 

are given by the formula $f(t) = (x_0e^{\lambda t}, y_0e^{\mu t})$. If one sets $\alpha = - \frac{\mu}{\lambda}$, one easily checks that the integral curves of $\vec{X}$ are level sets of the function $ (x,y) \longmapsto yx^{\alpha}$.  
This function  is a \emph{Morse function} if and only if $\alpha=1$, or equivalently $\mu = -\lambda$.
 A non-linear version of this discussion can be obtained by applying a differentiable Hartman-Grobman theorem to the vector field $\vec{X}$ to bring ourselves back to the linear case.

\bibliographystyle{plain}
	\bibliography{biblio.bib}

\begin{bibdiv}
\begin{biblist}

\bib{Arnold}{article}{
      author={Arnol\cprime~d, V.~I.},
       title={Small denominators. {I}. {M}apping the circle onto itself},
        date={1961},
        ISSN={0373-2436},
     journal={Izv. Akad. Nauk SSSR Ser. Mat.},
      volume={25},
       pages={21\ndash 86},
      review={\MR{0140699}},
}

\bib{Arnoux}{incollection}{
      author={Arnoux, Pierre},
       title={\'{E}changes d'intervalles et flots sur les surfaces},
        date={1981},
   booktitle={Ergodic theory ({S}em., {L}es {P}lans-sur-{B}ex, 1980)
  ({F}rench)},
      series={Monograph. Enseign. Math.},
      volume={29},
   publisher={Univ. Gen\`eve, Geneva},
       pages={5\ndash 38},
      review={\MR{609891}},
}

\bib{AB:exp}{incollection}{
      author={Avila, Artur},
      author={Bufetov, Alexander},
       title={Exponential decay of correlations for the
  {R}auzy-{V}eech-{Z}orich induction map},
        date={2007},
   booktitle={Partially hyperbolic dynamics, laminations, and {T}eichm\"{u}ller
  flow},
      series={Fields Inst. Commun.},
      volume={51},
   publisher={Amer. Math. Soc., Providence, RI},
       pages={203\ndash 211},
      review={\MR{2388696}},
}

\bib{AF:wea}{article}{
      author={Avila, Artur},
      author={Forni, Giovanni},
       title={Weak mixing for interval exchange transformations and translation
  flows},
        date={2007},
        ISSN={0003-486X},
     journal={Ann. of Math. (2)},
      volume={165},
      number={2},
       pages={637\ndash 664},
         url={http://dx.doi.org/10.4007/annals.2007.165.637},
      review={\MR{2299743}},
}

\bib{AGY}{article}{
      author={Avila, Artur},
      author={Gou{\"e}zel, S{\'e}bastien},
      author={Yoccoz, Jean-Christophe},
       title={Exponential mixing for the {T}eichm\"uller flow},
        date={2006},
        ISSN={0073-8301},
     journal={Publ. Math. Inst. Hautes \'Etudes Sci.},
      number={104},
       pages={143\ndash 211},
         url={http://dx.doi.org/10.1007/s10240-006-0001-5},
      review={\MR{2264836}},
}

\bib{AvilaLyubich}{article}{
      author={Avila, Artur},
      author={Lyubich, Mikhail},
       title={The full renormalization horseshoe for unimodal maps of higher
  degree: exponential contraction along hybrid classes},
        date={2011},
        ISSN={0073-8301},
     journal={Publ. Math. Inst. Hautes \'{E}tudes Sci.},
      number={114},
       pages={171\ndash 223},
         url={https://doi-org.revues.math.u-psud.fr/10.1007/s10240-011-0034-2},
      review={\MR{2854860}},
}

\bib{AV:sim}{article}{
      author={Avila, Artur},
      author={Viana, Marcelo},
       title={Simplicity of {L}yapunov spectra: proof of the
  {Z}orich-{K}ontsevich conjecture},
        date={2007},
        ISSN={0001-5962},
     journal={Acta Math.},
      volume={198},
      number={1},
       pages={1\ndash 56},
         url={http://dx.doi.org/10.1007/s11511-007-0012-1},
      review={\MR{2316268}},
}

\bib{BressaudHubertMaass}{article}{
      author={Bressaud, Xavier},
      author={Hubert, Pascal},
      author={Maass, Alejandro},
       title={Persistence of wandering intervals in self-similar affine
  interval exchange transformations},
        date={2010},
        ISSN={0143-3857},
     journal={Ergodic Theory Dynam. Systems},
      volume={30},
      number={3},
       pages={665\ndash 686},
         url={https://doi-org.revues.math.u-psud.fr/10.1017/S0143385709000418},
      review={\MR{2643707}},
}

\bib{Bu:dec}{article}{
      author={Bufetov, Alexander~I.},
       title={Decay of correlations for the {R}auzy-{V}eech-{Z}orich induction
  map on the space of interval exchange transformations and the central limit
  theorem for the {T}eichm\"{u}ller flow on the moduli space of abelian
  differentials},
        date={2006},
        ISSN={0894-0347},
     journal={J. Amer. Math. Soc.},
      volume={19},
      number={3},
       pages={579\ndash 623},
         url={https://doi.org/10.1090/S0894-0347-06-00528-5},
      review={\MR{2220100}},
}

\bib{Ca:fol}{incollection}{
      author={Calabi, Eugenio},
       title={An intrinsic characterization of harmonic one-forms},
        date={1969},
   booktitle={Global {A}nalysis ({P}apers in {H}onor of {K}. {K}odaira)},
   publisher={Univ. Tokyo Press, Tokyo},
       pages={101\ndash 117},
      review={\MR{0253370}},
}

\bib{CamelierGutierrez}{article}{
      author={Camelier, Ricardo},
      author={Gutierrez, Carlos},
       title={Affine interval exchange transformations with wandering
  intervals},
        date={1997},
        ISSN={0143-3857},
     journal={Ergodic Theory Dynam. Systems},
      volume={17},
      number={6},
       pages={1315\ndash 1338},
         url={https://doi-org.revues.math.u-psud.fr/10.1017/S0143385797097666},
      review={\MR{1488320}},
}

\bib{Ch:dis}{article}{
      author={Chaika, Jon},
       title={Every ergodic transformation is disjoint from almost every
  interval exchange transformation},
        date={2012},
        ISSN={0003-486X},
     journal={Ann. of Math. (2)},
      volume={175},
      number={1},
       pages={237\ndash 253},
         url={http://dx.doi.org/10.4007/annals.2012.175.1.6},
      review={\MR{2874642}},
}

\bib{Cobo}{article}{
      author={Cobo, Milton},
       title={Piece-wise affine maps conjugate to interval exchanges},
        date={2002},
        ISSN={0143-3857},
     journal={Ergodic Theory Dynam. Systems},
      volume={22},
      number={2},
       pages={375\ndash 407},
         url={https://doi-org.revues.math.u-psud.fr/10.1017/S0143385702000196},
      review={\MR{1898797}},
}

\bib{CS:ren}{article}{
      author={Cunha, Kleyber},
      author={Smania, Daniel},
       title={Renormalization for piecewise smooth homeomorphisms on the
  circle},
        date={2013},
        ISSN={0294-1449},
     journal={Ann. Inst. H. Poincar\'{e} Anal. Non Lin\'{e}aire},
      volume={30},
      number={3},
       pages={441\ndash 462},
         url={https://doi.org/10.1016/j.anihpc.2012.09.004},
      review={\MR{3061431}},
}

\bib{CS:rig}{article}{
      author={Cunha, Kleyber},
      author={Smania, Daniel},
       title={Rigidity for piecewise smooth homeomorphisms on the circle},
        date={2014},
        ISSN={0001-8708},
     journal={Adv. Math.},
      volume={250},
       pages={193\ndash 226},
         url={https://doi.org/10.1016/j.aim.2013.09.017},
      review={\MR{3122166}},
}

\bib{deFariadeMelo}{article}{
      author={de~Faria, Edson},
      author={de~Melo, Welington},
       title={Rigidity of critical circle mappings. {I}},
        date={1999},
        ISSN={1435-9855},
     journal={J. Eur. Math. Soc. (JEMS)},
      volume={1},
      number={4},
       pages={339\ndash 392},
         url={https://doi-org.revues.math.u-psud.fr/10.1007/s100970050011},
      review={\MR{1728375}},
}

\bib{deFariadeMelo2}{article}{
      author={de~Faria, Edson},
      author={de~Melo, Welington},
       title={Rigidity of critical circle mappings. {II}},
        date={2000},
        ISSN={0894-0347},
     journal={J. Amer. Math. Soc.},
      volume={13},
      number={2},
       pages={343\ndash 370},
  url={https://doi-org.revues.math.u-psud.fr/10.1090/S0894-0347-99-00324-0},
      review={\MR{1711394}},
}

\bib{Denjoy}{article}{
      author={{Denjoy}, Arnaud},
       title={{Sur les courbes definies par les \'equations diff\'erentielles
  \`a la surface du tore}},
    language={French},
        date={1932},
        ISSN={0021-7824},
     journal={{J. Math. Pures Appl. (9)}},
      volume={11},
       pages={333\ndash 375},
}

\bib{Feigenbaum}{article}{
      author={Feigenbaum, Mitchell~J.},
       title={Quantitative universality for a class of nonlinear
  transformations},
        date={1978},
        ISSN={0022-4715},
     journal={J. Statist. Phys.},
      volume={19},
      number={1},
       pages={25\ndash 52},
         url={https://doi-org.revues.math.u-psud.fr/10.1007/BF01020332},
      review={\MR{501179}},
}

\bib{FF}{article}{
      author={Flaminio, Livio},
      author={Forni, Giovanni},
       title={Invariant distributions and time averages for horocycle flows},
        date={2003},
        ISSN={0012-7094},
     journal={Duke Math. J.},
      volume={119},
      number={3},
       pages={465\ndash 526},
  url={https://doi-org.revues.math.u-psud.fr/10.1215/S0012-7094-03-11932-8},
      review={\MR{2003124}},
}

\bib{FF2}{article}{
      author={Flaminio, Livio},
      author={Forni, Giovanni},
       title={On the cohomological equation for nilflows},
        date={2007},
        ISSN={1930-5311},
     journal={J. Mod. Dyn.},
      volume={1},
      number={1},
       pages={37\ndash 60},
         url={https://doi-org.revues.math.u-psud.fr/10.3934/jmd.2007.1.37},
      review={\MR{2261071}},
}

\bib{Fo:asy}{inproceedings}{
      author={Forni, G.},
       title={Asymptotic behaviour of ergodic integrals of `renormalizable'
  parabolic flows},
        date={2002},
   booktitle={Proceedings of the {I}nternational {C}ongress of
  {M}athematicians, {V}ol. {III} ({B}eijing, 2002)},
   publisher={Higher Ed. Press, Beijing},
       pages={317\ndash 326},
      review={\MR{1957542}},
}

\bib{Fo:ann}{article}{
      author={Forni, Giovanni},
       title={Solutions of the cohomological equation for area-preserving flows
  on compact surfaces of higher genus},
        date={1997},
        ISSN={0003-486X},
     journal={Ann. of Math. (2)},
      volume={146},
      number={2},
       pages={295\ndash 344},
         url={https://doi-org.revues.math.u-psud.fr/10.2307/2952464},
      review={\MR{1477760}},
}

\bib{Forni2}{article}{
      author={Forni, Giovanni},
       title={Sobolev regularity of solutions of the cohomological equation},
        date={2021},
        ISSN={0143-3857},
     journal={Ergodic Theory Dynam. Systems},
      volume={41},
      number={3},
       pages={685\ndash 789},
         url={https://doi-org.revues.math.u-psud.fr/10.1017/etds.2019.108},
      review={\MR{4211919}},
}

\bib{FU:gro}{unpublished}{
      author={Fr\k{a}czek, Krzysztof},
      author={Ulcigrai, Corinna},
       title={On the growth of birkhoff integrals for locally hamiltonian flows
  and ergodicity of extensions},
        note={Preprint},
}

\bib{Selim:loc}{unpublished}{
      author={Ghaozuani, Selim},
       title={Local rigidity for periodic generalised interval exchange
  transformations},
        note={arXiv:1907.05646},
}

\bib{Selim:survey}{unpublished}{
      author={Ghazouani, Selim},
       title={Une invitation aux surfaces de dilatation},
        note={arXiv:1901.08856},
}

\bib{Herman}{article}{
      author={Herman, Michael-Robert},
       title={Sur la conjugaison diff\'{e}rentiable des diff\'{e}omorphismes du
  cercle \`a des rotations},
        date={1979},
        ISSN={0073-8301},
     journal={Inst. Hautes \'{E}tudes Sci. Publ. Math.},
      number={49},
       pages={5\ndash 233},
  url={http://www.numdam.org.revues.math.u-psud.fr:2048/item?id=PMIHES_1979__49__5_0},
      review={\MR{538680}},
}

\bib{KKU}{article}{
      author={Kanigowski, Adam},
      author={Ku\l{}aga-Przymus, Joanna},
       title={Ratner's property and mild mixing for smooth flows on surfaces},
        date={2016},
        ISSN={0143-3857},
     journal={Ergodic Theory Dynam. Systems},
      volume={36},
      number={8},
       pages={2512\ndash 2537},
         url={http://dx.doi.org/10.1017/etds.2015.35},
      review={\MR{3570023}},
}

\bib{Karaliolios}{article}{
      author={Karaliolios, Nikolaos},
       title={Local rigidity of {D}iophantine translations in
  higher-dimensional tori},
        date={2018},
        ISSN={1560-3547},
     journal={Regul. Chaotic Dyn.},
      volume={23},
      number={1},
       pages={12\ndash 25},
         url={https://doi-org.revues.math.u-psud.fr/10.1134/S1560354718010021},
      review={\MR{3759967}},
}

\bib{Ka:inv}{article}{
      author={Katok, A.~B.},
       title={Invariant measures of flows on orientable surfaces},
        date={1973},
        ISSN={0002-3264},
     journal={Dokl. Akad. Nauk SSSR},
      volume={211},
       pages={775\ndash 778},
      review={\MR{0331438}},
}

\bib{KhaninKhmelev}{article}{
      author={Khanin, K.},
      author={Khmelev, D.},
       title={Renormalizations and rigidity theory for circle homeomorphisms
  with singularities of the break type},
        date={2003},
        ISSN={0010-3616},
     journal={Comm. Math. Phys.},
      volume={235},
      number={1},
       pages={69\ndash 124},
         url={https://doi-org.revues.math.u-psud.fr/10.1007/s00220-003-0809-5},
      review={\MR{1969721}},
}

\bib{KT:Her}{article}{
      author={Khanin, K.},
      author={Teplinsky, A.},
       title={Herman's theory revisited},
        date={2009},
        ISSN={0020-9910},
     journal={Invent. Math.},
      volume={178},
      number={2},
       pages={333\ndash 344},
         url={https://doi.org/10.1007/s00222-009-0200-z},
      review={\MR{2545684}},
}

\bib{KS:Her}{article}{
      author={Khanin, K.~M.},
      author={Sina\u{\i}, Ya.~G.},
       title={A new proof of {M}. {H}erman's theorem},
        date={1987},
        ISSN={0010-3616},
     journal={Comm. Math. Phys.},
      volume={112},
      number={1},
       pages={89\ndash 101},
         url={http://projecteuclid.org/euclid.cmp/1104159810},
      review={\MR{904139}},
}

\bib{Kh:sur}{inproceedings}{
      author={Khanin, Konstantin},
       title={Renormalization and rigidity},
        date={2018},
   booktitle={Proceedings of the {I}nternational {C}ongress of
  {M}athematicians---{R}io de {J}aneiro 2018. {V}ol. {III}. {I}nvited
  lectures},
   publisher={World Sci. Publ., Hackensack, NJ},
       pages={1973\ndash 1993},
      review={\MR{3966838}},
}

\bib{KhaninKocicMazzeo}{article}{
      author={Khanin, Konstantin},
      author={Koci\'{c}, Sa\v{s}a},
      author={Mazzeo, Elio},
       title={{$C^1$}-rigidity of circle maps with breaks for almost all
  rotation numbers},
        date={2017},
        ISSN={0012-9593},
     journal={Ann. Sci. \'{E}c. Norm. Sup\'{e}r. (4)},
      volume={50},
      number={5},
       pages={1163\ndash 1203},
         url={https://doi-org.revues.math.u-psud.fr/10.24033/asens.2642},
      review={\MR{3720027}},
}

\bib{KhaninTeplinsky}{article}{
      author={Khanin, Konstantin},
      author={Teplinsky, Alexey},
       title={Renormalization horseshoe and rigidity for circle diffeomorphisms
  with breaks},
        date={2013},
        ISSN={0010-3616},
     journal={Comm. Math. Phys.},
      volume={320},
      number={2},
       pages={347\ndash 377},
         url={https://doi-org.revues.math.u-psud.fr/10.1007/s00220-013-1706-1},
      review={\MR{3053764}},
}

\bib{Zo:Rau}{article}{
      author={Kontsevich, Maxim},
      author={Zorich, Anton},
       title={Connected components of the moduli spaces of {A}belian
  differentials with prescribed singularities},
        date={2003},
        ISSN={0020-9910},
     journal={Invent. Math.},
      volume={153},
      number={3},
       pages={631\ndash 678},
         url={https://doi-org.revues.math.u-psud.fr/10.1007/s00222-003-0303-x},
      review={\MR{2000471}},
}

\bib{Levitt2}{article}{
      author={Levitt, Gilbert},
       title={Feuilletages des surfaces},
        date={1982},
        ISSN={0373-0956},
     journal={Ann. Inst. Fourier (Grenoble)},
      volume={32},
      number={2},
       pages={x, 179\ndash 217},
  url={http://www.numdam.org.revues.math.u-psud.fr:2048/item?id=AIF_1982__32_2_179_0},
      review={\MR{662443}},
}

\bib{Levitt1}{article}{
      author={Levitt, Gilbert},
       title={Pantalons et feuilletages des surfaces},
        date={1982},
        ISSN={0040-9383},
     journal={Topology},
      volume={21},
      number={1},
       pages={9\ndash 33},
  url={https://doi-org.revues.math.u-psud.fr/10.1016/0040-9383(82)90039-8},
      review={\MR{630878}},
}

\bib{Levitt3}{article}{
      author={Levitt, Gilbert},
       title={La d\'{e}composition dynamique et la diff\'{e}rentiabilit\'{e}
  des feuilletages des surfaces},
        date={1987},
        ISSN={0373-0956},
     journal={Ann. Inst. Fourier (Grenoble)},
      volume={37},
      number={3},
       pages={85\ndash 116},
  url={http://www.numdam.org.revues.math.u-psud.fr:2048/item?id=AIF_1987__37_3_85_0},
      review={\MR{916275}},
}

\bib{Lyubich}{article}{
      author={Lyubich, Mikhail},
       title={Feigenbaum-{C}oullet-{T}resser universality and {M}ilnor's
  hairiness conjecture},
        date={1999},
        ISSN={0003-486X},
     journal={Ann. of Math. (2)},
      volume={149},
      number={2},
       pages={319\ndash 420},
         url={https://doi-org.revues.math.u-psud.fr/10.2307/120968},
      review={\MR{1689333}},
}

\bib{MMY}{article}{
      author={Marmi, S.},
      author={Moussa, P.},
      author={Yoccoz, J.-C.},
       title={The cohomological equation for {R}oth-type interval exchange
  maps},
        date={2005},
        ISSN={0894-0347},
     journal={J. Amer. Math. Soc.},
      volume={18},
      number={4},
       pages={823\ndash 872},
  url={https://doi-org.revues.math.u-psud.fr/10.1090/S0894-0347-05-00490-X},
      review={\MR{2163864}},
}

\bib{MMY2}{article}{
      author={Marmi, S.},
      author={Moussa, P.},
      author={Yoccoz, J.-C.},
       title={Affine interval exchange maps with a wandering interval},
        date={2010},
        ISSN={0024-6115},
     journal={Proc. Lond. Math. Soc. (3)},
      volume={100},
      number={3},
       pages={639\ndash 669},
         url={https://doi-org.revues.math.u-psud.fr/10.1112/plms/pdp037},
      review={\MR{2640286}},
}

\bib{MMY3}{article}{
      author={Marmi, Stefano},
      author={Moussa, Pierre},
      author={Yoccoz, Jean-Christophe},
       title={Linearization of generalized interval exchange maps},
        date={2012},
        ISSN={0003-486X},
     journal={Ann. of Math. (2)},
      volume={176},
      number={3},
       pages={1583\ndash 1646},
  url={https://doi-org.revues.math.u-psud.fr/10.4007/annals.2012.176.3.5},
      review={\MR{2979858}},
}

\bib{MUY}{article}{
      author={Marmi, Stefano},
      author={Ulcigrai, Corinna},
      author={Yoccoz, Jean-Christophe},
       title={On {R}oth type conditions, duality and central {B}irkhoff sums
  for {I}.{E}.{M}},
        date={2020},
        ISSN={0303-1179},
     journal={Ast\'{e}risque},
      number={416},
       pages={65\ndash 132},
         url={https://doi.org/10.24033/ast},
      review={\MR{4142457}},
}

\bib{MarmiYoccoz}{article}{
      author={Marmi, Stefano},
      author={Yoccoz, Jean-Christophe},
       title={H\"{o}lder regularity of the solutions of the cohomological
  equation for {R}oth type interval exchange maps},
        date={2016},
        ISSN={0010-3616},
     journal={Comm. Math. Phys.},
      volume={344},
      number={1},
       pages={117\ndash 139},
         url={https://doi-org.revues.math.u-psud.fr/10.1007/s00220-016-2624-9},
      review={\MR{3493139}},
}

\bib{Martens}{article}{
      author={Martens, Marco},
       title={The periodic points of renormalization},
        date={1998},
        ISSN={0003-486X},
     journal={Ann. of Math. (2)},
      volume={147},
      number={3},
       pages={543\ndash 584},
         url={https://doi.org/10.2307/120959},
      review={\MR{1637651}},
}

\bib{MartensPalmisano}{article}{
      author={Martens, Marco},
      author={Palmisano, Liviana},
       title={Invariant manifolds for non-differentiable operators},
     journal={Preprint},
         url={https://arxiv.org/abs/1704.06328},
}

\bib{MartensWinckler}{article}{
      author={Martens, Marco},
      author={Winckler, Bj\"{o}rn},
       title={On the hyperbolicity of {L}orenz renormalization},
        date={2014},
        ISSN={0010-3616},
     journal={Comm. Math. Phys.},
      volume={325},
      number={1},
       pages={185\ndash 257},
         url={https://doi-org.revues.math.u-psud.fr/10.1007/s00220-013-1858-z},
      review={\MR{3147438}},
}

\bib{Ma:int}{article}{
      author={Masur, Howard},
       title={Interval exchange transformations and measured foliations},
        date={1982},
     journal={Annals of Mathematics},
      volume={115},
       pages={169\ndash 200},
}

\bib{McMullen1}{book}{
      author={McMullen, Curtis~T.},
       title={Complex dynamics and renormalization},
      series={Annals of Mathematics Studies},
   publisher={Princeton University Press, Princeton, NJ},
        date={1994},
      volume={135},
        ISBN={0-691-02982-2; 0-691-02981-4},
      review={\MR{1312365}},
}

\bib{McMullen2}{book}{
      author={McMullen, Curtis~T.},
       title={Renormalization and 3-manifolds which fiber over the circle},
      series={Annals of Mathematics Studies},
   publisher={Princeton University Press, Princeton, NJ},
        date={1996},
      volume={142},
        ISBN={0-691-01154-0; 0-691-01153-2},
         url={https://doi-org.revues.math.u-psud.fr/10.1515/9781400865178},
      review={\MR{1401347}},
}

\bib{PM:sur}{article}{
      author={P\'{e}rez~Marco, Ricardo},
       title={Sur les dynamiques holomorphes non lin\'{e}arisables et une
  conjecture de {V}. {I}. {A}rnol\cprime d},
        date={1993},
        ISSN={0012-9593},
     journal={Ann. Sci. \'{E}cole Norm. Sup. (4)},
      volume={26},
      number={5},
       pages={565\ndash 644},
         url={http://www.numdam.org/item?id=ASENS_1993_4_26_5_565_0},
      review={\MR{1241470}},
}

\bib{Ra:ech}{article}{
      author={Rauzy, G\'erard},
       title={\'{E}changes d'intervalles et trasformations induites},
        date={1979},
     journal={Acta Arithmetica},
      volume={XXXIV},
       pages={315\ndash 328},
}

\bib{Sullivan}{incollection}{
      author={Sullivan, Dennis},
       title={Bounds, quadratic differentials, and renormalization
  conjectures},
        date={1992},
   booktitle={American {M}athematical {S}ociety centennial publications, {V}ol.
  {II} ({P}rovidence, {RI}, 1988)},
   publisher={Amer. Math. Soc., Providence, RI},
       pages={417\ndash 466},
      review={\MR{1184622}},
}

\bib{CoulletTresser}{article}{
      author={Tresser, Charles},
      author={Coullet, Pierre},
       title={It\'{e}rations d'endomorphismes et groupe de renormalisation},
        date={1978},
        ISSN={0151-0509},
     journal={C. R. Acad. Sci. Paris S\'{e}r. A-B},
      volume={287},
      number={7},
       pages={A577\ndash A580},
      review={\MR{512110}},
}

\bib{Ul:mix}{article}{
      author={Ulcigrai, Corinna},
       title={Mixing of asymmetric logarithmic suspension flows over interval
  exchange transformations},
        date={2007},
        ISSN={0143-3857},
     journal={Ergodic Theory Dynam. Systems},
      volume={27},
      number={3},
       pages={991\ndash 1035},
         url={http://dx.doi.org/10.1017/S0143385706000836},
      review={\MR{2322189 (2008m:37011)}},
}

\bib{Ul:abs}{article}{
      author={Ulcigrai, Corinna},
       title={Absence of mixing in area-preserving flows on surfaces},
        date={2011},
        ISSN={0003-486X},
     journal={Ann. of Math. (2)},
      volume={173},
      number={3},
       pages={1743\ndash 1778},
         url={http://dx.doi.org/10.4007/annals.2011.173.3.10},
      review={\MR{2800723 (2012j:37006)}},
}

\bib{Ul:she}{incollection}{
      author={Ulcigrai, Corinna},
       title={Shearing and mixing in parabolic flows},
        date={2013},
   booktitle={European {C}ongress of {M}athematics},
   publisher={Eur. Math. Soc., Z\"{u}rich},
       pages={691\ndash 705},
      review={\MR{3469153}},
}

\bib{Ve:gau}{article}{
      author={Veech, William~A.},
       title={Gauss measures for transformations on the space of interval
  exchange maps},
        date={1982},
     journal={Annals of Mathematics},
      volume={115},
       pages={201\ndash 242},
}

\bib{Ve:inI}{article}{
      author={Veech, William~A.},
       title={The metric theory of interval exchange transformations {I}.
  generic spectral properties.},
        date={1984},
     journal={American Journal of Mathematics},
      volume={107},
      number={6},
       pages={1331\ndash 1359},
}

\bib{Ve:Teich}{article}{
      author={Veech, William~A.},
       title={The {T}eichm\"{u}ller geodesic flow},
        date={1986},
        ISSN={0003-486X},
     journal={Ann. of Math. (2)},
      volume={124},
      number={3},
       pages={441\ndash 530},
         url={https://doi-org.revues.math.u-psud.fr/10.2307/2007091},
      review={\MR{866707}},
}

\bib{Vi:IET}{misc}{
      author={Viana, M.},
       title={Dynamics of interval exchange transformations and teichm\"uller
  flows},
         how={Available from \url{http://w3.impa.br/~viana}},
        note={Lecture Notes},
}

\bib{ViaLP}{book}{
      author={Viana, Marcelo},
       title={Lectures on {L}yapunov exponents},
      series={Cambridge Studies in Advanced Mathematics},
   publisher={Cambridge University Press, Cambridge},
        date={2014},
      volume={145},
        ISBN={978-1-107-08173-4},
         url={https://doi.org/10.1017/CBO9781139976602},
      review={\MR{3289050}},
}

\bib{VulKhanin}{article}{
      author={Vul, E.~B.},
      author={Khanin, K.~M.},
       title={Homeomorphisms of the circle with fracture-type singularities},
        date={1990},
        ISSN={0042-1316},
     journal={Uspekhi Mat. Nauk},
      volume={45},
      number={3(273)},
       pages={189\ndash 190},
         url={https://doi.org/10.1070/RM1990v045n03ABEH002353},
      review={\MR{1071941}},
}

\bib{Winckler}{article}{
      author={Winckler, Bj\"{o}rn},
       title={A renormalization fixed point for {L}orenz maps},
        date={2010},
        ISSN={0951-7715},
     journal={Nonlinearity},
      volume={23},
      number={6},
       pages={1291\ndash 1302},
  url={https://doi-org.revues.math.u-psud.fr/10.1088/0951-7715/23/6/003},
      review={\MR{2646067}},
}

\bib{Ya:att}{article}{
      author={Yampolsky, Michael},
       title={The attractor of renormalization and rigidity of towers of
  critical circle maps},
        date={2001},
        ISSN={0010-3616},
     journal={Comm. Math. Phys.},
      volume={218},
      number={3},
       pages={537\ndash 568},
         url={https://doi.org/10.1007/PL00005561},
      review={\MR{1828852}},
}

\bib{Ya:hyp}{article}{
      author={Yampolsky, Michael},
       title={Hyperbolicity of renormalization of critical circle maps},
        date={2002},
        ISSN={0073-8301},
     journal={Publ. Math. Inst. Hautes \'{E}tudes Sci.},
      number={96},
       pages={1\ndash 41 (2003)},
         url={https://doi.org/10.1007/s10240-003-0007-1},
      review={\MR{1985030}},
}

\bib{Yoccoz}{article}{
      author={Yoccoz, J.-C.},
       title={Conjugaison diff\'{e}rentiable des diff\'{e}omorphismes du cercle
  dont le nombre de rotation v\'{e}rifie une condition diophantienne},
        date={1984},
        ISSN={0012-9593},
     journal={Ann. Sci. \'{E}cole Norm. Sup. (4)},
      volume={17},
      number={3},
       pages={333\ndash 359},
  url={http://www.numdam.org.revues.math.u-psud.fr:2048/item?id=ASENS_1984_4_17_3_333_0},
      review={\MR{777374}},
}

\bib{Yo:CIME}{incollection}{
      author={Yoccoz, Jean-Christophe},
       title={Analytic linearization of circle diffeomorphisms},
        date={2002},
   booktitle={Dynamical systems and small divisors ({C}etraro, 1998)},
      series={Lecture Notes in Math.},
      volume={1784},
   publisher={Springer, Berlin},
       pages={125\ndash 173},
         url={https://doi.org/10.1007/978-3-540-47928-4_3},
      review={\MR{1924912}},
}

\bib{Yoc:cours}{article}{
      author={Yoccoz, Jean-Christophe},
       title={Echanges d'intervalles},
        date={2005},
     journal={Cours, Coll{\`e}ge de France},
  url={https://www.college-de-france.fr/site/jean-christophe-yoccoz/course-2004-2005.htm},
}

\bib{Yoc:CF}{incollection}{
      author={Yoccoz, Jean-Christophe},
       title={Continued fraction algorithms for interval exchange maps: an
  introduction},
        date={2006},
   booktitle={Frontiers in number theory, physics, and geometry. {I}},
   publisher={Springer, Berlin},
       pages={401\ndash 435},
         url={https://doi.org/10.1007/978-3-540-31347-2_12},
      review={\MR{2261103}},
}

\bib{Yoc:B}{incollection}{
      author={Yoccoz, Jean-Christophe},
       title={\'{E}changes d'intervalles et surfaces de translation},
        date={2009},
       pages={Exp. No. 996, x, 387\ndash 409 (2010)},
        note={S\'{e}minaire Bourbaki. Vol. 2007/2008},
      review={\MR{2605330}},
}

\bib{Yoc:Clay}{incollection}{
      author={Yoccoz, Jean-Christophe},
       title={Interval exchange maps and translation surfaces},
        date={2010},
   booktitle={Homogeneous flows, moduli spaces and arithmetic},
      series={Clay Math. Proc.},
      volume={10},
   publisher={Amer. Math. Soc., Providence, RI},
       pages={1\ndash 69},
      review={\MR{2648692}},
}

\bib{Zo:gau}{article}{
      author={Zorich, Anton},
       title={Finite {G}auss measure on the space of interval exchange
  transformation. {L}yapunov exponents},
        date={1996},
     journal={Ann. Inst. Fourier, Grenoble},
      volume={46},
       pages={325\ndash 370},
}

\bib{Zo:dev}{article}{
      author={Zorich, Anton},
       title={Deviation for interval exchange transformations},
        date={1997},
        ISSN={0143-3857},
     journal={Ergodic Theory Dynam. Systems},
      volume={17},
      number={6},
       pages={1477\ndash 1499},
         url={http://dx.doi.org/10.1017/S0143385797086215},
      review={\MR{1488330}},
}

\bib{Zo:how}{incollection}{
      author={Zorich, Anton},
       title={How do the leaves of a closed {$1$}-form wind around a surface?},
        date={1999},
   booktitle={Pseudoperiodic topology},
      series={Amer. Math. Soc. Transl. Ser. 2},
      volume={197},
   publisher={Amer. Math. Soc., Providence, RI},
       pages={135\ndash 178},
         url={https://doi.org/10.1090/trans2/197/05},
      review={\MR{1733872}},
}

\end{biblist}
\end{bibdiv}

\end{document}